\documentclass[english,reqno, 11pt]{amsart}
\usepackage{amssymb}
\usepackage{amsfonts}
\usepackage{amscd}
\usepackage[T1]{fontenc}
\usepackage{enumerate}%
\usepackage{amsthm,lipsum}
\usepackage[colorlinks, linkcolor=red]{hyperref}
\usepackage{color}
\usepackage[utf8]{inputenc}

\newcommand\red[1]{{\color{red}#1}}

\definecolor{immi}{rgb}{0,.6,.1}

\newbox\removebox
\newcommand\remove[2]{%
\setbox\removebox=\ifmmode\hbox{$#2$}\else\hbox{#2}\fi%
\leavevmode
\rlap{\textcolor{#1}{\vrule height0.8ex depth-0.6ex width\wd\removebox}}%
\box\removebox
}
\long\def\bigremove#1{%
\par\setbox\removebox=\vbox{#1}%
\vbox{%
\vbox to0pt{\hbox{\tikz\draw[color=blue,thick] (0,0) -- (\wd\removebox,-\ht\removebox)  (\wd\removebox,0) -- (0,-\ht\removebox);}}
\box\removebox
}
}

\usepackage{mathrsfs} 

\usepackage{geometry}
\def\moi{\operatorname{moi}}

\def\deg{\operatorname{deg}}
\def\ac{{\overline{\rm ac}}}
\def\acc{{{\rm ac}}}

\def\red{\operatorname{red}}
\def\SL{\operatorname{SL}}
\def\rk{\operatorname{rank}}

\def\SL{\operatorname{SL}}

\def\11{{\mathbf 1}}
\def\AA{{\mathbb A}}

\def\CC{{\mathbb C}}

\def\FF{{\mathbb F}}

\def\MM{{\mathbb M}}
\def\NN{{\mathbb N}}

\def\PP{{\mathbb P}}
\def\QQ{{\mathbb Q}}
\def\RR{{\mathbb R}}

\def\ZZ{{\mathbb Z}}

\def\cA{{\mathcal A}}
\def\cB{{\mathcal B}}
\def\cC{{\mathscr C}}
\def\cD{{\mathcal D}}
\def\cE{{\mathcal E}}

\def\cH{{\mathcal H}}
\def\cI{{\mathcal I}}
\def\cJ{{\mathcal J}}
\def\cK{{\mathcal K}}
\def\cL{{\mathcal L}}
\def\cM{{\mathcal M}}
\def\cN{{\mathcal N}}
\def\cO{{\mathcal O}}

\def\cQ{{\mathcal Q}}
\def\cR{{\mathcal R}}
\def\cS{{\mathcal S}}

\def\cU{{\mathcal U}}

\def\cZ{{\mathcal Z}}

\newtheorem{thm}[subsubsection]{Theorem}
\newtheorem{lem}[subsubsection]{Lemma}
\newtheorem{cor}[subsubsection]{Corollary}
\newtheorem{prop}[subsubsection]{Proposition}
\newtheorem{conj}
{Conjecture}

{Theorem}

\newtheorem{Qu}{Question}
\theoremstyle{definition}
\newtheorem{defn}[subsubsection]{Definition}

\newtheorem{example}[subsubsection]{Example}

\newtheorem{def-prop}[subsubsection]{Proposition-Definition}
\newtheorem{def-theorem}[subsubsection]{Theorem-Definition}
\newtheorem{def-lem}[subsubsection]{Lemma-Definition}

\theoremstyle{remark}
\newtheorem{remark}[subsubsection]{Remark}

\theoremstyle{plain}
\newtheorem*{namedthm}{\namedthmname}
\newcounter{namedthm}
\makeatletter
\newenvironment{named}[1]
  {\def\namedthmname{#1}%
   \refstepcounter{namedthm}%
   \namedthm\def\@currentlabel{#1}}
  {\endnamedthm}
\makeatother

\numberwithin{equation}{subsection}

  {\par\medskip\noindent #1\par\begingroup%
    \advance\leftskip by 1em\advance\rightskip by 1em}%
  {\par\endgroup}

\DeclareMathOperator*{\rank}{rank}

\DeclareMathOperator*{\Spec}{Spec}
\DeclareMathOperator*{\Specm}{Specmax}

\newcommand{\ord}{\operatorname{ord}}

\newcommand{\Sm}{\operatorname{Sym}}
\newcommand{\Sing}{\operatorname{Sing}}
\newcommand{\lct}{\operatorname{lct}}

\newcommand{\codim}{\operatorname{codim}}

\begin{document}
\author[K.~H.~Nguyen]
{Kien Huu Nguyen}

\address{KU Leuven, Department of Mathematics,
Celestijnenlaan 200B, B-3001 Leu\-ven, Bel\-gium}
\email{kien.nguyenhuu@kuleuven.be}

\address{Institute of Mathematics, Vietnam Academy of Science and Technology, 18 Hoang Quoc Viet, Nghia Do, Hanoi, Vietnam}
\email{nhkien@math.ac.vn}

\title[]{Uniform estimate of Lang-Weil type for counting points of schemes over finite rings}


\begin{abstract}

In this paper, we study the number of points of $\ZZ$-schemes $X$ of finite type over the finite rings $\ZZ/p^m\ZZ$ for positive integers $m$ and prime numbers $p$. Motivated by Igusa's conjecture for exponential sums and the Lang-Weil estimate, we establish an estimate of $X(\ZZ/p^m\ZZ)$ in terms of $m$, $X(\ZZ/p^{m-1}\ZZ)$ and poles of the $p$-adic Igusa local zeta function of $X$ uniformly in $m\geq 2$ and $p$. Our estimate might be viewed as an inductive version for counting points of schemes over finite rings of the Lang-Weil estimate. Our estimate agrees with the expected bound of Igusa's conjecture for exponential sums if $X_\QQ$ is not a locally  complete intersection having only rational singularities. This can be regarded as an analogous version of the estimate of exponential sums modulo powers of primes of polynomials with non-rational singularities given by Cluckers, Musta\c{t}\v{a} and the author in 2019. In the case that $X_\QQ$ is a locally  complete intersection having only rational singularities, our estimate  is very close to the expected bound of Igusa's conjecture for exponential sums. To prove our estimate, beside using the usual tools such as resolutions of singularities and Denef's formula for Igusa local zeta functions, we also use the Hermit interpolation formula to study the numerators of Igusa local zeta functions. By the same method, we are able to approach Igusa's conjecture for exponential sums in the case of rational singularities by establishing a uniform bound of exponential sums modulo powers of primes of polynomials in terms of their motivic oscillation indexes. Lastly, to have further applications of our result in the future, we study heavily on $p$-adic analysis to give an upper bound of the real parts of non-trivial poles of Igusa local zeta functions in terms of suitable generalized log canonical thresholds. 
\end{abstract}

\setcounter{tocdepth}{1}

\maketitle 
{\tableofcontents}

\section{Introduction} 
Let $X$ be a scheme of finite type over $\ZZ$. Let $p$ be a prime and $l$ be a positive integer, we denote by $N_0(\FF_{p^l},X)$  the number of geometrically irreducible components of dimension $n_p(X)=\dim(X_{\FF_p})$ of $X_{\FF_{p^l}}$. Generic flatness implies that  $n_p(X)=\dim(X_\QQ)$ if $p$ is large enough. Moreover, it follows from \cite[Part 3, Theorem 9.7.7]{EGA4} that if $X_\QQ$ is geometrically irreducible then $N_0(\FF_{p^l},X)=1$ provided that $p$ is large enough. The usual Lang-Weil estimate in \cite{LWeil} asserted that  
\begin{equation}\label{LWeil}
\left|\#X(\FF_{p^l})-N_0(\FF_{p^l},X)p^{l n_p(X)}\right|\leq C_Xp^{l(n_p(X)-\frac{1}{2})}
\end{equation}
for a constant $C_X$ depending only on the complexity of $X$. Here, we say that the complexity of $X$ is bounded by $D$ if $X$ can be covered by $D$ affine open sets that can be embedded as closed subsets of $\AA_{\ZZ}^D$ defined by $D$ polynomials of degree at most $D$.   Essentially,  the Lang-Weil estimate can be explained by the Riemann hypothesis over finite fields (the Weil conjectures) (see, e.g., \cite{DeligneWI,DeligneWII}). More precisely, we consider the Hasse-Weil zeta function of $X$ over $\FF_p$:
$$\zeta_{p,X}(t):=\exp\left(\sum_{l\geq 1}\frac{\#X(\FF_{p^l})}{l}t^l\right).$$
Roughly speaking, the Weil conjectures say that if $X_{\FF_p}$ is smooth and projective then $\zeta_{p,X}(t)$ is a rational function of $t$, i.e., 
$$\zeta_{p,X}(t)=\frac{P_p(t)}{Q_p(t)}$$
for monic polynomials $P_p(t),Q_p(t)\in\ZZ[T]$, such that $(P_p(t),Q_p(t))=1$ and if $\alpha$ is a root of $Q_p(t)$ (resp. $P_p(t)$) then all conjugates of $\alpha$ in $\CC$ have absolute value of the form $p^{\frac{w}{2}}$ for a natural  number $w\leq 2n_p(X)$ (resp. $w\leq 2n_p(X)-1$). Deligne proved the part of the Weil conjectures  related to the absolute value of roots of $P_p(t)Q_p(t)$ in \cite{DeligneWI} and generalized his result by removing both the smoothness and the projectivity of $X_{\FF_p}$ in \cite{DeligneWII}. In fact, if $P_p(\alpha)Q_p(\alpha)=0$ then there is $0\leq i\leq 2n_p(X)$ such that $\alpha$ is an eigenvalue of the Frobenius action on the $i^{\textnormal{th}}$ $\ell$-adic cohomology group with compact support $H_c^i(X_{\overline{\FF}_p},\QQ_\ell)$. It follows from the Weil conjectures that
\begin{equation}\label{exproot}\#X(\FF_{p^l})=\sum_{\alpha, Q(\alpha)=0}\mathfrak{m}_\alpha\alpha^l-\sum_{\beta, P(\beta)=0}\mathfrak{m}_\beta\beta^l
\end{equation}
for all $l\geq 1$ and 
\begin{equation}\label{logaform}\mathfrak{L}_{p,X}(t):=\sum_{l\geq 1}\#X(\FF_{p^l})t^l=\sum_{\alpha, Q(\alpha)=0}\frac{\mathfrak{m}_\alpha\alpha t}{1-\alpha t}-\sum_{\beta, P(\beta)=0}\frac{\mathfrak{m}_\beta\beta t}{1-\beta t}\in\CC[[t]],
\end{equation} where $\mathfrak{m}_\lambda$ is the multiplicity of $t-\lambda$ in $P_p(t)Q_p(t)$.

Moreover, it is well known  (see, e.g., \cite{DeligneWII}) that $\deg_t(P_p(t)Q_p(t))$ is at most $\cB(X,p)$ which is the sum of the Betti numbers of $X_{\overline{\FF}_p}$, i.e., 
$$\cB(X,p)=\sum_{0\leq i\leq 2n_p(X)}\dim_{\QQ_\ell}H_c^i(X_{\overline{\FF}_p},\QQ_\ell).$$
Note that $\cB(X,p)$ can be bounded by the complexity of $X$ (see, e.g.,  \cite{Katz-Bet}), thus (\ref{LWeil}) follows from (\ref{exproot}) with $$N_0(\FF_{p^l},X)=\sum_{\alpha, Q(\alpha)=0, \alpha^l=p^{l n_p(X)}}\mathfrak{m}_\alpha.$$

Now, it is very natural to ask for an analogous version of (\ref{LWeil}) for uniform counting points of $X$ over finite rings $R$ if $R=\ZZ/p^m\ZZ$ or  $R=\FF_p[t]/(t^m)$  for primes $p$ and positive integers $m$.  More generally,  we also ask this question if $R=\cO_L/\cM_L^m$ for $m\geq 1$, where $L$ is a non-Archimedean local field (i.e., a  finite extension of $\QQ_p$ or $\FF_p((t))$), $\cO_L$ is  the ring of integers of $L$,  and  $\cM_L$ is the maximal ideal of $\cO_L$. 

In the first approach, we can use the method of Meuser in \cite{Meuser} related to Teichmüller representatives of such an $L$ to show that for each $m\geq 1$ there is a $k_L$-scheme $X^{m,L}$ of finite type with $k_L:=\cO_L/\cM_L$ such that 
$$\#X(\cO_{L^{(l)}}/\cM_{L^{(l)}}^m)=\#X^{m,L}(k_{L^{(l)}})$$
for all $l\geq 1$, where $L^{(l)}$ is the unramified extension of degree $l$ of $L$. If $L$ is of large enough residue field characteristic, one can take $X^{m,L}$ to be $X_{m-1}\otimes_{\ZZ} k_L$, where $X_{m-1}$ is the $(m-1)^{\textnormal{th}}$ jet scheme associated to $X$ (see \cite[Chapter 3, Section 2]{CNic}). When $X_\QQ$ is a locally complete intersection of pure dimension having only rational singularities (see its definition in Subsection \ref{JSL}), we can control well $\dim(X_m\otimes_\ZZ\QQ)$ and $N_0(\FF_{p^l},X_m)$ if $p$ is large enough (see \cite{Mustata1}). If this is the case,  Aizenbud and Avni  gave a version of (\ref{LWeil}) in \cite[Theorem A]{AizenAvni}. Namely, by putting $q_L:=\#k_L$, one has
\begin{equation}\label{LWmod0}\left|X(\cO_L/\cM_L^m)-q_L^{(m-1)\dim(X_\QQ)}\#X(k_L)\right|\leq Cq_L^{m\dim(X_\QQ)-1} 
\end{equation}
and 
\begin{equation}\label{LWmod}
\left|X(\cO_L/\cM_L^m)-q_L^{m\dim(X_\QQ)}\right|\leq Cq_L^{m\dim(X_\QQ)-\frac{1}{2}} 
\end{equation}
for all $m\geq 1$, all non-Archimedean local fields $L$ of large enough residue field characteristic and a constant $C$ independent of $L,m$ provided that $X_\QQ$ is geometrically irreducible and a locally complete intersection having only rational singularities. However, when $X_\QQ$ is not a locally complete intersection or has non-rational singularities then it is very difficult  to compute exactly the value of $\dim(X_m\otimes_\ZZ\QQ)$ and $N_0(\FF_{p^l},X_m)$ for all $m$. If this is the case, we just know exactly the value of $\dim(X_m\otimes_\ZZ\QQ)$ for some special integers $m$ (see \cite{MustJAMS}). 

In another approach, we can study the above question by using Igusa local zeta functions. The idea is that instead of considering the Hasse-Weil zeta function
$$\zeta_{L,X,m}(t):=\exp\left(\sum_{l\geq 1}\frac{\#X^{m,L}(k_{L^{(l)}})}{l}t^l\right)=\exp\left(\sum_{l\geq 1}\frac{\#X(\cO_{L^{(l)}}/\cM_{L^{(l)}}^m)}{l}t^l\right)$$
for each non-Archimedean local field $L$  and each positive integer $m$, we consider the Igusa local zeta function of $X$ associated to each non-Archimedean local field $L$:
\begin{equation}\label{Izeta} \cZ_{L,X}(t):=\sum_{m\geq 1}\#X(\cO_L/\cM_L^m)t^{m}.
\end{equation}
In Subsection \ref{exp-zetafunction}, we will redefine  $\cZ_{L,X}(t)$ by using $q_L^{-s}$ instead of $t$ to be more compatible with integrals over non-Archimedean local fields. Historically, $\cZ_{L,X}(t)$ was firstly studied by Borevich and Shafarevich in \cite{Bosha} when $X$ is defined by a non-constant polynomial in $\ZZ[x_1,...,x_n]$. They conjectured that $\cZ_{L,X}(t)$ is a rational function of $t$ if $L$ is of characteristic zero. Igusa proved this conjecture in \cite{Igusa1} and Meuser generalized Igusa's result to the case that $X$ is defined by more than one non-constant polynomial in $\ZZ[x_1,....,x_n]$ (see \cite{Meura}). Another proof for the rationality of $\cZ_{L,X}(t)$ was given by Denef in \cite{Denef}. The rationality of  $\cZ_{L,X}(t)$ if $L$ is of positive characteristic is still an open question due to the lack of resolution of singularities in this case. However, it is well known that there is an integer $M$ depending on $X$ such that if $L,\tilde{L}$ are a non-Archimedean local fields of residue field characteristic at least $M$ then  $\cZ_{L,X}(t)=\cZ_{\tilde{L},X}(t)$ provided that $k_L\simeq k_{\tilde{L}}$ (see Corollary \ref{transfer2}). Thus, we have the rationality of $\cZ_{L,X}(t)$ if $L$ is a non-Archimedean local field of large enough residue field characteristic.  Note that the function $\mathfrak{L}_{p,X}(t)$ in (\ref{logaform}) is of similar form to $\cZ_{L,X}(t)$. Thus, by the same phenomenon as with $\mathfrak{L}_{p,X}(t)$,  we see that poles of  $\cZ_{L,X}(t)$ will give us the most important information of $\#X(\cO_L/\cM_L^m)$ for all $m$. Now, we suppose that $L$ is a non-Archimedean local field of characteristic zero.  Let us write 
\begin{equation}\label{fraction}\cZ_{L,X}(t)=G_{L,X}(t)+\sum_{\alpha\in A_{L,X}}\frac{H_{L,\alpha}(t)}{(1-\alpha t)^{\mathfrak{e}_\alpha}}
\end{equation}
for a finite set $A_{L,X}\subset \CC$,  a polynomial $G_{L,X}(t)\in\CC[t]$ together with a positive integer $\mathfrak{e}_\alpha$ and a non-zero polynomial $H_{L,\alpha}(t)\in\CC[t]$ of degree at most $\mathfrak{e}_\alpha-1$ for each $\alpha\in A_{L,X}$. It is clear that (\ref{fraction}) is much more complicated than (\ref{logaform}). We have not yet known too much about $G_{L,X}(t)$ and $H_{L,\alpha}(t)$. We just know that $\deg_t(G_{L,X}(t))\leq 0$ if $L$ is of large enough residue field characteristic (see \cite{Denefdegree,VeysZ}). On the other hand,  we have a more conceptual way to understand the set $A_{L,X}$. More precisely,  the set $A_{L,X}$ depends on singularities of $X$. Let us fix $\alpha\in A_{L,X}$.  Then  one can show that there are positive integers $a,b$  such that $\alpha^{b}=q_L^{a}$. When $X$ is a closed subscheme of $\AA_\ZZ^n$ then such pair $(b,a)$ is a piece of numerical information come from any  strong log resolution of singularities of $X_\QQ$ in $\AA_\QQ^n$ (see \cite{Igusa3,VeysZ} or Subsection \ref{exp-zetafunction}).  More profoundly, if $X$ is defined by a non-constant polynomial $f\in\ZZ[x_1,...,x_n]$ then the monodromy conjecture of Igusa predicts that $\exp(\frac{2\pi ia}{b})$ is an eigenvalue of the monodromy action on some cohomology group of the Milnor fiber of $f$ around some point $x^{(0)}\in X(\CC)$ (see \cite{Igubf,DenefBour}). This conjecture is quite similar to the Weil conjectures, namely,  the Frobenius action on $\ell$-adic cohomology groups with compact support is replaced by the monodromy action on cohomology groups of Milnor fibers.  A stronger version of the monodromy conjecture (the strong monodromy conjecture) is stated that $\frac{a}{b}-n$ is a root of the Bernstein-Sato polynomial $b_f(s)$ of $f$ (see \cite{Igubf,DenefBour}). Where, $b_f(s)$ is the monic polynomial $b(s)$ of smallest degree such that there is $P\in\QQ[x_1,...,x_n,\frac{\partial f}{\partial x_1},...,\frac{\partial f}{\partial x_n},s]$ such that $Pf^{s+1}=b_f(s)f^s$ (see \cite{Bern}). Moreover, one also expects that $\frac{a}{b}-n$ is a root of multiplicity at most $\mathfrak{e}_\alpha$ of $b_f(s)$. If $X$ is an arbitrary $\ZZ$-scheme of finite type, we also have a version of the strong monodromy conjecture relating the poles of $\cZ_{L,X}(t)$ with the roots of the Bernstein-Sato polynomial of $X_\QQ$ defined in \cite{BMS}. When we already know about the set $A_{L,X}$,  by using (\ref{fraction}), one can show as in \cite{Loeasym,VeysZ} that
\begin{equation}\label{lims}
\limsup_m \#X(\cO_L/\cM_L^m)^{\frac{1}{m}}=\max_{\alpha\in A_{L,X}}|\alpha|.
\end{equation}
However, it is still far from a bound of Lang-Weil type for $\#X(\cO_L/\cM_L^m)$ uniformly in all $m$.

In this paper, based on \cite{MoiIde}, we give a new perspective on bounds of Lang-Weil type of $X(\cO_L/\cM_L^m)$ uniformly in positive integers $m\geq 2$ and all non-Archimedean local fields $L$ of large enough residue field characteristic. We will give more detail in Subsection \ref{exp-zetafunction}. Here, we sketch the main idea.  Suppose that $X$ is defined by an ideal of $\ZZ[x_1,...,x_n]$ generated by $r$ polynomials $f_1,...,f_r$. For each $m\geq 2$ and each non-Archimedean local field $L$, we can show that 
$$\#X(\cO_L/\cM_L^m)-q_L^{n-r}\#X(\cO_L/\cM_L^{m-1})=q_L^{mn}\int_{(x,y)\in\cO_L^n\times \left(\cO_L^{r}\setminus\cM_L\cO_L^r\right)}\psi(\sum_{1\leq i\leq r}y_if_i(x))|dx\wedge dy|,$$
where $\psi$ is any additive character of conductor $m$ of $L$ and $|dx\wedge dy|$ is the Haar measure on $L^{n+r}$ such that the volume of $\cO_L^{n+r}$ is $1$. Here, an additive character $\psi$ of conductor $m$ of $L$ is a continuous group homomorphism $\psi$ from $(L,+)$ to the multiplicative group $(\CC^{\times},\times)$  such that $\psi|_{\cM_L^m\cO_L}=1$ but $\psi|_{\cM_L^{m-1}\cO_L}\neq 1$. 
Recall that if $f$ is a non-constant polynomial in $\ZZ[x_1,...,x_l]$ and $Z$ be a $\ZZ$-subscheme of finite type of $\AA_\ZZ^n$ then for each non-Archimedean local field $L$ of  characteristic zero, there is  $\sigma_{L}(f,Z)\in\QQ_{>0}\cup\{+\infty\}$ which is the supremum of  all real numbers $\sigma$ satisfying
$$\left|E_{f,L,Z,\psi}\right|:=\left|\int_{\{x\in\cO_L^l\mid (x\mod \cM_L)\in Z(k_L)\}}\psi(f(x))|dx|\right|\leq q_L^{-m\sigma}$$
for all large enough integers $m$ and all additive characters $\psi$ of $L$ of conductor $m$. We call $\sigma_L(f,Z)$ the $L$-oscillation index of $f$ at $Z$. In fact, $-\sigma_L(f,Z)$ is  the real part of a non-trivial pole of some Igusa local zeta function associated to $f$ and $Z$ (see Subsection \ref{exp-zetafunction}). If the strong monodromy conjecture is true then $\sigma_L(f,Z)$ is larger than the minimal exponent $\tilde{\alpha}_f$ of $f$ (the minimal root of  $\frac{b_f(-s)}{1-s}$).  Igusa's conjecture for exponential sums (see Subsection \ref{exp-zetafunction}) predicts that for each real number $\sigma$ satisfying $\sigma\leq\sigma_{L}(f,Z)$ for all non-Archimedean local fields $L\supset \QQ$ of large enough residue field characteristic then 
$$\left|E_{f,L,Z,\psi}\right|\leq C(f,\sigma) m^{l-1}q_L^{-m\sigma}$$
for all $m\geq 2$, all non-Archimedean local fields of large enough residue field characteristic, all additive characters $\psi$ of $L$ of conductor $m$ and a constant $C(f,\sigma)$ depending only on $f,\sigma$. Moreover, one may also expect to replace the factor $m^{l-1}$ above by $m^{\max\{\mathfrak{m}_\sigma,1\}-1}$, where $\mathfrak{m}_\sigma$ is the maximum  of all natural numbers $e$ such that $e$ is the multiplicity of a non-trivial pole with real part $-\sigma$ of an Igusa local zeta function associated to $f,L,Z$ mentioned in Subsection \ref{exp-zetafunction}.  
When $f=\sum_{1\leq i\leq r}y_if_i(x)$ and $Z=\AA_\ZZ^n\times(\AA_\ZZ^r\setminus\Spec(\ZZ[x_1,...,x_r]/(x_1,...,x_r)))$ then we can show that 
\begin{equation}\label{polsig}q_L^{n-\sigma_L(f,Z)}=\max_{\alpha\in B_{L,X,n-r}}|\alpha|
\end{equation}
for all non-Archimedean local fields $L$ of characteristic zero, where 
$$B_{L,X,n-r}=\{\alpha\mid \alpha^{-1} \textnormal{ is a pole of } Z_{L,X}(t)(1-tq_L^{n-r})\}.$$
Thus, by Igusa's conjecture for exponential sums, we may expect that 
\begin{equation}\label{Iguexpt}\left|\#X(\cO_L/\cM_L^m)-q_L^{n-r}\#X(\cO_L/\cM_L^{m-1})\right|\leq C_{X,n-r}\max_{\alpha\in B_{L,X,n-r}}m^{n-1}|\alpha|^m
\end{equation}
for all $m\geq 2$, all non-Archimedean local fields $L$ of large enough residue field characteristic and a positive constant $C_{X,n-r}$ depending only on $X,n,r$.  In the general case, when $X$ is not an affine $\ZZ$-scheme, we may also expect that (\ref{Iguexpt}) holds if we replace  $n-r$ by an arbitrary real number $l$ and $m^{n-1}$ by $m^{\mathfrak{e}_{\alpha,l}-1}$, where $\mathfrak{e}_{\alpha,l}$ is the multiplicity of the pole $\alpha^{-1}$ of $Z_{L,X}(t)(1-tq_L^{l})$ if $$\alpha\in B_{L,X,l}:=\{\alpha\mid \alpha^{-1} \textnormal{ is a pole of } Z_{L,X}(t)(1-tq_L^{l})\}.$$  If this is true and $L$ is a finite extension of $\QQ_p$ for some large enough prime $p$  then we can relate $\zeta_{L,X,m}(t)$ and $\zeta_{L,X,m-1}(t)$ for all $m\geq 2$ and get a lot of information about $\#X(\cO_L/\cM_L^m)=\#X_m(k_L)$ inductively such as the numbers $N_0(k_L,X_m)$, $n_{p}(X_m)$ and $\cB(X_m,p)$. If this is the case, by using (\ref{LWeil}) and (\ref{Iguexpt}), we can also reprove (\ref{LWmod}) and (\ref{lims}). Thus, it suggests that an analogous version of the Lang-Weil estimate for counting points of $X$ over finite rings may be of the same form as  (\ref{Iguexpt}). Now, we state the first main result of this paper. 

\begin{named}{Theorem A}\label{uupp00} Let $X$ be a $\ZZ$-scheme of finite type. There is an integer $M$ depending only on $X$, a constant $C_X$ depending  only on the complexity of $X$, an integer $e_{X,l}$ for each real number $l$, a constant $C_{X,L}$ and an integer $e_{X,l,L}$ for each $l\in \RR$ and each non-Archimedean local field $L$ of characteristic zero such that:
\begin{itemize}
\item $C_{X,L}=C_X$ and $e_{X,l,L}=e_{X,l}$ if $L$ is of residue field characteristic at least $M$.
\item  For all $l\in\RR$, all $m\geq 2$ and all non-Archimedean local fields $L$ of  characteristic zero, we have 
\begin{equation}\label{boun000}\left|\#X(\cO_L/\cM_L^m)-q_L^{l}\#X(\cO_L/\cM_L^{m-1})\right|\leq C_{X,L}\max_{\alpha\in B_{L,X,l}}m^{\mathfrak{e}_{\alpha,l}-1}(1+q_L^{l}|\alpha|^{-1})|\alpha|^{m}.
\end{equation}
\item For each $l\in \RR$ and each non-Archimedean local field $L$ of characteristic zero, the inequality
$$\left|\#X(\cO_L/\cM_L^m)-q_L^{l}\#X(\cO_L/\cM_L^{m-1})\right|\geq q_L^{e_{X,l,L}}\max_{\alpha\in B_{L,X,l}}m^{\mathfrak{e}_{\alpha,l}-1}|\alpha|^{m}$$
holds for infinitely many integers $m\geq 2$. 
\end{itemize}
\end{named}
We will generalize \ref{uupp00} and prove it in Propositions \ref{lowe}, \ref{ulowbou} and Theorem \ref{uupp}. The ingredients to prove \ref{uupp00} are  Denef's formula for Igusa local zeta functions and the Hermit interpolation formula (see Subsections \ref{exp-zetafunction} and  \ref{Herf}). Suppose that $X_\QQ$ is equi-dimensional and a closed subscheme of the affine space $\AA_\QQ^n$. If either $X_\QQ$ is not a locally complete intersection  or $X_\QQ$ has non-rational singularities then we can know how to find $\max_{\alpha\in B_{L,X,l}}|\alpha|$ in any strong log resolution of $X_\QQ$ in $\AA_\QQ^n$ (see Subsection \ref{exp-zetafunction}), thus we can use a similar calculation to \cite[Section 4]{Saskia-Kien} related to Denef's formula to deduce \ref{uupp00}. Otherwise, the problem becomes much  harder.  Fortunately, we can overcome this problem by using Denef's formula and the Hermit interpolation formula to estimate the coefficients of the polynomials $H_{L,\alpha}(t)$ appearing in (\ref{fraction}) uniformly in $L$. The upper bound in \ref{uupp00} is as expected of Igusa's conjecture for exponential sums in almost all cases except for the case that $X_\QQ$ is locally a complete intersection of dimension $l$ having only rational singularities where we want to remove the factor $(1+q_L^{l}|\alpha|^{-1})$. This is quite similar to the fact that the Lang-Weil estimate (\ref{LWeil}) can be improved if $X_\QQ$ is a complete intersection in a projective space and the singularities of $X$ are good enough (see, e.g., \cite[Theorem 6.1]{GhLa}).  On the other hand, motivated by the work in \cite{CluckerNguyen,MoiIde}, if $X$ is defined by polynomials $f_1,...,f_r\in\ZZ[x_1,...,x_n]$ of degree $d>0$ then we have a weak version of (\ref{boun000}) in a similar form to (\ref{Iguexpt}) but with all $m\geq 1$ as follows.
\begin{cor}\label{nsdI}Let $X$ be the closed subscheme of $\AA_\ZZ^n$ defined by polynomials $f_1,...,f_r$ of degree $d>0$. Suppose that $\dim(X_\QQ)=n-r$. For each $1\leq i\leq r$, we denote by $f_{id}$ the homogeneous part of degree $d$ of $f_i$. We set $s(f_1,...,f_r)=\dim\{z\in\CC^n\mid \rk((\frac{\partial f_i}{\partial x_j}(z))_{1\leq i\leq r,1\leq j\leq n})<r\}.$ Then, there is an integer $M$ depending only on $X$ and a positive constant $C_{n,d}$ depending only on $n,d$ such that   
$$\left|\#X(\cO_L/\cM_L^m)-q_L^{n-r}\#X(\cO_L/\cM_L^{m-1})\right|\leq C_{n,d}m^{n-1}\max\Biggl\{q_L^{\frac{n-r}{d+2}}\left(\max_{\alpha\in B_{L,X,n-r}}|\alpha|\right)^{\frac{d+1}{d+2}}, q_L^{n-\frac{n-s(f_1,...,f_r)}{d}}\Biggr\}^{m}$$
for all $m\geq 1$ and all non-Archimedean local fields $L$ of residue field characteristic at least $M$.
\end{cor}
Moreover, the following corollary  of \ref{uupp00}  improves (\ref{LWmod0}):

\begin{cor}\label{coutp}Let $X$ be a $\ZZ$-scheme of finite type such that $X_\QQ$ is geometrically irreducible, locally a complete intersection of dimension $l$ having only rational singularities. Then there is an integer $M$ such that $\beta_{L,X,l}:=\max\{|\alpha|\mid \alpha\in B_{L,X,l}\}<q_L^l$ if $L$ is a non-Archimedean local field  of residue field characteristic at least $M$. Moreover, there is a constant $C_X$ depending only the complexity of $X$ such that 
$$\left|X(\cO_L/\cM_L^m)-q_L^{(m-1)l}\#X(k_L)\right|\leq C_Xq_L^{ml-\lceil l-\log_{q_L}(|\beta_{L,X,l}|)\rceil}$$
for all non-Archimedean local fields $L$ of residue field characteristic at least $M$, where if $a$ is a real number then $\lceil a\rceil=\min\{i\in \ZZ\mid i\geq a\}$.
\end{cor}

Notice that Igusa's conjecture for exponential sums in the case of non-rational singularities  was proved in \cite{CMN}. In Section \ref{flaw},  by using the above method related to the Hermit interpolation formula, we can approach Igusa's conjecture for exponential sums in the case of rational singularities as follows.
\begin{named}{Theorem B}\label{Iguapp}Let $f\in\ZZ[x_1,...,x_n]$ be a non-constant polynomial and $Z$ be an $\ZZ$-subscheme of finite type of $\AA_\ZZ^n$.  Suppose that the hypersurface defined by $f(x)-f(a)$ has only rational singularities in a neighbourhood of $a$ for all $a\in Z(\overline{\QQ})$. Then there is an integer $M$ (depending only on $f,Z$),  a constant $C$ (depending only on $n,\deg(f)$), and a constant $C_L$ for each non-Archimedean local field $L$ of characteristic zero such that $C_L=C$ if $L$ is of residue field characteristic at least $M$, and 
\begin{equation}\label{mabounex}\left|E_{f,L,Z,\psi}\right|\leq C_Lm^{n-1}q_L^{-m\sigma_L(f,Z)+\sigma_L(f,Z)-\frac{1}{2}}
\end{equation}
for all non-Archimedean local fields $L$ of characteristic zero and all additive characters $\psi$ of $L$ of conductor $m\geq 2$. 
\end{named} 
 \ref{Iguapp} will be generalized to Theorem \ref{main}. Similar to Corollary \ref{nsdI}, we give the following corollary of \ref{Iguapp} to approach the conjecture stated in \cite[Conjecture 1]{CluckerNguyen}. 
\begin{cor}\label{Igcoexp}
Let $f\in\ZZ[x_1,...,x_n]$ be a non-constant polynomial of degree $d\geq 1$ and $Z$ be a $\ZZ$-subscheme of finite type of $\AA_\ZZ^n$. Let $s(f)$ be as in Corollary \ref{nsdI}. For each non-Archimedean local field $L$ of characteristic zero, we set $$\beta_\QQ(f,Z):=\min\biggl\{\frac{d+1}{d+2}\moi_\QQ(f,Z)+\frac{1}{2(d+2)},\frac{n-s(f)}{d}\biggr\},$$
where $$\moi_\QQ(f,Z):=\lim_{\ell\to +\infty}\inf\{\sigma_L(f,Z)\mid \exists p>\ell, L \textnormal{ is a finite extension of } \QQ_p\}.$$ There is an integer $M$ depending only on $f,Z$ and a constant $C_{n,d}$ depending only on $n,d$ such that 
\[\left|E_{f,L,Z,\psi}\right|\leq C_{n,d}m^{n-1}q_L^{-m\beta_\QQ(f,Z)}\] 
for all non-Archimedean local fields $L$ of residue field characteristic at least $M$  and all additive characters $\psi$ of $L$ of conductor $m\geq 1$. In other words, Igusa's conjecture for exponential sums holds for $f,Z$ and all $\sigma\leq \beta_\QQ(f,Z)$. 
\end{cor}

As in \cite{Loeasym}, in order to use \ref{uupp00} and \ref{Iguapp} together with  their corollaries, we need to estimate the $L$-oscillation indexes $\sigma_L(f,Z)$ for non-constant polynomials $f\in\ZZ[x_1,...,x_n]$, $\ZZ$-subschemes $Z$ of finite type of $\AA_\ZZ^n$ and  non-Archimedean local fields $L$ of characteristic zero. Let us fix such $f, Z$.   As mentioned above, if  $L$ is a non-Archimedean local field of characteristic zero then the strong monodromy conjecture implies that 
$\sigma_{L}(f,Z)$ is larger than the minimal exponent $\tilde{\alpha}_f$ of $f$. However, the strong monodromy conjecture  is still open in the general case. 
Thus, we may try to find another way to give lower bounds of $\sigma_{L}(f,Z)$. It follows from \cite{Igusa3} that 
\begin{equation}\label{lowlct}\sigma_{L}(f,Z)\geq \min_{a\in Z(\overline{\QQ})}\lct_a(f-f(a))
\end{equation}
 for all non-Archimedean local fields $L$ of characteristic zero (see Subsection \ref{exp-zetafunction}), where $\lct_a(f-f(a))$ is the log canonical threshold of $f-f(a)$ at $a$ (see Subsection \ref{JSL} for the definition of log canonical thresholds). Moreover, the equality in (\ref{lowlct}) holds for infinitely many such local fields $L$ if $f-f(a)$ defines a hypersurface in $\AA_{\overline{\QQ}}^n$  having non-rational singularities in any neighbourhood of $a$ for some $a\in Z(\overline{\QQ})$ (see \cite[Proposition 3.10]{CMN}). In \cite[Proposition 5.2]{CKMu}, this lower bound of $\sigma_{L}(f,Z)$ was improved  quite easily  to
\begin{equation}\label{lolctja}\sigma_{L}(f,Z)\geq \min_{a\in Z(\overline{\QQ})}\lct_a((f-f(a))+J_f^2),
\end{equation}
if $L\supset\QQ$ is a non-Archimedean local field of large enough residue field characteristic, where $J_f$ is the Jacobian ideal of $f$ generated by $\frac{\partial f}{\partial x_1},...,\frac{\partial f}{\partial x_n}$ and $\lct_a((f-f(a))+J_f^2)$ is the log canonical threshold of the ideal $(f-f(a))+J_f^2$. In Section \ref{lobmoti}, we give an improvement of \cite[Theorem 6.1 and Proposition 6.4]{CKMu}. We sketch the idea and result here. To simplify the argument, we will suppose that $f(Z(\overline{\QQ}))=0$. In \cite[Proposition 6.4]{CKMu}, for each $m\geq 2$ and each non-Archimedean local field $L$ of large enough residue field characteristic, the authors constructed a canonical cylinder $A_L(Z,f,m)\subset L^n$  depending on $L,Z,f,m$  such that  for all additive characters $\psi$ of $L$ of conductor $m$, the integral of the function $\psi(f(x))$ on $A_L(Z,f,m)$ is zero. Then we can calculate easily the volume of the complement $A'_L(Z,f,m)$ of $A_L(Z,f,m)$ in $\{x\in\cO_L^n\mid (x\mod \cM_L)\in Z(k_L)\}$ to give a common upper bound of $\left|E_{f,L,Z,\psi}\right|$ for all additive characters $\psi$ of $L$ of conductor $m$ and obtain (\ref{lolctja}). A natural idea is to replace the set $A_L(Z,f,m)$ by another set of larger volume. In this paper, for each  non-Archimedean local field $L$ of large enough residue field characteristic, each point $y\in\cO_L^n$ such that $(y\mod \cM_L)\in Z(k_L)$ and each $m\geq 2$, by using some special properties of matrices over local rings and the Hessian matrix $\cH_f(y)$ of $f$ at $y$,  we can construct a cylinder $\cA_L(y,Z,f,m)$  depending on $L,y,Z,f,m$ such that the integral of the function $\psi(f(x))$ on $\cA_L(y,Z,f,m)$ is zero for all additive characters $\psi$ of $L$ of conductor $m$. The main problem is that the set $\cA_L(y,Z,f,m)$ depends on $y$. Fortunately, we can show that there is $\mathsf{r}_m\in (0,1)$ depending on $m$ such that  $\cA_L(y,Z,f,m)$ is contained in the ball of radius $\mathsf{r}_m$ centered at $y$ for each $y\in \cO_L^n$. Moreover, the volume of $\cA_L(x,Z,f,m)$ viewed as a function in $x$  is constant on each $n$-dimensional closed  ball $B(y,\mathsf{r}_m)\subset \cO_L^n$. Our idea is to write  $\{x\in\cO_L^n\mid (x\mod \cM_L)\in Z(k_L)\}$ as a disjoint union of balls $B_{j,m}$ of radius $\mathsf{r}_m$ for $1\leq j\leq e(m)$ and choose a point $x^{(j)}\in B_{j,m}$ for each $j$. Then, the integral of the function $\psi(f(x))$ on $\cA_L(Z,f,m)=\cup_{1\leq j\leq e(m)}\cA_L(x^{(j)},Z,f,m)$ is zero for all additive characters $\psi$ of $L$ of conductor $m$. Thus,  for all additive characters $\psi$ of $L$ of conductor $m$, $\left|E_{f,L,Z,\psi}\right|$ is less than the volume of the set $$\cA'_L(Z,f,m):=\{x\in\cO_L^n\mid (x\mod \cM_L)\in Z(k_L)\}\setminus \cA_L(Z,f,m)=\sqcup_{1\leq j\leq e(m)}\cA'_L(x^{(j)},Z,f,m)$$ with $\cA'_L(x^{(j)},Z,f,m):=B_{j,m}\setminus \cA_L(x^{(j)},Z,f,m)$. To compute the volume of $\cA'_L(Z,f,m)$,  for each $j$, we write $\cA'_L(x^{(j)},Z,f,m)$ as a disjoint union of cylinders $C_{jl}$ of volume depending only on $m$ and the Hessian matrix of $f$ at $x^{(j)}$ for $1\leq l\leq d_j$. Moreover, for each $j$ and $l$, we give a version of Hensel's lemma for some suitable tuples of polynomials and use it to construct a cylinder $\tilde{C}_{jl}\subset B_{j,m}$ such that the volume of $\tilde{C}_{jl}$ differs from the volume of $C_{jl}$ by a factor depending only on $m$ and the Hessian matrix of $f$ at $x^{(j)}$. In addition, we can show that  $\tilde{C}_{jl}$ for $1\leq j\leq e(m)$ and $1\leq l\leq d_j$  are pairwise disjoint sets and  $\tilde{\cC}_L(Z,f,m):=\cup_{1\leq j\leq e(m), 1\leq l\leq d_j}\tilde{C}_{jl}$ is a canonical set depending only on $L, m, Z, f$ and the Jacobian ideal of $f$. Furthermore, we can show that the volumes of $C_{jl}$ and $\tilde{C}_{jl}$  are independent of the choice of $x^{(j)}$.   In summarize, we can calculate the volume of $\cA'_L(Z,f,m)$ by an integral of a function in terms of $m$ and the Hessian matrix of $f$ on $\tilde{\cC}_L(Z,f,m)$. It yields a lower bound of $\sigma_L(f,Z)$. Moreover, by using this lower bound of $\sigma_L(f,Z)$ and \ref{Iguapp}, we also obtain a uniform  upper bound of $\left|E_{f,L,Z,\psi}\right|$. The exact statement is as follows.

\begin{named}{Theorem C}\label{lobsig}Let $f\in\ZZ[x_1,...,x_n]$ be a non-constant polynomial. For each $c\in \overline{\QQ}$, we  denote by $\cI_{f-c}$ the ideal of $\QQ[c,x_1,...,x_n]$ generated by $(f-c)^2$ and $(\frac{\partial f}{\partial x_i})^3$ for $1\leq i\leq n$. 
For each $0\leq l\leq n$, we denote by $\tilde{\mathsf{D}}_{f,l}$ the ideal of $\QQ[x_1,...,x_n]$ generated by all principal minors of order $l$ of  the Hessian matrix $\cH_{f}(x)$ with the convention that $\tilde{\mathsf{D}}_{f,0}=(1)$. We set  $$\sigma_{f,Z}:=\min_{a\in Z(\overline{\QQ})}\min_{0\leq l\leq n}\left(2\lct_{a}\left(\cI_{f-f(a)},\tilde{\mathsf{D}}_{f,l},\frac{1}{2}\right)-\frac{l}{6}\right),$$
where $\lct_{a}\left(\cI_{f-f(a)},\tilde{\mathsf{D}}_{f,l},\frac{1}{2}\right)$ is the generalized log canonical threshold of the triple $(\cI_{f-f(a)},\tilde{\mathsf{D}}_{f,l},\frac{1}{2})$ defined in Subsection \ref{JSL}. There is an integer $M$ depending only on $f,Z$ such that $\sigma_{L}(f,Z)\geq \sigma_{f,Z}$ for all non-Archimedean local fields $L\supset\QQ$ of residue field characteristic at least $M$. Consequently, there is an integer $M$ depending only on $f,Z$ and a positive constant $C_{n,d}$ depending only on $n,d$ such that
$$|E_{f,L,Z,\psi}|\leq C_{n,d}m^{n-1}q_L^{-(m-1)\sigma_{f,Z}-\frac{1}{2}}$$
for all $m\geq 2$, all non-Archimedean local fields $L\supset\QQ$ of  residue field characteristic at least $M$ and all  additive characters $\psi$ of $L$ of conductor $m$. 
\end{named}
 
In fact, we can show that $\sigma_{f,Z}\geq \min_{a\in Z(\overline{\QQ})}\lct_a((f-f(a))+J_f^2)$ for all $f$ and $Z$ (see Proposition \ref{compaCKM}). Furthermore, by the definition of $\sigma_L(f,Z)$ in Subsection \ref{exp-zetafunction}, \ref{lobsig} will imply that the real part of any pole of any Igusa local zeta function associated to $f,Z,L$  is $-1$ or less than $-\sigma_{f,Z}$ if $L$ is a non-Archimedean  local field of large enough residue field characteristic. By (\ref{polsig}) and working with affine coverings, this also helps to give an upper bound of $\max_{\alpha\in B_{L,X,l}}|\alpha|$ in terms of geometric invariants of $X$ if $X$ is a scheme of finite type over $\ZZ$, $\ell\in\RR$ an $L$ is a non-Archimedean local field of large enough residue field characteristic. Thus, we can obtain a weaker version of (\ref{boun000}) (resp. (\ref{mabounex})) with the right-hand side depending on the geometry of $X$ (resp. $f$) provided that $L$ is a non-Archimedean local field of large enough residue field characteristic. We will present a stronger and more general statement of \ref{lobsig} in Section \ref{lobmoti}.

Note that Igusa proposed the conjecture for exponential sums as the first step to study the local-global principle for the existence of non-zero integral solutions  of homogeneous polynomials of degree $d\geq 3$ (see \cite{Igusa3}). In another approach to this local-global principle, the Hardy-Littlewood circle method splits the problem to the estimations for the minor arcs and major arcs (see \cite{Birch,Browning-HB,BrowPren}). Then Igusa's conjecture for exponential sums also helps to validate the estimation for the major arcs (see \cite{CluckerNguyen,Nguyennsd}). 
 In the general case, to achieve the local-global principle as Igusa's expectation, there is still a huge gap related to the estimation for the minor arcs in the Hardy-Littlewood circle method  or the next steps of Igusa's method.  However, when $d=3$, the delta method of Heath-Brown in \cite{Heath-Brown-cubic} focuses on the estimation of the major arcs. By using the delta method, Hooley showed that the local-global principle holds for a  homogeneous polynomial $f$ of degree $3$ in $n$ variables over the ring of integers if $f$ defines a smooth hypersurface in $\PP_\QQ^{n-1}$ (i.e., $s(f)=0$ with $s(f)$ is as in Corollary \ref{nsdI}) and either $n=8$ and a Riemann hypothesis is true for certain Hasse-Weil \textsf{L}-functions (see \cite{Hoo14}) or $n\geq 9$ (see \cite{Hoo13}). On the other hand, Igusa expected that the local-global principle holds for a homogeneous polynomial $f$ of degree $3$ in at least $7$ variables over the ring of integers if $s(f)=0$. Thus, there may be a possible room to improve Hooley's results. In order to do this, we may need to give as best as lower bound of $\sigma_L(f,Z)$ for polynomials $f$ of degree $3$. In \cite{CluckerNguyen}, one expects that
 $$\sigma_L(f,Z)\geq \frac{n-s(f)}{\deg(f)}$$ 
 for all non-Archimedean local fields $L$ of characteristic zero. As seen in Example \ref{exmpol}, we can not have 
 $$\sigma_{f,Z}\geq\frac{n-s(f)}{\deg(f)}$$
 in the general case. However, if $f\in\ZZ[x_1,...,x_n]$ is a homogeneous polynomial of degree $3$ such that $s(f)=0$ then $\sigma_{f,Z}=\frac{n}{3}$. Thus, it is natural to ask whether we have  $$\sigma_{f,Z}\geq\frac{n-s(f)}{3}$$
if $f\in \ZZ[x_1,...,x_n]$ is a non-constant polynomial of degree $3$. If it is true then we may hope to improve the result of Hooley in \cite{Hoo14}. We will generalize this question in Section \ref{EXQu}.

\section{Preliminaries}
In this section, we will introduce the notations an notions  used throughout this paper. We will recall some known facts about resolutions of singularities, log canonical thresholds, Igusa's local zeta functions and exponential sums modulo $p^m$. Moreover we also recall the Hermite interpolation formula and introduce its application  on the study of Igusa's local zeta functions.

\subsection{Non-Archimedean local fields.}\label{nuloc}Let $K$ be a number field and $\cO_K$ be its ring of integers. Throughout this paper, a local field $L$ over $\cO_K$ is a finite extension of $\QQ_p$ or $\FF_p((t))$ which is endowed with a ring homomorphism $\varphi:\cO_K\to L$. Let $L$ be a local field over $\cO_K$, we denote by $\cO_L$ the ring of integers in $L$ and $\cM_L$ the maximal ideal of $\cO_L$. Let $M$ be an integer, we denote by $\cL_{K,M}$ the set of all local fields over $\cO_K$ of characteristic zero and of residue field characteristic at least $M$. We also denote by $\cL'_{K,M}$ the set of all local fields over $\cO_K$ of characteristic at least $M$. We set $\tilde{\cL}_{K,M}:=\cL_{K,M}\cup\cL'_{K,M}$.

Let $L\in\tilde{\cL}_{\QQ,1}$. We denote by $k_L=\cO_L/\cM_L$ the residue field of $L$,  $p_L$ be the characteristic of $k_L$ and $q_L$ the cardinality of $k_L$. If $z\in\cO_L$, we denote by $\overline{z}$ the image of $z$ by the reduction map $\cO_L\to k_L$. Let $\ord_L: L\to \ZZ\cup\{+\infty\}$ be the valuation map of $L$. We fix a uniformizing element $\varpi_L$ of $\cO_L$ and denote by $\acc:L \to\cO_L^*\cup\{0\}$ the angular component map given by $\acc(0)=0$ and $\acc(z)=z\varpi_L^{-\ord_L(z)}$. Let $i\geq j$ be positive integers, we denote by $\pi_i:\cO_L\to\cO_L/(\varpi_L^i)$ and $
\pi_{ij}:\cO_L/(\varpi_L^i)\to \cO_L/(\varpi_L^j)$ the canonical projections. If $z\in\cO_L$, then  $\overline{z}=\pi_1(z)$. For each $\cO_L$-scheme $X$ and $i\geq j\geq 1$, $\pi_{ij}$ induces the canonical projection $\pi_{ij}: X(\cO_L/(\varpi_L^i))\to X(\cO_L/(\varpi_L^j))$. If $\cI$ is an ideal of $\cO_L[x_1,\dots,x_n]$ generated by $f_1,...,f_r$ and $x\in \cO_L^n$, then we set $\ord_L\left(\cI(x)\right):=\min_{f\in \cI}\ord_L(f(x))=\min_{1\leq i\leq r}\ord_L(f_i(x))$. We denote by $\left|dx_1\wedge...\wedge dx_n\right|$  the normalized Haar measure on $L^n$ such that the volume of $\cO_L^n$ is $1$.

Recall that an {\em additive character} $\psi$ of $L$ is a continuous group homomorphism $\psi: (L,+)\to (\CC^{\times},\times)$. Let $\psi$ be a non-trivial additive character of $L$, then the {\em conductor} $m_\psi$ of $\psi$ is the integer $m$ such that $\psi$ is trivial on $\varpi_L^m\cO_L$, but non-trivial on $\varpi_L^{m-1}\cO_L$. Let $\psi$  be an additive character of $L$ of conductor $0$ then every additive character of $L$ is of the form $\psi_z:(L,+)\to(\CC^{\times},\times), a\mapsto \psi(az)$ for a unique element $z\in L$, moreover $m_{\psi_z}=-\ord_L(z)$. From now on, if we take a local field $L$  over $\cO_K$ then we understand that we take a tripe $(L,\varpi_L,\psi)$ with a uniformizing parameter $\varpi_L$ of $\cO_L$ and a non-trivial character $\psi$ of $L$ of conductor $0$. We usually write $L\in \tilde{\cL}_{K,M}$ but also write $(L,\varpi_L,\psi)\in \tilde{\cL}_{K,M}$ in somewhere. Let $(L,\varpi_L,\psi)$ be a local field over $\cO_K$, then $\psi_{\varpi_L^{-1}}$ is of conductor $1$ and  induces a character $\overline{\psi}:k_L\to \CC^{\times}, (x\mod \cM_L)\mapsto \psi_{\varpi_L^{-1}}(x)$ for all $x\in\cO_L$.

A {\em multiplicative character} $\chi$ of $\cO_L^{\times}$ is a continuous homomorphism $\chi:(\cO_L^{\times},\times)\to (\CC^{\times},\times)$ with finite image. We also put $\chi(0)=0$.
If $\chi$ is a multiplicative character, the {\em conductor} $c(\chi)$ of $\chi$ is the smallest integer $e\geq 1$ such that $\chi$ is trivial on $1+\cM_L^e$.  If $c(\chi)=1$ then $\chi$ induces a character $\overline{\chi}:k_L^{\times}\to\CC^{\times}, (x \mod \cM_L)\mapsto\chi(x)$ for all $x\in \cO_L^{\times}$. We denote by $\chi_{\rm triv}$ the trivial character $\cO_L^{\times}\to \CC^{\times}, z\mapsto 1$. Let $\chi$ be a multiplicative character, we denote by $d(\chi)$ the order of $\chi$, i.e., the smallest positive integer $d$ such that $\chi^d=\chi_{\rm triv}$. 

Let $\mathfrak{p}\in\Specm(\cO_K)$, then we denote by $K_\mathfrak{p}$  the completion of $K$ with respect to the place $\mathfrak{p}$ and set $\cN_\mathfrak{p}:=\left|N_{K/\QQ}(\mathfrak{p})\right|=\#(\cO_K/\mathfrak{p})=q_{K_\mathfrak{p}}$.

\subsection{Resolution of singularities  and log canonical thresholds}\label{JSL}
Let $R$ be a commutative ring with $1\neq 0$. If $\cI$ is an ideal of $R[x_1,\dots,x_n]$ and $R'$ is an $R$-algebra, then we will write $\cI_{R'}$ instead of $\cI\otimes_R R'$. Similarly, if $X, Y$ are $R$-schemes, $\varphi:X\to Y$ is a $R$-morphism and $R'$ is a $R$-algebra then we write $\varphi_{R'}:X_{R'}=X\times_{\Spec(R)}\Spec(R')=X\otimes_R R'\to Y_{R'}=Y\times_{\Spec(R)}\Spec(R')=Y\otimes_R R'$ for the extension of scalars  to $R'$ of $\varphi$. In this paper, we will write $X_{R'}$ instead of $X\otimes_R R'$.  Suppose that $R$ is Noetherian and $X$ is an $R$-scheme of finite type. We write $X_{\red}$ to denote the reduced scheme associated to $X$. We say that the complexity of $X$ is bounded by $N$ if $X$ can be covered by at most $N$ open affine subschemes of finite type over $R$, each of them can be embedded into $\AA_{R}^N$ as a closed subscheme defined by at most $N$ polynomials of degree at most $N$. In this paper, a hypersurface in $\AA_R^n$ (resp. $\PP_R^n$) is the closed subscheme of $\AA_R^n$ (resp. $\PP_R^n$) defined by a non-constant polynomial (resp. homogeneous polynomial) $f\in R[x_1,...,x_n]$.

Let $k$ be a field. We denote by $\overline{k}$ a fixed algebraic closure of $k$.  A $k$-variety is a reduced and separated $k$-scheme of finite type. 
Let $X$ be a $k$-variety and $V$ be a subvariety of $X$, we denote by $\dim_V(X)$ the dimension of $X$ at $V$, i.e., $\dim_V(X)=\min_{V\subset U, U \textnormal{ is a open subset of } X}\dim (U)$. We say that a $k$-variety $X$ is equi-dimensional if $\dim_x(X)=\dim(X)$ for all $x\in X$. If this is the case, we also say that $X$ is of pure dimension $\dim(X)$. If $X$ is a $k$-variety of pure dimension, $Y$ is a $k$-subvariety of $X$ and $V$ is a $k$-subvariety of $Y$, the codimension of $Y$ in $X$ at $V$ is defined by $\codim_V(Y,X)=\dim(X)-\dim_V(Y)$.

Let $k$ be a field.  Let $\cI$ be a non-zero ideal of $k[x_1,\dots,x_n]$ such that $\cI\neq (1)$ and $X$ be the closed subscheme of $\AA_k^n$ associated to $\cI$. A strong log resolution of $\cI$ over $k$ is  a proper, birational $k$-morphism $\pi: Y\to \AA_{k}^n$, with a smooth $k$-variety $Y$ such that $\cI\cO_Y=\cO_Y(-\sum_{i\in I}N_iE_i)$  and $\cK_{Y/\AA_k^n}=\sum_{i\in T}(\nu_i-1)E_i$ for positive integers $\nu_i, N_i$ such that $\sum_{i\in T}E_i$ is a divisor with simple normal crossings, where the relative canonical divisor $\cK_{Y/\AA_{k}^n}$ is locally defined by the determinant of the Jacobian matrix of $\pi$. We call $((\nu_i, N_i))_{i\in T}$ the numerical data of $\pi$ and $(\nu_i,N_i)$ the numerical data of $E_i$. Suppose that $k$ is of characteristic zero. By the work of Hironaka in \cite{Hir:Res}, such a strong log resolution of $\cI$ exists. Let $Z$ be a subscheme of $\AA_k^n$. The log canonical threshold 
 of $\cI$ at $Z$ is defined by
$$\lct_Z(\cI):=\min_{i\in T, Z\cap\pi(E_i)\neq\emptyset}\frac{\nu_i}{N_i}=\min_{x\in Z(\overline{k})}\lct_x(\cI_{\overline{k}}).$$
We use the convention that $\lct_Z(\cI)=+\infty$ if $Z\cap\pi(E_i)=\emptyset$ for all $i\in T$. This definition does not depend on the choice of $\pi$ (see \cite{Mustata2}). We write simple $\lct(\cI)$ instead of $\lct_X(\cI)$.  If $\cI=(f)$ for a non-constant polynomial $f$, we also write $\lct_Z(f)$ instead of $\lct_Z((f))$.  It follows from \cite[Corollary 0.2]{MustJAMS} that $\lct_Z(\cI)\leq n-\dim_{Z\cap X}(X)$. By using the generic smoothness as in \cite[Property 1.23]{Mustata2}, if $\cI$ can be generated by at most $e$ polynomials of degree at most $e$ then we can choose $(Y,\pi)$ such that $\#I$, $\nu_i, N_i$ for $i\in I$ and the complexity of $(\sqcup_{i\in I}\tilde{E}_i)\sqcup Y$ are bounded by a constant $C_{e,n}$ depending only on $e,n$. If this is the case, we say that the complexity of $(Y,\pi)$ is bounded by $C_{e,n}$. 
More generally, let $\cJ$ be another ideal of $k[x_1,\dots,x_n]$ and $\beta$ be a positive rational number. By \cite{Hir:Res}, there exists a strong common log resolution $\tilde{\pi}:\tilde{Y}\to \AA_k^n$ of $\cI$ and $\cJ$, i.e., $\tilde{Y}$ is smooth, $\tilde{\pi}$ is a proper and birational morphism,   $\cI\cO_{\tilde{Y}}=\cO_{\tilde{Y}}(-\sum_{i\in T}\tilde{N}_i\tilde{E}_i), \cJ\cO_{\tilde{Y}}=\cO_{\tilde{Y}}(-\sum_{i\in T}\cN_i\tilde{E}_i)$ and $\cK_{\tilde{Y}/\AA_k^n}=\sum_{i\in T}(\tilde{\nu}_i-1)\tilde{E}_i$ for positive integers $\tilde{\nu}_i$ and non-negative integers $\cN_i,\tilde{N}_i$ such that $\sum_{i\in T}\tilde{E}_i$ is a divisor with simple normal crossings and $\tilde{N}_i+\cN_i>0$ for all $i$. It is clear that a strong log resolution of $\cI\cJ$ is a strong common log-resolution of $\cI$ and $\cJ$.  In \cite{ELM}, the generalized log canonical threshold $\lct_Z(\cI,\cJ,\beta)$  is defined by 
\begin{equation}\label{genlogdef}\lct_Z(\cI,\cJ,\beta)=\min_{i\in T, Z\cap\tilde{\pi}(\tilde{E}_i)\neq\emptyset}\frac{\tilde{\nu}_i+\beta\cN_i}{\tilde{N}_i}=\min_{x\in Z(\overline{k})}\lct_x(\cI,\cJ,\beta)
\end{equation}
with the convention that $\frac{a}{0}=+\infty$ if $a>0$. We also write $\lct(\cI,\cJ,\beta)$ instead of $\lct_{\AA_k^n}(\cI,\cJ,\beta)$. 
Note that if $k'$ is a field extension of $k$ then $\lct_Z(\cI)=\lct_{Z_{k'}}(\cI_{k'})$ and $\lct_Z(\cI,\cJ,\beta)=\lct_{Z_{k'}}(\cI_{k'},\cJ_{k'},\beta)$.

When $k=K$ for a number field $K$ then we can say about strong log resolutions over $K$ with good reduction modulo $\cM_L$ for $L\in\tilde{\cL}_{K,1}$. Let $\cI$ be a non-zero ideal of $\cO_K[x_1,...,x_n]$ such that $\cI_K\neq (1)$. Let $(Y,\pi)$ be a strong log resolution of $\cI_K$ over $K$. Let  $(E_{i})_{i\in T}$ and $((\nu_i,N_i))_{i\in T}$ be the prime divisors and the numerical associated to  $(Y,\pi)$. For each subset $I\subset T$, we consider the schemes
\[
E_{I}:=\cap_{i\in I} E_{i} \qquad \text{and} \qquad \overset{\circ}{E_{I}}:=E_{I}\backslash\cup_{j\in T\backslash I}E_{j}.
\]
In particular, when $I=\emptyset$ we have $E_{\emptyset}=Y$. If $L\in\cL_{K,1}$ then $(Y,\pi)$ induces an embedding resolution $(Y_L,\pi_L)$ of $\cI_L$ over  $L$. 
If $W$ is a closed subscheme of $Y$, we denote the reduction mod $\cM_L$ of $W_L$ by $\overline{W}_{k_L}$ (see \cite{Shim}). We say that the resolution $(Y,\pi)$  has \emph{good reduction modulo $\cM_L$} if $\overline{Y}_{k_L}$ and all $\overline{E}_{ik_L}$ are smooth, $\cup_{i\in T}\overline{E}_{ik_L}$ has normal crossings, and the schemes $\overline{E}_{ik_L}$ and $\overline{E}_{jk_L}$ have no common components whenever $i\neq j$. 
We remark that there exists an integer $M$ such that for all local fields $L\in\cL_{K,M}$, $(Y,\pi)$ has good reduction modulo $\cM_L$ (see \cite[Theorem 2.4]{Denefdegree}). 

Let  $L\in\cL_{K,1}$. Suppose that $(Y,\pi)$ has good reduction modulo $\cM_L$. Then $(Y,\pi)$ induces a strong log resolution $(\overline{Y}_{k_L},\overline{\pi}_{k_L})$  of $\cI_{k_L}$ over $k_L$. For $I\subset T$, one then verifies that $\overline{E}_{Ik_L}=\cap_{i\in I}\overline{E}_{ik_L}$, and we put analogously $\overset{\circ}{\overline{E}}_{Ik_L}:=\overline{E}_{Ik_L}\backslash\cup_{j\notin I}\overline{E}_{jk_L}$. Moreover, if $\tilde{L}\in\cL'_{K,1}$ such that $k_L\cong k_{\tilde{L}}$ then $(\overline{Y}_{k_L},\overline{\pi}_{k_L})$ also induces a strong log resolution $(\tilde{Y}_{\tilde{L}},\tilde{\pi}_{\tilde{L}})$ of $\cI_{\tilde{L}}$ over $\tilde{L}$ with good reduction modulo $\cM_{\tilde{L}}$ in the above sense. Thus, we also say that $(Y,\pi)$ has good reduction modulo $\cM_{\tilde{L}}$.

\begin{prop}\label{genlogc}Let $K$ be a number field.  Let $\cI, \cJ$ be ideals of $\cO_L[x_1,...,x_n]$ such that $\cI_K\neq (0), \cI_K\neq (1)$. Let $Z$ be a subscheme of $\AA_{\cO_K}^n$ and $\beta$ be a positive rational number. Suppose that $\cI_K$ and $\cJ_K$ can be generated by $e$ polynomials of degree at most $e$.  Then there is an integer $M$ (depending only on $\cI,\cJ,Z$) and a positive constant $C_{e,n}$ (depending only on $n,e$) such that 
$$\int_{\{x\in \cO_L^n\mid \overline{x}\in Z(k_L),\ord_L(\cI_{\cO_L}(x))\geq m\}}q_L^{-\beta\ord_L(\cJ_{\cO_L}(x))}|dx|\leq C_{e,n}m^{n-1}q_L^{-m\lct_{Z_K}(\cI_K,\cJ_K,\beta)}$$
for all $L\in\tilde{\cL}_{K,M}$ and all $m\geq 1$. In addition, there is a constant $d_{e,n,\beta}$ (depending only on $n,e,\beta$) and a finite extension $L_p$ of $\QQ_p$ for each $p>M$ such that the inequality
$$\int_{\{x\in \cO_{L_p}^n\mid \overline{x}\in Z(k_{L_p}),\ord_{L_p}(\cI_{\cO_{L_p}}(x))\geq m\}}q_{L_p}^{-\beta\ord_{L_p}(\cJ_{\cO_{L_p}}(x))}|dx|\geq q_{L_p}^{d_{e,n}}q_{L_p}^{-m\lct_{Z_K}(\cI_K,\cJ_K,\beta)}$$
holds for infinitely many $m\geq 1$. 
\end{prop}
\begin{proof}
Let $\pi:Y\to X=\AA_K^n$ be a strong common log resolution of $\cI$ and $\cJ$  mentioned in Section \ref{JSL} such that the complexity of $(Y,\pi)$ is bounded by a constant $D_{e,n}$ depending only on $n,e$. We write $\cI\cO_Y=\cO_Y(-\sum_{i\in T}N_iE_i), \cJ\cO_Y=\cO_Y(-\sum_{i\in T}\cN_iE_i)$ and $\cK_{Y/\AA_K^n}=\sum_{i\in T}(\nu_i-1)E_i$ for positive integers $(\nu_i)_{i\in T}$, non-negative integers $(N_i,\cN_i)_{i\in T}$ and prime divisors $(E_i)_{i\in T}$ of $Y$ such that $\sum_{i\in T}E_i$ is a divisor with simple normal crossings. We set $T_1=\{i\in T\mid N_i>0\}$, $T_2=T\setminus T_1$. 

Let $M$ be an integer such that $(Y,\pi)$ has good reduction modulo $\cM_L$ for all $L\in\tilde{\cL}_{K,M}$.  Then $(Y,\pi)$ induces a strong common log resolution $(\overline{Y}_{k_L},\overline{\pi}_{k_L})$ of $\cI_{k_L}$ and $\cJ_{k_L}$ over $k_L$ for each $L\in\tilde{\cL}_{K,M}$. 

Let $L\in\tilde{\cL}_{K,M}$. We set $$\cA_{L,Z}=\{(I,J)\in T_1\times T_2\mid \overline{\pi}_{k_L}(\overset{\circ}{\overline{E}}_{(I\cup J)k_L})(k_L)\cap Z(k_L)\neq \emptyset\}.$$
By (\ref{genlogdef}) and enlarging $M$ if needed, we can suppose that 
\begin{equation}\label{upplogen}\lct_{Z_K}(\cI_K,\cJ_K,\beta)\leq \min_{(I,J)\in \cA_{L,Z}}\min_{i\in I}\frac{\nu_i+\beta\cN_i}{N_i}=\min_{(I,J)\in \cA_L}\min_{i\in I\cup J}\frac{\nu_i+\beta\cN_i}{N_i}
\end{equation}
 for all $L\in\tilde{\cL}_{K,M}$. Moreover, we can suppose that for each prime $p\geq M$, the inequality (\ref{upplogen}) holds for a finite extension $L_p$ of $\QQ_p$. 
 
  Now, we can use the method in the proof of \cite[Theorem 3.1]{Denefdegree} to have 
\begin{align}\label{equlogge}&\int_{\{x\in \cO_L^n\mid \overline{x}\in Z(k_L),\ord_L(\cI_{\cO_L}(x))\geq m\}}q_L^{-\beta\ord_L(\cJ_{\cO_L}(x))}|dx|\nonumber\\=&\sum_{(I,J)\in\cA_{L,Z}}C_{I,J,Z,L}\sum_{\ell=(\ell_i)_{i\in I\cup J}\in\NN_{\geq 1}^{\#I+\#J}, \sum_{i\in I}\ell_iN_i\geq m}q_L^{-\sum_{i\in I\cup J}\ell_i(\nu_i+\beta\cN_i)},
\end{align}
 where $$C_{I,J,Z,L}=\#\left(\left(\overset{\circ}{\overline{E}}_{(I\cup J)k_L}\cap\overline{\pi}_{k_L}^{-1}(Z_{k_L})\right)(k_L)\right)(q_L-1)^{\#I+\#J}q_L^{-n}$$
 for all $I,J\in \cA_{L,Z}$.
 By using the Lang-Weil estimate (\ref{LWeil}) together with the facts that the complexity of $(Y,\pi)$ is bounded by $D_{e,n}$ and $\dim\overset{\circ}{\overline{E}}_{(I\cup J)k_L}=n-\#I-\#J$, there is a constant $Q_{e,n}$ depending only on $n,e$ such that 
 \begin{equation}\label{upbougene}
 q_L^{-n}\leq C_{I,J,Z,L}\leq Q_{n,e}
 \end{equation}
 for all $L\in\tilde{\cL}_{K,M}$ and all $(I,J)\in\cA_{L,Z}$.  Thus, the first assertion follows from (\ref{upplogen}), (\ref{equlogge}), (\ref{upbougene}) and Lemma \ref{upnu} below. 
 
To prove the second assertion, for each prime $p>M$, we take $(I,J)\in\cA_{L_p,Z}$ such that 
 $$\lct_{Z_K}(\cI_K,\cJ_K,\beta)=\min_{i\in I}\frac{\nu_i+\beta\cN_i}{N_i}.$$
 Let $i_0\in I$ such that $\lct_{Z_K}(\cI_K,\cJ_K,\beta)=\frac{\nu_{i_0}+\beta\cN_{i_0}}{N_{i_0}}$. For each $m\in\NN$ such that $m\equiv\sum_{i\in I\setminus\{i_0\}}N_i\mod N_{i_0}$, let  $\ell^{(m)}\in \NN_{\geq 1}^{\#I+\#J}$ such that 
 $\ell^{(m)}_i=1$ if $i\in I\cup J\setminus \{i_0\}$ and $\ell^{(m)}_{i_0}=\frac{m-\sum_{i\in I\setminus\{i_0\}}N_i}{N_{i_0}}$.  Then we can use (\ref{equlogge}), (\ref{upbougene}) and the fact that the complexity of $(Y,\pi)$ is bounded by $D_{e,n}$ to have 
\begin{align*}&\int_{\{x\in \cO_{L_p}^n\mid \overline{x}\in Z(k_{L_p}),\ord_L(\cI_{\cO_{L_p}}(x))\geq m\}}q_{L_p}^{-\beta\ord_{L_p}(\cJ_{\cO_{L_p}}(x))}|dx|\\ \geq& q_{L_p}^{-n}q_{L_p}^{-\sum_{i\in I\cup J}\ell_i^{(m)}(\nu_i+\beta\cN_i)}\\=&q_{L_p}^{-n}q_{L_p}^{-\sum_{i\in I\cup J\setminus \{i_0\}}(\nu_i+\beta\cN_i-N_i\frac{\nu_{i_0}+\beta\cN_{i_0}}{N_{i_0}})}q_{L_p}^{-m\lct_{Z_K}(\cI_K,\cJ_K,\beta)}\\
\geq& q_{L_p}^{d_{e,n,\beta}}q_{L_p}^{-m\lct_{Z_K}(\cI_K,\cJ_K,\beta)}
\end{align*}
for all $m\in \left(\sum_{i\in I\setminus\{i_0\}}N_i\right)+ N_{i_0}\NN_{\geq 1}$, where 
$$d_{e,n,\beta}=-n-(1+\beta)D_{e,n}^2.$$
This completes our proof. 
\end{proof}
\begin{lem}\label{upnu}Let $I$ be a finite set and $\sigma$ be a positive real number. For each $i\in I$, we take a positive real number $a_i$ and a non-negative real number $b_i$ such that $a_i\geq\sigma b_i$.  Then there is a positive constant $C$ depending only on $((a_i,b_i))_{i\in I}$ such that 
$$\sum_{(m_i)_{i\in I}\in\NN^{\#I}, \sum_{i\in I}m_ib_i\geq m}q^{-\sum_{i\in I}m_ia_i}\leq Cm^{\#I-1}q^{-m\sigma}$$
for all $q\geq 2$ and all $m\geq 1$.
\end{lem}
\begin{proof}
It was proved in \cite[Lemma 4.0.1]{Saskia-Kien} for the case that $a_i,b_i,q\geq 2,m\geq 1$ are positive integers. The general case can be proved without any modifications.
\end{proof}

To finish this subsection, we recall the notion of rational singularities that will be used in this paper. Let $k$ be a field of characteristic zero and  $X$ be an $k$-variety. Let $h:Y\to X$ be a \index{resolution of singularities}resolution of singularities of $X$, i.e. $h$ is a proper birational morphism, $Y$ is a smooth variety, $h$ is an isomorphism outside the singular locus $\Sing(X)$ of $X$ and $(h^{-1}(\Sing(X)))_{red}$ is a divisor with simple normal crossings. The existence of such a resolution of singularities of $X$ also follows  from \cite{Hir:Res}. We say that $X$ has only {\em rational singularities} if the higher direct images $\cR^jh_*(\cO_Y)=0$ for all $j>0$. Otherwise, we say that $X$ has {\em non-rational singularities}. Let  $x_0\in X$.  Then we say that $X$ has only rational singularity at $x_0$ if there is an open neighbourhood $U$ of $x_0$ in $X$ such that $U$ has only rational singularities. If this is not the case, we say that $X$ has non-rational singularity at $x_0$. 

\subsection{Exponential sums modulo $p^m$ and Igusa local zeta functions}\label{exp-zetafunction}
In this subsection, we fix a number field $K$. For each  $\cO_K$-subscheme $Z$ of $\AA_{\cO_K}^n$ of finite type and each local field  $L$ over $\cO_K$, we denote by $\Phi_{L,Z}=\textbf{1}_{\{x\in \cO_L^n|\overline{x}\in Z(k_L)\}}$ the Schwartz-Bruhat function on $L^n$ associated to $Z$. We remark that $\Phi_{L,Z}(x)$ depends only on $\overline{x}= x \mod \cM_L$ hence $\Phi_{L,Z}(x)$ induces a function $\overline{\Phi}_{L,Z}$ on $k_L^n$.  
\subsubsection{The case of polynomials}Let $L\in\tilde{\cL}_{K,1}$.
 Let $\psi$ be an additive character of $L$ and  $\chi$ be a multiplicative character of $\cO_L^{\times}$. Let $f$ be a non-constant polynomial in $\cO_L[x_1,...,x_n]$ and $Z$ be an $\cO_K$-subscheme of $\AA_{\cO_K}^n$ of finite type. We associate our data with the \emph{exponential sum}
$$E_{f,L,Z,\psi}=\int_{\cO_L^n}\Phi_{L,Z}\psi(f(x))|dx|,$$
and the \emph{Igusa local zeta function}
$$\cZ(f,L,Z,\chi,s)=\int_{L^n}\Phi_{L,Z}\chi(ac(f(x)))|f(x)|^s|dx|.$$
where $s\in\CC$ with $\Re(s)>0$. If $L\in\cL_{K,1}$, then it is well known that $\cZ(f,L,Z,\chi,s)$ is holomorphic  in this region and has meromorphic continuity to $\CC$, moreover it is a rational function in $q_L^{-s}$ (see, e.g. \cite{Igusa3}). We say that $s$ is a non-trivial pole of $\cZ(f,L,Z,\chi,s)$ if $s$ is a pole of $\epsilon_{\chi}(s)\cZ(f,L,Z,\chi,s)$, where $\epsilon_{\chi}(s)=1$ if $\chi\neq\chi_{\textnormal{triv}}$ and $\epsilon_{\chi_{\textnormal{triv}}}(s)=(1-q_L^{s+1})$.

From now on, we suppose that  $f\in \cO_K[x_1,...,x_n]$.  The {\em motivic oscillation index} $\moi_K(f,Z)$ of $f$ at $Z$ over $K$ is the largest number $\alpha\leq +\infty$ with the property that there exists an integer $M$ and a real number $\beta$  such that 
\begin{equation}\label{inequality1}
|E_{f,L,Z,\psi}|\leq q_L^{\beta}m_\psi^{n-1}q_L^{-m_\psi\alpha}
\end{equation}
for all local fields $L\in\cL_{K,M}$ and all additive characters $\psi$ of $L$ of conductor $m_\psi\geq 2$ (see \cite{CMN,NguyenVeys}). 

We can give an alternative definition of $\moi_K(f,Z)$ as follows (see \cite{CMN} and \cite{NguyenVeys}). It follows from \cite{Igusa3} (see also \cite[Theorem 1.3.2]{DenefBour}) that the set 
\begin{equation}\label{pode}\textnormal{Pol}_{f}:=\cup_{W\subset \AA_{\cO_K}^n}\cup_{L\in\cL_{K,1}}\textnormal{Pol}_{f,L,W}
\end{equation} is a finite set. Here, the first union is taken over all $\cO_K$-subschemes $W$  of $\AA_{\cO_K}^n$ of finite type and  $\textnormal{Pol}_{f,L,W}$ is the set consisting of 
$\sigma\in \RR$ such that there exists  $c\in L$, a multiplicative  character $\chi$ of $\cO_L^{\times}$  and a non-trivial pole  $s_0$ of  $\cZ(f-c,L,W,\chi,s)$ with $\Re(s_0)=-\sigma$. For each $L\in\cL_{K,1}$, we set
\begin{equation}\label{npol}\sigma_L(f,Z):=\min \textnormal{Pol}_{f,L,Z}\leq+\infty
\end{equation}
with the convention that $\sigma_L(f,Z)=+\infty$ if $\textnormal{Pol}_{f,L,Z}=\emptyset$. Note that if we fix a strong log resolution $(Y_c,\pi_c)$ of $(f-c)$ over $\overline{K}$ associated to numerical data $((\nu_{ic},N_{ic}))_{i\in T_c}$ for each  $c$ in the set $V_f$ of critical values of $f$ over $\overline{K}$ then 
\begin{equation}\label{pode1}\textnormal{Pol}_f\subset \cup_{c\in V_f}\{\frac{\nu_{ic}}{N_{ic}}\mid i\in T_c\}.
\end{equation}
Thus, it follows from Subsection \ref{JSL} that 
$$\sigma_L(f,Z)\geq\min_{a\in Z(\overline{K})}\lct_a(f-f(a))$$
if $L\in\cL_{K,1}$.
 By \cite[Proposition 2.7]{DenefVeys}, for each $L\in\cL_{K,1}$, there is a constant $c_L>0$ and infinitely many positive integers $m$ such that 
\begin{equation}\label{lowbL}|E_{f,L,Z,\psi}|\geq c_Lq_L^{-m\sigma_L(f,Z)}
\end{equation}
for some additive character $\psi$ of $L$ of conductor $m$. Moreover, there is a constant $C_L$ such that
\begin{equation}\label{rboun}|E_{f,L,Z,\psi}|\leq C_Lm^{n-1}q_L^{-m_\psi\sigma_L(f,Z)}
\end{equation}
for all additive characters $\psi$ of $L$ of conductor $m_\psi\geq 1$. Then, we have 
\begin{align}\label{exmo}\moi_K(f,Z)&=\liminf_{L\in \cL_{K,M}, M\to+\infty}\sigma_L(f,Z)\nonumber\\&=\min\{\sigma\leq+\infty\mid \exists (L_i)_{i\geq 1}\subset \cL_{K,1}, p_{L_i}<p_{L_{i+1}}\hspace{0.1cm}\forall i\geq 1, \sigma=\sigma_{L_i}(f,Z)\}.
\end{align}
In practice, we use a more explicit description of $\moi_K(f,Z)$ by looking only at poles of the Igusa local zeta functions $\cZ(f-c,L,Z,\chi,s)$ for critical values $c$ of $f$. Let $c_1,...,c_r$ be the critical values of $f$ such that $c_i\in f(Z(\overline{K}))\subset\overline{K}$. We denote by $Z_{c_i}=(Z\otimes_{\cO_K} \cO_K[c_i])\cap \{f-c_i=0\}$. Note that there exists an integer $M$ such that for all local fields $L\in\cL_{K,M}$ with a structure map $\varphi:\cO_K\to \cO_L$, we can extend this structure map to a map $\varphi:\cO_K[c_1,...,c_r]\to \overline{\cO}_L$, where $\overline{\cO}_L$ is the integral closure of $\cO_L$ in $\overline{L}$ and
\begin{equation}\label{decomsum}
E_{f,L,Z,\psi_z}=\sum_{1\leq i\leq r, \varphi(c_i)\subset L} \psi(\varphi(c_i)) \cdot  E_{f-c_i,L,Z_{c_i},\psi}
\end{equation}
for all  $L\in\cL_{K,M}$ and all additive characters $\psi$ of $L$ of conductor at least $1$.  Moreover, it is easy to show that if $c_i\in f(Z(\overline{K}))$ is a critical value of $f$ and $L\in\cL_{K,1}$ such that $c_i\in L$ then $\cZ(f-c_i,L,Z,\chi,s)$ differs from $\cZ(f-c_i,L,Z_{c_i},\chi,s)$ by a polynomial in $q_L^{-s}$. This together with (\ref{npol}), (\ref{pode1}), (\ref{lowbL}), (\ref{rboun}), (\ref{exmo}) and (\ref{decomsum}) yield 
\begin{equation}\label{mimoi}
\displaystyle \sigma_L(f,Z)=\min_{\{i\mid c_i\in L\}}\sigma_L(f-c_i,Z_{c_i})
\end{equation} and 
\begin{equation}\label{moif}\moi_K(f,Z)=\liminf_{L\in\cL_{K,M}, M\to+\infty}\sigma_L(f,Z)=\min_{1\leq i\leq r}\moi_{K[c_i]}(f-c_i,Z_{c_i}).
\end{equation}
Here, we used  the convention that if $f$ has no critical point in $Z(\overline{K})$ then we put $\moi_K(f,Z)=+\infty$.

There is a relation between exponential sums and Igusa's local zeta functions of polynomials by combining \cite[Proposition 1.4.4]{DenefBour}, \cite[Theorem 2.1]{Denef91} and \cite[Proposition 5.4]{NguyenVeys} as follows.
\begin{prop}\label{zetaex}Let $f, Z$ be as above. Suppose moreover that $f(Z(\overline{K}))=0$. Then there exists an integer $M$ such that  $E_{f,L,Z,\psi_z}$ is equal to
\begin{align*}
\cZ(f,L,Z,\chi_\textnormal{triv},0)&+\mbox{\textnormal{Coeff}}_{t^{m-1}}\Big(\dfrac{(t-q_L)\cZ(f,L,Z,\chi_\textnormal{triv},s)}{(q_L-1)(1-t)}\Big)\\
&+\sum_{\chi\neq\chi_\textnormal{triv}, c(\chi)=1}g_{\chi^{-1}}\chi(\acc(z))\mbox{\textnormal{Coeff}}_{t^{m-1}}\big(\cZ(f,L,Z,\chi,s)\big)
\end{align*}
for all local fields $L\in \cL_{K,M}$,  all  $m\geq 1$ and all $z\in L$ with $\ord_L(z)=-m$, where $g_{\chi}$ is the Gaussian sum
\[
g_{\chi}=\frac{1}{q_L-1}\sum_{\overline{v}\in k_L^*}\chi(v)\psi\bigl(v/\varpi_L\bigr)
\]
with absolute value $\dfrac{|q_L|^{\frac{1}{2}}}{q_L-1}$.
\end{prop}

We recall a short description of Igusa local zeta functions in terms of strong log resolutions of singularities. 
\begin{prop}[\cite{CMN}, Proposition 3.4]
\label{zeta}
We keep the notations and assumptions of Proposition \ref{zetaex}.  
Let $(Y,\pi)$ is a strong log resolution of $(f)$ over $K$. Let $((\nu_i,T_i))_{i\in T(\pi)}$ be the numerical data associated to $(Y,\pi)$. Then there is an integer $M$ depending only on $f$ and a constant $C$ depending only on the complexity of $Y\sqcup (\sqcup_{i\in T(\pi)}E_i)$ such that the following formula
holds for all local fields $L\in\tilde{\cL}_{K,M}$ and all multiplicative characters $\chi$ on $\cO_{L}^{\times}$:
\begin{equation}\label{eq:ZPhichi}
\cZ(f,L,Z,\chi,s)=
\sum_{\substack{I\subseteq T(\pi)}}
c_{I,L,Z,\chi}\prod_{i\in I}\frac{q_L^{-N_{i}s-\nu_{i}}}{1-q_L^{-N_{i}s-\nu_{i}}},
\end{equation}
where the complex numbers $c_{I,L,Z,\chi}$ are independent of $s$ and satisfy
\begin{equation}\label{eq:cI1}
|c_{I,L,Z,\chi}|\leq C.
\end{equation}
Moreover, for such $L$ we further have
\begin{equation}\label{eq:cI2}
c_{I,L,Z,\chi} = 0
\end{equation}
if $c(\chi)>1$ or if the order of $\chi$ does not divide $N_i$ for some $i\in I$ or $\pi(E_I^{\circ}(\overline{K}))\cap Z(\overline{K})=\emptyset$.
\end{prop}

We also recall Igusa's conjecture for exponential sums in \cite[Page 25]{CMN}.
\begin{conj}[Igusa's conjecture for exponential sums]\label{Iguconj}Let $f, Z$ be as above. 
 Then there are positive constants $c,M$ depending only on $f,Z$ such that
$$\left|E_{f,L,Z,\psi}\right|\leq cm_{\psi}^{n-1}q_L^{-m_\psi\moi_{K,Z}(f)}$$
for all local fields $L\in\tilde{\cL}_{K,M}$ and all additive characters $\psi$ of $L$ of conductor $m_{\psi}\geq 2$.
\end{conj} 
The fact is that we only need to study Conjecture \ref{Iguconj} with $L\in\cL_{K,M}$. This is a corollary of the following proposition.
\begin{prop}[Transfer principle,\cite{NguyenVeys}, Proposition 5.4]\label{transfer1}Let $f,Z$ be as above. 
Assume that there is at most one critical value of $f$ contained in $f(Z(\overline{K}))$.
Then there is an integer $M$ depending only on $f$, such that for all pairs of local fields $(L,\varpi_L,\psi)$ and $(L',\varpi_{L'},\psi')$ in  $\tilde{\cL}_{K,M}$  with $k_L=k_{L'}$ and $\overline{\psi}=\overline{\psi'}$, and for all $z\in L$ and $z'\in L'$ with $\ord_{L}(z)=\ord_{L'}(z')\leq -1$ and $\ac(z)=\ac(z')$, we have
 \begin{equation}
|E_{f,L,Z,\psi_z}|=|E_{f,L',Z,\psi'_{z'}}|.
\end{equation}
\end{prop}
We also need the following result in \cite{CMN}.
\begin{thm}[\cite{CMN}, Theorem 1.5 and Proposition 3.10]\label{rCMN}
Conjecture \ref{Iguconj} holds 
if $f-f(x_0)$ defines a hypersurface having non-rational singularity at $x_0$ for some $x_0\in Z(\overline{K})$. If this is the case, we have $\moi_K(f,Z)=\min_{a\in\overline{K}}\lct_{Z_{\overline{K}}}(f-a)\leq 1$. 
On the other hand, $\moi_K(f,Z)>1$ if any only if $f-f(x_0)$ defines a hypersurface  having only rational singularity at $x_0$ for all $x_0\in Z(\overline{K})$.
\end{thm}
\subsubsection{The case of ideals}
Let $\cI$ be an ideal of $\cO_K[x_1,...,x_n]$ such that $\cI_K\neq (1)$ and $X$ be the closed subscheme of $\AA_{\cO_K}^n$ associated to $\cI$. Let $Z$ be an $\cO_K$-subscheme of $\AA_{\cO_K}^n$ of finite type and $L\in\tilde{\cL}_{K,1}$. We associated  $L,Z,\cI$ to the \emph{Igusa local zeta function} 
\begin{equation}\label{defzeta}\cZ_{L,Z,\cI}(s):=\int_{\cO_L^n}\Phi_{L,Z}(x)q_L^{-s\ord_L\left(\cI_{\cO_L}(x)\right)}\left|dx\right|,
\end{equation}
for $\Re(s)>0$. Note that if $\cI$ is generated by $f$ then $\cZ_{L,Z,\cI}(s)=\cZ(f,L,Z,\chi_{\textnormal{triv}},s)$. By \cite[Theorem 2.4]{VeysZ},  $\cZ_{L,Z,\cI}(s)$ is also a rational function in $q_L^{-s}$ if $L\in\cL_{K,1}$.  Let $m\geq 1$ be an integer and $r\in\RR$.  We slightly generalize \cite[Definition 3.8]{MoiIde} by associating $\cI,r,m,L,Z$ to the \emph{exponential sum} 
of $\cI$ (or $X$) at $Z$ over $L$ defined by 
\begin{equation}\label{exsum1}
E_{L,Z,\cI}^{(r)}(m)=E_{L,Z,X}^{(r)}(m):=q_L^{-mn}\left(\#\pi_{m1}^{-1}(X\cap Z)(k_L)-q_L^{n-r}\#\pi_{(m-1)1}^{-1}(X\cap Z)(k_L)\right)
\end{equation}
if $m\geq 2$ and 
\begin{equation}\label{exsum2}
E_{L,Z,\cI}^{(r)}(m)=E_{L,Z,X}^{(r)}(m):=q_L^{-n}\left(\#(X\cap Z)(k_L)-q_L^{-r}\#Z(k_L)\right)
\end{equation}
if $m=1$,  where if $i\geq j\geq 0$ then the map $\pi_{ij}:X(\cO_L/(\varpi_L)^i)\to X(\cO_L/(\varpi_L)^j)$ is defined in Subsection \ref{nuloc}. 
By the definition, we have 
\begin{equation}\label{intersec}E_{L,Z,\cI}^{(r)}(m) =E_{L,Z\cap X,\cI}^{(r)}(m)
\end{equation}
for all $m\geq 2$.

We have a relation between exponential sums of polynomials and exponential sums of ideals as follows.
\begin{prop}[\cite{MoiIde}, Remark 3.9]\label{ridepol}Let $\cI,Z$ be as above. If $r\in\ZZ_{>0}$ and there exists an open neighborhood $U$ of $Z$ in $\AA_{\cO_K}^n$ together with polynomials $f_1,...,f_r\in\cO_K[x_1,...,x_n]$ such that $\cI\cO_U=(f_1,...,f_r)$ then 
\begin{equation}\label{idepoly}
E_{L,Z,\cI}^{(r)}(m)=E_{\sum_{i=1}^ry_if_i,L,Z\times_{\Spec(\cO_K)}(\AA_{\cO_K}^r\setminus\Spec(\cO_K[y_1,...,y_r]/(y_1,...,y_r))),\psi}
\end{equation}
for all $L\in \tilde{\cL}_{K,1}$ and all additive characters $\psi$ of $L$ of conductor $m\geq 1$. 
\end{prop} From now on, we write $Z^{(r)}$ instead of $Z\times_{\Spec(\cO_K)}(\AA_{\cO_K}^r\setminus\Spec(\cO_K[y_1,...,y_r]/(y_1,...,y_r))$.  We recall the following definition. 
\begin{defn}[Lemma-Definition 3.1 and Corollary 3.12, \cite{MoiIde}]\label{defmoi0}Let $\cI,Z$ be as above and  $r\in\ZZ_{>0}$. The {\em $r^{\textnormal{th}}$-motivic oscillation index}  of $\cI$ at $Z$ over $K$, denoted by $\moi_{K,Z}^{(r)}(\cI)$, is defined to be the motivic oscillation index $\moi_K(\sum_{i=1}^ry_if_i,Z^{(r)})$ if $\cI_K=(f_1,...,f_r)$. Note that this definition does not depend on the choice of $f_1,...,f_r$. We also put $\moi_{K,Z}^{(r)}(\cI)=0$ if $\cI_K$ cannot be generated by $r$ elements. 
\end{defn}
In addition, there is a relation between $E_{L,Z,\cI}^{(r)}(m)$ and $\cZ_{L,Z,\cI}(s)$ as follows.
\begin{prop}[\cite{MoiIde}, Lemma 3.11 and Corollary 3.12]\label{eqzeta} Let $\cI, X, Z$ be as above and $r\in\RR$. Suppose that $Z\subset X$. Then have 
 \begin{equation}\label{eqexzeta0}
 \left(1-q_L^{-(s+r)}\right)\cZ_{L,Z,\cI}(s)=q_L^{-(n+s)}\left(1-q_L^{-r}\right)\#Z(k_L)+\left(1-q_L^s\right)\sum_{m\geq 2}E_{L,Z,\cI}^{(r)}(m)q_L^{-ms}
 \end{equation}
 for all $L\in\tilde{\cL}_{K,1}$. Moreover, if $r\in\ZZ_{>0}$ and  $\cI_K$ can be generated by $r$ elements then $\moi_{K,Z}^{(r)}(\cI)=\moi_{K,Z}^{r,\textnormal{naive}}(\cI)$, where if $\ell\in\ZZ_{>0}$ then
 $$\moi_{K,Z}^{\ell,\textnormal{naive}}(\cI):=\liminf_{L\in \cL_{K,M}, M\to+\infty}\sigma_{L,Z}^{\ell,\textnormal{naive}}$$
 is the $\ell^{\textnormal{th}}$-naive motivic oscillation index of $\cI$ at $Z$ over $K$ and $\sigma_{L,Z}^{\ell,\textnormal{naive}}$ is the minimal number $\sigma$ such that there is a pole $s_0$ of $\left(1-q_L^{-(s+\ell)}\right)\cZ_{L,Z,\cI}(s)$ satisfying $\Re(s_0)=-\sigma$. 
 
\end{prop}
\begin{remark}
Note that (\ref{exsum1}), (\ref{exsum2}) and Proposition \ref{eqzeta} were stated only for $r\in\ZZ_{>0}$ in \cite{MoiIde}. However, we can extend them to all real numbers $r$ without any problem.
\end{remark}
\begin{lem}\label{znpo}Let $\cI, X, Z$ be as above.  If $Z\subset X$ then we have $$\left(1-q_L^{-r}\right)\cZ_{L,Z,\cI}(0)=\left(1-q_L^{-r}\right)\int_{\cO_L^n}\Phi_{L,Z}(x)\left|dx\right|=q_L^{-n}\left(1-q_L^{-r}\right)\#Z(k_L)$$
for all $L\in\tilde{\cL}_{K,1}$ and all $r\in\RR$.

Consequently, if $q_L^{s_0}=1$ then $s_0$ is not a pole of the function $\sum_{m\geq 2}E_{L,Z,\cI}^{(r)}(m)q_L^{-ms}$
for all $L\in\cL_{K,1}$ and all $r\in\RR$.
\end{lem}
\begin{proof} The first assertion follows immediately from the definition of $\Phi_{L,Z}(x)$ and (\ref{defzeta}). The second assertion follows from the first assertion and (\ref{eqexzeta0}).
\end{proof}

We also have an explicit formula of $\cZ_{L,Z,\cI}(s)$ as follows. 
\begin{prop}[\cite{VeysZ}, Theorems 2.7 and 2.10]\label{Defm}Let $\cI,Z$ be as above. Let $(Y,\pi)$ be a strong log resolution of $\cI_K$ over $K$ associated to prime divisors $(E_i)_{i\in T}$ and numerical data $((\nu_i,N_i))_{i\in T}$. Let $L\in\cL_{K,1}$ such that $(Y_L,\pi_L)$ has good reduction modulo $\cM_L$. Then, for all $\tilde{L}\in\tilde{\cL}_{K,1}$ with $k_L\simeq k_{\tilde{L}}$,  we have 
$$\cZ_{\tilde{L},Z,\cI}(s)=\cZ_{L,Z,\cI}(s)=q_L^{-n}\sum_{I\subset T}\cC_{I,L,Z}\prod_{i\in I}\frac{(q_L-1)q_L^{-N_{i}s-\nu_{i}}}{1-q_L^{-N_{i}s-\nu_{i}}},$$
where 
$$\cC_{I,L,Z}=\underset{a\in\overset{\circ}{\overline{E}}_{Ik_L}(k_{L})}{\sum}\overline{\Phi}_{L,Z}(\overline{\pi}(a)).$$
In particular, $s_{L0}:=-\min\{\frac{\nu_i}{N_i}\mid i\in T,\cC_{\{i\},L,Z}\neq 0\}$ is a pole of $\cZ_{L,Z,\cI}(s)$ of maximal real part.
\end{prop}
\begin{lem}\label{upc}Keeping the notations and assumptions of Proposition \ref{Defm}. We have 
$$0<\left(\left(1-q_L^{-N_Ls-\nu_L}\right)^{\kappa_L}\cZ_{L,Z,\cI}(s)\right)(s_{L0})<+\infty,$$
where $s_{L0}=-\frac{\nu_L}{N_L}$ with $\gcd(\nu_L,N_L)=1$ and $$\kappa_L:=\max\{\#I\mid I\subset T,\cC_{I,L,Z}\neq 0,  \frac{\nu_i}{N_i}=-s_{L0}\hspace{0.1cm}\forall i\in I\}.$$ Consequently, $s_{L0}$ is a pole of $\cZ_{L,Z,\cI}(s)$ of multiplicity $\kappa_L$.

In addition, suppose that $\cI_K$ can be generated by $e$ polynomials of degree at most $e$.  Then there is a positive constant $C_{e,n}$ and a strong log resolution $(Y,\pi)$ of $\cI_K$ over $K$  whose complexity bounded by $C_{e,n}$ such that $C_{e,n}$ depends only on $n$ and $e$, and 
\begin{equation}\label{ulw}0\leq q_L^{-n}(q_L-1)^{\#I}\cC_{I,L,Z}\leq C_{e,n}
\end{equation}
for all $I\subset T$.

\end{lem}
\begin{proof}
We set $T_{L0}:=\{i\in T\mid \frac{\nu_{i}}{N_i}=-s_{L0}\}$, $\cA_{L0}:=\{I\subset T_{L0}\mid \#I=\kappa_L\}$,  $d_i:=\frac{N_i}{N_L}=\frac{\nu_i}{\nu_L}$  if $i\in T_{L0}$. By a simple calculation, we have 
\begin{equation*}\left(\left(1-q_L^{-N_Ls-\nu_L}\right)^{\kappa_L}\cZ_{L,Z,\cI}(s)\right)(s_{L0})=q_L^{-n}\sum_{I\in \cA_{L0}}\sum_{I\subset J\subset T}\cC_{J,L,Z}(q_L-1)^{\#J}\prod_{i\in I}\frac{1}{d_i}\prod_{i\in J\setminus I}\frac{q_L^{-N_{i}s_{L0}-\nu_{i}}}{1-q_L^{-N_{i}s_{L0}-\nu_{i}}}.
\end{equation*}
This together with the definition of $s_{L0}$ and $\kappa_L$ yield $$0<\left(\left(1-q_L^{-N_Ls-\nu_L}\right)^{\kappa_L}\cZ_{L,Z,\cI}(s)\right)(s_{L0})<+\infty.$$

On the other hand, as mentioned in Section \ref{JSL}, there is a constant $C_{e,n}$ depending only on $n,e$ and a choice of  $(Y,\pi)$ such that $1\leq \nu_i,N_i\leq C_{e,n}$ for all $i\in T$  and the complexity of $Y$ together with $E_i$ for $i\in T$ is bounded by $C_{e,n}$.  Note that $\dim(\overset{\circ}{\overline{E}}_{Ik_L})=n-\#I$ for all $I\subset T$ and all $L\in\cL_{K,M}$. Now we can enlarge $C_{e,n}$ if needed and use the Lang-Weil estimate (\ref{LWeil}) to count the number of $k_L$-points of $\overset{\circ}{\overline{E}}_{Ik_L}$  to obtain (\ref{ulw}).
\end{proof}
As explained in \cite[Remark 3.18]{MoiIde}, if $\cI$ is generated by $f_1,...,f_r\in\cO_K[x_1,...,x_r]$, then we can rewrite Conjecture \ref{Iguconj} for $f=\sum_{1\leq i\leq r}y_if_i(x)$ and $Z^{(r)}$ by using the exponential sums of $\cI$ at $Z$  as follows. 
\begin{conj}[Averaged Igusa conjecture for exponential sums]\label{avelocIgu}
Let $\cI,Z$ be as above. 
For each positive integer $r$, there is an integer $M$ depending on $\cI, Z$ and a positive constant $c_r$ depending on $\cI,r$ such that for all local fields $L\in\tilde{\cL}_{K,M}$ and all $m\geq 2$, we have 
$$\left|E_{L,Z,\cI}^{(r)}(m)\right|\leq c_rm^{n+r-1}q_L^{-m\moi_{K,Z}^{(r)}(\cI)}.$$
\end{conj}
\begin{remark}In \cite[Conjecture 3.17]{MoiIde}, Conjecture \ref{avelocIgu} was stated in a stronger form requiring that $M$ depends only on $\cI$. However, it is not true in the general case. Indeed, for each integer $N$, we set $Z_N=\Spec(\AA_{\cO_K}^n/(N,x_1,...,x_n))$. If  $\cI$ is generated by a homogeneous polynomial of degree at least 2, then it is easy to verify that $\moi_{K,Z_N}^{(r)}(\cI)=+\infty$. However, there is not any integer $M$ such that $E_{L,Z_N,\cI}^{(r)}(m)=0$ for all $N>M$ and all $m\geq 2$. Thus, we need to correct \cite[Conjecture 3.17]{MoiIde} by Conjecture \ref{avelocIgu}.
\end{remark}
More generally, if $X$ is any scheme  of finite type over $\cO_K$ ($X$ is no longer an affine scheme) and  $Z$ is an $\cO_K$-subscheme of $X$ of finite type, then we can still use (\ref{exsum1}) to define the absolute exponential sums $\cE_{r,L,Z,X}(m)$ of $X$ at $Z$ over $L$ for $m\geq 2$ as follows.
\begin{defn}\label{gzeta}Let $Z\subset X$ be $\cO_K$-schemes of finite type. Let $r$ be a real number and $L\in\tilde{\cL}_{K,1}$. For $m\geq 2$, the $r$-absolute exponential sums  of $X$ at $Z$ over $L$ is defined by 
$$\cE_{r,L,Z,X}(m)=\#\pi_{m1}^{-1}\left((X\cap Z)(k_L)\right)-q_L^{r}\#\pi_{(m-1)1}^{-1}((X\cap Z)(k_L)).$$  
We write $\cE_{r,X,L}(m)$ instead of $\cE_{r,L,Z,X}(m)$ if $Z=X$ and $m\geq 2$. We define
$$\cZ_{r,L,Z,X}(s):=\sum_{m\geq 2}\cE_{r,L,Z,X}(m)q_L^{-ms}.$$
\end{defn}
\begin{lem}\label{moexpo}Let $Z\subset X$ be $\cO_K$-schemes of finite type and $r\in\RR$. Then $\cZ_{r,L,Z,X}(s)$
is a rational function in $q_L^{-s}$ for all $L\in\cL_{K,1}$. Moreover, there is a finite subset $\cA(X)$ of $\NN_{>0}^3$ and a natural number, all of which depend only on the complexity of $X$ such that $$\cZ_{r,L,Z,X}(s+n)\prod_{(a,b,c)\in \cA(X)}(1-q_L^{-sa}q_L^{-b})^c$$ is a polynomial in $q_L^{-s}$ for all $L\in\cL_{K,1}$ and all $r\in \RR$. In particular, we have $$\Lambda_r(L,Z,X):=\{\lambda\in\RR\mid \exists s_0, \Re(s_0)=-\lambda, \cZ_{r,L,Z,X}(s_0)=\infty\}\subset \Lambda(X):=\{\frac{b}{a}-n\mid (a,b,c)\in\cA(X)\}$$ for all $L\in \cL_{K,1}$ and all $r\in\RR$.  
\end{lem}
\begin{proof}

We take a finite stratification of $Z$ whose strata  $Z_i$ for $i\in I$ are $\cO_K$-schemes of finite type such that for each $i\in I$ there is an open affine subscheme $U_i$ of $X$ of finite type over $\cO_K$ satisfying $Z_i\subset U_i$. Then it is easy to verify that
$$\cE_{r,L,Z,X}(m)=\sum_{i\in I}\cE_{r,L,Z_i,U_i}(m)$$
for all $L\in\tilde{\cL}_{K,1}$ and all $r\in\RR$. Thus we have 
\begin{equation}\label{sepsu}
\cZ_{r,L,Z,X}(s)=\sum_{i\in I}\cZ_{r,L,Z_i,U_i}(s)
\end{equation}
for all $L\in\tilde{\cL}_{K,1}$ and all $r\in\RR$. 

For each $i\in I$, let $\iota_i:U_i\to \AA_{\cO_K}^{n_i}$ be a closed embedding. We set $\tilde{Z}_i=\iota_i(Z_i)$, $\tilde{U}_i=\iota(U_i)$. Let $\cI_i$ be the ideal of $\cO_K[x_1,...,x_{n_i}]$ associated to $\tilde{U}_i$.  Then we have $\cZ_{r,L,Z_i,U_i}(s)=\cZ_{r,L,\tilde{Z}_i,\tilde{U}_i}(s)$ for all $L\in\tilde{\cL}_{K,1}$ and all $r\in \RR$. By (\ref{exsum1}) and Definition \ref{gzeta}, we have \begin{equation}\label{tras}\cZ_{r,L,Z_i,U_i}(s)=\cZ_{r,L,\tilde{Z}_i,\tilde{U}_i}(s)=\sum_{m\geq 2}E_{L,Z_i,\cI_i}^{(n_i-r)}(m)q_L^{-m(s-n_i)}
\end{equation}
for all $L\in\tilde{\cL}_{K,1}$ and all $i\in I$.
Therefore, the assertions follow from (\ref{sepsu}), (\ref{tras}), \cite[Theorem 2.4]{VeysZ}, Proposition \ref{eqzeta} and Lemma \ref{znpo}.
\end{proof}
\begin{defn}\label{defmoi}Let $Z\subset X$ be $\cO_K$-schemes of finite type and $r$ be a real number. We set $\sigma_r^{\textnormal{abs}}(L,Z,X)=-\max\Lambda_r(L,Z,X)$ and call it the {\em $r$-absolute $L$-oscillation index} of $X$ at $Z$. Here, we use the convention that $\sigma_r^{\textnormal{abs}}(L,Z,X)=+\infty$ if $\cZ_{r,L,Z,X}(s)$ is an entire function. The {\em $r$-absolute motivic oscillation} of $X$ at $Z$ over $K$ is defined by $$\moi_r^{\textnormal{abs}}(K,Z,X):=\liminf_{L\in\cL_{K,M},M\to+\infty}\sigma_r^{\textnormal{abs}}(L,Z,X)\leq +\infty.$$ 

\end{defn}
\begin{remark}\label{remoi}We use the notations of Proposition \ref{eqzeta} and  Definition \ref{defmoi}.  Suppose that $X=\Spec(\cO_K[x_1,...,x_n]/\cI)$, $r\in\ZZ_{>0}$ and $n\geq r$.  Then it follows easily from the definition that $\sigma_r^{\textnormal{abs}}(L,Z,X)=\sigma_{L,Z}^{n-r,\textnormal{naive}}(\cI)-n$ and $\moi_r^{\textnormal{abs}}(K,Z,X)=\moi_{K,Z}^{n-r,\textnormal{naive}}(\cI)-n$.
\end{remark}
\begin{prop}\label{prepol}Let $Z\subset X$ be $\cO_K$-schemes of finite type. There are  positive integers $n,\ell$,  non-zero ideals $\cI_i$ of $\cO_K[x_1,...,x_n]$  and  $\cO_K$-schemes $\tilde{Z}_i$ of finite type for $1\leq i\leq \ell$ such that $n,\ell,\cI_i$ are independent of $Z$, $\cI_i\otimes K\neq(1)$ and $\tilde{Z}_i\subset\Spec(\cO_K[x_1,...,x_n]/\cI_i)$  for all $1\leq i\leq \ell$, and the equalities
\begin{equation}\label{eq-1}\left(1-q_L^{-(s+n-r)}\right)\sum_{i\in I}\cZ_{L,\tilde{Z}_i,\cI_i}(s)=q_L^{-(n+s)}\left(1-q_L^{-n+r}\right)\#Z(k_L)+\left(1-q_L^s\right)\cZ_{r,L,Z,X}(s+n),
\end{equation}
$$\#Z(k_L)=\sum_{i\in I}\#Z_i(k_L)$$ hold for all $L\in\tilde{\cL}_{K,1}$ and all $r\in\RR$. In particular, we have 
\begin{align*}\min_{r\in \RR}\sigma_r^{\textnormal{abs}}(L,Z,X)=\min\{\sigma\leq+\infty\mid \exists 1\leq i\leq e, \exists s_0\in\CC, \Re(s_0)=-\sigma-n,  \cZ_{L,\tilde{Z}_i,\cI_i}(s_0)=\infty\}
\end{align*}
for all $L\in\cL_{K,1}$.
\end{prop}
\begin{proof}
We proceed similarly to the proof of Lemma \ref{moexpo}. Let $\cU=\{U_i\mid 1\leq i\leq\ell\}$ be a finite covering of $X$ by open affine subschemes $U_i$ of $X$ of finite type over $\cO_K$. Then it is easy to construct a stratification of $Z$ whose strata $Z_i$ for $1\leq i\leq \ell$ are $\cO_K$-schemes of finite type such that $Z_i\subset U_i$ for all $i$.  For each $1\leq i\leq \ell$, there is a closed embedding $\iota_i:U_i\to \AA_{\cO_K}^{n_i}$ for some $n_i$. We set $n=\max_{i\in I}n_i$. By using the natural embeddings $\AA_{\cO_K}^{n_i}\to \AA_{\cO_K}^{n}$ for $1\leq i\leq\ell$, we can suppose that $n_i=n$ for all $i$. We set $\tilde{U}_{i}=\iota_{i}(U_i)$ and let $\cI_{i}$ be the ideal of $\cO_K[x_1,...,x_n]$ associated to $\tilde{U}_{i}$. We also put $\tilde{Z}_{i}=\iota_{i}(Z_i)$. Then we have 
$$\#Z(k_L)=\sum_{i\in I}\#Z_i(k_L)=\sum_{i\in I}\#\tilde{Z}_{i}(k_L)$$
for all $L\in \tilde{\cL}_{K,1}$. Thus, it follows from (\ref{tras}), Proposition \ref{eqzeta} and Lemma \ref{znpo} that 
\begin{equation*}\left(1-q_L^{-(s+n-r)}\right)\sum_{i\in I}\cZ_{L,\tilde{Z}_{i},\cI_i}(s)=q_L^{-(n+s)}\left(1-q_L^{-n+r}\right)\#Z(k_L)+\left(1-q_L^s\right)\cZ_{r,L,Z,X}(s+n)
\end{equation*}
for all $L\in\tilde{\cL}_{K,1}$ and all $r\in\RR$. Moreover, if $L\in\cL_{K,1}$ then $s_0$ is not a pole $\cZ_{r,L,Z,X}(s+n)$ if $q_L^{s_0}=1$. 
Therefore,
 $\min_{r \in \RR}\sigma_r^{\textnormal{abs}}(L,Z,X)$ is equal to 
 \begin{align*}\min\{\sigma\leq+\infty\mid \exists s_0\in\CC, \Re(s_0)=-\sigma-n,  \sum_{i\in I}\cZ_{L,\tilde{Z}_i,\cI_i}(s_0)=\infty\}
\end{align*}
for all $L\in \cL_{K,1}$. By the same reason as \cite[theorem 2.7]{VeysZ}, we have 

\begin{align*}&\min\{\sigma\leq+\infty\mid \exists s_0\in\CC, \Re(s_0)=-\sigma-n,  \sum_{i\in I}\cZ_{L,\tilde{Z}_i,\cI_i}(s_0)=\infty\}\\
=&\min\{\sigma\leq+\infty\mid \exists 1\leq i\leq e, \exists s_0\in\CC, \Re(s_0)=-\sigma-n,  \cZ_{L,\tilde{Z}_i,\cI_i}(s_0)=\infty\}
\end{align*} for all $L\in\cL_{K,1}$.
This completes the proof.
\end{proof}
\begin{cor}[Transfer principle]\label{transfer2}There exists an integer $M$ depending only on $X$ such that 
$$\cE_{r,L_1,Z,X}(m)=\cE_{r,L_2,Z,X}(m)$$
for all $L_1,L_2\in \tilde{\cL}_{K,M}$ with $k_{L_1}\simeq k_{L_2}$, all $\cO_K$-subschemes $Z$ of $X$ of finite type, all  $r\in \RR$ and all $m\geq 2$.
\end{cor}
\begin{proof}We use the notation of Proposition \ref{prepol}.  We fix a strong log resolution $(Y_i,\pi_i)$ of $\cI_{iK}$ over $K$ for each $i\in I$. Then, there is an integer $M$ depending on $(Y_i,\pi_i)$ for $i\in I$ such that $(Y_i,\pi_i)$ has good reduction modulo $\cM_L$ for all $L\in\cL_{K,M}$. Then, the assertion follows from Propositions \ref{eqzeta} and \ref{Defm}.
\end{proof}

\begin{cor}\label{po1}Let $Z\subset X$ be $\cO_K$-schemes of finite type. For each $L\in\cL_{K,1}$, we set $$\sigma^{\textnormal{abs}}(L,Z,X)=\min_{r\in \RR}\sigma_r^{\textnormal{abs}}(L,Z,X)$$
and  $$\tilde{\sigma}^{\textnormal{abs}}(L,Z,X)=\max_{r\in \RR}\sigma_r^{\textnormal{abs}}(L,Z,X)$$
then $\{\sigma_r^{\textnormal{abs}}(L,Z,X)\mid r\in\RR\}=\{\sigma^{\textnormal{abs}}(L,Z,X),\tilde{\sigma}^{\textnormal{abs}}(L,Z,X)\}$. Moreover, there is at most one number $r\in\RR$ such that $\sigma_r^{\textnormal{abs}}(L,Z,X)\neq \sigma^{\textnormal{abs}}(L,Z,X)$. 
\end{cor}
\begin{proof}
The assertion follows immediately from the first assertion of Proposition \ref{prepol} and Definition \ref{defmoi}.
\end{proof}
The following proposition characterizes locally complete intersections having only rational singularities.
\begin{prop}\label{rati}
With the notation of Corollary \ref{po1}, suppose that $X_K$ is equi-dimensional then the following conditions are equivalent:
\begin{itemize}
\item[(i),]$X_{\overline{K}}$ is a locally complete intersection having only rational singularity at every point $P$ of $Z(\overline{K})$.
\item[(ii),]There exists an integer $M$ such that $-\dim(X_K)< \sigma_{\dim(X)}^{\textnormal{abs}}(L,Z,X)$ for all $L\in\cL_{K,M}$.
\end{itemize}
\end{prop}
\begin{proof} The assertions follow  from \cite[Theorem F]{MoiIde}. We hereby amend the statement of \cite[Item (i), Theorem F]{MoiIde} and \cite[Item (ii), Proposition 4.3]{MoiIde}. More precisely, in \cite[Item (ii), Proposition 4.3]{MoiIde}, we need to state a stronger claim that if $L\in \cL_{K,1}$, $s_1$ is a pole of $\cZ_{L,\AA_{\cO_K}^n,\cI}(s)$ and $s_1\neq -r+2e\pi i$ for all $e\in\ZZ$ then $\Re(s_1)<-r$. Similarly, in \cite[Item (i), Theorem F]{MoiIde}, we need to state that if $L\in\cL_{K,M}$ and $s_0$ is a pole of $\cZ_{L,Z,\cI}(s)$ then either $s_0=\dim(X_K)-n+2e\pi i$ for some $e\in\ZZ$  and $\cZ_{L,Z,\cI}(s)$ has a pole of multiplicity at most $1$ at $s_0$ or $\Re(s_0)<\dim(X_K)-n$. 
\end{proof}
\begin{defn}\label{multi}Let $Z\subset X$ be $\cO_K$-schemes of finite type. Let $r\in\RR$ and $L\in\cL_{K,1}$. For each $L\in\cL_{K,1}$ and each pole $s_0$ of $
\cZ_{r,L,Z,X}(s)$, we denote by $\mathfrak{e}_{r,L,Z,X}(s_0)$ the \emph{multiplicity} of the pole $s_0$ of $\cZ_{r,L,Z,X}(s)$. We set 
$$\mathfrak{e}_r^{\textnormal{abs}}(L,Z,X)=\max\{m>0\mid \exists s_0, \mathfrak{e}_{r,L,Z,X}(s_0)=m,  \Re(s_0)=- \sigma_r^{\textnormal{abs}}(L,Z,X)\}$$
and call it the \emph{zeta multiplicity} of $X$ at $Z$ over $L$. We use the convention that $\mathfrak{e}_r^{\textnormal{abs}}(L,Z,X)=1$ if $\sigma_r^{\textnormal{abs}}(L,Z,X)=+\infty$.
\end{defn}
\begin{remark}\label{tild}Let $M$ be as in Corollary \ref{transfer2}. For $L\in\tilde{\cL}_{K,M}$ and $r\in\RR$, we can also define 
$$\sigma_r^{\textnormal{abs}}(L,Z,X):=\sigma_r^{\textnormal{abs}}(\tilde{L},Z,X),\sigma^{\textnormal{abs}}(L,Z,X):=\sigma^{\textnormal{abs}}(\tilde{L},Z,X),$$
 $$\tilde{\sigma}^{\textnormal{abs}}(L,Z,X):=\tilde{\sigma}^{\textnormal{abs}}(\tilde{L},Z,X) \textnormal{ and } \mathfrak{e}_r^{\textnormal{abs}}(L,Z,X):=\mathfrak{e}_r^{\textnormal{abs}}(\tilde{L},Z,X)$$ for any $\tilde{L}\in\cL_{K,M}$ such that $k_L\simeq k_{\tilde{L}}$.
\end{remark}

\subsection{Hermite interpolation formula}\label{Herf}
Let $\mathsf{F}$ be a field of characteristic zero and $h\in \mathsf{F}[x]$ be a polynomial of degree $d$. Let $a_1,...,a_e\in \mathsf{F}$ and $d_1,...,d_e\in\NN_{>0}$ such that $d_1+...+d_e=d+1$. Suppose that $h^{(j)}(a_i)=\alpha_{ij}$ for $1\leq i\leq e$ and $0\leq j\leq d_i-1$, where $h^{(j)}$ is the $j^{\textnormal{th}}$-derivative of $h$. Then, it follows from the Hermit interpolation formula (see \cite[Chapter 2]{BeShi}) that
\begin{equation}\label{Hinter}h(x)=\sum_{i=1}^e\sum_{j=0}^{d_i-1}\sum_{\ell=0}^{d_i-j-1}\frac{\alpha_{ij}}{j!\ell!}\left(\frac{(x-a_i)^{d_i}}{\Omega(x)}\right)^{(\ell)}(a_i)\frac{\Omega(x)}{(x-a_i)^{d_i-j-\ell}},
\end{equation}
where $\Omega(x)=\prod_{i=1}^e(x-a_i)^{d_i}$.

Let $N$ be a positive integer, we consider the field $\CC(y^{\frac{1}{N}})$ with the valuation $\deg_y:\CC(y^{\frac{1}{N}})\to\frac{1}{N}\ZZ\cup\{-\infty\}$ given by $y^{\frac{1}{N}}\mapsto \frac{1}{N}, \frac{r(y)}{s(y)}\mapsto\deg_y(r)-\deg_y(s)\in\frac{1}{N}\ZZ$ if $r(y),s(y)\in \CC[y^{\frac{1}{N}}]\setminus\{0\}$ and $0\mapsto -\infty$. We also consider the ring $\CC(y^{\frac{1}{N}})[t]$ of polynomials in $t$ with coefficients in $\CC(y^{\frac{1}{N}})$ and use the convention that $\deg_t(0)=-\infty$. Then $\deg_t$ also extends to a valuation on $\CC(y^{\frac{1}{N}})(t)$ in an obvious way. If $N\mid N'$ we indentify $\CC(y^{\frac{1}{N}})$ (resp. ring $\CC(y^{\frac{1}{N}})[t]$) as a subfield (resp. subring) of $\CC(y^{\frac{1}{N'}})$ (resp. $\CC(y^{\frac{1}{N'}})[t]$). We view $\QQ(y)[t]$ naturally as a subring of $\CC(y^{\frac{1}{N}})[t]$.

\begin{lem}\label{separate} Let $\{(\nu_i,N_i)\mid i\in \mathsf{I}_1\}, \{(\nu_i,N_i)\mid i\in \mathsf{I}_2\}$ be finite subsets  of $\NN_{\geq 1}^2$. We set $\mathfrak{E}_j=\{\frac{\nu_i}{N_i}\mid i\in \mathsf{I}_j\}$ and $l_j=\sum_{i\in \mathsf{I}_j} N_i$ for $j\in\{1,2\}$. Suppose that $\mathfrak{E}_1\cap \mathfrak{E}_2=\emptyset$ and $\mathfrak{E}_1\cup \mathfrak{E}_2\neq\emptyset$. 
There exist unique polynomials $g_{\mathsf{I}_1}(y,t)$ and $g_{\mathsf{I}_2}(y,t)$ in $\QQ(y)[t]$ such that $\deg_t(g_{\mathsf{I}_1}(y,t))<l_{1}, \deg_t(g_{\mathsf{I}_2}(y,t))<l_2$ and  the following properties hold:
\begin{itemize}
\item[(i),] We have 
\begin{equation}\label{decomposition}
\prod_{i\in \mathsf{I}_1\sqcup \mathsf{I}_2}\dfrac{1}{1-t^{N_i}y^{-\nu_i}}=\dfrac{g_{\mathsf{I}_1}(y,t)}{\prod_{i\in \mathsf{I}_1}(1-t^{N_i}y^{-\nu_i})}+\dfrac{g_{\mathsf{I}_2}(y,t)}{\prod_{i\in \mathsf{I}_2}(1-t^{N_i}y^{-\nu_i})}.
\end{equation}
\item[(ii),] For all $j\in\{1,2\}$, $l_j>0$ if and only if $g_{\mathsf{I}_j}(y,t)\neq 0$. Moreover, if $y_0$ is a non-zero complex number with $|y_0|\neq 1$ then $g_{\mathsf{I}_j}(y_0,t)\in\CC[t]$ for all $j$.
\item[(iii),]
Let $N$ be a positive integer such that $\displaystyle\big(\prod_{i\in \mathsf{I}_1\cup \mathsf{I}_2}N_i\big)\mid N$. Let $j\in \mathsf{I}_1$ and $\zeta$ be a $N_j^{\textnormal{th}}$-root of unity, then we denote by $m_{j,\zeta}$ the multiplicity of $\displaystyle\prod_{i\in \mathsf{I}_1}(1-t^{N_i}y^{-\nu_i})$ at  $t=\zeta y^{\nu_j/N_j}\in \CC(y^{\frac{1}{N}})$. If $l_1l_2>0$ and $\min\mathfrak{E}_1>\max\mathfrak{E}_2$ then $$\deg_y\left(g_{\mathsf{I}_1}^{(r)}\left(y,\zeta y^{\nu_j/N_j}\right)\right)\leq-r\frac{\nu_j}{N_j}-\sum_{i\in \mathsf{I}_2} \left(\frac{N_i\nu_j}{N_j}-\nu_i\right)$$ for all $0\leq r\leq m_{j,\zeta}-1$, where $g_{\mathsf{I}_1}^{(r)}(y,t)$ is the $r^{\textnormal{th}}$- derivative  in $t$ of $g_{\mathsf{I}_1}(y,t)$. 
\end{itemize}
\end{lem}
\begin{proof}
 The assertion is trivial if $l_1l_2=0$. So we can suppose that $l_1>0$ and $l_2>0$. It follows from $\mathfrak{E}_1\cap\mathfrak{E}_2=\emptyset$ that $$\displaystyle \bigg(\prod_{i\in \mathsf{I}_1}(1-t^{N_i}y^{-\nu_i}), \prod_{i\in \mathsf{I}_2}(1-t^{N_i}y^{-\nu_i})\bigg)=1$$ in the  ring $\QQ(y)[t]$. We observe that $\QQ(y)[t]$ is an Euclidean domain. So there exist unique polynomials $g_{\mathsf{I}_1}(y,t), g_{\mathsf{I}_2}(y,t)$ in $\QQ(y)[t]$ such that $\deg_t(g_{\mathsf{I}_1}(y,t))<l_{1}, \deg_t(g_{\mathsf{I}_2}(y,t))<l_2$ and 
\begin{equation}\label{coprime}
g_{\mathsf{I}_1}(y,t)\prod_{i\in \mathsf{I}_2}(1-t^{N_i}y^{-\nu_i})+g_{\mathsf{I}_2}(y,t)\prod_{i\in \mathsf{I}_1}(1-t^{N_i}y^{-\nu_i})=1.
\end{equation}
Hence, the item (i) follows from (\ref{coprime}). In particular, we have $g_{\mathsf{I}_1}(y,t)\neq 0$, $g_{\mathsf{I}_2}(y,t)\neq 0$. Moreover, there are polynomials $\tilde{g}_1(y,t),\tilde{g}_2(y,t)\in \QQ[y,t]$ and $h(y)\in \QQ[y]$ such that $g_{\mathsf{I}_j}(y,t)h(y)=\tilde{g}_j(y,t)$ for all $j\in\{1,2\}$ and $(\tilde{g}_1(a,t),\tilde{g}_2(a,t))\neq (0,0)$ if $h(a)=0$. In the ring $\QQ[y,y^{-1},t]$, (\ref{coprime}) becomes
\begin{equation}\label{coprime1}
\tilde{g}_{1}(y,t)\prod_{i\in \mathsf{I}_2}(1-t^{N_i}y^{-\nu_i})+\tilde{g}_{2}(y,t)\prod_{i\in \mathsf{I}_1}(1-t^{N_i}y^{-\nu_i})=h(y).
\end{equation}
Let $y_0\in\CC\setminus\{0\}$ with $|y_0|\neq 1$ then the equality $$\displaystyle \bigg(\prod_{i\in \mathsf{I}_1}(1-t^{N_i}y_0^{-\nu_i}), \prod_{i\in \mathsf{I}_2}(1-t^{N_i}y_0^{-\nu_i})\bigg)=1$$ holds in the  ring $\CC[t]$. Thus,  by substituting  $y_0$ for $y$ in (\ref{coprime1}) and observing that $\deg_t(\tilde{g}_{j}(y_0,t))\leq \deg_t(g_{\mathsf{I}_j}(y,t))< l_j$ for all $j\in\{1,2\}$, we must have $h(y_0)\neq 0$. Therefore $$g_{\mathsf{I}_j}(y_0,t)=\frac{\tilde{g}_j(y_0,t)}{h(y_0)}\in \CC[t]$$ for all $j\in\{1,2\}$. The item (ii) is proved.

If $j\in \mathsf{I}_1$ and $\zeta$ is a $N_j^{\textnormal{th}}$-root of unity, then we set $a_{j,\zeta}=\zeta y^{\nu_j/N_j}$.  By (\ref{coprime}), we have
$$g_{\mathsf{I}_1}(y,a_{j,\zeta})\prod_{i\in \mathsf{I}_2}(1-\zeta^{N_i}y^{\nu_jN_i/N_j-\nu_i})=1.$$
It implies that $$\displaystyle\deg_y(g_{\mathsf{I}_1}(y,a_{j,\zeta}))=-\sum_{i\in \mathsf{I}_2} \left(\frac{N_i\nu_j}{N_j}-\nu_i\right).$$

If $m_{j,\zeta}>1$, we take the $r^{\textnormal{th}}$-derivatives in $t$ for both sides of (\ref{coprime}) and $1\leq r\leq m_{j,\zeta}-1$. By using Leibniz's rule, we have 
$$\sum_{\ell=0}^{r}\binom{r}{\ell} g_{\mathsf{I}_1}^{(\ell)}(y,t)\left(\prod_{i\in \mathsf{I}_2}(1-t^{N_i}y^{-\nu_i})\right)^{(r-\ell)}+\left(g_{\mathsf{I}_2}(y,t)\prod_{i\in \mathsf{I}_1}(1-t^{N_i}y^{-\nu_i})\right)^{(r)}=0.$$
Since $m_{j,\zeta}$ is the multiplicity of $\displaystyle\prod_{i\in \mathsf{I}_1}(1-t^{N_i}y^{-\nu_i})$ at  $a_{j,\zeta}$, for each $1\leq r\leq m_{j,\zeta}-1$, we have 
$$\left(g_{\mathsf{I}_2}(y,t)\prod_{i\in \mathsf{I}_1}(1-t^{N_i}y^{-\nu_i})\right)^{(r)}(a_{j,\zeta})=0.$$
On the other hand, since $\min\mathfrak{E}_1>\max\mathfrak{E}_2$, it is easy to check that
$$\displaystyle\deg_y\left(\frac{\left(\prod_{i\in \mathsf{I}_2}(1-t^{N_i}y^{-\nu_i})\right)^{(r-\ell)}(a_{j,\zeta})}{\prod_{i\in \mathsf{I}_2}(1-t^{N_i}y^{-\nu_i})(a_{j,\zeta})}\right)\leq -(r-\ell)\frac{\nu_j}{N_j}$$
for all $0\leq \ell\leq r$.
By using induction on $r$, we have 
$$\deg_y(g_{\mathsf{I}_1}^{(r)}(y,a_{j,\zeta}))\leq -r\frac{\nu_j}{N_j}-\sum_{i\in \mathsf{I}_2} \left(\frac{N_i\nu_j}{N_j}-\nu_i\right)$$
for all $0\leq r\leq m_{i,\zeta}-1$.
\end{proof}
\begin{lem}\label{estimatedegree}We keep the notations of Lemma \ref{separate}. Suppose that $l_1l_2>0$ and $\min\mathfrak{E}_1>\max\mathfrak{E}_2$.  We set $g_{\mathsf{I}_1}(y,t)=\sum_{0\leq l\leq l_1-1}b_{l}(y)t^l$. Then we have
$$\deg_{y}b_l(y)\leq \max_{j\in \mathsf{I}_1}\biggl(-l\frac{\nu_j}{N_j}-\sum_{i\in \mathsf{I}_2} \left(\frac{N_i\nu_j}{N_j}-\nu_i\right)\biggr)$$
for all $0\leq l\leq l_1-1$.
\end{lem}
\begin{proof}We set $A=\{a_{j,\zeta}| j\in \mathsf{I}_1, \zeta^{N_j}=1\}$.  We use (\ref{Hinter}) to see that
\begin{equation}\label{equ001}
g_{\mathsf{I}_1}(y,t)=\sum_{a_{j,\zeta}\in A}\sum_{r=0}^{m_{j,\zeta}-1}\sum_{\ell=0}^{m_{j,\zeta}-1-r}g_{1}^{(r)}(y,a_{j,\zeta})\frac{1}{r!\ell!}\left(\frac{(t-a_{j,\zeta})^{m_{j,\zeta}}}{\Theta(t)}\right)^{(\ell)}(a_{j,\zeta})
 \times \frac{\Theta(t)}{(t-a_{j,\zeta})^{m_{j,\zeta}-r-\ell}},
\end{equation}
where $\Theta(t)=\prod_{a_{i,\xi}\in A}(t-a_{i,\xi})^{m_{i,\xi}}$.
By Lemma \ref{separate}, we have 
\begin{equation}\label{boundd}
\deg_y(g_{\mathsf{I}_1}^{(r)}(y,a_{j,\zeta}))\leq -r\frac{\nu_j}{N_j}-\sum_{i\in \mathsf{I}_2} \left(\frac{N_i\nu_j}{N_j}-\nu_i\right)
\end{equation}
for all $a_{j,\zeta}\in A$ and all $0\leq r\leq m_{i,\zeta}-1$. 

By looking at the coefficients of the polynomial $g_{\mathsf{I}_1}(y,t)\in\CC(y)[t]$, it follows from (\ref{equ001}) that $b_l(y)$ is a $\CC$-linear combination of terms of the form
$$a_{j,\zeta}^{r-l_0}g_{\mathsf{I}_1}^{(r)}(y,a_{j,\zeta})\prod_{a_{i,\xi}\in A\setminus\{a_{j,\zeta}\}}\frac{a_{i,\xi}^{m_{i,\xi}-l_{i,\xi}}a_{j,\zeta}^{\ell_{i,\xi}}}{(a_{i,\xi}-a_{j,\zeta})^{m_{i,\xi}+\ell_{i,\xi}}},$$
where $l_0, l_{i,\xi}, \ell_{i,\xi}$ are non-negative integers such that $$\displaystyle l=l_0+\sum_{a_{i,\xi}\in A\setminus\{a_{j,\zeta}\}}l_{i,\xi}$$ and $$\displaystyle\sum_{a_{i,\xi}\in A\setminus\{a_{j,\zeta}\}}\ell_{i,\xi}\leq m_{j,\zeta}-r-1.$$ Moreover, we observe that $\deg_y(a_{i,\xi}-a_{j,\zeta})=\max\{\deg_y(a_{i,\xi}),\deg_y(a_{j,\zeta})\}$ for all $a_{i,\xi}\in A\setminus\{a_{j,\zeta}\}$. Thus, by combining these facts with (\ref{boundd}), we can conclude that
$$\deg_{y}b_l(y)\leq \max_{j\in \mathsf{I}_1}\biggl(-l\frac{\nu_j}{N_j}-\sum_{i\in \mathsf{I}_2} \left(\frac{N_i\nu_j}{N_j}-\nu_i\right)\biggr).$$
\end{proof}
\begin{prop}\label{rewr}Let $\mathsf{I}_1,\mathsf{I}_2,\mathsf{I}_3$ be finite sets and $\mathfrak{N}_j=\sqcup_{i\in \mathsf{I}_j}\{(\nu_i,N_i)\}$ for $1\leq j\leq 3$ be finite multisets of elements in $\NN_{\geq 1}^2$ such that the sets $\{\frac{\nu_i}{N_i}| i\in \mathsf{I}_j\}$ are  pairwise disjoint. We set $\displaystyle \mathsf{I}=\sqcup_{i=1}^3\mathsf{I}_j$.  For each subset $\mathsf{J}$ of $\mathsf{I}$, we take a map  $c_\mathsf{J}:\CC\to \CC[t]$. 
We set 
$$F(y,t)=\sum_{\mathsf{J}\subset \mathsf{I}}\frac{c_\mathsf{J}(y)}{\prod_{i\in \mathsf{J}}(1-t^{N_i}y^{-\nu_i})}$$
with the convention that $\prod_{i\in \mathsf{J}}(1-t^{N_i}y^{-\nu_i})=1$ if $\mathsf{J}=\emptyset$. If $\mathsf{J}\subset \mathsf{I}$, then we set $\mathsf{J}_j=\mathsf{J}\cap \mathsf{I}_j$ for $1\leq j\leq 3$. We use Lemma \ref{separate} to write
$$\prod_{i\in \mathsf{J}}\frac{1}{1-t^{N_i}y^{-\nu_i}}=\dfrac{g_{\mathsf{J},1,2}(y,t)}{\prod_{i\in \mathsf{J}\setminus \mathsf{J}_3}(1-t^{N_i}y^{-\nu_i})}+\dfrac{g_{\mathsf{J},3}(y,t)}{\prod_{i\in \mathsf{J}_{3}}(1-t^{N_i}y^{-\nu_i})}$$
for each $\mathsf{J}\subset \mathsf{I}$ with $\mathsf{J}_3\neq\emptyset$.

Let $y_0\in\CC\setminus\{0\}$. Suppose that every pole $t_0$ of $F_{y_0}(t):=F(y_0,t)$  satisfies $1-t_0^{N_i}y_0^{-\nu_i}\neq 0$ for all $i\in \mathsf{I}_2$. Then there is $d_{\mathsf{J},y_0}(t)\in \CC[t]$ for each $\mathsf{J}\subset \mathsf{I}_1$  such that $\deg_t(d_{\mathsf{J},y_0}(t))<\sum_{i\in \mathsf{J}}N_i$ if $\mathsf{J}\neq\emptyset$, $\deg_t(d_{\mathsf{J},y_0}(t))\leq\max_{\mathsf{J}}\deg_t(c_\mathsf{J}(y_0))$, and 
\begin{equation}\label{rewrt}F(y_0,t)=d_{\emptyset,y_0}(t)+\sum_{\mathsf{J}\subset \mathsf{I}_1, \mathsf{J}\neq\emptyset}\frac{d_{\mathsf{J},y_0}(t)}{\prod_{i\in \mathsf{J}}(1-t^{N_i}y_0^{-\nu_i})}+\sum_{\mathsf{J}\subset \mathsf{I}, \mathsf{J}_3\neq\emptyset}\frac{c_\mathsf{J}(y_0)g_{\mathsf{J},3}(y_0,t)}{\prod_{i\in \mathsf{J}_3}(1-t^{N_i}y_0^{-\nu_i})}.
\end{equation}
In particular, if $\deg_t(c_\mathsf{J}(y_0))\leq 0$ for all $\mathsf{J}$ then $d_{\emptyset,y_0}(t)\in\CC$.
\end{prop}
\begin{proof}For each $\mathsf{J}\subset \mathsf{I}$, we also use Lemma \ref{separate} to write 
$$\dfrac{1}{\prod_{i\in \mathsf{J}\setminus \mathsf{J}_3}(1-t^{N_i}y^{-\nu_i})}=\dfrac{g_{\mathsf{J},1}(y,t)}{\prod_{i\in \mathsf{J}_1}(1-t^{N_i}y^{-\nu_i})}+\dfrac{g_{\mathsf{J},2}(y,t)}{\prod_{i\in \mathsf{J}_2}(1-t^{N_i}y^{-\nu_i})}$$
if $\mathsf{J}\setminus \mathsf{J}_3\neq\emptyset$. 
Now, we have 
\begin{align*}F(y,t)=&c_{\emptyset}(y)+\sum_{\mathsf{J}\subset \mathsf{I}, \mathsf{J}_3\neq\emptyset}\frac{c_\mathsf{J}(y)}{\prod_{i\in \mathsf{J}}(1-t^{N_i}y^{-\nu_i})}+\sum_{\emptyset\neq \mathsf{J}\subset \mathsf{I}, \mathsf{J}_3=\emptyset}\frac{c_\mathsf{J}(y)}{\prod_{i\in \mathsf{J}_1\cup \mathsf{J}_2}(1-t^{N_i}y^{-\nu_i})}\nonumber\\
=&c_{\emptyset}(y)+\sum_{\mathsf{J}\subset \mathsf{I}, \mathsf{J}_3\neq\emptyset}\left(\frac{c_\mathsf{J}(y)g_{\mathsf{J},3}(y,t)}{\prod_{i\in \mathsf{J}_3}(1-t^{N_i}y^{-\nu_i})}+\frac{c_\mathsf{J}(y)g_{\mathsf{J},1,2}(y,t)}{\prod_{i\in \mathsf{J}\setminus \mathsf{J}_3}(1-t^{N_i}y^{-\nu_i})}\right)\nonumber\\&\hspace{0.8cm}+\sum_{\emptyset\neq \mathsf{J}\subset \mathsf{I}, \mathsf{J}_3=\emptyset}\left(\frac{c_\mathsf{J}(y)g_{\mathsf{J},1}(y,t)}{\prod_{i\in \mathsf{J}_1}(1-t^{N_i}y^{-\nu_i})}+\frac{c_\mathsf{J}(y)g_{\mathsf{J},2}(y,t)}{\prod_{i\in \mathsf{J}_2}(1-t^{N_i}y^{-\nu_i})}\right)\nonumber\\
=&c_{\emptyset}(y)+\sum_{\mathsf{J}\subset \mathsf{I}, \mathsf{J}_3\neq\emptyset}\left(\frac{c_\mathsf{J}(y)g_{\mathsf{J},3}(y,t)}{\prod_{i\in \mathsf{J}_3}(1-t^{N_i}y^{-\nu_i})}\right)\nonumber\\
&\hspace{0.8cm}+\sum_{\mathsf{J}\subset \mathsf{I}, \mathsf{J}_3\neq\emptyset}\left(\frac{c_\mathsf{J}(y)g_{\mathsf{J},1,2}(y,t)g_{\mathsf{J},1}(y,t)}{\prod_{i\in \mathsf{J}_1}(1-t^{N_i}y^{-\nu_i})}+\frac{c_\mathsf{J}(y)g_{\mathsf{J},1,2}(y,t)g_{\mathsf{J},2}(y,t)}{\prod_{i\in \mathsf{J}_2}(1-t^{N_i}y^{-\nu_i})}\right)\nonumber\\&\hspace{0.8cm}+\sum_{\emptyset\neq \mathsf{J}\subset \mathsf{I}, \mathsf{J}_3=\emptyset}\left(\frac{c_\mathsf{J}(y)g_{\mathsf{J},1}(y,t)}{\prod_{i\in \mathsf{J}_1}(1-t^{N_i}y^{-\nu_i})}+\frac{c_\mathsf{J}(y)g_{\mathsf{J},2}(y,t)}{\prod_{i\in \mathsf{J}_2}(1-t^{N_i}y^{-\nu_i})}\right).\nonumber  \\
\end{align*}
We set 
\begin{equation*}G_y(t)=\left(\sum_{\mathsf{J}\subset \mathsf{I}, \mathsf{J}_3\neq\emptyset}\frac{c_\mathsf{J}(y)g_{\mathsf{J},1,2}(y,t)g_{\mathsf{J},2}(y,t)}{\prod_{i\in \mathsf{J}_2}(1-t^{N_i}y^{-\nu_i})}\right)+ \left(\sum_{\emptyset\neq \mathsf{J}\subset \mathsf{I}, \mathsf{J}_3=\emptyset}\frac{c_\mathsf{J}(y)g_{\mathsf{J},2}(y,t)}{\prod_{i\in \mathsf{J}_2}(1-t^{N_i}y^{-\nu_i})}\right).
\end{equation*}
By the assumption, if $1-t_0^{N_i}y_0^{-\nu_i}\neq 0$ for some $i\in \mathsf{I}_2$ then $t_0$ is not a pole of $F_{y_0}(t)$. This and the assumption that the sets $\{\frac{\nu_i}{N_i}| i\in \mathsf{I}_j\}$ are  pairwise disjoint imply that $G_{y_0}(t)$ is a polynomial in $t$. Thus, the first assertion follows by writing 
$$\frac{c_\mathsf{J}(y_0)g_{\mathsf{J},1,2}(y_0,t)g_{\mathsf{J},1}(y_0,t)}{\prod_{i\in \mathsf{J}_1}(1-t^{N_i}y_0^{-\nu_i})}=H_{\mathsf{J},y_0}(t)+\frac{Q_{\mathsf{J},y_0}(t)}{\prod_{i\in \mathsf{J}_1}(1-t^{N_i}y_0^{-\nu_i})}$$
if $\mathsf{J}_3\neq \emptyset$, and
$$\frac{c_\mathsf{J}(y_0)g_{\mathsf{J}_{1}}(y_0,t)}{\prod_{i\in \mathsf{J}_1}(1-t^{N_i}y_0^{-\nu_i})}=H_{\mathsf{J},y_0}(t)+\frac{Q_{\mathsf{J},y_0}(t)}{\prod_{i\in \mathsf{J}_1}(1-t^{N_i}y_0^{-\nu_i})}$$
if $\mathsf{J}_3=\emptyset\neq \mathsf{J}$ such that $H_{\mathsf{J},y_0}(t),Q_{\mathsf{J},y_0}(t)\in \CC[t]$ and $\deg_t(Q_{\mathsf{J},y_0}(t))<\sum_{i\in \mathsf{J}_1}N_i$ for all $\mathsf{J}\neq\emptyset$.

By Lemma \ref{separate}, $\deg_t(g_{\mathsf{J},3}(y,t))<\sum_{i\in \mathsf{J}_3}N_i$ if $\mathsf{J}_3\neq \emptyset$. Therefore, we can evaluate the degree function $\deg_t$ on both sides of (\ref{rewrt}) to see that $\deg_t(d_{\emptyset,y_0}(t))\leq\max_\mathsf{J}\deg_t(c_\mathsf{J}(y_0))$. In particular, if $\deg_t(c_\mathsf{J}(y_0))\leq 0$ for all $\mathsf{J}$, then $\deg_t(d_{\emptyset,y_0}(t))\leq 0$ thus $d_{\emptyset,y_0}(t)\in\CC$.
\end{proof}
\section{Estimate of Lang-Weil type for counting points of schemes over finite rings}\label{Opes}
In this section, we fix a number field $K$ and $\cO_K$-schemes $Z,X$ of finite type with $Z\subset X$. We use the convention that $a^{-\infty}=0$ if $a>1$. With the notations of Lemma \ref{moexpo} and Definition \ref{multi}, for each $L\in\cL_{K,1}, r\in \RR$ and $\lambda\in\Lambda_r(L,Z,X)$, we set $$\mathfrak{e}_{r,L}(\lambda)=\max\{\mathfrak{e}_{r,L,Z,X}(s_0)\mid \cZ_{r,L,Z,X}(s_0)=\infty,\Re(s_0)=-\lambda\}.$$ First of all, we give a lower bound and an upper bound of $|\cE_{r,L,Z,X}(m)|$ for  $L\in \cL_{K,1}$ and $m\geq 2$  as follows. 
\begin{prop}\label{lowe}For each $L\in \cL_{K,1}$ and each $r\in\RR$, there exist  positive constants $C_{r,L,Z,X}$ and $c_{r,L,Z,X}$  such that the inequality
$$|\cE_{r,L,Z,X}(m)|\leq C_{r,L,Z,X}m^{\mathfrak{e}_r^{\textnormal{abs}}(L,Z,X)-1}q_L^{-m\sigma_r^{\textnormal{abs}}(L,Z,X)}$$
holds for all $m\geq 2$, and the inequality
$$|\cE_{r,L,Z,X}(m)|\geq c_{r,L,Z,X}m^{\mathfrak{e}_r^{\textnormal{abs}}(L,Z,X)-1}q_L^{-m\sigma_r^{\textnormal{abs}}(L,Z,X)}$$
holds for infinitely many $m\geq 2$. 
\end{prop}
\begin{proof}
To simplify the notation, we set $(\lambda_{L0},\mathfrak{e}_{L0}):=(\sigma_r^{\textnormal{abs}}(L,Z,X),\mathfrak{e}_r^{\textnormal{abs}}(L,Z,X))$. 

Let $L\in\cL_{K,1}$. Recall that a Presburger subset of $\NN$ is a set which can be written as a finite Boolean combination of sets of the form $\{ m\in \NN\mid a_1\leq m, m\equiv a_2 \mod a_3\}$ for integers $a_1,a_2,a_3$. By Lemma \ref{moexpo}, using partial fractions and expanding $\cZ_{r,L,Z,X}(s)$ as a power series in $q_L^{-s}$, for each $\lambda\in \Lambda_r(L,Z,X)$ and $0\leq i\leq \mathfrak{e}_{r,L}(\lambda)-1$, there  is a positive integer $\ell_{\lambda,i}$, constants $c_{r,L,Z,X,\lambda,i,\ell}$ and Presburger subsets $\cB_{\lambda,i,\ell}$ of $\NN$ with $1\leq \ell\leq \ell_{\lambda,i}$ such that
$$\cE_{r,L,Z,X}(m)=\sum_{\lambda\in \Lambda_r(L,Z,X)}\sum_{0\leq i\leq \mathfrak{e}_{r,L}(\lambda)-1}\sum_{1\leq \ell\leq\ell_{\lambda,i}}\11_{\cB_{\lambda,i,\ell}}(m)c_{r,L,Z,X,\lambda,i,\ell}m^iq_L^{-m\lambda}$$
for all $m\geq 2$. Moreover,  if $\lambda\in \Lambda_r(L,Z,X)$ and  $0\leq i\leq \mathfrak{e}_{r,L}(\lambda)-1$ then either 
\begin{equation}\label{Lnzero}c_{L,\lambda,i}(m):=\sum_{1\leq \ell\leq\ell_{\lambda,i}}\11_{\cB_{\lambda,i,\ell}}(m)c_{r,L,Z,X,\lambda,i,\ell}\neq 0
\end{equation} for infinitely many $m\geq 2$ or $c_{L,\lambda,i}(m)=0$ for all $m\geq 2$. In particular, (\ref{Lnzero}) holds for $\lambda\in \Lambda_r(L,Z,X)$, $i=\mathfrak{e}_{r,L}(\lambda)-1$ and infinitely many $m\geq 2$. Thus, it follows from the definition of $\sigma_r^{\textnormal{abs}}(L,Z,X)$ and $\mathfrak{e}_r^{\textnormal{abs}}(L,Z,X)$ that 
$$\left|\cE_{r,L,Z,X}(m)\right|\geq \frac{1}{2}\bigg|\sum_{1\leq \ell\leq\ell_{\lambda_{L0},\mathfrak{e}_{L0}-1}}\11_{\cB_{\lambda_{L0},\mathfrak{e}_{L0}-1,\ell}}(m)c_{r,L,Z,X,\lambda_{L0},\mathfrak{e}_{L0}-1,\ell}\bigg|m^{\mathfrak{e}_{L0}-1}q_L^{-m\lambda_0}$$
hold if $m$ is large enough. Therefore, we can take 
$$c_{r,L,Z,X}:=\limsup_{m\to +\infty}\frac{1}{2}\bigg|\sum_{1\leq \ell\leq\ell_{\lambda_{L0},\mathfrak{e}_{L0}-1}}\11_{\cB_{\lambda_{L0},\mathfrak{e}_{L0}-1,\ell}}(m)c_{r,L,Z,X,\lambda_{L0},\mathfrak{e}_{L0}-1,\ell}\bigg|>0.$$
Moreover, by the definition of $\mathfrak{e}_{L0}$ and the fact that $\lambda_0=\min\Lambda_r(L,Z,X)$, it is easy to have a positive constant $\tilde{C}_{r,L,Z,X}$ depending only on $r,L,Z,X$ such that
$$m^iq_L^{-m\lambda}\leq \tilde{C}_{r,L,Z,X}m^{\mathfrak{e}_{L0}-1}q_L^{-m\lambda_0}$$
for all $\lambda\in \Lambda_r(L,Z,X)$, all $0\leq i\leq \mathfrak{e}_{r,L}(\lambda)-1$ and all $m\geq 2$. Thus, we can take 
$$C_{r,L,Z,X}:=\tilde{C}_{r,L,Z,X}\sum_{\lambda\in \Lambda_r(L,Z,X)}\sum_{0\leq i\leq \mathfrak{e}_{r,L}(\lambda)-1}\sum_{1\leq \ell\leq\ell_{\lambda,i}}\left|c_{r,L,Z,X,\lambda,i,\ell}\right|>0.$$
\end{proof}

Moreover, we give the following uniform lower bound of  $|\cE_{r,L,Z,X}(m)|$ for  local fields $L$ over $K$ of large enough residue field characteristic and $m\geq 2$ that improves the last part of \cite[Corollary 3.13]{MoiIde}. 
 \begin{prop}\label{ulowbou}Let $r\in\RR$. There is an integer $M$ depending only on $X$ and an integer $e_{X,r}$ depending only on $X,r$ such that for each $L\in\tilde{\cL}_{K,M}$, the inequality
$$|\cE_{r,L,Z,X}(m)|\geq q_L^{e_{X,r}}m^{\mathfrak{e}_r^{\textnormal{abs}}(L,Z,X)-1}q_L^{-m\sigma_r^{\textnormal{abs}}(L,Z,X)}$$
holds for infinitely many $m\geq 2$.
 \end{prop}
\begin{proof}Let us fix $r\in\RR$. By Corollary \ref{transfer2} and Remark \ref{tild}, it is sufficient to work with $L\in\cL_{K,1}$.
We use the idea of \cite[Corollary 3.13]{MoiIde}. In fact, the proof of \cite[Corollary 3.13]{MoiIde} has a minor mistake on using partial fractions. Here, we correct it by using Lemma \ref{separate}. Let $\cA(X),\Lambda(X),n$ be as in Lemma \ref{moexpo}.     By Lemma \ref{separate}, for each function $\ell:\cA(X)\to \NN$, we can write 
$$\frac{1}{\prod_{(a,b,c)\in \cA(X)}(1-q_L^{-sa}q_L^{-b})^{\ell(a,b,c)}}=\sum_{\lambda\in \Lambda(X)}\frac{g_{\lambda,\ell}(q_L,q_L^{-s})}{\prod_{(a,b,c)\in \cA_\lambda(X)}(1-q_L^{-sa}q_L^{-b})^{\ell(a,b,c)}}$$
for all $L\in \cL_{K,1}$, where $g_{\lambda,\ell}(y,t)\in \QQ(y)[t]$ for all $\lambda\in \Lambda(X)$ and $$\cA_\lambda(X):=\{(a,b,c)\in\cA(X)\mid \frac{b}{a}=\lambda+n\}.$$

To expand $\displaystyle \frac{1}{\prod_{(a,b,c)\in \cA_\lambda(X)}(1-t^{a}y^{-b})^{\ell(a,b,c)}}$ as a power series in $t$, we need to count the number of functions $\alpha:\cA_\lambda(X)\to\NN$ satisfying 
$$\sum_{(a,b,c)\in \cA_\lambda(X)}\alpha(a,b,c)a\ell(a,b,c)=m$$
for each $m\in \NN$. It is reduced to computing the Hilbert function of the graded polynomial ring over $\QQ$ generated by variables  $x_{a,b,c}$ of degree  $a\ell(a,b,c)$ for $(a,b,c)\in\cA_{\lambda}(X)$. Therefore, it follows from Propositions \ref{Defm} and \ref{prepol} together with \cite[Proposition 3 and Remark 5]{CaMa} that there is an integer $M$, a positive integer $\ell_{\lambda,i}$, Presburger subsets $\cB_{\lambda,i,\ell}$ of $\NN$, rational numbers $c_{\lambda,i,\ell},c'_{\lambda,i,\ell},b_{\lambda,i,\ell}$ together with $\AA_{\cO_K}^n$-schemes  $Y_{\lambda,i,\ell}$ of finite type and non-zero polynomials $H_{\lambda,i,\ell}(y)\in \QQ[y]$ for $\lambda\in \Lambda(X), 0\leq i\leq \kappa_r(\lambda):=\max\{\mathfrak{e}_{r,L}(\lambda)\mid L\in\cL_{K,1}\}-1$ and $1\leq \ell\leq \ell_{\lambda,i}$, all of which depends only on $X,r$ such that
\begin{equation}\label{expressione}
\cE_{r,L,Z,X}(m)=\sum_{\lambda\in \Lambda_r(X)}\sum_{0\leq i\leq \kappa_r(\lambda)}\sum_{1\leq \ell\leq \ell_{\lambda,i}}\11_{\cB_{\lambda,i,\ell}}(m)\frac{c_{\lambda,i,\ell}+c'_{\lambda,i,\ell}q_L^r}{H_{\lambda,i,\ell}(q_L)}\#Y_{\lambda,i,\ell,L,Z}(k_L)m^{i}q_L^{b_{\lambda,i,\ell}-\lambda m}
\end{equation}for all $L\in \cL_{K,M}$ and all $m\geq 2$,  where $Y_{\lambda,i,\ell,L,Z}= Y_{\lambda,i,\ell}\times_{\AA_{\cO_{K}}^n}Z_{k_L}$ and 
\begin{equation}\label{intlam}b_{\lambda,i,\ell}-\lambda m\in \ZZ
\end{equation} whenever $\11_{\cB_{\lambda,i,\ell}}(m)=1$. Moreover, for $L\in\cL_{K,M}$, $\lambda\in \Lambda(X)$ and $0\leq i\leq \mathfrak{e}_{r,L}(\lambda)-1$, either
\begin{equation}\label{nzero}c_{L,\lambda,i}(m):=\sum_{1\leq \ell\leq \ell_{\lambda,i}}\11_{\cB_{\lambda,i,\ell}}(m)\frac{c_{\lambda,i,\ell}+c'_{\lambda,i,\ell}q_L^r}{H_{\lambda,i,\ell}(q_L)}\#Y_{\lambda,i,\ell,L,Z}(k_L)q_L^{b_{\lambda,i,\ell}}\neq 0
\end{equation}
for infinitely many $m\geq 2$ or $c_{L,\lambda,i}(m)=0$ for all $m\geq 2$. In particular, (\ref{nzero}) holds for  $\lambda\in \Lambda_r(L,Z,X)$, $i=\mathfrak{e}_{r,L}(\lambda)-1$ and infinitely many $m\geq 2$.  Note that (\ref{expressione}) corrects \cite[Equality (13)]{MoiIde}.

Let $L\in \cL_{K,M}$. We set $(\lambda_{L0},\kappa_{L0}):=(\sigma_r^{\textnormal{abs}}(L,Z,X),\mathfrak{e}_r^{\textnormal{abs}}(L,Z,X)-1)$. Since (\ref{nzero}) holds for $\lambda=\lambda_{L0},i=\kappa_{L0}$ and infinitely many $m\geq 2$, there is a non-empty subset $I$ of $\{1,...,\ell_{\lambda_{L0},\kappa_{L0}}\}$ such that $\cB_{I,L}:=\cap_{\ell\in I}\cB_{\lambda_{L0},\kappa_{L0},\ell}\setminus\cup_{\ell\notin I}\cB_{\lambda_{L0},\kappa_{L0},\ell}$ is an infinite set and 
\begin{equation}\label{nze}\sum_{\ell\in I}\frac{c_{\lambda_{L0},\kappa_{L0},\ell}+c'_{\lambda_{L0},\kappa_{L0},\ell}q_L^r}{H_{\lambda_{L0},\kappa_{L0},\ell}(q_L)}\#Y_{\lambda_{L0},\kappa_{L0},\ell,L,Z}(k_L)q_L^{b_{\lambda_{L0},\kappa_{L0},\ell}}\neq 0.
\end{equation}
Let $b_{I,L}$ be the smallest positive rational number in the set  $\cap_{m\in \cB_{I,L}}\left(\lambda_{L0}m+\ZZ\right)$. It follows from (\ref{intlam}) that 
\begin{equation}\label{int}b_{\lambda_{L0},\kappa_{L0},\ell}-b_{I,L}=\beta_\ell\in\ZZ
\end{equation} for all $\ell\in I$. It is easy to see that we can write
\begin{equation}\label{rewpo}\sum_{\ell\in I}\frac{c_{\lambda_{L0},\kappa_{L0},\ell}+c'_{\lambda_{L0},\kappa_{L0},\ell}q_L^r}{H_{\lambda_{L0},\kappa_{L0},\ell}(y)}\#Y_{\lambda_{L0},\kappa_{L0},\ell,L, Z}(k_L)y^{\beta_\ell}=\frac{P_{I,L}(y)+q_L^r\tilde{P}_{I,L}(y)}{Q_{I}(y)}\in \QQ(y)
\end{equation}
such that the following conditions hold:
\begin{itemize}
\item $P_{I,L}(y), \tilde{P}_{I,L}(y), Q_{I}(y)\in\ZZ[y]$, $\left(P_{I,L}(q_L)+q_L^r\tilde{P}_{I,L}(q_L)\right)Q_{I}(q_L)\neq 0$.
\item $\deg(Q_{I}(y))\leq D$ for a constant $D$ depending only on $\beta_{\ell}$ and the degree of $H_{\lambda_{L0},\kappa_{L0},\ell}(y)$ for  $\ell\in I$.
\item  There is a positive constant $C$ depending only on $c_{\lambda_{L0},\kappa_{L0},\ell},c'_{\lambda_{L0},\kappa_{L0},\ell}$ and coefficients of  $H_{\lambda_{L0},\kappa_{L0},\ell}(y)$ for $\ell\in I$ such that the absolute value of any coefficient of $Q_{I}(y)$ is less than $C$.
\end{itemize}
By our construction, $\Xi:=\{(Q_{I}(y),b_{I,L})\in\ZZ[y]\times\QQ\mid L\in\cL_{K,M}, \#\cB_{I,L}=+\infty\}$ is a finite set depending only on $X,r$. Therefore, there is an integer $e$ depending only on $X,r$ such that 
\begin{equation}\label{lbo}\min\Biggl\{\left|P_{I,L}(q_L)+q_L^r\tilde{P}_{I,L}(q_L)\right|,\left|\frac{q_L^{b_{I,L}}}{2Q_{I}(q_L)}\right|\Biggr\}\geq q_L^{\frac{e}{2}}
\end{equation}
for all $L\in\cL_{K,M}$ and all $I$ with $\#\cB_{I,L}=+\infty$. By combining (\ref{expressione}), (\ref{nze}), (\ref{int}), (\ref{rewpo}) and (\ref{lbo}), if $L\in\cL_{K,M}$, $m$ is large enough and $m\in \cB_{I,L}$ with $\#\cB_{I,L}=+\infty$ then 
\begin{align*}\left|\cE_{r,L,Z,X}(m)\right|&\geq \left|\sum_{\ell\in I}\frac{c_{\lambda_{L0},\kappa_{L0},\ell}+q_L^rc'_{\lambda_{L0},\kappa_{L0},\ell}}{2H_{\lambda_{L0},\kappa_{L0},\ell}(q_L)}\#Y_{\lambda_{L0},\kappa_{L0},\ell,L, Z}(k_L)\right|m^{\kappa_{L0}}q_L^{b_{\lambda_{L0},\kappa_{L0},\ell}-\lambda_{L0}m}\\
&=\left|\frac{q_L^{b_{I,L}}\left(P_{I,L}(q_L)+q_L^r\tilde{P}_{I,L}(q_L)\right)}{2Q_{I}(q_L)}\right|m^{\kappa_{L0}}q_L^{-\lambda_{L0}m}\\
&\geq q_L^em^{\kappa_{L0}}q_L^{-\lambda_{L0}m}.
\end{align*}
This completes our proof.
\end{proof} 
Now, we give a uniform upper bound of $|\cE_{r,L,Z,X}(m)|$ for $m\geq 2$. 
\begin{thm}\label{uupp} There is an integer $M$ depending only on $X$ and a constant $C_X$ depending  only on the complexity of $X$,  such that 
$$|\cE_{r,L,Z,X}(m)|\leq C_{X}m^{\mathfrak{e}_r^{\textnormal{abs}}(L,Z,X)-1}(1+q_L^{\sigma_r^{\textnormal{abs}}(L,Z,X)+r})q_L^{-m\sigma_r^{\textnormal{abs}}(L,Z,X)}$$
for all $r\in\RR$, all $L\in\tilde{\cL}_{K,M}$ and all $m\geq 2$. 
\end{thm}

Before proving Theorem \ref{uupp}, we need some preparation. Let $\sigma^{\textnormal{abs}}(L,Z,X)$ be as in Corollary \ref{po1}. Then, there is at most one number $r\in\RR$ such that $\sigma^{\textnormal{abs}}(L,Z,X)\neq \sigma_r^{\textnormal{abs}}(L,Z,X)$. By  Proposition \ref{prepol}, there is a  positive integer $n$ and non-zero ideals $\cI_1,...,\cI_\ell$ of $\cO_K[x_1,...,x_n]$, all of them depend only on $X$, such that the following conditions hold:
\begin{itemize}
\item $\cI_{iK}\neq (1)$ for all $1\leq i\leq \ell$.
\item  For each $1\leq i\leq \ell$, there exists an $\cO_K$-subscheme $Z_i$ of $\Spec(\cO_K[x_1,...,x_n]/\cI_i)$ of finite type such that
\begin{equation}\label{eq-20}\left(1-q_L^{-(s+n-r)}\right)\sum_{1\leq i\leq \ell}\cZ_{L,Z_i,\cI_i}(s)=q_L^{-(n+s)}\left(1-q_L^{-n+r}\right)\#Z(k_L)+\left(1-q_L^s\right)\cZ_{r,L,Z,X}(s+n)
\end{equation}
and 
\begin{equation}\label{eq-3}\#Z(k_L)=\sum_{1\leq i\leq \ell}\#Z_i(k_L)
\end{equation}
for all $L\in\cL_{K,1}$ and all $r\in\ZZ$.
 \end{itemize}
From now on, suppose that the complexity of $X$ is bounded by $e$. Then we can choose $n,\ell,\cI_i$ for $1\leq i\leq\ell$ such that $1\leq n,\ell\leq e$ and  the ideals $\cI_i$ can be generated by at most $e$ polynomials of degree at most $e$. As mentioned in Section \ref{JSL},  for each $1\leq i\leq \ell$, there is a strong log resolution $(Y_i,\pi_i)$ of $\cI_{iK}$ associated to prime divisors $(E_{ij})_{j\in T_i}$ of $Y_i$ and numerical data $((\nu_{ij},N_{ij}))_{j\in T_i}$ such that the complexity of $(Y_i,\pi_i)$ is bounded by a constant $D(e)$ depending only on $e$.
Let $M$ be an  integer such that $(Y_i,\pi_i)$ has good reduction modulo $\cM_L$ for all $1\leq i\leq \ell$ and all $L\in\cL_{K,M}$. By Propositions \ref{Defm} and \ref{prepol}, we have 
\begin{equation}\label{sin}\sigma^{\textnormal{abs}}(L,Z,X)+n=\min\{\frac{\nu_{ij}}{N_{ij}}\mid 1\leq i\leq \ell, j\in T_i, \cC_{\{j\},L,Z_i}\neq 0\}
\end{equation} for all $L\in\cL_{K,M}$. For $1\leq i\leq \ell$ and $L\in\cL_{K,M}$, we set 
\begin{equation}\label{deT}
T_{iL}:=\{j\in T_i\mid \cC_{\{j\},L,Z_i}\neq 0\}.
\end{equation}
Then 
\begin{equation}\label{lesp}\sigma^{\textnormal{abs}}(L,Z,X)+n\leq \frac{\nu_{ij}}{N_{ij}}
\end{equation}
for all $1\leq i\leq \ell$ and all $j\in T_{iL}$. For $L\in\cL_{K,M}$, $1\leq i\leq \ell$ and $I\subset T_{iL}$, we set $$I_L=\{j\in I\mid \frac{\nu_{ij}}{N_{ij}}\geq \sigma_r^{\textnormal{abs}}(L,Z,X)+n\}.$$
As in Section \ref{flaw},  we also use Lemma \ref{separate} to write 
\begin{equation}\label{decomposition00}
\prod_{j\in I}\dfrac{1}{1-t^{N_{ij}}y^{-\nu_{ij}}}=\dfrac{g_{I,L}(y,t)}{\prod_{j\in I_L}(1-t^{N_{ij}}y^{-\nu_{ij}})}+\dfrac{h_{I,L}(y,t)}{\prod_{j\in I\setminus I_L}(1-t^{N_{ij}}y^{-\nu_{ij}})}
\end{equation}
for  $L\in \cL_{K,M}$, $1\leq i\leq \ell$ and $I\subset T_{iL}$ such that
\begin{equation*}
\deg_t\left(\dfrac{h_{I,L}(y,t)}{\prod_{j\in I\setminus I_L}(1-t^{N_{ij}}y^{-\nu_{ij}})}\right)<0 \textnormal{ and }\deg_t\left(\dfrac{g_{I,L}(y,t)}{\prod_{j\in I_L}(1-t^{N_{ij}}y^{-\nu_{ij}})}\right)<0,
\end{equation*} with the convention that 
\begin{itemize}
\item $\prod_{j\in I}(1-t^{N_{ij}}y^{-\nu_{ij}})=1 \textnormal{ if } I=\emptyset,$ 
\item $g_{I,L}(y,t)=1 \textnormal{ and } h_{I,L}(y,t)=0 \textnormal{ if } I=I_L,$
 \item $g_{I,L}(y,t)=0 \textnormal{ and } h_{I,L}(y,t)=1 \textnormal{ if } I_L=\emptyset \textnormal{ and } I\neq\emptyset.$
\end{itemize}

 We have the following lemmas.
\begin{lem}\label{coff}Let $L\in\cL_{K,1}$. We set $t=q_L^{-s}$ and view $\cZ_{L,Z_i,\cI_i}(s)$ as a power series in $t$ for all $1\leq i\leq \ell$. Then we have 
$$q_L^{-mn}\cE_{r,L,Z,X}(m)=\textnormal{Coeff}_{t^{m-1}}\left(\frac{(1-tq_L^{-n+r})\sum_{1\leq i\leq \ell}\cZ_{L,Z_i,\cI_i}(s)}{t-1}\right)+(1-q_L^{-n+r})\sum_{1\leq i\leq\ell}\cZ_{L,Z_i,\cI_i}(0)$$
for all $m\geq 2$.
\end{lem}
\begin{proof}
 By Definition \ref{gzeta}, we have 
 $$\cZ_{r,L,Z,X}(s+n)=\sum_{m\geq 2}q_L^{-mn}\cE_{r,L,Z,X}(m)t^{m}.$$
This together with the equalities  (\ref{eq-20}), (\ref{eq-3}) and Lemma \ref{znpo} yield
\begin{align*}&\left(1-tq_L^{-n+r}\right)\bigg(\sum_{1\leq i\leq \ell}\cZ_{L,Z_i,\cI_i}(s)\bigg)-\left(1-q_L^{-n+r}\right)\bigg(\sum_{1\leq i\leq \ell}\cZ_{L,Z_i,\cI_i}(0)\bigg)\\=&(t-1)\left(1-q_L^{-n+r}\right)\bigg(\sum_{1\leq i\leq \ell}\cZ_{L,Z_i,\cI_i}(0)\bigg)+(1-t^{-1})\bigg(\sum_{m\geq 2}q_L^{-mn}\cE_{r,L,Z,X}(m)t^{m}\bigg).
\end{align*}
Therefore, we have 
\begin{align*}&t\left(1-q_L^{-n+r}\right)\bigg(\sum_{1\leq i\leq \ell}\cZ_{L,Z_i,\cI_i}(0)\bigg)+\bigg(\sum_{m\geq 2}q_L^{-mn}\cE_{r,L,Z,X}(m)t^{m}\bigg)\\=&\frac{t\left(1-tq_L^{-n+r}\right)\sum_{1\leq i\leq \ell}\cZ_{L,Z_i,\cI_i}(s)}{t-1}+\frac{t}{1-t}\left(1-q_L^{-n+r}\right)\bigg(\sum_{1\leq i\leq \ell}\cZ_{L,Z_i,\cI_i}(0)\bigg)\\
=&\frac{t\left(1-tq_L^{-n+r}\right)\sum_{1\leq i\leq \ell}\cZ_{L,Z_i,\cI_i}(s)}{t-1}+\bigg(\sum_{m\geq 1}\left(1-q_L^{-n+r}\right)\sum_{1\leq i\leq \ell}\cZ_{L,Z_i,\cI_i}(0)t^m\bigg).
\end{align*}
This implies our assertion.
 \end{proof}
 \begin{lem}\label{coff1}Let $R(t)=\sum_{i\geq 0}a_it^i$ be a power series converging absolutely at $t=1$ and $q$ be a complex number. Then we have 
 $$\textnormal{Coeff}_{t^m}\left(\frac{(1-qt)R(t)}{t-1}\right)+(1-q)R(1)=\sum_{i\geq m+1}a_i-q\sum_{i\geq m}a_i$$
 for all $m\geq 0$.
 \end{lem}
 \begin{proof}It follows from a basic computation.
 \end{proof}
\begin{prop}\label{equpol}There is a constant $C_{0e}$ depending only on $e$  
such that 
$$|\cE_{r,L,Z,X}(m)|\leq C_{0e}m^{n-1}(1+q_L^{\sigma^{\textnormal{abs}}(L,Z,X)+r})q_L^{-m\sigma^{\textnormal{abs}}(L,Z,X)}$$
for all $r\in\RR$, all $L\in\cL_{K,M}$ and all $m\geq 2$. 
\end{prop}
\begin{proof}Since $(Y_i,\pi_i)$ has good reduction modulo $\cM_L$ for all $1\leq i\leq \ell$ and all $L\in\cL_{K,M}$, we can use Lemma \ref{Defm} to have 
\begin{equation}\label{eq00}
\sum_{1\leq i\leq \ell}\cZ_{L,Z_i,\cI_i}(s)=\sum_{i=1}^\ell\sum_{I\subset T_i}q_L^{-n}\cC_{I,L,Z_i}\prod_{j\in I}\frac{(q_L-1)q_L^{-N_{ij}s-\nu_{ij}}}{1-q_L^{-N_{ij}s-\nu_{ij}}}
\end{equation}
for all $L\in\cL_{K,M}$.  

By (\ref{deT}) together with the definition of $\cC_{I,L,Z_i}$  and the condition $Z_i\subset \Spec(\cO_K[x_1,...,x_n]/\cI_i)$, we have $\cC_{I,L,Z_i}=0$ if either $I\nsubseteq T_{iL}$ or $I=\emptyset$. Thus we have 
\begin{equation}\label{eq01}
\sum_{1\leq i\leq \ell}\cZ_{L,Z_i,\cI_i}(s)=\sum_{i=1}^\ell\sum_{\emptyset\neq I\subset T_{iL}}q_L^{-n}\cC_{I,L,Z_i}\prod_{j\in I}\frac{(q_L-1)q_L^{-N_{ij}s-\nu_{ij}}}{1-q_L^{-N_{ij}s-\nu_{ij}}}
\end{equation}
for all $L\in\cL_{K,M}$. If $1\leq i\leq \ell$ and $\emptyset\neq I\subset T_{iL}$, then it follows from Lemma  \ref{upnu} and (\ref{lesp}) that for all $L\in\cL_{K,M}$ and all $m\geq 2$ we have 
\begin{align}
&\bigg|\sum_{e\geq m}\textnormal{Coeff}_{t^{e}}\left(\prod_{j\in I}\frac{t^{N_{ij}}q_L^{-\nu_{ij}}}{1-t^{N_{ij}}q_L^{-\nu_{ij}}}\right)-q_L^{-n+r}\sum_{e\geq m-1}\textnormal{Coeff}_{t^{e}}\left(\prod_{j\in I}\frac{t^{N_{ij}}q_L^{-\nu_{ij}}}{1-t^{N_{ij}}q_L^{-\nu_{ij}}}\right)\bigg|\nonumber\\
=&\bigg|\sum_{(m_j)_{j\in I}\in \NN_{>0}^{\#I}, \sum_{j\in I}m_jN_{ij}\geq m}q_L^{-\sum_{j\in I}m_j\nu_{ij}}-q_L^{-n+r}\sum_{(m_j)_{j\in I}\in \NN_{>0}^{\#I}, \sum_{j\in I}m_jN_{ij}\geq m-1}q_L^{-\sum_{j\in I}m_j\nu_{ij}}\bigg|\nonumber\\
\leq&\sum_{(m_j)_{j\in I}\in \NN_{>0}^{\#I}, \sum_{j\in I}m_jN_{ij}\geq m}q_L^{-\sum_{j\in I}m_j\nu_{ij}}+q_L^{-n+r}\sum_{(m_j)_{j\in I}\in \NN_{>0}^{\#I}, \sum_{j\in I}m_jN_{ij}\geq m-1}q_L^{-\sum_{j\in I}m_j\nu_{ij}}\nonumber\\
\leq&C_Im^{\#I-1}\left(q_L^{-m(\sigma^{\textnormal{abs}}(L,Z,X)+n)}+q_L^{-n+r}q_L^{-(m-1)(\sigma^{\textnormal{abs}}(L,Z,X)+n)}\right)\nonumber
\end{align}
for a positive constant $C_I$ depending only on $((\nu_{ij},N_{ij}))_{j\in I}$. This together with Lemmas \ref{upc}, \ref{coff}, \ref{coff1} and (\ref{eq01}) imply our assertion.
\end{proof}
\begin{lem}\label{polr}Let $L\in\cL_{K,M}$ and $r\in\RR$ such that $\sigma^{\textnormal{abs}}(L,Z,X)\neq \sigma_r^{\textnormal{abs}}(L,Z,X)$. Then $\sigma^{\textnormal{abs}}(L,Z,X)=-r\in\ZZ$. Moreover,  with the notations of Proposition \ref{Defm}, if $1\leq i\leq \ell$ and $j_1\neq j_2\in T_{iL}$ such that $\frac{\nu_{ij_1}}{N_{ij_1}}=\frac{\nu_{ij_2}}{N_{ij_2}}=n-r$, then  $\cC_{I,L,Z_i}=0$ whenever $\{j_1,j_2\}\subset I$. 

In addition, there is $a\in \QQ^{\times}$ and a polynomial $G_L(t)\in \QQ[t]$ such that 
$$\sum_{1\leq i\leq \ell}\sum_{1\leq i\leq \ell}\cZ_{L,Z_i,\cI_i}(s)=\frac{a}{1-q_L^{-(s+n-r)}}+\frac{G_L(q_L^{-s})}{\prod_{1\leq i\leq \ell, j\in T_{iL}, \frac{\nu_{ij}}{N_{ij}}>n-r}\left(1-q_L^{-N_{ij}s-\nu_{ij}}\right)}.$$

\end{lem}
\begin{proof}It follows from the condition  $\sigma^{\textnormal{abs}}(L,Z,X)\neq \sigma_r^{\textnormal{abs}}(L,Z,X)$ and the definition of $\sigma^{\textnormal{abs}}(L,Z,X)$, $\sigma_r^{\textnormal{abs}}(L,Z,X)$ that  $\sigma^{\textnormal{abs}}(L,Z,X)=-r<\sigma_r^{\textnormal{abs}}(L,Z,X)$ and $r-n$  is a pole of multiplicity $1$ of $$H(s):=\sum_{1\leq i\leq \ell}\sum_{1\leq i\leq \ell}\cZ_{L,Z_i,\cI_i}(s).$$
This together with (\ref{lesp}), (\ref{eq01}) and Lemma \ref{upc} imply that if $1\leq i\leq \ell$ and $\{j_1,j_2\}\subset I\subset T_{iL}$ such that $j_1\neq j_2$ and $\frac{\nu_{ij_1}}{N_{ij_1}}=\frac{\nu_{ij_2}}{N_{ij_2}}=n-r$, then  $\cC_{I,L,Z_i}=0$. Therefore, we can use (\ref{eq01}) and Lemma \ref{separate} to write 
$$H(s)=\frac{G_L(q_L^{-s})}{\prod_{1\leq i\leq \ell, j\in T_{iL}, \frac{\nu_{ij}}{N_{ij}}>n-r}\left(1-q_L^{-N_{ij}s-\nu_{ij}}\right)}+\sum_{1\leq i\leq\ell, j\in T_{iL}, \frac{\nu_{ij}}{N_{ij}}=n-r}\frac{a_{ij}(q_L^{-s})}{1-q_L^{-N_{ij}s-\nu_{ij}}}$$
for polynomials $a_{ij}(t),G_L(t)\in \QQ[t]$ with $\deg(a_{ij}(t))<N_{ij}$ for all $i,j$. 

We write $n-r=\frac{\nu}{N}$ for positive integers $\nu,N$ such that $\gcd(\nu,N)=1$. Let $\phi_1(t),...,\phi_{d}(t)$ be the irreducible factors in $\QQ[t]$ of $$\prod_{1\leq i\leq \ell, j\in T_{iL}, \frac{\nu_{ij}}{N_{ij}}=n-r}(1-t^{N_{ij}}q_L^{-\nu_{ij}})$$
with $\phi_1(t)=1-t^{N}q_L^{-\nu}$. Then the real part of any zero of $\phi_j(q_L^{-s})$ is $r-n$ for all $1\leq j\leq d$. Moreover, if $j\neq i$ then $\phi_{j}(q_L^{-s})$ and $\phi_i(q_L^{-s})$ have no common zero.  
On the other hand, we can find polynomials $h_1(t),...,h_d(t)\in \QQ[t]$ such that 
$$\sum_{1\leq i\leq\ell, j\in T_i, \frac{\nu_{ij}}{N_{ij}}=n-r}\frac{a_{ij}(q_L^{-s})}{1-q_L^{-N_{ij}s-\nu_{ij}}}=\frac{h_1(q_L^{-s})}{1-q_L^{-Ns}q_L^{-\nu}}+\sum_{j=2}^d\frac{h_j(q_L^{-s})}{\phi_i(q_L^{-s})}$$
and $\deg(h_i(t))\leq \deg(\phi_i(t))-1$ for all $1\leq i\leq d$. 

Suppose that $h_j(t)\neq 0$ for some $2\leq j\leq d$. It follows from the above conditions on $\phi_i(t)$ and $h_i(t)$ that $s_0$ is a pole of $H(s)$  if $\phi_j(q_L^{-s_0})=0$. Since $(1-tq_L^{r-n})$ is a factor of $\phi_1(t)$ in $\CC[t]$ and $\gcd(\phi_1(t),\phi_j(t))=1$ in $\QQ[t]$, we have $\gcd(1-tq_L^{r-n},\phi_j(t))=1$ in $\CC[t]$. Thus, $s_0$ is still a pole of
$(1-q_L^{-(s+n-r)})H(s)$
for all zeros $s_0$ of $\phi_j(q_L^{-s})$. So $\sigma^{\textnormal{abs}}(L,Z,X)=-r=-\Re(s_0)-n\geq\sigma_r^{\textnormal{abs}}(L,Z,X)$. We have a contradiction. Therefore, $h_j(t)=0$ for all $2\leq j\leq d$.

Since $r-n$ is a pole of multiplicity $1$ of $H(s)$, we must have $h_1(t)\neq 0$. By the above argument, we also have that $s_0$ is a pole of $H(s)$ for all zeros $s_0$ of $\phi_1(q_L^{-s})$. If $N>1$, then there is a zero $s_1$ of $\phi_1(q_L^{-s})$ such that $s_1$ is not a zero of $1-q_L^{-(s+n-r)}$. Then, $s_1$ is still a pole of $(1-q_L^{-(s+n-r)})H(s)$. Thus, $\sigma^{\textnormal{abs}}(L,Z,X)=-r=-\Re(s_1)-n\geq \sigma_r^{\textnormal{abs}}(L,Z,X)$ and we also get a contradiction.     So we must have $N=1, r=n-\frac{\nu}{N}\in\ZZ$ and $h_1(t)=a$ for a rational number $a\neq 0$ as desired.
\end{proof}
\begin{cor}\label{exrewr}With the assumptions of Lemma \ref{polr}, there is $a_L\in \QQ^{\times}$ and $b_L\in\QQ$ such that 
\begin{align*}\sum_{1\leq i\leq \ell}\sum_{1\leq i\leq \ell}\cZ_{L,Z_i,\cI_i}(s)=&b_L+\frac{a_L}{1-q_L^{-(s+n-r)}}\\&+\sum_{i=1}^\ell\sum_{\emptyset\neq I\subset T_{iL}}\frac{(q_L-1)^{\#I}}{q_L^{n}}\cC_{I,L,Z_i} \sum_{J\subset I, J_L\neq\emptyset}(-1)^{\#I-\#J}\frac{g_{J,L}(q_L,q_L^{-s})}{\prod_{j\in J_L}\left(1-q_L^{-N_{ij}s-\nu_{ij}}\right)},
\end{align*}
where $g_{J,L}$ is given by (\ref{decomposition00}). 
\end{cor}
\begin{proof}We have 
\begin{equation}\label{recom}
\prod_{j\in I}\left(\frac{q_L^{-N_{ij}s-\nu_{ij}}}{1-q_L^{-N_{ij}s-\nu_{ij}}}\right)\\
=\prod_{j\in I}\left(\frac{1}{1-q_L^{-N_{ij}s-\nu_{ij}}}-1\right)
=\sum_{J\subset I}(-1)^{\#I-\#J}\frac{1}{\prod_{j\in J}\left(1-q_L^{-N_{ij}s-\nu_{ij}}\right)}
\end{equation}
for all $1\leq i\leq \ell$ and all $\emptyset\neq I\subset T_{iL}$. Now, we set
$$\mathsf{I}_1=\{(i,j)\mid  1\leq i\leq \ell, j\in T_{iL}, \nu_{ij}=(n-r)N_{ij}\},$$ $$\mathsf{I}_2=\{(i,j)\mid 1\leq i\leq \ell, j\in T_{iL}, \nu_{ij}-(\sigma_r^{\textnormal{abs}}(L,Z,X)+n)N_{ij}<0\}\setminus \mathsf{I}_1,$$ $$\mathsf{I}_3=\{(i,j)\mid 1\leq i\leq \ell, j\in T_{iL}, \nu_{ij}-(\sigma_r^{\textnormal{abs}}(L,Z,X)+n)N_{ij}\geq 0\},$$
$$\mathfrak{N}_1=\sqcup_{(i,j)\in \mathsf{I}_1}\{(\nu_{ij},N_{ij})\},$$ 
$$\mathfrak{N}_2=\sqcup_{(i,j)\in \mathsf{I}_2}\{(\nu_{ij}, N_{ij})\}$$ 
and
$$\mathfrak{N}_3=\sqcup_{(i,j)\in \mathsf{I}_3}\{(\nu_{ij}, N_{ij})\}.$$ Moreover, we choose a map $c_\mathsf{J}:\CC\to \CC\subset\CC[t]$ for each $\mathsf{J}\subset \mathsf{I}_1\sqcup \mathsf{I}_2\sqcup \mathsf{I}_3$ such that the following conditions hold:
\begin{itemize}
\item $c_\mathsf{J}\equiv 0$ if $\emptyset\neq \mathsf{J}\nsubseteq \{(i,j)\mid j\in T_{iL}\}$ for all $1\leq i\leq \ell$.
\item If $\emptyset\neq \mathsf{J}\subset \{(i,j)\mid j\in T_{iL}\}$ for some $1\leq i\leq\ell$ then   $$c_\mathsf{J}(q_L)=\sum_{J=\{j\mid (i,j)\in \mathsf{J}\}\subset I\subset T_{iL}}(-1)^{\#I-\#J}\frac{(q_L-1)^{\#I}}{q^n}\cC_{I,L,Z_i}$$ for  all $L\in \cL_{K,M}$.
\item $$c_\emptyset(q_L)=\sum_{1\leq i\leq \ell}\sum_{\emptyset\neq I\subset T_{iL}}\frac{(1-q_L)^{\#I}}{q_L^{n}}\cC_{I,L,Z_i}$$ for all $L\in\cL_{K,M}$.
\end{itemize} Then, we can use (\ref{decomposition00}), (\ref{eq01}), (\ref{recom}) together with  Proposition \ref{rewr} and Lemma \ref{polr} to complete the proof. 
\end{proof}

\begin{lem}\label{upbou}Let $L\in\cL_{K,1}$. We set $t=q_L^{-s}$. For each $1\leq i\leq \ell$ and $ J\subset T_{iL}$ such that $\emptyset\neq J_L$, we set $$\sigma_{JL}:=\min_{j\in J_L}\frac{\nu_{ij}}{N_{ij}}$$ and $$\gamma_{JL}:=\sum_{j\in J\setminus J_L}\left(N_{ij}\sigma_{JL}-\nu_{ij}\right).$$  Then $\gamma_{J,L}\geq 0$ for all $1\leq i\leq \ell$ and  all $ J\subset T_{iL}$with $\emptyset\neq J_L$. Moreover, there is a constant $C_{1e}$ depending only on $e$ such that 
$$\left|\sum_{e\geq m}\textnormal{Coeff}_{t^e}\left(\frac{g_{J,L}(q_L,q_L^{-s})}{\prod_{j\in J_L}(1-q_L^{-N_{ij}s-\nu_{ij}})}\right)\right|\leq C_{1e}m^{\#J_L-1}q_L^{-m\sigma_{JL}}q_L^{-\gamma_{JL}}$$
for all $m\geq 1$, all $1\leq i\leq \ell$ and all $ J\subset T_{iL}$ with $\emptyset\neq J_L$.
\end{lem}
\begin{proof}Let $1\leq i\leq \ell$ and $J\subset T_{iL}$ with $\emptyset\neq J_L$. By the definition of $J_L$, we have  $\gamma_{JL}\geq 0$. 

Now, we  write $g_{J,L}(y,t)=\sum_{0\leq l\leq d_{JL}}a_{Jl}(y)t^l$ for elements $a_{Jl}(y)\in \QQ(y)$ and $d_{JL}=-1+\sum_{j\in J_L}N_{ij}$. In particular, we have $g_{J,L}(y,t)=1$ and $\gamma_{JL}=0$ if $J=J_L$.   By Lemmas \ref{separate} and \ref{estimatedegree} together with the fact that $1\leq \nu_{ij},N_{ij}\leq D(e)$ for all $i,j$, there is a constant $C_{1e}'$ depending only on $e$ such that 
\begin{align}\label{majoa}
\left|a_{Jl}(q_L)\right|&\leq C_{1e}'q_L^{\max_{j\in J_L}\left(-l\frac{\nu_{ij}}{N_{ij}}-\sum_{j'\in J\setminus J_L} \left(\frac{N_{ij'}\nu_{ij}}{N_{ij}}-\nu_{ij'}\right)\right)}\nonumber\\&\leq C_{1e}'q_L^{-l\sigma_{JL}}q_L^{-\gamma_{JL}}
\end{align}
for all $0\leq l\leq d_{JL}$. Moreover, by a simple calculation, we have 
\begin{align}\label{uppid}
\left|\sum_{e\geq m}\textnormal{Coeff}_{t^e}\left(\frac{g_{J,L}(q_L,q_L^{-s})}{\prod_{j\in J_L}(1-q_L^{-N_{ij}s-\nu_{ij}})}\right)\right|&=\left|\sum_{l=0}^{d_{JL}}a_{Jl}(q_L)\sum_{(m_j)_{j\in J_L}\in B_{L,J,l}}q_L^{-\sum_{j\in J_L}m_j\nu_{ij}}\right|\nonumber\\
&\leq \sum_{l=0}^{d_{JL}}\left|a_{Jl}(q_L)\right|\sum_{(m_j)_{j\in J_L}\in B_{L,J,l}}q_L^{-\sum_{j\in J_L}m_j\nu_{ij}}
\end{align}
where
$B_{L,J,l}=\{(m_j)_{j\in J_L}\in \NN^{\#J_L}\mid l+\sum_{j\in I_L}m_jN_{ij}\geq m\}$.
Thus, the second assertion  follows from Lemma \ref{upnu}, (\ref{majoa}) and (\ref{uppid}).
\end{proof}
\begin{lem}\label{estimation2} Let $ a,b,c\in \CC$ with $b\neq 1$ and $H(t)=c+\frac{a}{1-tb^{-1}}$ then 
$$H(1)+\text{Coeff}_{t^{m-1}}\left(\frac{(t-b)H(t)}{(b-1)(1-t)}\right)=0$$
for all $m\geq 2$.
\end{lem}
\begin{proof}
For all $m\geq 2$, we have 
\begin{align*}
H(1)+\text{Coeff}_{t^{m-1}}\left(\frac{(t-b)H(t)}{(b-1)(1-t)}\right)&=\frac{ab}{b-1}+c+ \text{Coeff}_{t^{m-1}}\left(\frac{(t-b)\left(a+c(1-tb^{-1})\right)}{(b-1)(1-t)(1-tb^{-1})}\right)\\
&=\frac{ab}{b-1}+c+\text{Coeff}_{t^{m-1}}\left(\frac{-ab}{(b-1)(1-t)}-\frac{c}{b-1}-\frac{c}{1-t}\right)\\
&=0.
\end{align*}
\end{proof}
\begin{prop}\label{bigg}Let $L\in\cL_{K,M}$ and $r\in\RR$. Suppose that  $\sigma^{\textnormal{abs}}(L,Z,X)\neq \sigma_r^{\textnormal{abs}}(L,Z,X)$. Then there is a constant $C_{2e}$ depending only on $e$  
such that 
$$|\cE_{r,L,Z,X}(m)|\leq C_{2e}m^{n-1}(1+q_L^{\sigma_r^{\textnormal{abs}}(L,Z,X)+r})q_L^{-m\sigma_r^{\textnormal{abs}}(L,Z,X)}$$
for all $L\in\cL_{K,M}$ and all $m\geq 2$. 
\end{prop}
\begin{proof}We set $t=q_L^{-s}$ as usual. Let $a_L,b_L$ be as in Corollary \ref{exrewr}. It follows from Lemma \ref{estimation2} that
\begin{equation}\label{coeq1}\textnormal{Coeff}_{t^{m-1}}\left(\frac{1-tq_L^{-n+r}}{t-1}\left(b_L+\frac{a_L}{1-tq_L^{-n+r}}\right)\right)+\left(1-q_L^{-n+r}\right)\left(b_L+\frac{a_L}{1-q_L^{-n+r}}\right)=0
\end{equation}
for all $m\geq 2$.
Thus, the assertion follows from (\ref{coeq1}), Corollary \ref{exrewr}, Lemmas \ref{upc}, \ref{coff}, \ref{coff1}, \ref{upbou} together with the fact that the  complexity of the strong  log resolution $(Y_i,\pi_i)$ is bounded by $D(e)$ for all $1\leq i\leq \ell$.
\end{proof}
\begin{proof}[Proof of Theorem \ref{uupp}]It follows from Corollary \ref{po1}, Propositions \ref{equpol} and \ref{bigg} that there is a constant $C_{2e}$ depending only on $e$ such that 
\begin{equation}\label{inen-1}
|\cE_{r,L,Z,X}(m)|\leq C_{2e}m^{n-1}(1+q_L^{\sigma_r^{\textnormal{abs}}(L,Z,X)+r})q_L^{-m\sigma_r^{\textnormal{abs}}(L,Z,X)}
\end{equation}
for all $L\in\cL_{K,M}$ and all $m\geq 2$. 

Let us use the notations in the proof of Proposition \ref{ulowbou}. In particular, recall that we have 
\begin{equation}\label{expressione1}
\cE_{r,L,Z,X}(m)=\sum_{\lambda\in \Lambda_r(X)}\sum_{0\leq i\leq \kappa_r(\lambda)}\sum_{1\leq \ell\leq \ell_{\lambda,i}}\11_{\cB_{\lambda,i,\ell}}(m)\frac{c_{\lambda,i,\ell}}{H_{\lambda,i,\ell}(q_L)}\#Y_{\lambda,i,\ell,L,Z}(k_L)m^{i}q_L^{b_{\lambda,i,\ell}-\lambda m}
\end{equation}for all $L\in \cL_{K,M}$ and all $m\geq 2$. Moreover, we set
\begin{equation}\label{nzero1}c_{L,\lambda,i}(m)=\sum_{1\leq \ell\leq \ell_{\lambda,i}}\11_{\cB_{\lambda,i,\ell}}(m)\frac{c_{\lambda,i,\ell}}{H_{\lambda,i,\ell}(q_L)}\#Y_{\lambda,i,\ell,L, Z}(k_L)q_L^{b_{\lambda,i,\ell}}.
\end{equation}
Note that the set $\cB_{\lambda,i,\ell}$ depends only on counting the number of functions $\alpha: I_l\to\NN$ satisfying 
$$\sum_{j\in I_l}\alpha(j)N_{lj}=m$$
for $1\leq l\leq \ell$ and non-empty subsets $I_l$ of $\{j\in T_l\mid \frac{\nu_{lj}}{N_{lj}}=\lambda+n\}$. Thus, by \cite[Proposition 3]{CaMa}, there is a constant $C_{3e}$ depending only on $e$ such that $c_{L,\lambda,i}(m)|_{\NN_{\geq C_{3e}}} $ is a periodic function of period at most $C_{3e}$ for all $L\in\cL_{K,M}$, all $\lambda\in\Lambda_r(X)$ and all $0\leq i\leq \kappa_r(\lambda)$. Now, we can use (\ref{inen-1}) and adapt the argument in \cite[Claim 3.2.7]{Saskia-Kien} to obtain a constant $C_{4e}$ depending only on $e$ such that for all $L\in\cL_{K,M}$, all  $\lambda\in\Lambda_r(X)$, all $0\leq i\leq \kappa_r(\lambda)$ and all $m\geq C_{3e}$, we have 
\begin{equation}\label{inloc1}
\left|c_{L,\sigma_r^{\textnormal{abs}}(L,Z,X),i}(m)q_L^{-m\sigma_r^{\textnormal{abs}}(L,Z,X)}\right|\leq C_{4e}(1+q_L^{\sigma_r^{\textnormal{abs}}(L,Z,X)+r})q_L^{-m\sigma_r^{\textnormal{abs}}(L,Z,X)}
\end{equation} and
\begin{equation}\label{inloc2}
\left|c_{L,\lambda,i}(m)m^{\kappa_r(\lambda)}q_L^{-\lambda m}\right|\leq C_{4e}(1+q_L^{\sigma_r^{\textnormal{abs}}(L,Z,X)+r})q_L^{-m\sigma_r^{\textnormal{abs}}(L,Z,X)}
\end{equation} if $\lambda>\sigma_r^{\textnormal{abs}}(L,Z,X)$. 
By combining (\ref{inen-1}), (\ref{expressione1}), (\ref{nzero1}), (\ref{inloc1}), (\ref{inloc2}) and using the fact that  the complexity of $(Y_l,\pi_l)$ is bounded by $D(e)$ for all $1\leq l\leq\ell$,  there is a constant $C_e$ depending only on $e$ such that
$$|\cE_{r,L,Z,X}(m)|\leq C_em^{\mathfrak{e}_r^{\textnormal{abs}}(L,Z,X)-1}(1+q_L^{\sigma_r^{\textnormal{abs}}(L,Z,X)+r})q_L^{-m\sigma_r^{\textnormal{abs}}(L,Z,X)}$$
for all $L\in\cL_{K,M}$  and all $m\geq 2$. This together with Corollary \ref{transfer2} and Remark \ref{tild} imply our assertion.
\end{proof}
\begin{cor}\label{coutn}Let $L\in\tilde{\cL}_{K,M}$. Suppose that  $\sigma^{\textnormal{abs}}(L,Z,X)\neq \sigma_r^{\textnormal{abs}}(L,Z,X)$. Then $0<\sigma_r^{\textnormal{abs}}(L,Z,X)+r$. Let $0<\epsilon<\sigma_r^{\textnormal{abs}}(L,Z,X)+r$. There is a constant $C_{\epsilon,X}$ depending only on $\epsilon$ and the complexity of $X$ such that 
$$\left|\frac{\#\pi_{m1}^{-1}\left((X\cap Z)(k_L)\right)}{q_L^{mr}}-\frac{\#(X\cap Z)(k_L)}{q_L^r}\right|\leq C_{\epsilon,X}q_L^{\epsilon-\sigma_r^{\textnormal{abs}}(L,Z,X)-r}$$
for all $m\geq 2$. In particular, if $X_K$ is equi-dimensional and locally a complete intersection having only rational singularities in a neighbourhood of $Z_K$ in $X_K$ then there is a constant $C_X$ depending only on the complexity of $X$ such that 
$$\left|\frac{\#\pi_{m1}^{-1}\left((X\cap Z)(k_L)\right)}{q_L^{m\dim(X_K)}}-\frac{\#(X\cap Z)(k_L)}{q_L^{\dim(X_K)}}\right|\leq C_{X}q_L^{-\lceil\moi_{\dim(X_K)}^{\textnormal{abs}}(K,Z,X)+\dim(X_K)\rceil}$$
for all $L\in\tilde{\cL}_{K,M}$ and all $m\geq 2$. Here, if $a$ is a real number then we denote by $\lceil a\rceil$ the smallest integer $\alpha$ such that $\alpha\geq a$. 
\end{cor}
\begin{proof}Let $L\in\tilde{\cL}_{K,M}$. By Lemma \ref{polr} and the fact that $\sigma^{\textnormal{abs}}(L,Z,X)\neq \sigma_r^{\textnormal{abs}}(L,Z,X)$, we have $$\sigma_r^{\textnormal{abs}}(L,Z,X)+r>0$$ if $\sigma^{\textnormal{abs}}(L,Z,X)\neq \sigma_r^{\textnormal{abs}}(L,Z,X)$.

By the definition, we have 
$$\cE_{r,L,Z,X}(m)=\#\pi_{m1}^{-1}\left((X\cap Z)(k_L)\right)-q_L^{r}\#\pi_{(m-1)1}^{-1}((X\cap Z)(k_L))$$
for all $m\geq 2$. 
Thus, it follows from Theorem \ref{uupp} that 
\begin{align}\label{inqm}&\left|\frac{\#\pi_{m1}^{-1}\left((X\cap Z)(k_L)\right)}{q_L^{mr}}-\frac{\#\pi_{(m-1)1}^{-1}((X\cap Z)(k_L))}{q_L^{r(m-1)}}\right|\nonumber\\ \leq& C_Xm^{\mathfrak{e}_r^{\textnormal{abs}}(L,Z,X)-1}q_L^{-(m-1)(\sigma_r^{\textnormal{abs}}(L,Z,X)+r)}
\end{align}
for all $m\geq 2$ and  a positive constant $C_X$ depending only on the complexity of $X$. Let $0<\epsilon <\sigma_r^{\textnormal{abs}}(L,Z,X)+r$ then there is a positive constant $c_{X,\epsilon}$ depending only on  $\epsilon$ and the complexity of $X$ such that 
\begin{equation}\label{eqma}C_Xm^{\mathfrak{e}_r^{\textnormal{abs}}(L,Z,X)-1}\leq c_{\epsilon,X}q^{\epsilon (m-1)}
\end{equation}
for all $q\geq 2$ and all $m\geq 2$. By combining (\ref{inqm}) and (\ref{eqma}), we have 
\begin{align*}
\left|\frac{\#\pi_{m1}^{-1}\left((X\cap Z)(k_L)\right)}{q_L^{mr}}-\frac{\#(X\cap Z)(k_L)}{q_L^r}\right|&\leq \sum_{i=2}^m\left|\frac{\#\pi_{i1}^{-1}\left((X\cap Z)(k_L)\right)}{q_L^{ir}}-\frac{\#\pi_{(i-1)1}^{-1}((X\cap Z)(k_L))}{q_L^{r(i-1)}}\right|\\
&\leq \sum_{i\geq 1}c_{\epsilon,X}q_L^{i(\epsilon-\sigma_r^{\textnormal{abs}}(L,Z,X)-r)}\\
&=\frac{c_{\epsilon,X}q_L^{\epsilon-\sigma_r^{\textnormal{abs}}(L,Z,X)-r}}{1-q_L^{\epsilon-\sigma_r^{\textnormal{abs}}(L,Z,X)-r}}.
\end{align*}
Thus, it is sufficient to take $$C_{\epsilon,X}=\frac{c_{\epsilon,X}}{1-2^{\epsilon-\sigma_r^{\textnormal{abs}}(L,Z,X)-r}}.$$

Now, we can use the first assertion and Proposition \ref{rati}, and proceed as in \cite[Theorem B]{MoiIde}  to prove the remaining assertion.
\end{proof}

\begin{cor}\label{AvIgnrati}
Let $\cI$ be a non-zero ideal of $\cO_K[x_1,...,x_n]$ such that $\cI_K\neq (1)$. We set $Y=\Spec(\cO_K[x_1,...,x_n]/\cI)$. Let $W$ be an $\cO_K$-subscheme  of $\AA_{\cO_K}^n$ of finite type. Suppose that $\cI$ can be generated by $e$ elements of degree at most $e$. Then there is an integer $M$ depending only on $\cI,W$ and a constant $C_{e,n}$ depending only on $n,e$ such that the following assertions hold for all $\ell\in\ZZ_{>0}$:
\begin{itemize}
\item We have $$\left|E_{L,W,\cI}^{(\ell)}(m)\right|\leq C_{e,n}m^{n-1}q_L^{\moi_{K,W}^{(\ell)}(\cI)-\ell}q_L^{-m\moi_{K,W}^{(\ell)}(\cI)}$$
for all $L\in\tilde{\cL}_{K,M}$ and all $m\geq 2$.
\item If  either $\ell\neq n-\dim_{W_{\overline{K}}\cap Y_{\overline{K}}}(Y_{\overline{K}})$ or $Y_{\overline{K}}\cap U$ has non-rational singularity at some point $P\in (W\cap Y)(\overline{K})$ for any small enough open neighbourhood $U\subset \AA_{\overline{K}}^n$  of $W_{\overline{K}}\cap Y_{\overline{K}}$ satisfying $\dim(U\cap Y_{\overline{K}})=\dim_{W_{\overline{K}}\cap Y_{\overline{K}}}(Y_{\overline{K}})$, then we have 
$$\left|E_{L,W,\cI}^{(\ell)}(m)\right|\leq C_{e,n}m^{n-1}q_L^{-m\moi_{K,W}^{(\ell)}(\cI)}$$
for all $L\in\tilde{\cL}_{K,M}$ and all $m\geq 2$. In other words, Conjecture \ref{avelocIgu} holds in this case.

\end{itemize}
\end{cor}
\begin{proof} By (\ref{intersec}), we can suppose that $W\subset Y$. Let $\ell\in\ZZ_{>0}$. By the definition, we have 
$$q_L^{-mn}E_{L,W,\cI}^{(\ell)}(m)=\cE_{n-\ell,L,W,Y}(m)$$
for all $L\in\tilde{\cL}_{K,1}$ and all $m\geq 2$. By Definition \ref{defmoi0}, Proposition \ref{eqzeta} and Remark \ref{remoi}, we have  $$\moi_{n-\ell}^{\textnormal{abs}}(K,W,Y)=\moi_{K,W}^{\ell,\textnormal{naive}}(\cI)-n\geq \moi_{K,W}^{(\ell)}(\cI)-n.$$ On the other hand, by Lemma \ref{moexpo}, Definition \ref{defmoi} and Remark \ref{tild}, there is an integer $M$ depending only on $\cI,W$ such that 
$$\moi_{n-\ell}^{\textnormal{abs}}(K,W,Y)=\min_{L\in\cL_{K,M}}\sigma_{n-\ell}^{\textnormal{abs}}(L,W,Y)=\min_{L\in\tilde{\cL}_{K,M}}\sigma_{n-\ell}^{\textnormal{abs}}(L,W,Y).$$
Thus, we have \begin{equation}\label{moupb}\moi_{K,W}^{(\ell)}(\cI)\leq n+\sigma_{n-\ell}^{\textnormal{abs}}(L,W,Y)
\end{equation} for all $L\in\tilde{\cL}_{K,M}$. This together with Theorem \ref{uupp}  imply the first assertion. 

The second assertion is trivial if $\moi_{K,W}^{(\ell)}(\cI)=0$. Thus, we can suppose that  $\moi_{K,W}^{(\ell)}(\cI)>0$. By Definition \ref{defmoi0}, $\cI_K$ can be generated by $\ell$ elements. Suppose that either $\ell\neq n-\dim_{W_{\overline{K}}}(Y_{\overline{K}})$ or $Y_{\overline{K}}\cap U$ has non-rational singularities at some point $P\in W(\overline{K})$ for any small enough open neighbourhood $U\subset \AA_{\overline{K}}^n$  of $W_{\overline{K}}$ satisfying $\dim(U\cap W_{\overline{K}})=\dim_{W_{\overline{K}}}(Y_{\overline{K}})$. We will prove that $\moi_{K,W}^{(\ell)}(\cI)\leq \ell$. Indeed, assume that $\moi_{K,W}^{(\ell)}(\cI)>\ell$. Since $\cI_K$ can be generated by $\ell$ polynomials, we have $\dim(V)\geq n-\ell$ for any irreducible component $V$ of $Y_{\overline{K}}$. In particular, $\dim_{W_{\overline{K}}}(Y_{\overline{K}})\geq n-\ell$. So $\ell\geq n-\dim_{W_{\overline{K}}}(Y_{\overline{K}})$. As mentioned in Section \ref{JSL}, we have $\lct_{W_{\overline{K}}}(\cI_{\overline{K}})\leq n-\dim_{W_{\overline{K}}}(Y_{\overline{K}})\leq \ell<\moi_{K,W}^{(\ell)}(\cI)$. It follows from \cite[Item (v), Corollary 3.14]{MoiIde} that $\ell\leq\lct_{W_{\overline{K}}}(\cI_{\overline{K}})$. Thus, we must have $\ell=n-\dim_{W_{\overline{K}}}(Y_{\overline{K}})=\lct_{W_{\overline{K}}}(\cI_{\overline{K}})$. Therefore, there is an open neighbourhood $U$ of $W_{\overline{K}}$ in $\AA_{\overline{K}}^n$ such that $U\cap Y_{\overline{K}}$ is of pure dimension $n-\ell$. Thus, we can use (\ref{moupb}) and Proposition \ref{rati} to see that  $Y_{\overline{K}}\cap U$ has only rational singularity at every point of $W(\overline{K})$. We have a contradiction since $\ell=n-\dim_{W_{\overline{K}}}(Y_{\overline{K}})$. Thus, $\moi_{K,W}^{(\ell)}(\cI)\leq \ell$. This fact and the first assertion imply  the second assertion.
\end{proof}

\begin{cor}\label{Igpol}
Let $\ell\geq 1$ and $f_1,...,f_{\ell}\in\cO_K[x_1,...,x_n]$ be non-constant polynomials of degree at most $d$. We set $\cI=(f_1,...,f_{\ell})$, $g=\sum_{i=1}^{\ell}y_if_i$ and $Y=\Spec(\cO_K[x_1,...,x_n]/\cI)$. Let $W$ be an $\cO_K$-subscheme  of $\AA_{\cO_K}^n$ of finite type. We set $$W^{(\ell)}:=W\times_{\Spec(\cO_K)}(\AA_{\cO_K}^{\ell}\setminus \Spec(\cO_K[y_1,...,y_{\ell}]/(y_1,...,y_{\ell})))$$ as in Section \ref{exp-zetafunction}.  Then there is an integer $M$ depending only on $\cI, W$ and a constant $C_{n,d,\ell}$ depending only on $n,d,\ell$ such that the following assertions hold:
\begin{itemize}
\item We have $$\left|E_{g,L,W^{(\ell)},\psi}\right|\leq C_{n,d,\ell}m^{n-1}q_L^{\moi_{K,W^{(\ell)}}(g)-\ell}q_L^{-m\moi_{K,W^{(\ell)}}(g)}$$
for all $L\in\tilde{\cL}_{K,M}$ and all additive characters $\psi$ of $L$ of conductor at least $2$.
\item If  either  $\ell\neq n-\dim_{W_{\overline{K}}\cap Y_{\overline{K}}}(Y_{\overline{K}})$ or $Y_{\overline{K}}\cap U$ has non-rational singularity at some point $P\in (W\cap Y)(\overline{K})$ for any small enough open neighbourhood $U\subset \AA_{\overline{K}}^n$  of $W_{\overline{K}}\cap Y_{\overline{K}}$ satisfying $\dim(U\cap Y_{\overline{K}})=\dim_{W_{\overline{K}}\cap Y_{\overline{K}}}(Y_{\overline{K}})$, then we have 
$$\left|E_{g,L,W^{(\ell)},\psi}\right|\leq C_{n,d,\ell}m^{n-1}q_L^{-m\moi_{K,W^{(\ell)}}(g)}$$
for all $L\in\tilde{\cL}_{K,M}$ and all additive characters $\psi$ of $L$ of conductor at least $2$. Thus,  Conjecture \ref{Iguconj} holds for $g$ and $W^{(\ell)}$ in this case.

\end{itemize}
\end{cor}
\begin{proof}
It follows immediately from Corollary \ref{AvIgnrati}, Proposition \ref{ridepol} and Definition \ref{defmoi0}. 
\end{proof}
\begin{proof}[Proof of \ref{uupp00} and Corollary \ref{coutp}]
With the notations of \ref{uupp00}, Corollary \ref{coutp} together with  Definitions \ref{defmoi} and \ref{multi}, we have $$\beta_{L,X,l}=\max_{\alpha\in B_{L,X,l}}|\alpha|=q_L^{-\sigma_l^{\textnormal{abs}}(L,X,X)}$$ and
$$\max_{\alpha\in B_{L,X,l}}m^{\mathfrak{e}_{\alpha,l}-1}\left(1+q_L^{l}|\alpha|^{-1}\right)|\alpha|^{m}= m^{\mathfrak{e}_l^{\textnormal{abs}}(L,X,X)-1}\left(1+q_L^{\sigma_l^{\textnormal{abs}}(L,X,X)+l}\right)q_L^{-m\sigma_l^{\textnormal{abs}}(L,X,X)}$$
for all $L\in\cL_{\QQ,1}$. Therefore \ref{uupp00} follows from Propositions \ref{lowe}, \ref{ulowbou} and Theorem \ref{uupp}. On the other hand,  Corollary \ref{coutp} is a consequence of Corollary \ref{coutn}. 
\end{proof}
To finish this section, we prove a weaker version of Conjecture \ref{avelocIgu} in the case that the ideal $\cI$ of $\cO_K[x_1,...,x_n]$ generated by polynomials $f_1,...,f_\ell$ of degree $d$ such that $f_1,...,f_\ell$ define a complete intersection in $\AA_K^n$ of dimension $n-\ell$ as follows. 
\begin{defn}\label{singudim}Let $k$ be a field and $f_1,...,f_\ell\in k[x_1,...,x_n]$  be non-constant polynomials. Suppose that $\deg(f_i)=d$ for all $1\leq i\leq \ell$. For each $1\leq i\leq \ell$, we denote by $f_{id}$ the homogeneous part of highest degree of $f_i$. We set
$$S(f_{1d},...,f_{rd}):=\Bigg\{x\in\overline{k}^n\mid \rk\left(\left(\frac{\partial f_{id}}{\partial x_j}(x)\right)_{1\leq i\leq \ell, 1\leq j\leq n}\right)<\ell\Bigg\}$$
and $s(f_1,...,f_\ell):=\dim(S(f_{1d},...,f_{\ell d}))$.

\end{defn}
\begin{defn}\label{alphaI}
Let $\ell\geq 1$ and $f_1,...,f_{\ell}\in\cO_K[x_1,...,x_n]$ be non-constant polynomials of degree $d$. We set $\cI=(f_1,...,f_{\ell})$ and $Y=\Spec(\cO_K[x_1,...,x_n]/\cI)$. Suppose that $Y_K$ is a complete intersection in $\AA_K^n$ of dimension $n-\ell$. Let $W$ be an $\cO_K$-subscheme  of $\AA_{\cO_K}^n$ of finite type. We set 
$$\alpha_K(\cI,W):=\min\biggl\{\frac{d+1}{d+2}\moi_{K,W}^{(\ell)}(\cI)+\frac{\ell}{d+2},\frac{n-s(f_1,...,f_\ell)}{d}\biggr\}$$
and 
$$\alpha_L(\cI,W):=\min\biggl\{\frac{d+1}{d+2}\sigma_{L,W}^{\ell,\textnormal{naive}}(\cI)+\frac{\ell}{d+2},\frac{n-s(f_1,...,f_\ell)}{d}\biggr\}$$
for each $L\in\cL_{K,1}$. 
\end{defn}
We have the following result.
\begin{prop}\label{uppallf} Let $n,d\geq 1$. Then there is an integer $M$ and a constant $c_{n,d}$, both of them depends only on $n,d$ such that
$$\left|E_{f,L,W,\psi}\right|\leq c_{n,d}q_L^{\frac{-m_\psi n}{d}}$$
for all $\cO_K$-subschemes $W$ of $\AA_{\cO_K}^n$ of finite type, all $L\in\tilde{\cL}_{K,M}$, all additive characters $\psi$ of $L$ of conductor $m_\psi\in[1,d+1]$ and all $f\in \cO_L[x_1,...,x_n]$ of degree $d$  such that the homogeneous part of degree $d$  of $(f\mod \cM_L)$ defines a smooth hypersuface in $\PP_{k_L}^{n-1}$. 
\end{prop}
\begin{proof}
The argument is the same as in the case $s=0$ of \cite[Proposition 3.7]{CluckerNguyen}.
\end{proof}
\begin{prop}\label{uppallide}
Let $\cI=(f_1,...,f_\ell)\subset \cO_K[x_1,...,x_n]$ such that $1\leq \ell\leq n$ and $\cI_K$ defines a complete intersection in $\AA_K^n$ of dimension $n-\ell$. We set $g(x,y)=\sum_{i=1}^\ell y_if_i(x)$. Suppose that $\deg(f_i)=d$ for all $1\leq i\leq n$. Then there is an integer $M$ depending only of $f_1,...,f_\ell$ and a positive constant $c_{n,d}$ depending only on $n,d$ such that 
$$\left|E_{L,W,\cI}^{(\ell)}(m)\right|=\left|E_{g,L,W^{(\ell)},\psi}\right|\leq c_{n,d}q_L^{-\frac{m(n-s(f_1,...,f_\ell))}{d}}$$
for all $L\in\tilde{\cL}_{K,M}$, all additive characters $\psi$ of $L$ of conductor $m_\psi=m\in [1,d+1]$ and  all $\cO_K$-subschemes $W$ of $\AA_{\cO_K}^n$ of finite type. 
\end{prop}
\begin{proof}Let $W$ be an $\cO_K$-subscheme of $\AA_{\cO_K}^n$ of finite type. Let $f_{1d},...,f_{\ell d}$ be the homogeneous part of degree $d$ of $f_1,...,f_\ell$ respectively. We set $s=s(f_1,...,f_\ell)$.  

Suppose that $s=0$. There is an integer $M$ depending only on $f_1,...,f_\ell$ such that  $$s(f_1\mod \cM_L,..., f_\ell\mod \cM_L)=0$$ and $\deg(f_i\mod \cM_L)=d$ for all $L\in\tilde{\cL}_{K,M}$ and all $1\leq i\leq \ell$. Therefore, we can verify that the homogeneous part of degree $d$ of $\overline{g}_{y}(x_1,...,x_n):=g(x,y)\mod \cM_L$ defines a smooth hypersuface in $\PP_{k_L}^{n-1}$ for all $L\in\tilde{\cL}_{K,M}$ and all $y\in\cO_L^\ell\setminus \varpi_L\cO_L^\ell$. Therefore, by using Proposition \ref{uppallf} and fiber integration, there is a constant $c_{n,d}$ depending only on $n,d$ such that 
\begin{align*}
\left|E_{g,L,W^{(\ell)},\psi}\right|=&\left|\int_{(x,y)\in\cO_L^n\times(\cO_L^\ell\setminus\varpi_L\cO_L^\ell)}\Phi_{L,W^{(\ell)}}(x,y)\psi(g(x,y))|dx\wedge dy|\right|\\
=&\left|\int_{\cO_L^\ell\setminus\varpi_L\cO_L^\ell}\left(\int_{\cO_L^n}\Phi_{L,W}(x)\psi(g(x,y))|dx|\right)|dy|\right|\leq \int_{\cO_L^\ell\setminus\varpi_L\cO_L^\ell}c_{n,d}q_L^{-\frac{m_\psi n}{d}}|dy|\leq c_{n,d}q_L^{-\frac{m_\psi n}{d}}
\end{align*}
for all $L\in\tilde{\cL}_{K,M}$ and all additive character $\psi$ of $L$ of conductor $m_\psi\in[1,d+1]$. 

Now, we suppose that $s>0$. By using Bertini's theorem (see, e.g., \cite[Chapter II, Theorem 8.18]{harts}), we can find $s$ linearly independent  forms $l_1,...,l_s\in (K^n)^*$ (depending only on $f_1,...,f_\ell$)  such that they define a linear subspace $H$ of $\AA_K^n$ of dimension $n-s$ satisfying $s(f_1|_{H},...,f_\ell|_{H})=0$. We can suppose that the coefficients of $l_i$  belong to $\cO_K$ for all $1\leq i\leq s$. By choosing a $K$-linear change of variables $z=\theta(x)$ (depending only on $f_1,...,f_\ell$) with coefficients in $\cO_K$, we can suppose that $l_i=z_i$ for all $1\leq i\leq s$. There is an integer $M$ depending only on $f_\ell,...,f_\ell$ such that the matrix of $\theta^{-1}$ can be defined over $\cO_K[\frac{1}{M}]$. We set $\tilde{f}_i=f_i(\theta^{-1}(z))\in \cO_K[\frac{1}{M},z_1,...,z_n]$ and $\tilde{g}=\sum_{1\leq i\leq \ell}y_i\tilde{f}_i$. Let  $\tilde{W}$ be  the subscheme of $\AA_{\cO_K}^n$ obtained from $W$ by this change of variables. Then, it is easy to verify that $$E_{g,L,W^{(\ell)},\psi}=E_{\tilde{g},L,\tilde{W}^{(\ell)},\psi}$$ for all $L\in\tilde{\cL}_{K,M}$ and all additive characters $\psi$ of $L$. Therefore, the assertion holds for $\tilde{f}_1,...,\tilde{f}_\ell$ and $\tilde{W}$ if and only if it holds for $f_1,...,f_\ell$ and $W$. After renaming, we can suppose that $H=\{x_1=...,x_s=0\}$ and $s(f_1|_{H},...,f_\ell|_{H})=0$. Then, there is an integer $M$ depending only on $f_1,...,f_\ell$ such that  $$s(f_1(a,x')\mod \cM_L,..., f_\ell(a,x')\mod \cM_L)=0$$ and $\deg(f_i(a,x')\mod \cM_L)=d$ for all $L\in\tilde{\cL}_{K,M}$, all $1\leq i\leq \ell$ and all $a\in\cO_L^s$, where $x'=(x_{s+1},...,x_n)$. As in the case $s=0$, we can verify that the homogeneous part of degree $d$ of $\overline{g}_{y,a}(x'):=g(a,x',y)\mod \cM_L$ defines a smooth hypersuface in $\PP_{k_L}^{n-s-1}$ for all $L\in\tilde{\cL}_{K,M}$, all $y\in\cO_L^\ell\setminus \varpi_L\cO_L^\ell$ and all $a\in\cO_L^s$. The assertion in this case follows by using Proposition \ref{uppallf} and fiber integration as above.
\end{proof}
\begin{thm}\label{bounalph}
Let $\cI=(f_1,...,f_\ell)\subset \cO_K[x_1,...,x_n]$ such that $1\leq \ell\leq n$ and $\cI_K$ defines a complete intersection in $\AA_K^n$ of dimension $n-\ell$. Suppose that $\deg(f_i)=d$ for all $1\leq i\leq n$. Let $W$ be an $\cO_K$-subscheme of $\AA_{\cO_K}^n$ of finite type. Then there is an integer $M$ depending only of $W,f_1,...,f_\ell$ and a positive constant $c_{n,d}$ depending only on $n,d$ such that 
$$\left|E_{L,W,\cI}^{(\ell)}(m)\right|\leq c_{n,d}m^{n-1}q_L^{-m\alpha_L(\cI,W)}$$
for all $L\in\cL_{K,M}$ and all integers $m\geq 1$. In particular, we have $$\left|E_{L,W,\cI}^{(\ell)}(m)\right|\leq c_{n,d}m^{n-1}q_L^{-m\alpha_K(\cI,W)}$$
for all $L\in\tilde{\cL}_{K,M}$ and all integers $m\geq 1$
\end{thm}
\begin{proof}

By Proposition \ref{eqzeta} and Corollary \ref{transfer2}, the first assertion implies the second assertion. 

Now, we prove the first assertion. It follows from \cite[Corollary 3.14, Item (vi)]{MoiIde} that \begin{equation}\label{lowerlctmoi}
\sigma_{L,W}^{\ell,\textnormal{naive}}(\cI)\geq \lct_{W_K}(\cI_K).
\end{equation}

If $\sigma_{L,W}^{\ell,\textnormal{naive}}(\cI)=+\infty$ then the assertion follows immediately from Definition \ref{alphaI}, Corollary \ref{AvIgnrati} and Proposition \ref{uppallide}.

If $\sigma_{L,W}^{\ell,\textnormal{naive}}(\cI)<\ell$, then we can use (\ref{lowerlctmoi}), \cite[Theorem 1.1 and Remark 4.5]{CDMO} and \cite[Proposition 3.19]{MoiIde} to have 
$$\ell>\sigma_{L,W}^{\ell,\textnormal{naive}}(\cI)\geq \lct_{W_K}(\cI_K)\geq \frac{n-s(f_1,...,f_\ell)}{d}.$$
This together with Definition \ref{alphaI}, Corollary \ref{AvIgnrati} and Proposition \ref{uppallide} imply the assertion in this case. 

If $+\infty>\sigma_{L,W}^{\ell,\textnormal{naive}}(\cI)\geq \ell$, we have 
$$m\left(\frac{d+1}{d+2}\sigma_{L,W}^{\ell,\textnormal{naive}}(\cI)+\frac{\ell}{d+2}\right)\leq (m-1)\sigma_{L,W}^{\ell,\textnormal{naive}}(\cI)+\ell$$
for all $m\geq d+2$. Thus, we can use Definition \ref{alphaI}, Corollary \ref{AvIgnrati} and Proposition \ref{uppallide}  again to deduce the assertion in this case. 
\end{proof}
\begin{proof}[Proof of Corollary \ref{nsdI}] 
By Proposition \ref{eqzeta}, we have 
$$q_L^{n-\sigma_{L,X}^{r,\textnormal{naive}}(\cI)}=\max_{\alpha\in B_{L,X,n-r}}|\alpha|.$$
It implies that 
$$\left(q_L^{\frac{n-r}{d+2}}\max_{\alpha\in B_{L,X,n-r}}|\alpha|^{\frac{d+1}{d+2}}\right)^m=q_L^{m(n-\alpha_L(\cI,X))}.$$
Thus Corollary \ref{nsdI} follows from Definition \ref{alphaI} and Theorem \ref{bounalph}.
\end{proof}

\section{Approach Igusa's conjecture for exponential sums in the case of rational singularities}\label{flaw}
In this section, we fix a number field $K$, a non constant polynomial $f$ in $\cO_K[x_1,...,x_n]$ and an $\cO_K$-subscheme $Z$ of $\AA_{\cO_K}^n$ of finite type such that $f(Z(\overline{K}))=0$. In this section, we assume that $f$ defines a hypersuface in $\AA_K^n$ 
having only rational singularity at $x_0$ for all $x_0\in Z(\overline{K})$.

\begin{lem}\label{min}
There is an integer $M$ depending on $f,Z$ such that $$\min_{L\in\cL_{K,M}}\sigma_L(f,Z)>1,$$ where $\sigma_L(f,Z)$ is defined by (\ref{npol}).
\end{lem}
\begin{proof}
Let $(Y,\pi)$ be a strong log resolution of $(f)$ over $K$ associated to divisors $(E_{i})_{i\in T}$ and numerical data $((\nu_i,N_i))_{i\in T}$. Let $T_0=\{i\in T\mid (\nu_i,N_i)=(1,1)\}$. As in the proof of \cite[Proposition 3.10]{CMN}, we have $\frac{\nu_i}{N_i}>1$ if $\pi(E_i)\cap Z_K\neq\emptyset$ and $i\notin T_0$. Moreover, we also have $\pi(E_i\cap E_j)\cap Z_K=\emptyset$ if $i,j\in T_0$. Note that $(Y_L,\pi_L)$ is a strong log resolution of $(f)$ over $L$ for all $L\in \cL_{K,1}$. Thus, the assertion follows from \cite{Igusa3} (see also \cite[Theorem 1.3.2]{DenefBour}) and (\ref{npol}).
\end{proof}

As mentioned in Section \ref{JSL},  there exists an integer $D$ depending only on $n,\deg(f)$ such that for each $c\in\overline{K}$, there is a strong log resolution of $(f-c)$ over $K[c]$ whose complexity is bounded by $D$. 
We fix such a strong log resolution $(Y,\pi)$ of $(f)$ over $K$. 
Let $(E_{i})_{i\in T(\pi)}$ be the divisors and $((\nu_i,N_i))_{i\in T(\pi)}$ be the numerical data associated to $(Y,\pi)$. Then we can suppose that $1\leq \nu_i,N_i\leq D$ for all $i$ and the complexity of $Y\sqcup(\sqcup_{i\in T(\pi)}E_i)$ is bounded by $D$.  Let $M$ be as in Corollary \ref{zeta} with respect to $(Y,\pi)$. By enlarging $M$ if needed, we can suppose that Lemma \ref{min} holds for such $M$. We set $T_0(\pi)=\{i\in T(\pi)\mid (\nu_i,N_i)=(1,1)\}$ and $T_Z(\pi)=\{i\in T(\pi)\mid \pi(E_i)\cap Z_K\neq \emptyset\}\setminus T_0(\pi)$. Then the strict transform of $(f^{-1}(0))_{\red}$ with respect to $\pi$ is $\cup_{i\in T_0(\pi)}E_i$. For each set $J\subset T(\pi)$ and $L\in \cL_{K,1}$, we set $J_L=\{i\in J\mid \nu_i\geq N_i\sigma_L(f,Z)\}$. By Lemma \ref{min},  we have $(T(\pi))_L\cap T_0(\pi)=\emptyset$  for all $L\in\cL_{K,M}$. As in Section \ref{Opes}, we also use Lemma \ref{separate} to write 
\begin{equation}\label{decomposition0}
\prod_{i\in J}\dfrac{1}{1-t^{N_i}y^{-\nu_i}}=\dfrac{g_{J,L}(y,t)}{\prod_{i\in J_L}(1-t^{N_i}y^{-\nu_i})}+\dfrac{h_{J,L}(y,t)}{\prod_{i\in J\setminus J_L}(1-t^{N_i}y^{-\nu_i})}
\end{equation}
for each set $J\subset T(\pi)$ and each $L\in\cL_{K,1}$ such that
\begin{equation*}
\deg_t\left(\dfrac{h_{J,L}(y,t)}{\prod_{i\in J\setminus J_L}(1-t^{N_i}y^{-\nu_i})}\right)<0 \textnormal{ and }\deg_t\left(\dfrac{g_{J,L}(y,t)}{\prod_{i\in J_L}(1-t^{N_i}y^{-\nu_i})}\right)<0,
\end{equation*} where we use the convention that 
\begin{itemize}
\item $\prod_{i\in J}(1-t^{N_i}y^{-\nu_i})=1 \textnormal{ if } J=\emptyset,$ 
\item $g_{J,L}(y,t)=1 \textnormal{ and } h_{J,L}(y,t)=0 \textnormal{ if } J=J_L,$
 \item $g_{J,L}(y,t)=0 \textnormal{ and } h_{J,L}(y,t)=1 \textnormal{ if } J_L=\emptyset \textnormal{ and } J\neq\emptyset.$
\end{itemize}

\begin{lem}\label{zetanontrivial}
Let $L\in\cL_{K,M}$ and $\chi$ be a non-trivial character of $\cO_L^{\times}$ such that $c(\chi)=1$. Let $c_{I,L,Z,\chi}$ be as in Proposition \ref{zeta} with respect to $(Y,\pi)$. Then we have
\begin{equation}\label{zeta1}
\cZ(f,L,Z,\chi,s)=
\sum_{\substack{\emptyset\neq I\subseteq T_{Z}(\pi)}, \#I\leq n, d(\chi)|N_{i} \forall i\in I}
c_{I,L,Z,\chi} \sum_{J\subset I}(-1)^{\#I-\#J}\frac{g_{J,L}(q_L,q_L^{-s})}{\prod_{i\in J_L}(1-q_L^{-N_{i}s-\nu_{i}})}.
\end{equation}
\end{lem}
\begin{proof}
We do similarly to Corollary \ref{exrewr}. 
By the definition of $T_Z(\pi)$, the definition of strong log resolutions, the facts that $L\in\cL_{K,M}$, $d(\chi)>1$ and $f(Z(\overline{K}))=0$, we have $c_{I,L,Z,\chi}=0$ if one of the following conditions holds:
\begin{itemize}
\item $I=\emptyset$,
\item $\#I>n$,
\item $I\nsubseteq T_Z(\pi)$, 
\item $\emptyset\neq I\subset T_Z(\pi)$ but $d(\chi)\nmid N_i$ for some $i\in I$.
\end{itemize}
Moreover, $\sigma_i=-\frac{\nu_i}{N_i}$ is not the real part of any pole of $\cZ(f,L,Z,\chi,s)$ if $i\notin (T_{Z}(\pi))_L$. Thus, the assertion follows by using Proposition \ref{zeta}, (\ref{recom}), (\ref{decomposition0}) and applying Proposition \ref{rewr} to the sets $\mathsf{I}_1=(T(\pi))_L$, $\mathsf{I}_2=T(\pi)\setminus (T(\pi))_L$, $\mathsf{I}_3=\emptyset$, $\mathfrak{N}_1=\sqcup_{i\in \mathsf{I}_1}\{(\nu_i,N_i)\}$, $\mathfrak{N}_2=\sqcup_{i\in \mathsf{I}_2}\{(\nu_i, N_i)\}$, $\mathfrak{N}_3=\emptyset$ and functions $c_J:\CC\to \CC$ for $J\subset T(\pi)$ satisfying $$c_J(q_L)=\sum_{J\subset I\subset T(\pi)}(-1)^{\#I-\#J}c_{I,L,Z,\chi}$$ for all $J\subset T(\pi)$ and all $L\in \cL_{K,M}$.
 
\end{proof}
\begin{lem}\label{zetatrivial}
Let $L\in\cL_{K,M}$ and  $c_{I,L,Z,\chi_{\textnormal{triv}}}$ be as in Corollary \ref{zeta} with respect to $(Y,\pi)$. Then there are complex numbers $a_L,b_L$ such that 
\begin{equation*}\label{zeta2}
\cZ(f,L,Z,\chi_{\textnormal{triv}},s)=a_L+\frac{b_L}{1-q_L^{-s-1}}+
\sum_{\substack{\emptyset\neq I\subseteq T_Z(\pi)},\#I\leq n}
c_{I,L,Z,\chi_{\textnormal{triv}}} \sum_{J\subset I}(-1)^{\#I-\#J}\frac{g_{J,L}(q_L,q_L^{-s})}{\prod_{i\in J_L}(1-q_L^{-N_{i}s-\nu_{i}})}.
\end{equation*}
\end{lem}
\begin{proof}
As in Lemma \ref{zetanontrivial},  we can use Proposition \ref{zeta}, (\ref{recom}), (\ref{decomposition0}) and apply Proposition \ref{rewr} with the sets $\mathsf{I}_1=(T(\pi))_L, \mathsf{I}_2=T(\pi)\setminus (T_0(\pi)\cup(T(\pi))_L), \mathsf{I}_3=T_0(\pi)$, the multisets  $\mathfrak{N}_1=\sqcup_{i\in \mathsf{I}_1}\{(\nu_i,N_i)\}$, $\mathfrak{N}_2=\sqcup_{i\in \mathsf{I}_2}\{(\nu_i, N_i)\}$, $\mathfrak{N}_3=\sqcup_{i\in \mathsf{I}_3}\{(\nu_i,N_i)\}$, and functions $c_J:\CC\to \CC\subset \CC[t]$ for $J\subset T(\pi)$ satisfying $$c_J(q_L)=\sum_{J\subset I\subset T(\pi)}(-1)^{\#I-\#J}c_{I,L,Z,\chi_{\textnormal{triv}}}$$ for all $J\subset T(\pi)$ and all $L\in \cL_{K,M}$.

\end{proof}
\begin{lem}\label{estimation1}For each set $J\subset T_Z(\pi)$, we set 
$$F_J(t)=\frac{g_{J,L}(q_L,t)}{\prod_{i\in J_L}(1-t^{N_{i}}q_L^{-\nu_{i}})}\in\CC[[t]].$$
There exists a constant $C_0>0$ depending only on $((\nu_i,N_i))_{i\in T(\pi)}$ such that
$$|\textnormal{Coeff}_{t^{m-1}}(F_J(t))|\leq C_0m^{\#J_L-1}q_L^{-(m-1)\sigma_L(f,Z)}$$
for all $L\in \cL_{K,1}$, all $J\subseteq T_Z(\pi)$ and all $m\geq 1$.
\end{lem}
\begin{proof}It follows by repeating the argument of Lemma \ref{upbou}.

\end{proof}

\begin{lem}\label{estimation3}With the notation of Lemma \ref{estimation1}, there exists a constant $C$ depending only on $((\nu_i,N_i))_{i\in T(\pi)}$ such that
$$\left|F_J(1)+\text{Coeff}_{t^{m-1}}\left(\frac{(t-q_L)F_J(t)}{(q_L-1)(1-t)}\right)\right|\leq Cm^{\#J_L-1}q_L^{-(m-1)\sigma_L(f,Z)-1}$$
for all $m\geq 1$, all $L\in\cL_{K,1}$ and all $J\subset T_Z(\pi)$.
\end{lem}
\begin{proof}By (\ref{decomposition0}), if $J_L=\emptyset$ then $F_J(t)\equiv 0$ or $F_J(t)\equiv 1$. 
It is easy to check that the assertion holds in this case whenever $C\geq 2$.

Suppose that $J_L\neq\emptyset$. We firstly observe that
\begin{equation}\label{q0}\frac{t-q_L}{(q_L-1)(1-t)}=-\left(\frac{1}{q_L-1}+\frac{1}{1-t}\right).
\end{equation}
By a basis computation we have
$$F_J(1)-\text{Coeff}_{t^{m-1}}\left(\frac{F_J(t)}{1-t}\right)=\sum_{(\ell,m_i)_{i\in J_L}\in \cB_{m,J,L}}a_{\ell,L}(q_L)q_L^{-\sum_{i\in J_L}m_i\nu_i}$$
where 
$$\cB_{m,J,L}=\{(\ell,m_i)_{i\in J_L}\in\NN^{\#J_L+1}|0\leq \ell\leq\ell_{J,L}-1, \ell+\sum_{i\in J_L}m_iN_i\geq m\},$$
and $g_{J,L}(y,t)=\sum_{\ell=0}^{\ell_{J,L}-1}a_{\ell,L}(y)t^\ell$ with $\ell_{J,L}=\sum_{i\in J_L}N_i$. As in the proof of Lemma \ref{upbou}, there is a constant $C_1$ depending only on $((\nu_i,N_i))_{i\in T(\pi)}$ such that $$a_{\ell,J}(q_L)\leq C_1q_L^{-\ell\sigma_L(f,Z)}$$ for all $0\leq \ell\leq \ell_{J,L}-1$. This and Lemma \ref{upnu} imply that there exists a constant $C_{2}$ depending only on $((\nu_i,N_i))_{i\in T(\pi)}$ such that 
\begin{equation}\label{iq1}\left|F_J(1)-\text{Coeff}_{t^{m-1}}\left(\frac{F_J(t)}{1-t}\right)\right|\leq C_{2}m^{\#J_L-1}q_L^{-m\sigma_L(f,Z)}
\end{equation}
for all $L\in\cL_{K,1}$, all $J\subset T(\pi)$ with $J_L\neq\emptyset$ and all $m\geq 1$.
On the other hand, let $C_0$ be as in Lemma \ref{estimation1} then
\begin{equation}\label{iq2}\left|\text{Coeff}_{t^{m-1}}\left(\frac{F_J(t)}{q_L-1}\right)\right|\leq 2C_{0}m^{\#J_L-1}q_L^{-(m-1)\sigma_L(f,Z)-1}
\end{equation}
for all $L\in\cL_{K,1}$, all $J\subset T(\pi)$ with $J_L\neq\emptyset$ and all $m\geq 1$.
Therefore, by combining (\ref{q0}), (\ref{iq1}), (\ref{iq2}) and putting $C=2C_0+C_2$, we have 
$$\left|F_J(1)+\text{Coeff}_{t^{m-1}}\left(\frac{(t-q_L)F_J(t)}{(q_L-1)(1-t)}\right)\right|\leq Cm^{\#J_L-1}q_L^{-(m-1)\sigma_L(f,Z)-1}$$
for all $L\in\cL_{K,1}$, all $J\subset T(\pi)$ and all $m\geq 1$.
\end{proof}

\begin{prop}\label{auxiliary}
 There are constants $C_L$ for $L\in \cL_{K,1}$, and a constant $C$ depending only on $n,\deg(f)$ such that $C=C_L$ 
if $L\in\cL_{K,M}$ and 
$$|E_{f,L,Z,\psi}|\leq C_L m_\psi^{n-1}q_L^{-(m_\psi-1)\sigma_{L}(f,Z)-\frac{1}{2}}$$
for all $L\in\cL_{K,1}$ and all additive characters $\psi$ of $L$ of conductor $m_\psi\geq 2$. 
\end{prop}
\begin{proof}Note that we chose $(Y,\pi)$ and a constant $D$ depending only on $n,\deg(f)$ such that $1\leq \nu_i,N_i\leq D$ for all $i\in T(\pi)$. Thus the assertion follows by combining (\ref{rboun}), Proposition \ref{zetaex}, Corollary \ref{zeta} and Lemmas \ref{estimation2}, \ref{zetanontrivial}, \ref{zetatrivial},  \ref{estimation1}, \ref{estimation3}. 
\end{proof}
From the above auxiliary results, we can state the main theorem of this section. 
\begin{thm}\label{main}Let $g$ be a non-constant polynomial in $\cO_K[x_1,...,x_n]$ and $W$ be an $\cO_K$-subscheme of $\AA_{\cO_K}^n$ of finite type such that $g-g(x_0)$ defines a hypersurface in $\AA_{\overline{K}}^n$ having only rational singularity at $x_0$ for all $x_0\in W(\overline{K})$. Then, 
there is an integer $M$ depending only on $g,W$, a constant $C$ depending only on $n,\deg(g)$, and constants $C_L$ for each $L\in\cL_{K,1}$ such that $C=C_L$  for all $L\in\cL_{K,M}$ and 
$$|E_{g,L,W,\psi}|\leq C_Lm_\psi^{n-1}q_L^{-(m_\psi-1)\sigma_L(g,W)-\frac{1}{2}}$$
for all $L\in \cL_{K,1}$ and all additive characters $\psi$ of $L$ of conductor $m_\psi\geq 2$.
\end{thm}
\begin{proof}
The assertion follows from (\ref{rboun}), (\ref{decomsum}), (\ref{mimoi}) and Proposition \ref{auxiliary}.
\end{proof}
\begin{cor}\label{upbconja}
Let $g\in\cO_K[x_1,...,x_n]$ be a non-constant polynomial and $W$ be an $\cO_K$-subscheme of $\AA_{\cO_K}^n$ of finite type. Then there are positive constants $M,C_1$ depending only on $g, W$ and a positive constant $C_2$ depending only on $n,\deg(g)$ such that
$$|E_{g,L,W,\psi}|\leq m_\psi^{n-1}\min\biggl\{C_1q_L^{-m_\psi\min_{a\in \overline{K}}\lct_{W_{\overline{K}}}(g-a)},C_2q_L^{-(m_\psi-1)\moi_K(g,W)-\frac{1}{2}}\biggr\}$$
for all $L\in \tilde{\cL}_{K,M}$ and all additive characters $\psi$ of $L$ of conductor $m_\psi\geq 2$. 
\end{cor}
\begin{proof}Note that by the calculation in \cite[Section 4]{Saskia-Kien}, Theorem  \ref{main} also holds for $g,W$  if $g-g(x_0)$ defines a hypersurface in $\AA_{\overline{K}}^n$ having non-rational singularity at $x_0$ for some $x_0\in W(\overline{K})$.
Thus, the assertion follows from (\ref{decomsum}), (\ref{mimoi}), (\ref{moif}), Theorems \ref{rCMN}, \ref{main} and  Proposition \ref{transfer1}.
\end{proof}
The following corollary generalizes Corollary \ref{Igcoexp}.
\begin{cor}\label{bounnsd}Let $g$ be as in Corollary \ref{upbconja}. We suppose that $\deg(g)=d$. 
Let $W$ be an $\cO_K$-subscheme of $\AA_{\cO_K}^n$ of finite type.  We set 
$$\beta_K(g,W):=\min\biggl\{\frac{d+1}{d+2}\moi_K(g,W)+\frac{1}{2(d+2)},\frac{n-s(g)}{d}\biggr\},$$
where $s(g)$ is defined in Definition \ref{singudim}. 
 Then there is an integer $M$ depending only on $f,W$ and a positive constant $C_{n,d}$ depending only on $n,d$ such that
$$|E_{g,L,W,\psi}|\leq C_{n,d}m_\psi^{n-1}q_L^{-\beta_K(g,W)m_\psi}$$
for all $L\in\tilde{\cL}_{K,M}$ and all additive characters $\psi$ of $L$ of conductor at least $1$.
\end{cor}
 \begin{proof}
 By Proposition \ref{uppallf}, there are positive constants $c_{n,d},M$ depending only on $n,d$ such that 
 \begin{equation}\label{ineqd1}|E_{g,L,W,\psi}|\leq c_{n,d}m_\psi^{n-1}q_L^{-\frac{n-s(g)}{d}m_\psi}\leq c_{n,d}m_\psi^{n-1}q_L^{-\beta_K(g,W)m_\psi}
 \end{equation}
for all $L\in\tilde{\cL}_{K,M}$ and all additive characters $\psi$ of $L$ of conductor $m_\psi\in [1,d+1]$.

Suppose firstly that $\moi_K(g,W)\leq\frac{1}{2}$. Then it follows from Theorem \ref{rCMN} that $$\moi_K(g,W)=\min_{a\in\overline{K}^n}\lct_{W_{\overline{K}}}(g-a).$$
 Thus, we can use the argument of \cite[Proposition 2.4]{CluckerNguyen} to have 
 $$\frac{n-s(g)}{d}\leq \min_{a\in\overline{K}^n}\lct_{W_{\overline{K}}}(g-a)=\moi_K(g,W)\leq \frac{1}{2}.$$
This and  Corollary  \ref{upbconja} imply that there is an integer $M$ depending only on $g, W$ and a positive constant  $c'_{n,d}$ depending only on $n,d$ such that 
\begin{equation}\label{ineqd12}|E_{g,L,W,\psi}|\leq c'_{n,d}m_\psi^{n-1}q_L^{-\moi_K(g,W)(m_\psi-1)-\frac{1}{2}}\leq c'_{n,d}m_\psi^{n-1}q_L^{-\frac{n-s(g)}{d}m_\psi}
\end{equation}
for all $L\in\tilde{\cL}_{K,M}$ and all additive characters $\psi$ of $L$ of conductor at least $2$.  Thus, we can deduce the assertion in this case by using  (\ref{ineqd1}) and (\ref{ineqd12}).
 
 Now, we suppose that $\moi_K(g,W)>\frac{1}{2}$. Then  it is easy to verify that 
 $$m\beta_K(g,W)\leq m\left(\frac{d+1}{d+2}\moi_K(g,W)+\frac{1}{2(d+2)}\right)\leq (m-1)\moi_K(g,W)+\frac{1}{2}$$
 for all $m\geq d+2$. This fact together with Corollary \ref{upbconja} and (\ref{ineqd1}) imply the assertion in this case. 
 \end{proof}

\section{Lower bound of the motivic oscillation indexes}\label{lobmoti}In this section, we fix  positive integers $n,d$ together with a number field  $K$ and an $\cO_K$-subscheme $Z$ of $\AA_{\cO_K}^{n}$ of finite type.
\subsection{A general bound on exponential sums over non-Archimedean local fields}
We fix  a local field  $L\in \tilde{\cL}_{K,d+1}$. We view elements of $L^n$ as row vectors. Let $h$ be a polynomial of degree $d$ in variables $x_1,...,x_{n}$ over $\cO_L$.   In this subsection, we will give an upper bound of the integrals 
\begin{equation}\label{exph}\tilde{E}_{L,h,Z,m,\psi}=\int_{\{x\in \cO_L^{n}\mid \ord_L(h(x))\geq m-1\}}\Phi_{L,Z}(x)\psi(h(x))|dx|
\end{equation}
for additive characters $\psi$ of $L$ of conductor $m\geq 2$.

First of all, we need to prepare some notations and notions related to matrices over local rings. 
\begin{defn} \label{orZZ}Let $r$ be a positive integer. We denote by $\ZZ_{\ord}^r$ the set of all strictly ordered sets $\{i_1<i_2<...<i_r\}$ for integers $i_1,...,i_r$. We set $|\mathsf{S}|=r$ if $\mathsf{S}\in\ZZ_{\ord}^r$.  If $\mathsf{S}=\{i_1<...<i_r\}\in \ZZ_{\ord}^r$ then we set $\Upsilon(\mathsf{S})=\{i_1,...,i_r\}\subset\ZZ$. Conversely, each subset $S$ of $\ZZ$ with $r$ elements is associated  to a unique element $\Xi(S)\in\ZZ_{\ord}^r$ satisfying $\Upsilon(\Xi(S))=S$.  We endow 	$\ZZ_{\ord}^r$ with the totally order given by $\mathsf{I}=\{i_1<...<i_r\}<\mathsf{J}=\{j_1<...<j_r\}$ if there is $1\leq \ell\leq r$ such that $i_l=j_l$ if $l<\ell$ and $i_\ell<j_\ell$. If $r'$ is a positive integer with $r'<r$, $\mathsf{S}=\{i_1<i_2<...<i_r\}\in \ZZ_{\ord}^r$, $\mathsf{S}'=\{j_1<j_2<...<j_{r'}\}\in \ZZ_{\ord}^{r'}$, then we say that $\mathsf{S}'$ is a subset of $\mathsf{S}$ and write $\mathsf{S}'\subset \mathsf{S}$ if $\{j_1,...,j_{r'}\}\subset\{i_1,...,i_r\}$. If this is the case,  we define $\mathsf{S}\setminus\mathsf{S}'=\Xi(\{i_1,...,i_r\}\setminus\{j_1,...,j_{r'}\})\in\ZZ_{\ord}^{r-r'}$ and write $\mathsf{S}=\left(\mathsf{S}\setminus\mathsf{S}'\right)\sqcup\mathsf{S}'$.


\end{defn}  
\begin{defn}\label{matrS}Let $\mathsf{S}\in\ZZ_{\ord}^n$. If $R$ is a ring, we denote by $\MM_{\mathsf{S}}(R)$ the ring of $n\times n$ matrices with entries in $R$ whose rows and columns are indexed by $\mathsf{S}$. We also denote by $\Sm_\mathsf{S}(R)$ the set of symmetric matrices in  $\MM_{\mathsf{S}}(R)$. If $\mathsf{S}=\{1<2<...<n\}$ then we write $\MM_{n}(R)$ and  $\Sm_n(R)$ instead of $\MM_{\mathsf{S}}(R)$ and $\Sm_\mathsf{S}(R)$ respectively. If $\mathsf{N}\in\MM_{\mathsf{S}}(R)$ and $\mathsf{I}$ is a non-empty subset  of $\mathsf{S}$, we denote by $\mathsf{N}_\mathsf{I}$  the principal submatrix of $\mathsf{N}$ given by deleting its $i^{\textnormal{th}}$-row and $i^{\textnormal{th}}$-column if $i\notin \mathsf{I}$. Let $\mathsf{N}=(z_{ij})_{i,j\in \mathsf{S}}\in \MM_\mathsf{S}(\cO_L)$ be an $n\times n$ matrix. We set $\ord_L(\mathsf{N}):=\min_{i,j\in\mathsf{S}}\ord_L(z_{ij})\leq +\infty$ with the convention that $\ord_L(\mathsf{N})=+\infty$ if $\mathsf{N}=0$. If $\mathsf{N}\neq 0$,  we set $\mathsf{N}^{\ac}:=\varpi_L^{-\ord_L(\mathsf{N})}\mathsf{N}\mod \varpi_L\cO_L\in \MM_{\mathsf{S}}(k_L)$ and $\rank_{\ac}(\mathsf{N}):=\rank(\mathsf{N}^{\ac})$. If $\mathsf{N}=0$ we also set $\mathsf{N}^{\ac}=0$ and $\rank_{\ac}(\mathsf{N})=0$.


\end{defn} 

\begin{defn}\label{inste}Let $\mathsf{S}\in\ZZ_{\ord}^n$ and $0\neq\mathsf{N}=(z_{ij})_{i,j\in\mathsf{S}}\in\Sm_\mathsf{S}(\cO_L)$. Let $r=\rank_{\ac}(\mathsf{N})$ and  $\mathsf{I}=\mathsf{I}(\mathsf{N})=\{i_1<...<i_r\}$ be the smallest element of $\ZZ_{\ord}^r$  such that  $\mathsf{I}\subset \mathsf{S}$ and $r=\rank_{\ac}(\mathsf{N}_\mathsf{I})$.  
By the definition of $r$ and $\mathsf{I}$, the rows of $\mathsf{N}_\mathsf{I}$ form a basis of the free $\cO_L$-module $\varpi_L^{\ord_L(\mathsf{N})}\cO_L^r$. Thus, for each $j\in \mathsf{S}\setminus \mathsf{I}$, there exists a unique element $(a_{jj'}(\mathsf{N}))_{j'\in \mathsf{I}}\in\cO_L^r$ such that 
\begin{equation}\label{linedep}(z_{ji_{1}},...,z_{ji_{r}})=\sum_{j'\in \mathsf{I}}a_{jj'}(\mathsf{N})(z_{j'i_{1}},...,z_{j'i_{r}}).
\end{equation} 
We also put $a_{jj'}(\mathsf{N})=0$ if $j,j'\in \mathsf{I}$. We set $\tilde{\mathsf{N}}=(\tilde{z}_{j\ell})_{\ell,j\in\mathsf{S}}$ with 
\begin{equation}\label{zij}\tilde{z}_{j\ell}=\begin{cases}z_{j\ell}-\sum_{j'\in \mathsf{I}}a_{jj'}(\mathsf{N})z_{j'\ell} \textnormal{ if either } j\notin\mathsf{I} \textnormal{ or } \ell\in\mathsf{I},\\
z_{j\ell}-\sum_{j'\in \mathsf{I}}a_{\ell j'}(\mathsf{N})z_{jj'} \textnormal{ if } \ell\notin\mathsf{I} \textnormal{ and } j\in\mathsf{I}.
\end{cases}
\end{equation}

Moreover, if $r<n$ then we set $\rho(\mathsf{N})=\tilde{\mathsf{N}}_{\mathsf{S}\setminus \mathsf{I}}$.

\end{defn}
\begin{lem}\label{symat}Let $\mathsf{S}\in \ZZ_{\ord}^n$ and $0\neq\mathsf{N}=(z_{ij})_{i,j\in\mathsf{S}}\in\Sm_\mathsf{S}(\cO_L)$ such that $\rank_{\ac}(\mathsf{N})<n$. We use the notations of Definition \ref{inste}. Then we have $\mathsf{B}\mathsf{N}\mathsf{B}^T=\tilde{\mathsf{N}}$, where $\mathsf{B}=(b_{ij})_{i,j\in \mathsf{S}}$ with $$b_{ij}=\begin{cases} 1 \textnormal{ if } i=j,\\
0 \textnormal{ if } i\neq j \textnormal{ and }  j\notin \mathsf{I}(\mathsf{N}),\\
-a_{ij}(\mathsf{N}) \textnormal{ if } i\neq j \textnormal{ and } j\in \mathsf{I}(\mathsf{N}).\end{cases}$$
In particular, $\mathsf{N}_{\mathsf{I}(\mathsf{N})}=\tilde{\mathsf{N}}_{\mathsf{I}(\mathsf{N})}$, $\rho(\mathsf{N})$ is a symmetric matrix, $\ord_L(\rho(\mathsf{N}))>\ord_L(\mathsf{N})$ and $$\det(\mathsf{N})=\det(\rho(\mathsf{N}))\times\det(\mathsf{N}_{\mathsf{I}(\mathsf{N})}).$$
Moreover, if $\mathsf{S}=\{1<2<...<n\}$ and $\mathsf{I}(\mathsf{N})=\{1<2<...<r\}$  then $\tilde{\mathsf{N}}$ is the diagonal block matrix $(\mathsf{N}_{\mathsf{I}(\mathsf{N})},\rho(\mathsf{N}))$.
\end{lem}
\begin{proof} We set $\mathsf{B}\mathsf{N}\mathsf{B}^T=(y_{ij})_{i,j\in\mathsf{S}}$. We will show that $y_{ij}=\tilde{z}_{ij}$ for all $i,j\in\mathsf{S}$. 

Note that $a_{l\ell}(\mathsf{N})=0$ if $l,\ell\in\mathsf{I}(\mathsf{N})$. Thus 
\begin{equation}\label{bdelta}
b_{l\ell}=\delta_{l\ell}
\end{equation} if $l,\ell\in\mathsf{I}(\mathsf{N})$, where $\delta_{l\ell}$ is the Kronecker delta. These facts and the definition of $\mathsf{B}$ imply 
\begin{equation}
y_{ij}=\sum_{l\in\mathsf{S}}\sum_{\ell\in\mathsf{S}}b_{il}z_{l\ell}b_{j\ell}=\sum_{l\in\mathsf{I}(\mathsf{N})}\sum_{\ell\in\mathsf{I}(\mathsf{N})}b_{il}z_{l\ell}b_{j\ell}=z_{ij}=z_{ij}-\sum_{\ell\in \mathsf{I}(\mathsf{N})}a_{i\ell}(\mathsf{N})z_{\ell j}=\tilde{z}_{ij}.
\end{equation}
for all $i,j\in\mathsf{I}(\mathsf{N})$.

Since $\mathsf{N}$ is symmetric, we have $y_{ij}=y_{ji}$ for all $i,j\in \mathsf{S}$. Moreover, we also have $\tilde{z}_{ij}=\tilde{z}_{ji}$ for all $j\in\mathsf{I}(\mathsf{N}), i\notin\mathsf{I}(\mathsf{N})$. Thus, it follows from the definition of $\mathsf{B}$,  (\ref{linedep}), (\ref{zij}) and (\ref{bdelta}) that 
\begin{align*}
y_{ji}=y_{ij}=\sum_{l\in\mathsf{S}}\sum_{\ell\in\mathsf{S}}b_{il}z_{l\ell}b_{j\ell}&=\bigg(\sum_{l\in\mathsf{I}(\mathsf{N})}\sum_{\ell\in\mathsf{I}(\mathsf{N})}b_{il}z_{l\ell}b_{j\ell}\bigg)+\bigg(\sum_{\ell\in\mathsf{I}(\mathsf{N})}b_{ii}z_{i\ell}b_{j\ell}\bigg)\\
&=\bigg(\sum_{l\in\mathsf{I}(\mathsf{N})}b_{il}z_{lj}\bigg)+z_{ij}=\bigg(\sum_{l\in\mathsf{I}(\mathsf{N})}-a_{il}(\mathsf{N})z_{lj}\bigg)+z_{ij}=\tilde{z}_{ij}=\tilde{z}_{ji}=0
\end{align*}
if $j\in\mathsf{I}(\mathsf{N})$ and $i\notin\mathsf{I}(\mathsf{N})$.

If $i,j\notin\mathsf{I}(\mathsf{N})$ then we use (\ref{linedep}), (\ref{zij}) and the definition of $\mathsf{B}$ to have 
\begin{align*}
y_{ij}=\sum_{l\in\mathsf{S}}\sum_{\ell\in\mathsf{S}}b_{il}z_{l\ell}b_{j\ell}&=\bigg(\sum_{l\in\mathsf{I}(\mathsf{N})}\sum_{\ell\in\mathsf{I}(\mathsf{N})}b_{il}z_{l\ell}b_{j\ell}\bigg)+\bigg(\sum_{\ell\in\mathsf{I}(\mathsf{N})}b_{ii}z_{i\ell}b_{j\ell}\bigg)+ \bigg(\sum_{l\in\mathsf{I}(\mathsf{N})}b_{il}z_{lj}b_{jj}\bigg)+b_{ii}z_{ij}b_{jj}
\\&=\bigg(\sum_{l\in\mathsf{I}(\mathsf{N})}\sum_{\ell\in\mathsf{I}(\mathsf{N})}a_{il}(\mathsf{N})z_{l\ell}a_{j\ell}(\mathsf{N})\bigg)-\bigg(\sum_{\ell\in\mathsf{I}(\mathsf{N})}z_{i\ell}a_{j\ell}(\mathsf{N})\bigg)- \bigg(\sum_{l\in\mathsf{I}(\mathsf{N})}a_{il}(\mathsf{N})z_{lj}\bigg)+z_{ij}\\
&=\bigg(\sum_{\ell\in\mathsf{I}(\mathsf{N})}z_{i\ell}a_{j\ell}(\mathsf{N})\bigg)-\bigg(\sum_{\ell\in\mathsf{I}(\mathsf{N})}z_{i\ell}a_{j\ell}(\mathsf{N})\bigg)- \bigg(\sum_{l\in\mathsf{I}(\mathsf{N})}a_{il}(\mathsf{N})z_{lj}\bigg)+z_{ij}\\
&=-\bigg(\sum_{l\in\mathsf{I}(\mathsf{N})}a_{il}(\mathsf{N})z_{lj}\bigg)+z_{ij}=\tilde{z}_{ij}.
\end{align*}
In summary, $\tilde{\mathsf{N}}$ equals  the symmetric matrix $\mathsf{B}\mathsf{N}\mathsf{B}^T$. Thus, $\rho(\mathsf{N})=\tilde{\mathsf{N}}_{\mathsf{S}\setminus\mathsf{I}(\mathsf{N})}$ is also a symmetric matrix. By the definition and the above calculation, it is easy to verify that $\det(\mathsf{B})=1$, $\left(\mathsf{B}\mathsf{N}\mathsf{B}^T\right)^{\ac}=\tilde{\mathsf{N}}^{\ac}$, $\mathsf{N}_{\mathsf{I}(\mathsf{N})}=\tilde{\mathsf{N}}_{\mathsf{I}(\mathsf{N})}$, $\rk_{\ac}(\mathsf{N})=\rk_{\ac}(\tilde{\mathsf{N}})=\rk_{\ac}\left(\tilde{\mathsf{N}}_{\mathsf{I}(\mathsf{N})}\right)$. These and the fact that $\tilde{z}_{ij}=\tilde{z}_{ji}=0$ if $j\in\mathsf{I}(\mathsf{N})$ and $i\notin\mathsf{I}(\mathsf{N})$ imply  the last assertion and  $$\det(\mathsf{N})=\det(\mathsf{B})^2\det(\mathsf{N})=\det(\tilde{\mathsf{N}})=\det(\rho(\mathsf{N}))\det(\tilde{\mathsf{N}}_{\mathsf{I}(\mathsf{N})})=\det(\rho(\mathsf{N}))\det(\mathsf{N}_{\mathsf{I}(\mathsf{N})}).$$
Moreover, we also have
$$\left(\tilde{\mathsf{N}}^{\ac}\right)_{\mathsf{S}\setminus\mathsf{I}(\mathsf{N})}=0.$$
Therefore, $\ord_L(\mathsf{N})<\ord_L(\rho(\mathsf{N}))$. 
\end{proof}
\begin{defn}\label{trnmatri}Let $\mathsf{S}=\{1<2<...<n\}$ and $\mathsf{N}=(z_{ij})_{1\leq i,j\leq n}\in\Sm_n(\cO_L)$. By using Definitions \ref{orZZ}, \ref{matrS}, \ref{inste} together with Lemma \ref{symat} and induction, we construct a finite sequence of tuples $$(\mathsf{S}_i,\mathsf{N}^{(i)},r_i(\mathsf{N}),\mu_i(\mathsf{N}))\in \ZZ_{\ord}^{n-\sum_{1\leq l\leq i-1}r_l(\mathsf{N})}\times \Sm_{\mathsf{S}_i}(\cO_L)\times \NN_{\geq 1}\times \left(\NN\cup\{+\infty\}\right)$$ for $i\geq 1$ and a matrix $(a_{ij}(\mathsf{N}))_{i,j\in\mathsf{S}}\in \MM_{\mathsf{S}}(\cO_L)$ as follows. Firstly, we set $\mathsf{S}_1=\mathsf{S}$, $\mathsf{N}^{(1)}=\mathsf{N}$. For each $i\geq 1$, if we have already constructed   $\mathsf{S}_i\in \ZZ_{\ord}^{n-\sum_{1\leq l\leq i-1}r_l(\mathsf{N})}$ with $\mathsf{S}_i\subset \mathsf{S}$ and  a matrix $\mathsf{N}^{(i)}\in\Sm_{\mathsf{S}_i}(\cO_L)$ then we set $\mu_i(\mathsf{N})=\ord_L\left(\mathsf{N}^{(i)}\right)$ and $$r_i(\mathsf{N})=\begin{cases}\rank_{\ac}\left(\mathsf{N}^{(i)}\right) \textnormal{ if } \mathsf{N}^{(i)}\neq 0,\\ |\mathsf{S}_i| \textnormal{ if } \mathsf{N}^{(i)}=0.\end{cases}$$ If $r_i(\mathsf{N})=|\mathsf{S}_i|$ then we end the construction and define $e(\mathsf{N}):=i$, $a_{jl}(\mathsf{N}):=0$ for all $j\in\mathsf{S}$ and all $l\in \mathsf{S}_i$. If $r_i(\mathsf{N})<|\mathsf{S}_i|$ then we set \begin{itemize}
\item $\mathsf{I}_i(\mathsf{N})=\mathsf{I}(\mathsf{N}^{(i)})$, 
\item $\mathsf{S}_{i+1}=\mathsf{S}_i\setminus \mathsf{I}_i(\mathsf{N})$, 
\item $\mathsf{N}^{(i+1)}=\rho(\mathsf{N}^{(i)})$, 
\item $a_{jl}(\mathsf{N})=a_{jl}(\mathsf{N}^{(i)})$ if $l\in \mathsf{I}_i(\mathsf{N}), j\in \mathsf{S}_i$,
\item  and $a_{jl}(\mathsf{N})=0$ if $l\in \mathsf{I}_i(\mathsf{N}), j\in \mathsf{S}\setminus \mathsf{S}_i$.
\end{itemize}

We set $\mathsf{I}_{e(\mathsf{N})}(\mathsf{N})=\mathsf{S}_{e(\mathsf{N})}$ and $I_i(\mathsf{N})=\Upsilon(\mathsf{I}_i(\mathsf{N}))$ if $1\leq i\leq e(\mathsf{N})$. Then $\Sigma_\mathsf{N}=\{I_1(\mathsf{N}),...,I_{e(\mathsf{N})}(\mathsf{N})\}$ forms a partition of $\{1,2,...,n\}$.

\end{defn}
\begin{lem}\label{detr}Let $\mathsf{N}=(z_{ij})_{1\leq i,j\leq n}\in\Sm_n(\cO_L)$ be an $n\times n$ symmetric matrix. For each $1\leq r\leq n$, we denote by $\mathfrak{M}_r(\mathsf{N})$ the set of all principal minors of order $r$ of $\mathsf{N}$. We set 
$$\mathcal{D}_r(\mathsf{N})=\min_{\mathsf{M}\in \mathfrak{M}_r(\mathsf{N})}\ord_L(\det(\mathsf{M})).$$
With the notations of Definition \ref{trnmatri}, we have 
$$\left|\det(\mathsf{N})\right|=q_L^{-\sum_{i=1}^{e(\mathsf{N})}\mu_i(\mathsf{N})r_i(\mathsf{N})}$$
and 
$$\mathcal{D}_r(\mathsf{N})=\bigg(r-\sum_{i\leq \ell_r-1}r_i(\mathsf{N})\bigg)\mu_{\ell_r}(\mathsf{N})+\sum_{i\leq\ell_r-1}\mu_i(\mathsf{N})r_i(\mathsf{N}),$$
where $\ell_r$ is the smallest integer number such that $r-\sum_{i\leq\ell_r}r_i(\mathsf{N})\leq 0$.

\end{lem}
\begin{proof}
For each matrix $\mathsf{A}\in \MM_n(\cO_L)$ and $1\leq r\leq n$, we denote by $\tilde{\mathfrak{M}}_r(\mathsf{A})$ the set of all minors of order $r$ of $\mathsf{A}$. We define  $\tilde{\cI}_r(\mathsf{A})$ to be the ideal of $\cO_L$ generated by  $\det(\mathsf{M})$ for $\mathsf{M}\in \tilde{\mathfrak{M}}_r(\mathsf{A})$. Then 	$$\tilde{\cI}_r(\mathsf{A})=\varpi_L^{\min_{\mathsf{M}\in \tilde{\mathfrak{M}}_r(\mathsf{A})}\ord_L(\det(\mathsf{M}))}\cO_L.$$

With the notations of Definition \ref{trnmatri}, we define $\mathsf{T}_n: \Sm_n(\cO_L)\to \Sm_n(\cO_L)$ to be the map that sends each matrix
$\mathsf{N}\in \Sm_n(\cO_L)$  to  the matrix $\mathsf{T}_n(\mathsf{N})=(u_{ij})_{1\leq i,j\leq n}\in \Sm_n(\cO_L)$ satisfying  
\begin{itemize}
\item $(\mathsf{T}_n(\mathsf{N}))_{\mathsf{I}_l(\mathsf{N})}=\mathsf{N}^{(l)}$ for each $1\leq l\leq e(\mathsf{N})$.
\item $u_{ij}=0$ if $i\in \mathsf{I}_l(\mathsf{N}), j\in\mathsf{I}_\ell(\mathsf{N})$ with $l\neq \ell$. 
\end{itemize}
By Lemma \ref{symat} and Definition \ref{trnmatri}, $\mathsf{T}_n(\mathsf{N})$ is obtained from $\mathsf{N}$ by the following fundamental operations:
\begin{itemize}
\item  Add a scalar multiple over $\cO_L$ of one row to another row.
\item  Add a scalar multiple over $\cO_L$ of one column to another column.
\end{itemize}
Therefore, we have $\tilde{\cI}_r(\mathsf{N})=\tilde{\cI}_r(\mathsf{\mathsf{T}_n(\mathsf{N})})$ for all $1\leq r\leq n$. On the other hand, by the definition of $\mathsf{T}_n(\mathsf{N})$ and Definition \ref{trnmatri}, it is easy to see that 
$$\tilde{\cI}_r(\mathsf{\mathsf{T}_n(\mathsf{N})})=\varpi_L^{\big(r-\sum_{i\leq \ell_r-1}r_i(\mathsf{N})\big)\mu_{\ell_r}(\mathsf{N})+\sum_{i\leq\ell_r-1}\mu_i(\mathsf{N})r_i(\mathsf{N})}\cO_L$$
for all $1\leq r\leq n$. Therefore, we have \begin{equation}\label{lomadet}\min_{\mathsf{M}\in \tilde{\mathfrak{M}}_r(\mathsf{N})}\ord_L(\det(\mathsf{M}))=\big(r-\sum_{i\leq \ell_r-1}r_i(\mathsf{N})\big)\mu_{\ell_r}(\mathsf{N})+\sum_{i\leq\ell_r-1}\mu_i(\mathsf{N})r_i(\mathsf{N})
\end{equation} for all $1\leq r\leq n$. This implies the first assertion by taking $r=n$. Moreover, since $\mathfrak{M}_r(\mathsf{N})\subset\tilde{\mathfrak{M}}_r(\mathsf{N})$ and (\ref{lomadet}), we have 
\begin{equation}\label{eqdetma1}\mathcal{D}_r(\mathsf{N})=\min_{\mathsf{M}\in \mathfrak{M}_r(\mathsf{N})}\ord_L(\det(\mathsf{M}))\geq \big(r-\sum_{i\leq \ell_r-1}r_i(\mathsf{N})\big)\mu_{\ell_r}(\mathsf{N})+\sum_{i\leq\ell_r-1}\mu_i(\mathsf{N})r_i(\mathsf{N})
\end{equation}
for all $1\leq r\leq n$.

On the other hand, for each $1\leq r\leq n$, let $\mathsf{I}_{r,\ell_r}$ be a subset of $\mathsf{I}_{\ell_r}$ satisfying $|\mathsf{I}_{r,\ell_r}|=r-\sum_{i\leq \ell_r-1}r_i(\mathsf{N})$. We set $$\mathsf{J}_r=\mathsf{I}_{r,\ell_r}\sqcup(\sqcup_{1\leq i\leq j_r-1}\mathsf{I}_{i})\subset\{1<...<n\}.$$
It follows from Definition \ref{trnmatri} that $\mathsf{T}_{|\mathsf{J}_r|}(\mathsf{N}_{\mathsf{J}_r})=(\mathsf{T}_n(\mathsf{N}))_{\mathsf{J}_r}$ for all $1\leq r\leq n$. Therefore, by the above argument, we have \begin{align}\label{eqdetma2}\ord_L(\det(\mathsf{N}_{\mathsf{J}_r}))=\ord_L(\det(\mathsf{T}_{|\mathsf{J}_r|}(\mathsf{N}_{\mathsf{J}_r})))=&\ord_L(\det((\mathsf{T}_n(\mathsf{N}))_{\mathsf{J}_r}))\nonumber\\=&\big(r-\sum_{i\leq \ell_r-1}r_i(\mathsf{N})\big)\mu_{\ell_r}(\mathsf{N})+\sum_{i\leq\ell_r-1}\mu_i(\mathsf{N})r_i(\mathsf{N})
\end{align}
for all $1\leq r\leq n$. The remaining assertion follows from (\ref{eqdetma1}), (\ref{eqdetma2}) and the fact that $\mathsf{N}_{\mathsf{J}_r}\in \mathfrak{M}_r(\mathsf{N})$ for all $1\leq r\leq n$.
\end{proof}

Now, we use the above definitions to construct a linear change of coordinates $\phi_X$ of $L^{n}$ for each $X\in \cO_L^n$ as follows.
\begin{defn}\label{typeh}Let $X\in\cO_L^{n}$ and $\cH_h(X)$ be the Hessian matrix  $\left(\frac{\partial^2 h}{\partial x_i\partial x_j}(X)\right)_{1\leq i,j\leq n}$. With the notations of Definition \ref{trnmatri}, we have the partition of $\{1,...,n\}$  into subsets $I_j(h,X):=I_j(\cH_h(X))$ for $1\leq j\leq e(h,X):=e(\cH_h(X))$ with respect to $\mathsf{N}=\cH_h(X)$. We say that $X$ is of $h$-Hessian-type $(r_1(h,X),...,r_{e(h,X)}(h,X))$ with $h$-multiplicity $(\mu_1(h,X),...,\mu_{e(h,X)}(h,X))$, where $$r_j(h,X):=r_j(\cH_h(X))=\#I_j(h,X)$$ and $$\mu_j(h,X):=\mu_j(\cH_h(X))=\ord_L(\cH_h(X)^{(j)})$$ for all $1\leq j\leq e(h,X)$.   Lastly,  for each $1\leq i\leq n$, we set $\mathfrak{j}_X(i)=j$ if $i\in I_j(h,X)$. We will write  $\mathfrak{j}$ instead of $\mathfrak{j}_X$ if $X$ is understood clearly in the context. 
\end{defn} 
\begin{defn}\label{changvar}Let $X\in\cO_L^n$. Let $(a_{ji}(\cH_h(X)))_{1\leq i,j\leq n}$ be the matrix defined in Definition \ref{trnmatri} associated to $\mathsf{N}=\cH_h(X)$. For each $1\leq e\leq e(h,X)$, we consider the linear change of coordinates $\phi_X^{(e)}:L^n\to L^n, (x_1,...,x_n)\mapsto (y_1,...,y_n)$ satisfying $x_i=y_i-\sum_{j=1}^{n} a_{ji}y_j$ if $i\in I_e(h,X)$ and $x_i=y_i$ if $i\notin I_e(h,X)$, where $a_{ji}=a_{ji}(\cH_h(X))$. We set $\phi_X=\phi_X^{(e(h,X))}\circ\dots\circ\phi_X^{(1)}$. We denote by $\mathsf{A}(X)=(\tilde{a}_{ij}(X))_{1\leq i,j\leq n}$ the matrix of $\phi_X^{-1}$ and set $g_X(y):=h(y\mathsf{A}(X))$. 
\end{defn}
We list here some important properties of $\phi_X$ for $X\in\cO_L^n$.
\begin{lem} \label{propAX}Let $X\in\cO_L^n$. We have $\mu_j(h,X)<\mu_{j'}(h,X)$ if $1\leq j<j'\leq e(h,X)$. If $\mu_j(h,X)<+\infty$ then $$\rk_{\ac}(\cH_h(X)^{(j)})=r_j(h,X).$$
If $\mu_j(h,X)=+\infty$ then $j=e(h,X)$. 

\end{lem}
\begin{proof}
The assertions follow from Lemma \ref{symat} and Definition \ref{trnmatri}.
\end{proof}

\begin{lem}\label{Hess}Let $X\in\cO_L^n$. We have
 $Y\mathsf{A}(X)=X$ with $Y=\phi_X(X)$ and $$g_X(y_1,...,y_n)=h\bigg(\sum_{j=1}^{n}\tilde{a}_{j1}(X)y_j,...,\sum_{j=1}^{n}\tilde{a}_{jn}(X)y_j\bigg).$$
Moreover, after a suitable change of the order of the coordinates $x_1,...,x_n$, the Hessian matrix \begin{equation*}\cH_{g_X}(Y)=\left(\frac{\partial^2 g_X}{\partial	 y_i\partial y_j}(Y)\right)_{1\leq i,j\leq n}
\end{equation*} equals the diagonal block matrix $\bigg(\cH_h(X)^{(1)},...,\cH_h(X)^{(e(h,X))}\bigg)$. 
\end{lem}
\begin{proof}
The first assertion follows from the definition of $\mathsf{A}(X)$. 

 With the notations of Definition \ref{trnmatri},  for each $1\leq e\leq e(h,X)$, we set 
$$b_{eij}(X):=\begin{cases} 1 \textnormal{ if } i=j,\\
0 \textnormal{ if } i\neq j \textnormal{ and }  j\notin \mathsf{I}_e(\cH_h(X)),\\
-a_{ij}(\cH_h(X)) \textnormal{ if } i\neq j \textnormal{ and } j\in \mathsf{I}_e(\cH_h(X)).
\end{cases}$$
Moreover, we also set $\mathsf{B}_e(X)=(b_{eij}(X))_{1\leq i,j\leq n}$. Then $\mathsf{B}_e(X)$ is the matrix of $\big(\phi_X^{(e)}\big)^{-1}$ for all $1\leq e\leq e(h,X)$. Moreover, we have $$\mathsf{A}(X)=\mathsf{B}_1(X)\mathsf{B}_2(X)\cdots\mathsf{B}_{e(h,X)}(X).$$
On the other hand, since $Y\mathsf{A}(X)=X$,  by a basic calculation, we have 
$$\cH_{g_X}(Y)=\mathsf{A}(X)\cH_h(X)\mathsf{A}(X)^T=\mathsf{B}_1(X)\cdots\mathsf{B}_{e(h,X)}(X)\cH_h(X)\mathsf{B}_{e(h,X)}(X)^T\cdots\mathsf{B}_1(X)^T.$$
By a suitable change of the order of the coordinates $x_1,...,x_n$, we can suppose that $$I_{e}(h,X)=\biggl\{1+\sum_{j=1}^{e-1}r_j(h,X),...,\sum_{j=1}^{e}r_j(h,X)\biggr\}$$ for all $1\leq e\leq e(h,X)$. Now, the second assertion follows by applying the last assertion of Lemma \ref{symat} several times.
\end{proof}

From the construction of $\phi_X$, we have the following properties of $\mathsf{A}(X)$. 
\begin{lem}\label{propt}Let $X\in\cO_L^n$. Then
$\mathsf{A}(X)\in \SL_n(\cO_L)$  and $\mathsf{A}(X)$ becomes a lower triangular matrix  whose all entries in the diagonal are  $1$ after a suitable change of the order of the coordinates  $x_1,...,x_{n}$.  Moreover, let $\vartheta_i$ for $1\leq i\leq n$ be integers or $+\infty$ such that $\vartheta_i=\vartheta_{i'}$ if $\mathfrak{j}(i)=\mathfrak{j}(i')$, $\vartheta_i\leq\vartheta_{i'}$ if $\mathfrak{j}(i)\leq\mathfrak{j}(i')$. Then $$(a+(\varpi_L^{\vartheta_1}\cO_L,...,\varpi_L^{\vartheta_{n}}\cO_L))\mathsf{A}(X)=a\mathsf{A}(X)+(\varpi_L^{\vartheta_1}\cO_L,...,\varpi_L^{\vartheta_{n}}\cO_L)$$
for all $a\in L^{n}$. In addition, if $u=(u_1,...,u_{n}), v=(v_1,...,v_{n})\in L^{n}$ such that $v=u\mathsf{A}(X)$ then $$\left(\frac{\partial g_X}{\partial y_1}(u),...,\frac{\partial g_X}{\partial y_n}(u)\right)=\left(\frac{\partial h}{\partial x_1}(v),...,\frac{\partial h}{\partial x_n}(v)\right)\mathsf{A}(X)^{\textnormal{T}}.$$
Moreover, if $\tau_i$ for $1\leq i\leq n$ are integers or $+\infty$ such that $\tau_i=\tau_{i'}$ if $\mathfrak{j}(i)=\mathfrak{j}(i')$, and $\tau_i\geq\tau_{i'}$ if $\mathfrak{j}(i)\leq \mathfrak{j}(i')$, then 
$$(a+(\varpi_L^{\tau_1}\cO_L,...,\varpi_L^{\tau_{n}}\cO_L))\mathsf{A}(X)^{\textnormal{T}}=a\mathsf{A}(X)^{\textnormal{T}}+(\varpi_L^{\tau_1}\cO_L,...,\varpi_L^{\tau_{n}}\cO_L)$$
for all $a\in L^{n}$.
\end{lem}
\begin{proof}
By a suitable change of the order of the coordinates $x_1,...,x_n$, we can suppose that $$I_{e}(h,X)=\biggl\{1+\sum_{j=1}^{e-1}r_j(h,X),...,\sum_{j=1}^{e}r_j(h,X)\biggr\}$$ for all $1\leq e\leq e(h,X)$. Let $\mathsf{B}_e(X)$ be as in the proof of Lemma \ref{Hess} for $1\leq e\leq e(h,X)$. Then it is easy to verify that $\mathsf{B}_e(X)$ is a lower triangular matrix whose all entries in the diagonal are  $1$ for all $1\leq e\leq e(h,X)$. This and the equality 
$$\mathsf{A}(X)=\mathsf{B}_1(X)\mathsf{B}_2(X)\cdots\mathsf{B}_{e(h,X)}(X)$$
imply the first assertion. The remaining assertions follow from the first assertion and the definition of $\mathsf{A}(X)$.
\end{proof}

\begin{def-lem}\label{kaet} Let $m\in\ZZ$, we set 
$\eta_{m}=\lceil\frac{m}{3}\rceil$ and $\kappa_m=\lfloor\frac{m+1}{3}\rfloor$, where if $\alpha$ is a real number then $\lceil\alpha\rceil$ (resp. $\lfloor\alpha\rfloor$) is the smallest (resp. biggest) integer $\eta$ (resp. $\kappa$) such that $\eta\geq\alpha$ (resp. $\kappa\leq\alpha$). Then we have  $\kappa_m\leq\eta_m$  and $\kappa_m-2\leq m-2\eta_m-1\leq\kappa_m-1$ for all $m\in\ZZ$.
\end{def-lem}
\begin{proof}
It follows from a basic argument.
\end{proof}
We can relate $\phi_X$ and $\phi_{X'}$ if $X$ and $X'$ are enough close as follows. 
\begin{lem}\label{samtype}Let $X\in\cO_L^n$, $m\in\ZZ_{\geq 2}$ and $Y=X\mathsf{A}(X)^{-1}$. Then we have
$$X+\varpi_L^{\eta_{m}}\cO_L^{n}=(Y+\varpi_L^{\eta_{m}}\cO_L^{n})\mathsf{A}(X).$$
Moreover, if $\ord_L(X'-X)\geq \eta_m$ then 
\begin{align*}
&\ord_L(\mathsf{A}(X)-\mathsf{A}(X'))\geq \eta_m, \ord_L(\mathsf{A}(X)^{-1}-\mathsf{A}(X')^{-1})\geq \eta_m,\\ &\phi_X(X)-\phi_{X'}(X')\geq \eta_m, \min\{\mu_j(h,X),\eta_m\}=\min\{\mu_j(h,X'),\eta_m\}
\end{align*}
and  $I_j(h,X)=I_j(h,X')$ if $\mu_j(h,X)<\eta_m$. In particular, for each $1\leq e\leq e(h,X)$, the ideal $\cI_{h,X,e}$ of $\cO_L[y_1,...,y_{n}]$ generated by $\frac{\partial g_X}{\partial y_i}$ for $i\in \cup_{1\leq j\leq e}I_j(h,X)$ is obtained from the ideal $\cJ_{h,I_1(h,X),...,I_e(h,X)}$ of $\cO_L[x_1,...,x_{n}]$ generated by $\frac{\partial h}{\partial x_i}$ for $i\in \cup_{1\leq j\leq e}I_j(h,X)$ via the change of variables $x=y\mathsf{A}(X)$.
\end{lem}
\begin{proof}The assertions follow easily from Lemma \ref{propt}. 
\end{proof}

\begin{def-lem}\label{rh-type1}Let $X\in\cO_L^n$ and $m\in\ZZ_{\geq 2}$. We set $$e(h,X,m)=\max\{j| \mu_j(h,X)<\kappa_m\}, I(h,X,m)=\cup_{j\leq e(h,X,m)}I_j(h,X)$$ and $r(h,X,m)=\#(I(h,X,m))$. Here, we use the convention that $$e(h,X,m)=0 \textnormal{ if }\{j| \mu_j(h,X)<\kappa_m\}=\emptyset.$$ Then $e(h,X,m)$, $I(h,X,m), r(h,X,m)$ and $\mu_j(h,X)$ for $j\leq e(h,X,m)$ depend only on $(X \mod \varpi_L^{\eta_m}\cO_L)$. 
\end{def-lem}
\begin{proof}
It follows from Lemma-Definition \ref{kaet} and Lemma \ref{samtype}.
\end{proof}
Now we introduce the notions of  type and multiplicity of a point $X\in\cO_L^n$ related to $h$ and each integer $m\geq 2$. 
\begin{defn}\label{rh-type2}Let $X\in\cO_L^n$ and $m\in\ZZ_{\geq 2}$. We say that $X$ is of $(m,h)$-type $$(I_1(h,X),...,I_{e(h,X,m)}(h,X))$$ with $(m,h)$-multiplicity $$(\mu_1(h,X),...,\mu_{e(h,X,m)}(h,X))$$ if $e(h,X,m)>0$. If $e(h,X,m)=0$, we say that $X$ is of $(m,h)$-type $(\emptyset)$ with $(m,h)$-multiplicity $0$.
\end{defn}

In the case that $X\in\cO_L^n$ is of $(m,h)$-type $(\emptyset)$ with $(m,h)$-multiplicity $0$,  we will construct the following set and show that the restriction of the integral in (\ref{exph}) to its complement in $X+\varpi_L^{\eta_m}\cO_L^n$ is zero. 
 
 \begin{defn}\label{rho0}Let $X\in\cO_L^n$ and $m\in\ZZ_{\geq 2}$. Suppose that $e(h,X,m)=0$. We set 
 $$\tilde{\cC}_{h,X,m,-\eta_m}=\bigg\{x\in X+\varpi_L^{\eta_m}\cO_L^n\mid \ord_L(h(x))\geq m-1, \ord_L\left(\frac{\partial h}{\partial x_i}(x)\right)\geq m-\eta_m \forall\hspace{0.1cm} 1\leq i\leq n\bigg\}.$$
 
 \end{defn}
\begin{lem}\label{rho00}Let $X=(X_1,...,X_n)\in\cO_L^n$ and $m\in\ZZ_{\geq 2}$. Suppose that $e(h,X,m)=0$. Then either $\tilde{\cC}_{h,X,m,-\eta_m}=\emptyset$ or $\tilde{\cC}_{h,X,m,-\eta_m}=\{x\in X+\varpi_L^{\eta_m}\cO_L^n\mid \ord_L(h(x))\geq m-1\}$. Moreover,  we have 
$$\int_{\{x\in X+\varpi_L^{\eta_m}\cO_L^n\mid\ord_L(h(x))\geq m-1\}}\Phi_{L,Z}(x)\psi(h(x))|dx|=\int_{\tilde{\cC}_{h,X,m,-\eta_m}}\Phi_{L,Z}(x)\psi(h(x))|dx|$$
for all additive characters $\psi$ of $L$ of conductor $m$.
\end{lem}
\begin{proof}Since 
$e(h,X,m)=0$, we must have 
\begin{equation}\label{Hbig}\ord_L\left(\frac{\partial^2 h}{\partial x_i\partial x_j}(x)\right)\geq \kappa_m
\end{equation} for all $1\leq i,j\leq n$ and all $x\in X+\varpi_L^{\eta_m}\cO_L^n$. Firstly, we will show that 
\begin{equation}\label{leJ}
\min\{\ord_L\left(\frac{\partial h}{\partial x_i}(x)\right),m-\eta_m\}=\min\{\ord_L\left(\frac{\partial h}{\partial x_i}(X)\right),m-\eta_m\}
\end{equation} for all $1\leq i\leq n$ and all $x\in X+\varpi_L^{\eta_m}\cO_L^n$. Indeed, let $x=(x_1,...,x_n)\in X+\varpi_L^{\eta_m}\cO_L^n$. Then, we can write 
$$\frac{\partial h}{\partial x_i}(x)=\frac{\partial h}{\partial x_i}(X)+\sum_{1\leq j\leq n}\frac{\partial^2 h}{\partial x_i\partial x_j}(X)(x_j-X_j)+\sum_{I\in \NN^n,|I|\geq 2}\frac{1}{I!}\frac{\partial^{|I|+1} h}{\partial x_i\partial x^I}(X)(x-X)^I,$$
where if $I=(\alpha_1,...,\alpha_n)\in\NN^n$ then $|I|=\sum_{j=1}^n\alpha_j$, $I!=\prod_{j=1}^n \alpha_j!$, $(x-X)^I=\prod_{j=1}^n (x_j-X_j)^{\alpha_j}$, $\partial x^I=\partial x_{1}^{\alpha_1}\cdots\partial x_n^{\alpha_n}$ as usual. By (\ref{Hbig}), Lemma-Definition \ref{kaet} and the fact that $\ord_L(x_j-X_j)\geq \eta_m$ for all $1\leq j\leq n$, we have 
$$\ord\left(\frac{\partial h}{\partial x_i}(x)-\frac{\partial h}{\partial x_i}(X)\right)\geq m-\eta_m.$$
This implies (\ref{leJ}). Thus, $\tilde{\cC}_{h,X,m,-\eta_m}=\emptyset$ or $\tilde{\cC}_{h,X,m,-\eta_m}=\{x\in X+\varpi_L^{\eta_m}\cO_L^n\mid \ord_L(h(x))\geq m-1\}$. The first assertion is proved.

Since $m\geq 2$, we have  $\eta_m\geq 1$. Thus, either $\Phi_{L,Z}(x)=0$ for all $x\in X+\varpi_L^{\eta_m}\cO_L^n$ or $\Phi_{L,Z}(x)=1$ for all $x\in X+\varpi_L^{\eta_m}\cO_L^n$. The second assertion is trivial if $\Phi_{L,Z}(x)=0$ for all $x\in X+\varpi_L^{\eta_m}\cO_L^n$. Thus, we can suppose that $\Phi_{L,Z}(x)=1$ for all $x\in X+\varpi_L^{\eta_m}\cO_L^n$.  By the first assertion, it remains to show that if $\tilde{\cC}_{h,X,m,-\eta_m}\neq \{x\in X+\varpi_L^{\eta_m}\cO_L^n\mid \ord_L(h(x))\geq m-1\}$ then 
\begin{equation}\label{zerooo}
\int_{\{x\in X+\varpi_L^{\eta_m}\cO_L^n\mid\ord_L(h(x))\geq m-1\}}\psi(h(x))|dx|=0.
\end{equation}
Indeed, if $\tilde{\cC}_{h,X,m,-\eta_m}\neq \{x\in X+\varpi_L^{\eta_m}\cO_L^n\mid \ord_L(h(x))\geq m-1\}$ then (\ref{leJ}) implies that the set $J:=\bigg\{1\leq i\leq n\mid \ord_L\left(\frac{\partial h}{\partial x_i}(X)\right)<m-\eta_m\bigg\}\neq\emptyset$. For each $1\leq i\leq n$, we set \begin{equation}\label{debet}\beta_i=\begin{cases}m-1-\ord_L\left(\frac{\partial h}{\partial x_i}(X)\right)\geq \eta_m \textnormal{ if } i\in J,\\ \eta_m \textnormal{ otherwise}.\end{cases}
\end{equation} We show that if $x^{(0)}=\left(x_1^{(0)},...,x_n^{(0)}\right)\in \{x\in X+\varpi_L^{\eta_m}\cO_L^n\mid \ord_L(h(x))\geq m-1\}$ then 
\begin{equation}\label{cupm-1} x^{(0)}+\prod_{i=1}^n\varpi_L^{\beta_i}\cO_L\subset \{x\in X+\varpi_L^{\eta_m}\cO_L^n\mid \ord_L(h(x))\geq m-1\}.
\end{equation} 
Let $x^{(1)}=\left(x_1^{(1)},...,x_n^{(1)}\right)\in x^{(0)}+\prod_{i=1}^n\varpi_L^{\beta_i}\cO_L$ then  we have 
\begin{align}\label{taydeve}
h\left(x^{(1)}\right)=&h\left(x^{(0)}\right)+\sum_{1\leq i\leq n}\frac{\partial h}{\partial x_i}\left(x^{(0)}\right)\left(x_i^{(1)}-x_i^{(0)}\right)+\sum_{I\in \NN^n,|I|>2}\frac{1}{I!}\frac{\partial^{|I|} h}{\partial x^I}\left(x^{(0)}\right)\left(x^{(1)}-x^{(0)}\right)^I\nonumber\\&+\sum_{1\leq i,j\leq n}\frac{\partial^2 h}{\partial x_i\partial x_j}\left(x^{(0)}\right)\left(x_j^{(1)}-x_j^{(0)}\right)\left(x_i^{(1)}-x_i^{(0)}\right).
\end{align}
It follows from (\ref{Hbig}), (\ref{leJ}), (\ref{debet}), (\ref{taydeve}) and Lemma-Definition \ref{kaet} that $$\ord_L\left(h\left(x^{(1)}\right)-h\left(x^{(0)}\right)\right)\geq m-1.$$ Thus (\ref{cupm-1}) is proved. Moreover, by using (\ref{Hbig}), (\ref{leJ}), (\ref{debet}), (\ref{taydeve}) and Lemma-Definition \ref{kaet} again, we have 
\begin{equation}\label{eqred}\ord_L\bigg(h\left(x^{(1)}\right)-h\left(x^{(0)}\right)-\sum_{i\in J}\frac{\partial h}{\partial x_i}\left(x^{(0)}\right)\left(x_i^{(1)}-x_i^{(0)}\right)\bigg)\geq m
\end{equation}
for all $x^{(1)}\in x^{(0)}+\prod_{i=1}^n\varpi_L^{\beta_i}\cO_L$. By using (\ref{leJ}), (\ref{debet}), (\ref{eqred}), the orthogonality of characters and the fact that $J\neq\emptyset$, we have 
\begin{align}\label{zerro}
\int_{x^{(0)}+\prod_{i=1}^n\varpi_L^{\beta_i}\cO_L}\psi(h(x))|dx|&=q_L^{-\sum_{1\leq i\leq n}\beta_i}\psi\left(h\left(x^{(0)}\right)\right)\int_{\cO_L^n}\psi_{\varpi_L^{m-1}}\bigg(\sum_{i\in J}\acc\left(\frac{\partial h}{\partial x_i}\left(x^{(0)}\right)\right)t_i\bigg)|dt|\nonumber\\&=0
\end{align}
for all additive characters $\psi$ of $L$ of conductor $m$. By (\ref{cupm-1}), the set $$\{x\in X+\varpi_L^{\eta_m}\cO_L^n\mid \ord_L(h(x))\geq m-1\}$$ can be written as a disjoint union of cylinders of the form $x^{(0)}+\prod_{i=1}^n\varpi_L^{\beta_i}\cO_L$. This and  (\ref{zerro}) imply (\ref{zerooo}) and the second assertion.
\end{proof}
If $X\in\cO_L^n$ is not of $(m,h)$-type $(\emptyset)$, then we also try to construct a similar set to the one in Definition \ref{rho0}. In order to do this, it is sufficient to construct such a set after using the change of coordinates  $\phi_X$. We first need to prepare some more notations. 
\begin{defn} \label{j-func} Let $X\in\cO_L^n$, $Y=X\mathsf{A}(X)^{-1}$ and $m\in\ZZ_{\geq 2}$. Suppose that $e(h,X,m)>0$.  Let $z\in Y+\varpi_L^{\eta_m}\cO_L^{n}$. 

We set $j_{m,z}=+\infty$ if for all  $j$ we have  $\displaystyle \min_{\ell\in I_{j}(h,X)}\ord_{L}\left(\frac{\partial g_X}{\partial y_\ell}(z)\right)>\frac{m+\tilde{\mu}_{j,m}(h,X)-1}{2}$, where $\tilde{\mu}_{j,m}(h,X)=\mu_j(h,X)-1$ if $\mu_j(h,X)<\kappa_m$ and $\tilde{\mu}_{j,m}(h,X)=m-1-2\eta_{m}$ if $\mu_j(h,X)\geq\kappa_m$. 

If there exists $e(h,X,m)<j\leq e(h,X)$ such that  $\displaystyle\min_{\ell\in I_{j}(h,X)}\ord_{L}\left(\frac{\partial g_X}{\partial y_\ell}(z)\right)\leq\frac{m+\tilde{\mu}_{j,m}(h,X)-1}{2}$, we define $j_{m,z}$ to be the smallest natural number $j$ with this property. 

If $\displaystyle\min_{\ell\in I_{i}(h,X)}\ord_{L}\left(\frac{\partial g_X}{\partial y_\ell}(z)\right)>\frac{m+\tilde{\mu}_{i,m}(h,X)-1}{2}$ for all $i>e(h,X,m)$ and there exists $1\leq j\leq e(h,X,m)$ such that  $\displaystyle\min_{\ell\in I_{j}(h,X)}\ord_{L}\left(\frac{\partial g_X}{\partial y_\ell}(z)\right)\leq\frac{m+\tilde{\mu}_{j,m}(h,X)-1}{2}$, we define $j_{m,z}$ to be the smallest natural number $j$ with this property. 
\end{defn}

\begin{lem}\label{bs1}Let $X\in\cO_L^n$ and $m\in\ZZ_{\geq 2}$. Suppose that $e(h,X,m)>0$.  We have
\[\tilde{\mu}_{j,m}(h,X)\leq \min\{\mu_{j}(h,X)-1,\max\{\kappa_m-2,m-2\eta_m-1\}\}=\min\{\mu_{j}(h,X)-1,m-2\eta_m-1\}\]
for all $m\geq 2$ and all $1\leq j\leq e(h,X)$.
\end{lem}
\begin{proof}The assertion follows immediately from the definition of $\tilde{\mu}_{j,m}(h,X)$ and Lemma-Definition \ref{kaet}. 
\end{proof}

For each $X\in\cO_L^n$ and $m\geq 2$ such that $e(h,X,m)>0$, after using $\phi_X$, we will construct the following subset $\cA_{h,X,m}$ of $\phi_X(X+\varpi_L^{\eta_m}\cO_L^n)$ and show that the restriction of the integral in (\ref{exph}) (after using $\phi_X$) to it equals zero. 
\begin{defn}\label{Aset}Let $X\in\cO_L^n$, $Y=X\mathsf{A}(X)^{-1}$ and $m\in\ZZ_{\geq 2}$. Suppose that $e(h,X,m)>0$. 
If $e(h,X)\geq j_{m,Y}>e(h,X,m)$, we set 
$$\cA_{h,X,m}:=\{z\in Y+\varpi_L^{\eta_{m}}\cO_L^{n}\mid \ord_L(g_X(z))\geq m-1\}.$$ If $1\leq j_{m,Y}\leq e(h,X,m)$ or $j_{m,Y}=+\infty$,  we set
\begin{align*}
\cA_{h,X,m}:=&\{z\in Y+\varpi_L^{\eta_{m}}\cO_L^{n}\mid j_{m,z}\leq e(h,X,m),\ord_L(g_X(z))\geq m-1\}.
\end{align*}
\end{defn}
\begin{defn}\label{Azset}Let $X\in\cO_L^n$ and $m\in\ZZ_{\geq 2}$. Suppose that $e(h,X,m)>0$. 
For each  $z\in \cA_{h,X,m}$, we set 
$$\cA_{h,X,m}(z):=z+\prod_{1\leq i\leq n} \varpi_L^{\delta(z,m,i)}\cO_L,$$
where 
$$\delta(z,m,i):=\begin{cases} &m \textnormal{ if } \mathfrak{j}(i)\neq j_{m,z},\\
&m-1-\min_{\ell\in I_{j_{m,z}}(h,X)}\ord_{L}\left(\frac{\partial g_X}{\partial y_\ell}(z)\right) \textnormal{ if } \mathfrak{j}(i)=j_{m,z}.
\end{cases}$$
\end{defn}
In fact, we will show that $\cA_{h,X,m}$ can be written as a disjoint union of sets of the form  $\cA_{h,X,m}(z)$ and the restriction of the integral in (\ref{exph}) (after using $\phi_X$) to each $\cA_{h,X,m}(z)$ is equal to zero. To do this, we need some lemmas. From now on until the end of Lemma \ref{count1}, we fix $m\geq 2$ and $X\in \cO_L^n$ such that $e(h,X,m)>0$.  We set $Y=X\mathsf{A}(X)^{-1}$ as usual. 

\begin{lem}\label{delta0}Let $z\in\cA_{h,X,m}$ then $\delta(z,m,i)\geq \eta_m$ for all $i$.
\end{lem}
\begin{proof}
Let $z\in\cA_{h,X,m}$. If $\mathfrak{j}(i)\neq j_{m,z}$  then it is clear that $\delta(z,m,i)=m\geq\eta_m$. If $\mathfrak{j}(i)=j_{m,z}\leq e(h,X,m)$ then it follows from Lemma-Definition \ref{kaet} and Definition \ref{j-func} that
\begin{equation*}\NN\ni\delta(z,m,i)=m-1-\min_{\ell\in I_j(h,X)}\ord_{L}\left(\frac{\partial g_X}{\partial y_\ell}(z)\right)\geq \frac{m-\mu_j(h,X)}{2}\geq \frac{m-\kappa_m+1}{2}\geq\eta_m-\frac{1}{2}.
\end{equation*} Thus $\delta(z,m,i)\geq \eta_m$. If  $\mathfrak{j}(i)=j_{m,z}>e(h,X,m)$, then Definition \ref{j-func} implies $$\delta(z,m,i)=m-1-\min_{\ell\in I_j(h,X)}\ord_{L}\left(\frac{\partial g_X}{\partial y_\ell}(z)\right)\geq m-1-\frac{m+\tilde{\mu}_{j,m}(h,X)-1}{2}=\eta_m.$$  Therefore, we have $\delta(z,m,i)\geq \eta_m$ for all $i$.
\end{proof}
\begin{lem}\label{valu}Let $z\in Y+\varpi_L^{\eta_m}\cO_L^n$ then $$\ord_L\left(\frac{\partial^2 g_X}{\partial y_\ell\partial y_i}(z)\right)\geq \mu_j(h,X)$$ if $\mathfrak{j}(i)=j=\mathfrak{j}(\ell)$, and 
$$\ord_L\left(\frac{\partial^2 g_X}{\partial y_\ell\partial y_i}(z)\right)\geq \eta_m$$
if $\mathfrak{j}(i)\neq\mathfrak{j}(\ell)$,
\end{lem}
\begin{proof}
The assertion follows immediately from the construction of $g_X$ and $Y$.
\end{proof}
\begin{lem}\label{Cond1}Let $z=(z_1,....,z_{n}),z'=(z'_1,...,z'_{n})\in Y+\varpi_L^{\eta_m}\cO_L^n$. Suppose that $\ord_L(g_X(z))\geq m-1$. Then $\ord_L(g_X(z'))\geq m-1$ if the following conditions holds:
\begin{itemize}
\item[i,]$\ord_L(z_i-z'_i)+\ord_L(\frac{\partial g_X}{\partial y_i}(z))\geq m-1$.
\item[ii,]$\sum_{i=1}^{n}j_i\ord_L(z_i-z'_i)\geq m-1$ for all $J=(j_1,...,j_{n})\in\NN^{n}$ such that $|J|\geq 3$.
\item[iii,]For all $1\leq j\leq e(h,X)$, we have $\mu_j(h,X)+\ord_L(z_i-z'_i)+\ord_L(z_\ell-z'_\ell)\geq m-1$ if $\mathfrak{j}(i)=j=\mathfrak{j}(\ell)$.
\item[iv,]We have $\eta_m+\ord_L(z_i-z'_i)+\ord_L(z_\ell-z'_\ell)\geq m-1$ if $\mathfrak{j}(i)\neq \mathfrak{j}(\ell)$.

\end{itemize}
\end{lem}
\begin{proof}
We write $z'-z=(u_1\varpi_L^{\beta_1},...,u_{n}\varpi_L^{\beta_{n}})$ for $u_i\in \cO_L^{\times}$ and $\eta_m\leq \beta_i\leq +\infty$ for all $1\leq i\leq n$ with the convention that $\varpi_L^{+\infty}=0$. Then we have
\begin{equation}\label{taylo1}
g_{X}(z')-g_X(z)=\sum_{1\leq i\leq n}\frac{\partial g_X}{\partial y_i}(z)\varpi_L^{\beta_i}u_i+\sum_{|J|\geq 3}\frac{1}{J!}\frac{\partial^{|J|}g_X}{\partial y^J}(z)\varpi_L^{\langle J,\beta\rangle}u^J+\frac{1}{2}\sum_{i,\ell=1}^{n}\frac{\partial^2 g_X}{\partial y_\ell\partial y_i}(z)\varpi_L^{\beta_\ell+\beta_i}u_\ell u_i,
\end{equation}
where $\beta=(\beta_1,...,\beta_{n})$, $u=(u_1,...,u_{n})$, $\langle \cdot ,\cdot\rangle$ is the usual scalar product on $\RR^{n}$ and we use the usual notations $y^J=y_1^{j_1}...y_{n}^{j_{n}}$, $|J|=j_1+...+j_{n}$ and $J!=j_1!...j_{n}!$ if $J=(j_1,...,j_{n})\in\NN^{n}$.
 This equality together with the assumptions and Lemma \ref{valu} imply the assertion.
\end{proof}

\begin{lem}\label{Cond2}
Let $z=(z_1,....,z_{n}),z'=(z'_1,...,z'_{n})\in Y+\varpi_L^{\eta_m}\cO_L^n$. Let $\theta, \ell$ be integers such that  $1\leq \ell\leq n$. Then $\ord_L\left(\frac{\partial g_X}{\partial y_\ell}(z')-\frac{\partial g_X}{\partial y_\ell}(z)\right)\geq \theta$ if the following conditions holds:
\begin{itemize}
\item[i,]$\ord_L(z_i-z'_i)+\mu_{\mathfrak{j}(\ell)}(h,X)\geq \theta$ if $\mathfrak{j}(i)=\mathfrak{j}(\ell)$.
\item[i,]$\ord_L(z_i-z'_i)+\eta_m\geq \theta$ if $\mathfrak{j}(i)\neq \mathfrak{j}(\ell)$.
\item[ii,]$\sum_{i=1}^{n}j_i\ord_L(z_i-z'_i)\geq \theta$ for all $J=(j_1,...,j_{n})\in\NN^{n}$ such that $|J|\geq 2$.
\end{itemize}
\end{lem}
\begin{proof}
We write $z'-z=(u_1\varpi_L^{\beta_1},...,u_{n}\varpi_L^{\beta_{n}})$ for $u_i\in \cO_L^{\times}$ and $\eta_m\leq \beta_i\leq +\infty$ for all $1\leq i\leq n$. Then we have
\begin{equation}
\frac{\partial g_X}{\partial y_\ell}(z')=\frac{\partial g_X}{\partial y_\ell}(z)+\sum_{i=1}^{n}\frac{\partial^2 g_X}{\partial y_\ell\partial y_i}(z)\varpi_L^{\beta_i}u_i+ \sum_{|J|\geq 2}\frac{1}{J!}\frac{\partial^{|J|+1}g_X}{\partial y_\ell\partial y^J}(z)\varpi_L^{\langle J,\beta\rangle}u^J.
\end{equation}
Thus, the assertion follows from this equality, the assumptions and Lemma \ref{valu}.
\end{proof}
\begin{defn}\label{sequ}
Let $\beta\in\NN^{n}$ and $a\in\NN$ such that $\beta_i=m-\eta_m$ for all $i\notin I(h,X,m)$. We say that a sequence $(z^{(\ell)}=(z_1^{(\ell)},...,z_n^{(\ell)}))_{0\leq \ell\leq a}$ in $Y+\varpi_L^{\eta_m}\cO_L^{n}$ is of type $(m,\beta)$ if the following conditions hold:
\begin{itemize}
\item $z_i^{(\ell)}=z_i^{(0)}$ for all $i\notin I(h,X,m)$ and $0\leq \ell\leq a$.
\item $\ord_L\left(\frac{\partial g_X}{\partial y_i}\left(z^{(\ell)}\right)\right)\geq \beta_i+\ell$ for all $i\in I(h,X,m)$ and all $0\leq \ell\leq a$.
\item $\ord_L\left(\frac{\partial g_X}{\partial y_i}\left(z^{(\ell)}\right)\right)\geq m-\eta_m$ for all $i\notin I(h,X,m)$ and all $0\leq \ell\leq a$.
\item $\ord_L\left(g_X\left(z^{(\ell)}\right)\right)\geq m-1$ for all $a\geq \ell\geq 0$.
\item $\ord_L\left(z_{i}^{(\ell)}-z_{i}^{(\ell-1)}\right)\geq \beta_i-\mu_{\mathfrak{j}(i)}(h,X)+\ell-1$ for all $i\in I(h,X,m)$ and all $1\leq \ell\leq a$.
 \end{itemize}
We also say that a sequence $\left(z^{(\ell)}\right)_{\ell\geq 0}$ in $Y+\varpi_L^{\eta_m}\cO_L^{n}$ is of type $(m,\beta)$ if $\left(z^{(\ell)}\right)_{0\leq \ell\leq a}$ is of type $(m,\beta)$ for all $a\geq 0$.
\end{defn}
\begin{lem} \label{detsolu}Let $n_0$ be a positive integer and $\mathsf{N}=(\xi_{jl})_{1\leq j,l\leq n_0}\in\MM_{n_0}(\cO_L)$ be a matrix. Suppose that there are non-empty subsets $\mathsf{I}_i$ of $\{1<...<n_0\}$ for $1\leq i\leq e$ such that $\{1<...<n_0\}=\sqcup_{1\leq i\leq e}\mathsf{I}_i$, $\left(\mathsf{N}_{\mathsf{I}_i}\right)^{\ac}$ is invertible for all $1\leq i\leq e$, and 
\begin{equation}\label{eqcond}\xi_{jl}=0
\end{equation} if $j\in \mathsf{I}_{i_1}, l\in \mathsf{I}_{i_2}$ with $i_1\neq i_2$. Let $\mathsf{a}=(a_1,...,a_{n_0})\in\cO_L^{n_0}$ such that 
\begin{equation}\label{coni1}\ord_L(a_j)\geq \ord_L(\mathsf{N}_{\mathsf{I}_i})
\end{equation} for all $1\leq i\leq e$ and all $j\in \mathsf{I}_i$. Let $\mathsf{M}=(\varsigma_{jl})_{1\leq j,l\leq n_0}\in \MM_{n_0}(\cO_L)$ such that 
\begin{equation}\label{coin2}\ord_L(\xi_{lj}-\varsigma_{lj})>\ord_L(\mathsf{N}_{\mathsf{I}_i})
\end{equation} for all $1\leq i\leq e$, all $l\in\mathsf{I}_i$ and all $1\leq j\leq n_0$. Then the equation $\mathsf{M}x^T+\mathsf{a}^T=0$ has a unique solution $\mathsf{x}_{(0)}\in L^{n_0}$. Moreover, we have $\mathsf{x}_{(0)}\in \cO_L^{n_0}$. 

In addition, let $\mathsf{b}=(b_1,...,b_{n_0})\in \cO_L^{n_0}$ and  $\mathfrak{m}\in\NN$ such that $\ord_L(b_j-a_j)\geq \ord_L(\mathsf{N}_{\mathsf{I}_i})+\mathfrak{m}$ for all $1\leq i\leq e$ and all $j\in \mathsf{I}_i$. Then  $\ord_L(\mathsf{x}_{(0)}-\mathsf{x}_{(1)})\geq \mathfrak{m}$ if $\mathsf{x}_{(1)}\in L^{n_0}$ is a solution of $\mathsf{M}x^T+\mathsf{b}^T=0$.

\end{lem}
\begin{proof}
By the assumption that $\left(\mathsf{N}_{\mathsf{I}_i}\right)^{\ac}$ is invertible for all $1\leq i\leq e$, we have $\det(\mathsf{N}_{\mathsf{I}_i})\neq 0$ and $\ord_L(\det(\mathsf{N}_{\mathsf{I}_i}))=|\mathsf{I}_i|\ord_L(\mathsf{N}_{\mathsf{I}_i})$ for all $1\leq i\leq e$. This and (\ref{eqcond}) imply
$$\det(\mathsf{N})=\prod_{1\leq i\leq e}\det(\mathsf{N}_{\mathsf{I}_i})\neq 0.$$ On the other hand, by using the Leibniz formula for determinants together with (\ref{eqcond}), (\ref{coin2}) and the assumption that  $\left(\mathsf{N}_{\mathsf{I}_i}\right)^{\ac}$ is invertible for all $1\leq i\leq e$, we have 
$\ord_L(\det(\mathsf{M}))=\ord_L(\det(\mathsf{N}))\neq 0$. Therefore, the equation $\mathsf{M}x^T+\mathsf{a}^T=0$ has a unique solution $\mathsf{x}_{(0)}\in L^{n_0}$. For each $1\leq j\leq n_0$, we denote by $\mathsf{M}_j$ the matrix obtained by replacing the $j^{\textnormal{th}}$-column of $\mathsf{M}$ by $-\mathsf{a}^T$. By using the Leibniz formula for determinants again together with (\ref{eqcond}), (\ref{coni1})  and (\ref{coin2}), we have 
$$\ord_L(\det(\mathsf{M}_{j}))\geq\sum_{1\leq i\leq e}|\mathsf{I}_i|\ord_L(\mathsf{N}_{\mathsf{I}_i})= \ord_L(\det(\mathsf{N}))=\ord_L(\det(\mathsf{M}))$$
for all $1\leq j\leq n_0$. Now, we can use Cramer's rule to deduce that $\mathsf{x}_{(0)}\in \cO_L^{n_0}$. 

Since $\ord_L(b_j-a_j)\geq \ord_L(\mathsf{N}_{\mathsf{I}_i})+\mathfrak{m}\geq \ord_L(\mathsf{N}_{\mathsf{I}_i}) $ and $\ord_L(a_j)\geq \ord_L(\mathsf{N}_{\mathsf{I}_i})$ for all $1\leq i\leq e$ and all $j\in \mathsf{I}_i$, we have $\ord_L(\varpi_L^{-\mathfrak{m}}(b_j-a_j))\geq \ord_L(\mathsf{N}_{\mathsf{I}_i})$ and $\ord_L(b_j)\geq \ord_L(\mathsf{N}_{\mathsf{I}_i})$ for all $1\leq i\leq e$ and all $j\in \mathsf{I}_i$. So the equation $\mathsf{M}x^T+\varpi_L^{-\mathfrak{m}}(\mathsf{b}^T-\mathsf{a}^T)=0$ has a solution $\mathsf{x}_{(2)}\in\cO_L^{n_0}$ as shown above. Moreover, we also have that $\mathsf{x}_{(1)}$ is the unique solution of the equation $\mathsf{M}x^T+\mathsf{b}^T=0$ in $L^{n_0}$. On the other hand, by a simple calculation, $\mathsf{x}_{(0)}+\varpi_L^\mathfrak{m}\mathsf{x}_{(2)}$ is a solution of $\mathsf{M}x^T+\mathsf{b}^T=0$ in $L^{n_0}$.  Thus, we must have $\mathsf{x}_{(1)}=\mathsf{x}_{(0)}+\varpi_L^\mathfrak{m}\mathsf{x}_{(2)}$ and  $\ord_L(\mathsf{x}_{(1)}-\mathsf{x}_{(0)})=\ord_L(\varpi_L^\mathfrak{m}\mathsf{x}_{(2)})\geq \mathfrak{m}$. This completes the proof. 
\end{proof}
\begin{lem}\label{Cond3}Let $z=(z_1,...,z_{n})\in Y+\varpi_L^{\eta_m}\cO_L^n$ such that $\ord_L(g_X(z))\geq m-1$, $\ord_L\left(\frac{\partial g_X}{\partial y_i}(z)\right)\geq m-\eta_m$ for all $i$. Let $\beta=(\beta_1,...,\beta_{n})\in \NN^{n}$ such that 
\begin{itemize}
\item[$(1)$,]$\beta_i=m-\eta_m$ if $i\notin I(h,X,m)$.
\item[$(2)$,]We have $\beta_i=\beta_\ell$ if $\mathfrak{j}(i)=\mathfrak{j}(\ell)$.
\item[$(3)$,]$\ord_L\left(\frac{\partial g_X}{\partial y_i}(z)\right)\geq\beta_i\geq m-\eta_m$ if $i\in I(h,X,m)$.
\item[$(4)$,]$\beta_{i_1}+\beta_{i_2}-\mu_{\mathfrak{j}(i_1)}(h,X)-\mu_{\mathfrak{j}(i_2)}(h,X)\geq\beta_{i_3}+1$ if $i_1,i_2,i_3\in I(h,X,m)$.
\item[$(5)$,]$\max_{i\in I(h,X,m)}\beta_i<\eta_m+\min_{i\in I(h,X,m)}(\beta_i-\mu_{\mathfrak{j}(i)}(h,X)).$
\end{itemize}
Then there exists a sequence $\left(z^{(\ell)}=\left(z_1^{(\ell)},...,z_{n}^{(\ell)}\right)\right)_{\ell\geq 0}$ in $Y+\varpi_L^{\eta_m}\cO_L^{n}$ of type $(m,\beta)$ such that $z^{(0)}=z$.  
\end{lem}

\begin{proof}Firstly, we set $$\upsilon_{i,\ell}=\begin{cases} \beta_i-\mu_{\mathfrak{j}(i)}(h,X)+\ell \textnormal{ if } i\in I(h,X,m),\\ +\infty \textnormal{ if } i\notin I(h,X,m). 
\end{cases}
$$
 We prove the assertion by induction on $\ell$. If $\ell=0$, we just need to take $z^{(0)}=z$.   
Suppose that there exists a sequence $\left(z^{(\ell)}\right)_{0\leq \ell\leq \ell_0}$ in $Y+\varpi_L^{\eta_m}\cO_L^{n}$ of type $(m,\beta)$ such that $z^{(0)}=z$.  
  Let $u=(u_1,...,u_{n})\in \cO_L^{n}$ such that $u_i=0$ for all $i\notin I(h,X,m)$, we set $\tilde{z}=z^{(\ell_0)}+(\varpi_L^{\upsilon_{1,\ell_0}}u_1,...,\varpi_L^{\upsilon_{n,\ell_0}}u_{n})$. We prove that $\ord_L(g_X(\tilde{z}))\geq m-1$. By the inductive hypothesis, we have $z^{(\ell_0)}\in Y+\varpi_L^{\eta_m}\cO_L$ and $\ord_L\left(g_X\left(z^{(\ell_0)}\right)\right)\geq m-1$. The assumptions on $\upsilon_{i,\ell_0},\beta_i$ and the fact that $\mu_j(h,X)\leq \kappa_m-1$ if $j\leq e(h,X,m)$ imply $\upsilon_{i,\ell_0}\geq m-\eta_m-\kappa_m+1$ for all $i$. Thus, the inductive hypothesis  and Lemma-Definition \ref{kaet} imply \begin{equation}\label{eqi1}\ord_L\left(\tilde{z}_i-z_i^{(\ell_0)}\right)+\ord_L\left(\frac{\partial g_X}{\partial y_i}\left(z^{(\ell_0)}\right)\right)\geq \upsilon_{i,\ell_0}+\beta_i+\ell_0\geq 2(m-\eta_m)-\kappa_m+1\geq m-1
\end{equation}
for all $i$. If $J=(j_1,...,j_{n})\in \NN^{n}$ such that $|J|=\sum_{i=1}^{n}j_i\geq 3$, then we can use Lemma-Definition \ref{kaet} to have 
\begin{equation}\label{eqi2}\sum_{i=0}^{n}j_i\ord_L\left(\tilde{z}_i-z_i^{(\ell_0)}\right)\geq \sum_{i=0}^{n}j_i\upsilon_{i,\ell_0}\geq 3(m-\eta_m-\kappa_m+1)\geq m-1.
\end{equation}
If $\mathfrak{j}(i)=\mathfrak{j}(l)=j$, by Lemma-Definition \ref{kaet} together with the definition of $\upsilon_{i,\ell_0}$ and the condition that $\mu_j(h,X)\leq \kappa_m-1$ if $j\leq e(h,X,m)$, we have
\begin{equation}\label{eqi3}\mu_j(h,X)+\ord_L\left(\tilde{z}_i-z_i^{(\ell_0)}\right)+\ord_L\left(\tilde{z}_l-z_l^{(\ell_0)}\right)=\mu_j(h,X)+2\upsilon_{i,\ell_0}
\geq 2\beta_i-\kappa_m+1\geq 2(m-\eta_m)-\kappa_m+1\geq m-1.
\end{equation}
Similarly, if $\mathfrak{j}(i)\neq\mathfrak{j}(l)$, we have 
\begin{align}\label{eqi4}
\eta_m+\ord_L\left(\tilde{z}_i-z_i^{(\ell_0)}\right)+\ord_L\left(\tilde{z}_l-z_l^{(\ell_0)}\right)\geq \eta_m+2(m-\eta_m-\kappa_m+1)\geq m-1.
\end{align} 
 Thus, we can use Lemma \ref{Cond1} to have $\ord_L(g_X(\tilde{z}))\geq m-1$.

We have 
\begin{align}\label{deriv}
\frac{\partial g_X}{\partial y_l}(\tilde{z})=&\frac{\partial g_X}{\partial y_l}\left(z^{(\ell_0)}\right)+\sum_{i=1}^{n}\frac{\partial^2 g_X}{\partial y_l\partial y_i}\left(z^{(\ell_0)}\right)\varpi_L^{\upsilon_{i,\ell_0}}u_i\\&+ \sum_{J=(j_1,...,j_{n})\in\NN^{n}, |J|\geq 2}\frac{1}{J!}\frac{\partial^{|J|+1}g_X}{\partial y_l\partial y^J}\left(z^{(\ell_0)}\right)\nonumber\varpi_L^{\sum_{1\leq i\leq n}j_i\upsilon_{i,\ell_0}}u^J\nonumber
\end{align}
for all $l$.

 Since $\upsilon_{i,\ell_0}\geq m-\eta_m-\kappa_m+1$ for all $i$, we have 
 \begin{equation}\label{eqi5}\sum_{1\leq i\leq n}j_i\upsilon_{i,\ell_0}\geq 2(m-\eta_m-\kappa_m+1)\geq m-\eta_m
 \end{equation}
 if $J\in \NN^{n}$ such that $|J|\geq 2$. By Lemma-Definition \ref{kaet}, Lemma \ref{valu}, the facts that $\upsilon_{i,\ell_0}=+\infty$ if $i\notin I(h,X,m)$ and $\upsilon_{i,\ell_0}\geq m-\eta_m-\kappa_m+1$ for all $i$, we have 
 \begin{equation}\label{eqi6}\ord_L\left(\frac{\partial^2 g_X}{\partial y_l\partial y_i}\left(z^{(\ell_0)}\right)\right)+\upsilon_{i,\ell_0}\geq \eta_m+m-\eta_m-\kappa_m+1\geq m-\eta_m
 \end{equation} for all  $l\notin I(h,X,m)$ and all $1\leq i\leq n$. Thus, since $\ord_L\left(\frac{\partial g_X}{\partial y_l}\left(z^{(\ell_0)}\right)\right)\geq m-\eta_m$ for all  $l\notin I(h,X,m)$, we have $\ord_L\left(\frac{\partial g_X}{\partial y_l}(\tilde{z})\right)\geq m-\eta_m$ for all $l\notin I(h,X,m)$.

We set $\mathsf{M}(t)=\left(\frac{\partial^2 g_X}{\partial y_l\partial y_i}(t)\varpi_L^{\upsilon_{i,\ell_0}}\right)_{l,i\in I(h,X,m)}$ for $t\in Y+\varpi_L^{\eta_m}\cO_L^{n}$. By the inductive hypothesis, we have $$\ord_L\left(\frac{\partial g_X}{\partial y_i}\left(z^{(\ell_0)}\right)\right)\geq \beta_i+\ell_0=\upsilon_{i,\ell_0}+\mu_{\mathfrak{j}(i)}(h,X)$$ if $i\in I(h,X,m)$. 
By this together with  Lemmas \ref{propAX}, \ref{Hess}, \ref{valu}, Definition \ref{typeh}, Lemma-Definition \ref{rh-type1}, the condition $(5)$ and the fact that $z^{(\ell_0)}\in Y+\varpi_L^{\eta_m}\cO_L^{n}$, we can verify that $\mathsf{N}=\mathsf{M}(Y), 
 \mathsf{M}=\mathsf{M}\left(z^{(\ell_0)}\right)$ and $\mathsf{a}=\left(\frac{\partial g_X}{\partial y_i}\left(z^{(\ell_0)}\right)\right)_{i\in I(h,X,m)}$ satisfy the assumptions of Lemma \ref{detsolu}. Therefore, we can use Lemma \ref{detsolu} to see that 
the equation 
\begin{equation}\label{eq0}
\mathsf{M}\left(z^{(\ell_0)}\right)(u_i)_{i\in I(h,X,m)}^{\textnormal{T}}+\left(\frac{\partial g_X}{\partial y_i}\left(z^{(\ell_0)}\right)\right)_{i\in I(h,X,m)}^{\textnormal{T}}=0
\end{equation} has a  solution $\left(u_{i}^{(\ell_0+1)}\right)_{i\in I(h,X,m)}\in \cO_L^{r(h,X,m)}$. 
Moreover, by the condition $(4)$, if $i_1,i_2,i_3\in I(h,X,m)$, we have 
\begin{equation}\label{lbou}\upsilon_{i_1,\ell_0}+\upsilon_{i_2,\ell_0}=\beta_{i_1}+\beta_{i_2}-\mu_{\mathfrak{j}(i_1)}(h,X)-\mu_{\mathfrak{j}(i_2)}(h,X)+2\ell_0\geq\beta_{i_3}+\ell_0+1.
\end{equation} By (\ref{deriv}), (\ref{eq0}) and (\ref{lbou}), we have  $\ord_L\left(\frac{\partial g_X}{\partial y_l}\left(z^{(\ell_0+1)}\right)\right)\geq \beta_l+\ell_0+1$ for all $l\in I(h,X,m)$, where $z^{(\ell_0+1)}=z^{(\ell_0)}+(\varpi_L^{\upsilon_{1,\ell_0}}u_1^{(\ell_0+1)},...,\varpi_L^{\upsilon_{n,\ell_0}}u_{n}^{(\ell_0+1)})$ with $u_i^{(\ell_0+1)}=0$ if $i\notin I(h,X,m)$. On the other hand, since $u_i^{(\ell_0+1)}\in\cO_L$ for all $i\in I(h,X,m)$ and $u_i^{(\ell_0+1)}=0$ for all $i\notin I(h,X,m)$, it follows from the above argument that $\ord_L\left(g_X\left(z^{(\ell_0+1)}\right)\right)\geq m-1$ and  $\ord_L\left(\frac{\partial g_X}{\partial y_l}\left(z^{(\ell_0+1)}\right)\right)\geq m-\eta_m$ for all $l\notin I(h,X,m)$. Thus, $\left(z^{(\ell)}\right)_{0\leq \ell\leq \ell_0+1}$ is a sequence in $Y+\varpi_L^{\eta_m}\cO_L^{n}$ of type $(m,\beta)$ with $z^{(0)}=z$. This completes the inductive argument. 

\end{proof}
\begin{lem}\label{uniqne}With the notations and the assumptions of Lemma \ref{Cond3}, let $\left(z^{(\ell)}=\left(z_1^{(\ell)},...,z_{n}^{(\ell)}\right)\right)_{\ell\geq 0}$ be the sequence in $Y+\varpi_L^{\eta_m}\cO_L^{n}$ of type $(m,\beta)$ constructed in Lemma \ref{Cond3}. Let $\nu\in \NN\cup\{+\infty\}$ and  $\left(z'^{(\ell)}=\left(z_1'^{(\ell)},...,z_{n}'^{(\ell)}\right)\right)_{\ell\geq 0}$ be a sequence in $Y+\varpi_L^{\eta_m}\cO_L^{n}$ of type $(m,\beta)$. Then $\ord_L\left(z_{i}'^{(0)}-z_i^{(0)}\right)\geq \min\{\nu,\upsilon_{i,0}\}$ for all  $1\leq i\leq n$ 
if and only if $\ord_L\left(z_{i}^{(\ell)}-z_{i}'^{(\ell)}\right)\geq \min\{\nu,\upsilon_{i,\ell}\}$ for all $1\leq i\leq n$ and all $\ell\geq 0$, where 
$$\upsilon_{i,\ell}=\begin{cases} \beta_i-\mu_{\mathfrak{j}(i)}(h,X)+\ell \textnormal{ if } i\in I(h,X,m),\\ +\infty \textnormal{ if } i\notin I(h,X,m). 
\end{cases}
$$




\end{lem}
\begin{proof}We also prove the assertion by induction on $\ell$. The assertion is trivial if $\ell=0$. Now, we suppose 
$\ord_L\left(z_{i}^{(\ell)}-z_{i}'^{(\ell)}\right)\geq \min\{\nu,\upsilon_{i,\ell}\}$ for all $1\leq i\leq n$ and all $\ell_0\geq \ell\geq 0$.  Then we need to show that $\ord_L\left(z_{i}^{(\ell_0+1)}-z_{i}'^{(\ell_0+1)}\right)\geq \min\{\nu,\upsilon_{i,\ell_0+1}\}$ for all $1\leq i\leq n$.
By Definition \ref{sequ} 
and the fact that $\left(z'^{(\ell)}=\left(z_1'^{(\ell)},...,z_{n}'^{(\ell)}\right)\right)_{\ell\geq 0}$ is a sequence of type $(m,\beta)$, we have $\ord_L\left(z_{i}'^{(\ell_0+1)}-z_{i}'^{(\ell_0)}\right)\geq\upsilon_{i,\ell_0}$ for all $1\leq i\leq n$. 
Thus, we have $\ord_L\left(z_i'^{(\ell_0+1)}-z_i^{(\ell_0)}\right)\geq \min\{\nu,\upsilon_{i,\ell_0}\}$ for all $1\leq i\leq n$. So we can write $z_i'^{(\ell_0+1)}=z_i^{(\ell_0)}+\varpi_L^{\min\{\nu,\upsilon_{i,\ell_0}\}}u_i'^{(\ell_0+1)}$ for all $1\leq i\leq n$. Therefore, we have 
\begin{align}\label{deriv1}
\frac{\partial g_X}{\partial y_l}\left(z'^{(\ell_0+1)}\right)=&\frac{\partial g_X}{\partial y_l}\left(z^{(\ell_0)}\right)+\sum_{i=1}^{n}\frac{\partial^2 g_X}{\partial y_l\partial y_i}\left(z^{(\ell_0)}\right)\varpi_L^{\min\{\nu,\upsilon_{i,\ell_0}\}}u_i'^{(\ell_0+1)}\nonumber\\&+ \sum_{J=(j_1,...,j_{n})\in\NN^{n}, |J|\geq 2}\frac{1}{J!}\frac{\partial^{|J|+1}g_X}{\partial y_l\partial y^J}\left(z^{(\ell_0)}\right)\varpi_L^{\sum_{1\leq i\leq n}j_i\min\{\nu,\upsilon_{i,\ell_0}\}}\left(u'^{(\ell_0+1)}\right)^J
\end{align}
for all $l$.
We set $\mathfrak{I}=\{1\leq i\leq n\mid \upsilon_{i,\ell_0}<\nu\}$. Then $\mathfrak{I}\subset I(h,X,m)$.  Let $i_0$ such that $\upsilon_{i_0,\ell_0}\geq \nu$. Since $\ord_L\left(z_{i_0}^{(\ell_0+1)}-z_{i_0}^{(\ell_0)}\right)\geq\upsilon_{i_0,\ell_0}$, we have $$\ord_L\left(z_{i_0}'^{(\ell_0+1)}-z_{i_0}^{(\ell_0+1)}\right)=\ord_L\left(z_{i_0}'^{(\ell_0+1)}-z_{i_0}^{(\ell_0)}+z_{i_0}^{(\ell_0)}-z_{i_0}^{(\ell_0+1)}\right)\geq\min\{\nu,\upsilon_{i_0,\ell_0}\}=\nu=\min\{\nu,\upsilon_{i_0,\ell_0+1}\}.$$ Thus, if $\mathfrak{J}=\emptyset$ then  $\ord_L\left(z_{i}'^{(\ell_0+1)}-z_{i}^{(\ell_0+1)}\right)\geq \min\{\nu,\upsilon_{i,\ell_0+1}\}$ for all $i$ and the inductive argument is completed. Suppose that $\mathfrak{J}\neq\emptyset$. It remains to show that $\ord_L\left(z_{i}'^{(\ell_0+1)}-z_{i}^{(\ell_0+1)}\right)\geq \min\{\nu,\upsilon_{i,\ell_0+1}\}$ for all $i\in\mathfrak{I}$.  
  If $l\in\mathfrak{I}$, we have $\nu\geq \upsilon_{l,\ell_0}+1$. This and (\ref{lbou}) together with the fact that $\min\{\nu,\upsilon_{i,\ell_0}\}=\nu$ if $i\notin I(h,X,m)$ imply that \begin{equation}\label{binu}\sum_{1\leq i\leq n}j_i\min\{\nu,\upsilon_{i,\ell_0}\}\geq \beta_{l}+\ell_0+1
  \end{equation} for all $l\in \mathfrak{I}$ and all $(j_1,...,j_{n})\in\NN^{n}$ with $\sum_{i=1}^{n}j_i\geq 2$. On the other hand, by Lemma \ref{valu}, we have 
\begin{equation}\label{eq-2}\ord_L\left(\frac{\partial^2 g_X}{\partial y_l\partial y_i}\left(z^{(\ell_0)}\right)\varpi_L^{\min\{\nu,\upsilon_{i,\ell_0}\}}\right)\geq \eta_m+ \nu\geq \mu_{\mathfrak{j}(l)}(h,X)+\upsilon_{l,\ell_0}+1=\beta_l+\ell_0+1
\end{equation}
for all $l\in\mathfrak{I}$ and all $i\notin \mathfrak{I}$. 
Now, we set 
$$\mathsf{a}_{\mathfrak{I}}=\left(\frac{\partial g_X}{\partial y_l}\left(z^{(\ell_0)}\right)+\sum_{i\notin\mathfrak{I}}\frac{\partial^2 g_X}{\partial y_l\partial y_i}\left(z^{(\ell_0)}\right)\varpi_L^{\min\{\nu,\upsilon_{i,\ell_0}\}}u_i^{(\ell_0+1)}\right)_{l\in\mathfrak{I}}$$
and
\begin{align*}\mathsf{b}_{\mathfrak{I}}=&\bigg(\frac{\partial g_X}{\partial y_l}\left(z^{(\ell_0)}\right)-\frac{\partial g_X}{\partial y_l}\left(z'^{(\ell_0+1)}\right)+\sum_{i\notin\mathfrak{I}}\frac{\partial^2 g_X}{\partial y_l\partial y_i}\left(z^{(\ell_0)}\right)\varpi_L^{\min\{\nu,\upsilon_{i,\ell_0}\}}u_i'^{(\ell_0+1)}\\
&+\sum_{J=(j_1,...,j_{n})\in\NN^{n}, |J|\geq 2}\frac{1}{J!}\frac{\partial^{|J|+1}g_X}{\partial y_l\partial y^J}\left(z^{(\ell_0)}\right)\varpi_L^{\sum_{1\leq i\leq n}j_i\min\{\nu,\upsilon_{i,\ell_0}\}}\left(u'^{(\ell_0+1)}\right)^J\bigg)_{l\in\mathfrak{I}}.
\end{align*}
Let $\left(u_i^{(\ell_0+1)}\right)_{1\leq i\leq n}$ and $\mathsf{M}(t)=\left(\frac{\partial^2 g_X}{\partial y_l\partial y_i}(t)\varpi_L^{\upsilon_{i,\ell_0}}\right)_{l,i\in I(h,X,m)}$ for $t\in Y+\varpi_L^{\eta_m}\cO_L^{n}$ be as in the proof of Lemma \ref{Cond3}. In particular, we have $z_i^{(\ell_0+1)}=z_i^{(\ell_0)}+\varpi_L^{\upsilon_{i,\ell_0}}u_i^{(\ell_0+1)}$ for all $i\in I(h,X,m)$. Moreover, we have $$\left(\mathsf{M}\left(z^{(\ell_0)}\right)\right)_\mathfrak{I}(u_i^{(\ell_0+1)})_{i\in\mathfrak{I}}^T+a_\mathfrak{I}^T=0=\left(\mathsf{M}\left(z^{(\ell_0)}\right)\right)_\mathfrak{I}(u_i'^{(\ell_0+1)})_{i\in\mathfrak{I}}^T+b_\mathfrak{I}^T.$$
As in the proof of Lemma \ref{Cond3}, we can show that $\mathsf{N}=\left(\mathsf{M}\left(Y\right)\right)_\mathfrak{I}$ and $\mathsf{M}=\left(\mathsf{M}\left(z^{(\ell_0)}\right)\right)_\mathfrak{I}$ satisfy the conditions of Lemma \ref{detsolu}. Moreover, by using (\ref{binu}), (\ref{eq-2}) and the fact that $\ord_L\left(\frac{\partial g_X}{\partial y_l}\left(z'^{(\ell_0+1)}\right)\right)\geq\beta_l+\ell_0+1$ for all $l\in I(h,X,m)$, we can also verify that $\mathsf{a}=\mathsf{a}_\mathfrak{I}$ and $\mathsf{b}=\mathsf{b}_\mathfrak{I}$ satisfy the conditions of Lemma \ref{detsolu} with respect to $\mathfrak{m}=1$. Therefore, by using Lemma \ref{detsolu}, we have $\ord_L\left(u_i^{(\ell_0+1)}-u_i'^{(\ell_0+1)}\right)\geq 1$ for all $i\in\mathfrak{I}$. 

Hence, 
\begin{equation*}\ord_L\left(z_i^{(\ell_0+1)}-z_i'^{(\ell_0+1)}\right)=\ord_L\left(z_i^{(\ell_0)}+\varpi_L^{\upsilon_{i,\ell_0}}u_i^{(\ell_0+1)}-z_i^{(\ell_0)}-\varpi_L^{\upsilon_{i,\ell_0}}u_i'^{(\ell_0+1)}\right)\geq\upsilon_{i,\ell_0+1}=\min\{\nu, \upsilon_{i,\ell_0+1}\}
\end{equation*}
for all $i\in \mathfrak{I}$. This completes our proof.
\end{proof}
\begin{cor}\label{exuni}
With the notations an assumptions of Lemmas \ref{Cond3} and \ref{uniqne},  there is a unique point $z'\in z+(\varpi_L^{\upsilon_{1,0}}\cO_L,...,\varpi_L^{\upsilon_{n,0}}\cO_L)$ such that $\ord_L(g_X(z))\geq m-1, \ord_L\left(\frac{\partial g_X}{\partial y_i}(z)\right)\geq m-\eta_m$ if $i\notin I(h,X,m)$ and $\frac{\partial g_X}{\partial y_i}(z)=0$ if $i\in I(h,X,m)$.
\end{cor}
\begin{proof}
By Lemma \ref{Cond3}, there is a sequence  $(z^{(\ell)})_{\ell\geq 0}$ in $Y+\varpi_L^{\eta_m}\cO_L^{n}$ of type $(m,\beta)$ such that $z^{(0)}=z$. This together with the completeness of $L$ imply that the limit  $\displaystyle z'=\lim_{\ell\to+\infty}z^{(\ell)}$ exists. Moreover, it is easy to verify that $z^{(\ell)}\in z+(\varpi_L^{\upsilon_{1,0}}\cO_L,...,\varpi_L^{\upsilon_{n,0}}\cO_L)$ for all $\ell\geq 0$. Therefore $z'\in z+(\varpi_L^{\upsilon_{1,0}}\cO_L,...,\varpi_L^{\upsilon_{n,0}}\cO_L)$. Let $\tilde{z}'\in z+(\varpi_L^{\upsilon_{1,0}}\cO_L,...,\varpi_L^{\upsilon_{n,0}}\cO_L)$ such that $\ord_L(g_X(\tilde{z}'))\geq m-1, \ord_L\left(\frac{\partial g_X}{\partial y_i}(\tilde{z}')\right)\geq m-\eta_m$ if $i\notin I(h,X,m)$ and $\frac{\partial g_X}{\partial y_i}(\tilde{z}')=0$ if $i\in I(h,X,m)$. We need to show that $\tilde{z}=z'$. Indeed, we consider the  sequence $(\tilde{z}^{(\ell)})_{\ell\geq 0}$ in $Y+\varpi_L^{\eta_m}\cO_L^{n}$ given by  $\tilde{z}^{(0)}=z$ and $\tilde{z}^{(\ell)}=\tilde{z}$ for all  $\ell\geq 1$. Then $(\tilde{z}^{(\ell)})_{\ell\geq 0}$ is of type $(m,\beta)$ with $\tilde{z}^{(0)}=z$. Thus, it follows from Lemma \ref{uniqne}  that $\ord_L\left(\tilde{z}_i-z_i^{(\ell)}\right)\geq \upsilon_{i,\ell}$ for all $\ell\geq 0$ and all $1\leq i\leq n$. Hence, $\displaystyle\tilde{z}=\lim_{\ell\to+\infty}z^{(\ell)}=z'$.
\end{proof}
\begin{lem}\label{partition}If $z\in\cA_{h,X,m}$ then $\cA_{h,X,m}(z)\subset\cA_{h,X,m}$.  If $z,z'\in\cA_{h,X,m}$ then either $\cA_{h,X,m}(z)=\cA_{h,X,m}(z')$ or $\cA_{h,X,m}(z)\cap\cA_{h,X,m}(z')=\emptyset$.

\end{lem}
\begin{proof}
 Let $z\in\cA_{h,X,m}$. By Lemma \ref{delta0}, we have $\delta(z,m,i)\geq \eta_m$ for all $i$. So $\cA_{h,X,m}(z)\subset Y+\varpi_L^{\eta_m}\cO_L^{n}$. Let $u=(u_1,...,u_{n})\in\cO_L^{n}$ and $z'=z+(\varpi_L^{\delta(z,m,1)}u_1,...,\varpi_L^{\delta(z,m,n)}u_{n})\in \cA_{h,X,m}(z)$. We will show that $z'\in\cA_{h,X,m}$ and $\cA_{h,X,m}(z)=\cA_{h,X,m}(z')$. By the definition of $j_{m,z}$ and $\cA_{h,X,m}(z)$,  it suffices to prove that $\ord_L(g_X(z'))\geq m-1$ and \begin{equation}\label{inq0}\ord_L\left(\frac{\partial g_X}{\partial y_\ell}(z)-\frac{\partial g_X}{\partial y_\ell}(z')\right)>\frac{m+\tilde{\mu}_{\mathfrak{j}(\ell),m}(h,X)-1}{2}
 \end{equation}
 for all $\ell$.
 
We have 
\begin{align*}&\ord_L(z_i-z'_i)+\ord_L\left(\frac{\partial g_X}{\partial y_i}(z)\right)
\geq\delta(z,m,i)+\ord_L\left(\frac{\partial g_X}{\partial y_i}(z)\right)\\ \geq&\min\biggl\{m, m-1-\delta_{\mathfrak{j}(i),j_{m,z}}\min_{\ell\in I_{j_{m,z}}(h,X)}\ord_{L}\left(\frac{\partial g_X}{\partial y_\ell}(z)\right)+\ord_L\left(\frac{\partial g_X}{\partial y_i}(z)\right)\biggr\} \geq m-1,
\end{align*}
where $\delta_{a,b}$ is the Kronecker delta function. Since $\delta(z,m,i)\geq \eta_m$ for all $i$, we have 
$$\sum_{i=1}^{n}j_i\ord_L(z_i-z'_i)\geq \sum_{i=1}^{n}j_i\delta(z,m,i)\geq 3\eta_m\geq m-1$$ for all $J=(j_1,...,j_{n})\in\NN^{n}$ such that $|J|\geq 3$.
If $j_{m,z}\leq e(h,X,m)$, then it follows from Definition \ref{j-func} 
that
\begin{align*}
\mu_{j_{m,z}}(h,X)+2\left(m-1-\min_{\ell\in I_{j_{m,z}}(h,X)}\ord_{L}\left(\frac{\partial g_X}{\partial y_\ell}(z)\right)\right)
\geq& 2(m-1)+\mu_{j_{m,z}}(h,X)-(m+\mu_{j_{m,z}}(h,X)-2)\\
=&m.
\end{align*}
If $e(h,X)\geq j_{m,z}>e(h,X,m)$, then by Definition \ref{j-func} 
 and Lemma-Definition \ref{kaet}, we have 
\begin{align*}
\mu_{j_{m,z}}(h,X)+2\left(m-1-\min_{\ell\in I_{j_{m,z}}(h,X)}\ord_{L}\left(\frac{\partial g_X}{\partial y_\ell}(z)\right)\right)
\geq& \mu_{j_{m,z}}(h,X)+2(m-1)-2(m-1-\eta_m)\\
\geq& \kappa_m+2\eta_m
\geq m-1.
\end{align*}
Thus, since $\mu_{+\infty}(h,X)=+\infty$ and $\delta(z,m,i)=m$ if $\mathfrak{j}(i)\neq j_{m,z}$,  we have 
\begin{align*}&\mu_j(h,X)+\ord_L(z_i-z'_i)+\ord_L(z_\ell-z'_\ell)\\ \geq& \min\Biggl\{2m, \mu_{j_{m,z}}(h,X)+2\left(m-1-\min_{\ell\in I_{j_{m,z}}(h,X)}\ord_{L}\left(\frac{\partial g_X}{\partial y_\ell}(z)\right)\right)\Biggr\}\geq m-1
\end{align*} for all $i,j,\ell$ such that  $\mathfrak{j}(i)=j=\mathfrak{j}(\ell)$. For all $i,\ell$ with $\mathfrak{j}(i)\neq\mathfrak{j}(\ell)$, we have $\delta(z,m,i)=m$ or $\delta(z,m,\ell)=m$, thus
$$\eta_m+\ord_L(z_i-z'_i)+\ord_L(z_\ell-z'_\ell)\geq  \eta_m+ \delta(z,m,i)+\delta(z,m,\ell)\\
\geq m.$$
Therefore, we can use Lemma \ref{Cond1} to have $\ord_L(g_X(z'))\geq m-1$.

Let $1\leq\ell\leq n$. By Lemma \ref{bs1} and the fact that $\delta(z,m,i)=m$ if $\mathfrak{j}(i)\neq j_{m,z}$, we have 
$$\ord_L(z_i-z'_i)+\mu_{\mathfrak{j}(\ell)}(h,X)\geq m>\frac{m+\tilde{\mu}_{\mathfrak{j}(\ell),m}(h,X)-1}{2}$$
if $\mathfrak{j}(i)=\mathfrak{j}(\ell)\neq j_{m,z}$, and
\begin{align*}\ord_L(z_i-z'_i)+\mu_{\mathfrak{j}(\ell)}(h,X)\geq& m-1-\min_{\ell\in I_{j_{m,z}}(h,X)}\ord_{L}\left(\frac{\partial g_X}{\partial y_\ell}(z)\right)+ \mu_{j_{m,z}}(h,X)\\ 
\geq& m-1-\frac{m+\tilde{\mu}_{j_{m,z},m}(h,X)-1}{2}+\mu_{j_{m,z}}(h,X)
>\frac{m+\tilde{\mu}_{\mathfrak{j}(\ell),m}(h,X)-1}{2}
\end{align*}
if $\mathfrak{j}(i)=\mathfrak{j}(\ell)=j_{m,z}$. Similarly, by Lemma \ref{bs1}, Lemma-Definition \ref{kaet} and the fact that $\delta(z,m,i)=m$ if $\mathfrak{j}(i)\neq j_{m,z}$, we also have
$$\ord_L(z_i-z'_i)+\eta_m\geq m>\frac{m+\tilde{\mu}_{\mathfrak{j}(\ell),m}(h,X)-1}{2}$$
 if $\mathfrak{j}(i)\neq \mathfrak{j}(\ell)$ and $\mathfrak{j}(i)\neq j_{m,z}$, and 
 \begin{align*}
\ord_L(z_i-z'_i)+\eta_m&\geq m-1-\min_{\ell\in I_{j_{m,z}}(h,X)}\ord_{L}\left(\frac{\partial g_X}{\partial y_\ell}(z)\right)+\eta_m\\
&\geq m-1-\frac{m+\tilde{\mu}_{j_{m,z},m}(h,X)-1}{2}+\eta_m
>\frac{m+\tilde{\mu}_{\mathfrak{j}(\ell),m}(h,X)-1}{2}
 \end{align*}
 if $\mathfrak{j}(i)=j_{m,z}\neq \mathfrak{j}(\ell)$. Lastly, it follows from Lemma-Definition \ref{kaet} and Lemmas \ref{bs1}, \ref{delta0} that
$$\sum_{i=1}^{n}j_i\ord_L(z_i-z'_i)\geq 2\eta_m>\frac{m+\tilde{\mu}_{\mathfrak{j}(\ell),m}(h,X)-1}{2}$$ for all $1\leq \ell\leq n$ and all $J=(j_1,...,j_{n})\in\NN^{n}$ such that $|J|\geq 2$. Therefore, we can use Lemma \ref{Cond2} to deduce (\ref{inq0}) as desired.  
\end{proof}
\begin{lem}\label{equj}Suppose that $e(h,X,m)<j_{m,\tilde{z}} \leq e(h,X)$ for some $\tilde{z}\in Y+\varpi_L^{\eta_m}\cO_L^n$. Then $j_{m,z}=j_{m,\tilde{z}}$ for all $z\in Y+\varpi_L^{\eta_m}\cO_L^{n}$.
\end{lem}
\begin{proof}Let $z\in Y+\varpi_L^{\eta_m}\cO_L^{n}$.
By Definition \ref{j-func}, it is sufficient to prove that \begin{equation}\label{inq1}\ord_L\left(\frac{\partial g_X}{\partial y_\ell}(z)-\frac{\partial g_X}{\partial y_\ell}(\tilde{z})\right)>\frac{m+\tilde{\mu}_{\mathfrak{j}(\ell),m}(h,X)-1}{2}
\end{equation}
 for all $\ell$ such that $\mathfrak{j}(\ell)>e(h,X,m)$.
 
Let $1\leq\ell\leq n$ such that $\mathfrak{j}(\ell)>e(h,X,m)$. Then we have $\mu_{\mathfrak{j}(\ell)}(h,X)\geq \kappa_m$ and $\tilde{\mu}_{\mathfrak{j}(\ell),m}(h,X)=m-2\eta_m-1$. Thus, by Lemma-Definition \ref{kaet} and Definition \ref{j-func},  we have  $$\ord_L(z_i-\tilde{z}_i)+\mu_{\mathfrak{j}(\ell)}(h,X)\geq \eta_m+ \kappa_m>m-\eta_m-1=\frac{m+\tilde{\mu}_{\mathfrak{j}(\ell),m}(h,X)-1}{2}$$ if $\mathfrak{j}(i)=\mathfrak{j}(\ell)$, and 
$$\ord_L(z_i-\tilde{z}_i)+\eta_m\geq 2\eta_m>m-\eta_m-1=\frac{m+\tilde{\mu}_{\mathfrak{j}(\ell),m}(h,X)-1}{2}$$ if $\mathfrak{j}(i)\neq \mathfrak{j}(\ell)$. Moreover, we also have 
$$\sum_{i=1}^{n}j_i\ord_L(z_i-\tilde{z}_i)\geq 2\eta_m>m-\eta_m-1=\frac{m+\tilde{\mu}_{\mathfrak{j}(\ell),m}(h,X)-1}{2}$$ for all $J=(j_1,...,j_{n})\in\NN^{n}$ such that $|J|\geq 2$. Therefore, we can use Lemma \ref{Cond2} to obtain (\ref{inq1}) and complete our proof. 
\end{proof}
\begin{lem}\label{eq11}Let $\psi$ be an additive character of $L$ of conductor $m$.
 We have 
$$\int_{\cA_{h,X,m}(z)}\psi(g_X(y))|dy|=\int_{\cA_{h,X,m}}\psi(g_X(y))|dy|=0$$
for all $X\in\cO_L^{n}$ and $z\in\cA_{h,X,m}$.
\end{lem}
\begin{proof}
Let $u=(u_1,...,u_{n})\in\cO_L^{n}$, we have 
\begin{align*}
&g_X(z+(\varpi_L^{\delta(z,m,1)}u_1,...,\varpi_L^{\delta(z,m,n)}u_{n}))\\=&g_X(z)+\sum_{1\leq i\leq n}\frac{\partial g_X}{\partial y_i}(z)\varpi_L^{\delta(z,m,i)}u_i+\frac{1}{2}\sum_{i,\ell=1}^{n}\frac{\partial^2 g_X}{\partial y_\ell\partial y_i}(z)\varpi_L^{\delta(z,m,\ell)+\delta(z,m,i)}u_\ell u_i\\&+ \sum_{|J|\geq 3}\frac{1}{J!}\frac{\partial^{|J|}g_X}{\partial y^J}(z)\varpi_L^{\langle J,\delta(z,m)\rangle}u^J.
\end{align*}
By  Lemma \ref{delta0}, we have $\delta(z,m,i)\geq \eta_m$ for all $i$. Thus, $\langle J,\delta(z,m)\rangle\geq m$ if $|J|\geq 3$. Moreover, we also have $$\ord_L\left(\frac{\partial^2 g_X}{\partial y_\ell\partial y_i}(z)\varpi_L^{\delta(z,m,\ell)+\delta(z,m,i)}\right)\geq m$$ for all $i,\ell$, and $\ord_L\left(\frac{\partial g_X}{\partial y_i}(z)\varpi_L^{\delta(z,m,i)}\right)\geq  m$ if $i\notin I_{j_{m,z}}(h,X)$. On the other hand, since $\delta(z,m,\ell)=m-1-\min_{i\in I_{j_{m,z}}(h,X)} \ord_L\left(\frac{\partial g_X}{\partial y_i}(z)\right)$ for all $\ell\in I_{j_{m,z}}(h,X)$, we have 
\begin{align*}
\min_{\ell\in I_{j_{m,z}}(h,X)} \ord_L\left(\frac{\partial g_X}{\partial y_\ell}(z)\varpi_L^{\delta(z,m,\ell)}\right)=&\min_{\ell\in I_{j_{m,z}}(h,X)} \ord_L\left(\frac{\partial g_X}{\partial y_\ell}(z)\right)+m-1-\min_{i\in I_{j_{m,z}}(h,X)} \ord_L\left(\frac{\partial g_X}{\partial y_i}(z)\right)\\=& m-1.
\end{align*}
By combining the above calculation, the orthogonality of characters and the fact that $\psi$ is of conductor $m\geq 2$, we have that $\psi_{\varpi_L^{m-1}}$ is of conductor $1$ and
\begin{align*}
\int_{\cA_{h,X,m}(z)}\psi(g_X(y))|dy|=& \psi(g_X(z))q_L^{-\sum_{1\leq i\leq n}\delta(z,m,i)}\int_{\cO_L^{n}}\psi_{\varpi_L^{m-1}}\bigg(\sum_{i\in I_{j_{m,z}}(h,X)}\frac{\partial g_X}{\partial y_i}(z)\varpi_L^{\delta(z,m,i)-m+1}u_i\bigg)|du|\\
\\=& 0.
\end{align*}
for all $X\in\cO_L^{n}$ and $z\in\cA_{h,X,m}$.
Lemma \ref{partition} implies that we can write $\cA_{h,X,m}$ as a disjoint union of open sets of the form $\cA_{h,X,m}(z)$ for $z\in\cA_{h,X,m}$. Thus we also have 
$$\int_{\cA_{h,X,m}(z)}\psi(g_X(y))|dy|=0$$
for all $X\in\cO_L^{n}$.
\end{proof}
Now, we try to understand the complement of $\cA_{h,X,m}$ in $\phi_X(X+\varpi_L^{\eta_m}\cO_L^n)$. To do this, we introduce some new sets. 
\begin{defn}\label{Bset}
We set 
\begin{align*}\cB_{h,X,m}:=&\biggl\{z\in Y+\varpi_L^{\eta_{m}}\cO_L^{n}\mid \frac{\partial g_X}{\partial y_i}(z)=0 \textnormal{ if } i\in I(h,X,m), j_{m,z}=+\infty, \ord_L(g_X(z))\geq m-1\biggr\}\\
=&\biggl\{z\in Y+\varpi_L^{\eta_{m}}\cO_L^{n}\mid \frac{\partial g_X}{\partial y_i}(z)=0 \textnormal{ if } i\in I(h,X,m), \\ &\ord_L\left(\frac{\partial g_X}{\partial y_i}(z)\right)\geq m-\eta_m \textnormal{ if } i\notin I(h,X,m),\ord_L(g_X(z))\geq m-1\biggr\}.
\end{align*}
and 
\begin{align*}\tilde{\cB}_{h,X,m}:=&\biggl\{z\in X+\varpi_L^{\eta_{m}}\cO_L^{n}\mid \frac{\partial h}{\partial x_i}(z)=0 \textnormal{ if } i\in I(h,X,m),\\ & \ord_L\left(\frac{\partial h}{\partial x_i}(z)\right)\geq m-\eta_m \textnormal{ if } i\notin I(h,X,m),\ord_L(h(z))\geq m-1\biggr\}.
\end{align*}
\end{defn}
\begin{defn}\label{Bzset}
Let $z\in\cB_{h,X,m}$.  We set
$$\cB_{h,X,m}(z):= z+\prod_{1\leq i\leq n} \varpi_L^{\gamma(h,X,m,i)}\cO_L$$
if $z\in \cB_{h,X,m}$, where $$\gamma(h,X,m,i):=m \textnormal{ if } \mathfrak{j}(i)>e(h,X,m)$$
and
$$\gamma(h,X,m,i):=\biggl\lfloor\frac{m+\tilde{\mu}_{\mathfrak{j}(i),m}(h,X)-1}{2}\biggr\rfloor+1-\mu_{\mathfrak{j}(i)}(h,X)= \biggl\lfloor\frac{m-\mu_{\mathfrak{j}(i)}(h,X)}{2}\biggr\rfloor$$ if $\mathfrak{j}(i)\leq e(h,X,m)$.
\end{defn}

The following lemma shows that the complement of $\cA_{h,X,m}$ in $\phi_X(X+\varpi_L^{\eta_m}\cO_L^n)\setminus \{y\mid \ord_L(g_X(y))<m-1\}$ can be written as a disjoint union of sets of the form $\cB_{h,X,m}(z)$. 
\begin{lem}\label{part1}If $z,z'\in\cB_{h,X,m}$ then either $\cB_{h,X,m}(z)\cap\cB_{h,X,m}(z')=\emptyset$ or $\cB_{h,X,m}(z)=\cB_{h,X,m}(z')$. Moreover, we have 
$$\{z\in Y+\varpi_L^{\eta_m}\cO_L^{n}| \ord_L(g_X(z))\geq m-1\}=\cA_{h,X,m}\sqcup(\cup_{z\in \cB_{h,X,m}}\cB_{h,X,m}(z)).$$

\end{lem}
\begin{proof}
The first claim is trivial by the definition.

For the second claim, we note that if $e(h,X,m)<j_{m,Y}\leq e(h,X)$ then $\{z\in Y+\varpi_L^{\eta_m}\cO_L^{n}| \ord_L(g_X(z))\geq m-1\}=\cA_{h,X,m}$ by the definition. Moreover, $\cB_{h,X,m}=\emptyset$ since $j_{m,z}=j_{m,Y}$ for all $z\in Y+\varpi_L^{\eta_m}\cO_L^{n}$ as shown in Lemma \ref{equj}. So our claim is trivial in this case.

Now, we suppose that $j_{m,Y}\leq e(h,X,m)$ or $j_{m,Y}=+\infty$.  We firstly prove that $\cB_{h,X,m}(z)\subset Y+\varpi_L^{\eta_m}\cO_L^{n}$. By Definition  \ref{Bzset} and Lemma-Definition \ref{kaet}, we have \begin{equation}\label{eq-m}\gamma(h,X,m,i)=\biggl\lfloor\frac{m-\mu_{\mathfrak{j}(i)}(h,X)}{2}\biggr\rfloor\geq\frac{m-\mu_{\mathfrak{j}(i)}(h,X)-1}{2}\geq \frac{m-\kappa_m}{2}>\eta_m-1
\end{equation} if  $\mathfrak{j}(i)\leq e(h,X,m)$ and $\gamma(h,X,m,i)=m$ if $\mathfrak{j}(i)>e(h,X,m)$. Since $z\in \cB_{h,X,m}\subset Y+\varpi_L^{\eta_m}\cO_L^{n}$, we have  $$\cB_{h,X,m}(z)=z+(\varpi_L^{\gamma(h,X,m,i)}\cO_L,...,\varpi_L^{\gamma(h,X,m,n)}\cO_L)\subset Y+\varpi_L^{\eta_m}\cO_L^{n}.$$ Let $z'\in \cB_{h,X,m}(z)$, we will prove that $z'\notin\cA_{h,X,m}$ by showing that $j_{m,z'}=+\infty$. Since $j_{m,z}=+\infty$ and $z\in \cB_{h,X,m}$, to show that $j_{m,z'}=+\infty$, it is sufficient to prove that 
\begin{equation}\label{eqm1}\ord_L\left(\frac{\partial g_X}{\partial y_\ell}(z)-\frac{\partial g_X}{\partial y_\ell}(z')\right)>\frac{m+\tilde{\mu}_{\mathfrak{j}(\ell),m}(h,X)-1}{2}
\end{equation} for all $\ell$. 

By Lemma \ref{bs1} and the fact that $\gamma(h,X,m,i)=m$ if $\mathfrak{j}(i)>e(h,X,m)$, we have 
\begin{align*}
\ord_L(z_i-z'_i)+\mu_{\mathfrak{j}(\ell)}(h,X)&\geq \gamma(h,X,m,i)+ \mu_{\mathfrak{j}(\ell)}(h,X)=\biggl\lfloor\frac{m-\mu_{\mathfrak{j}(i)}(h,X)}{2}\biggr\rfloor+\mu_{\mathfrak{j}(\ell)}(h,X)\\
&>\frac{m+\tilde{\mu}_{\mathfrak{j}(\ell),m}(h,X)-1}{2}
\end{align*} if $\mathfrak{j}(i)=\mathfrak{j}(\ell)\leq e(h,X,m)$, and $$\ord_L(z_i-z'_i)+\mu_{\mathfrak{j}(\ell)}(h,X)\geq m>\frac{m+\tilde{\mu}_{\mathfrak{j}(\ell),m}(h,X)-1}{2}$$ if $\mathfrak{j}(i)=\mathfrak{j}(\ell)>e(h,X,m)$. By Lemma \ref{bs1} and (\ref{eq-m}), we have 
\begin{align*}
\ord_L(z_i-z'_i)+\eta_m\geq \gamma(h,X,m,i)+\eta_m\geq 2\eta_m>\frac{m+\tilde{\mu}_{\mathfrak{j}(\ell),m}(h,X)-1}{2}
\end{align*} 
if $\mathfrak{j}(i)\neq \mathfrak{j}(\ell)$, and
$$\sum_{i=1}^{n}j_i\ord_L(z_i-z'_i)\geq \sum_{i=1}^{n}j_i\gamma(h,X,m,i)\geq 2\eta_m>\frac{m+\tilde{\mu}_{\mathfrak{j}(\ell),m}(h,X)-1}{2}$$ for all $J=(j_1,...,j_{n})\in\NN^{n}$ such that $|J|\geq 2$. Thus, we can use Lemma \ref{Cond2} to obtain (\ref{eqm1}). Therefore, we claim that $\cB_{h,X,m}(z)\cap \cA_{h,X,m}=\emptyset$ for all $z\in \cB_{h,X,m}$ and $$\{z\in Y+\varpi_L^{\eta_m}\cO_L^{n}| \ord_L(g_X(z))\geq m-1\}\supseteq\cA_{h,X,m}\sqcup(\cup_{z\in \cB_{h,X,m}}\cB_{h,X,m}(z)).$$


By Lemma \ref{equj}, to complete the proof, it remains to show that if $z'\in Y+\varpi_L^{\eta_m}\cO_L^{n}$ such that $\ord_L(g_X(z'))\geq m-1$ and $j_{m,z'}=+\infty$ then there is $z\in \cB_{h,X,m}$ such that $z'\in \cB_{h,X,m}(z)$. Indeed, let $z'$ such that $\ord_L(g_X(z'))\geq m-1$ and $j_{m,z'}=+\infty$.  By Definitions \ref{j-func} and \ref{Bzset} 
we have 
\begin{align*}\gamma(h,X,m,i)+\mu_{\mathfrak{j}(i)}(h,X)&=\biggl\lfloor\frac{m-\mu_{\mathfrak{j}(i)}(h,X)}{2}\biggr\rfloor+\mu_{\mathfrak{j}(i)}(h,X)=\biggl\lfloor\frac{m+\mu_{\mathfrak{j}(i)}(h,X)}{2}\biggr\rfloor\\
&=\biggl\lfloor\frac{m+\tilde{\mu}_{\mathfrak{j}(i),m}(h,X)+1}{2}\biggr\rfloor\leq \min_{\ell\in I_{\mathfrak{j}(i)}(h,X)}\ord_L\left(\frac{\partial g_X}{\partial y_\ell
}(z')\right)
\end{align*}
if $i\in I(h,X,m)$. On the other hand, by Lemma-Definition \ref{kaet} and the fact that $\mu_{\mathfrak{j}(i)}(h,X)\leq \kappa_m-1$ if $i\in I(h,X,m)$, we have  
\begin{align}\label{ieqgam}
\gamma(h,X,m,i_1)+\gamma(h,X,m,i_2)&=\biggl\lfloor\frac{m-\mu_{\mathfrak{j}(i_1)}(h,X)}{2}\biggr\rfloor+\biggl\lfloor\frac{m-\mu_{\mathfrak{j}(i_1)}(h,X)}{2}\biggr\rfloor\geq 2\biggl\lfloor\frac{m-\kappa_m+1}{2}\biggr\rfloor\nonumber\\
&\geq \biggl\lfloor\frac{m+\kappa_m-1}{2}\biggr\rfloor+1\geq \biggl\lfloor\frac{m+\mu_{\mathfrak{j}(i_3)}(h,X)}{2}\biggr\rfloor+1.
\end{align}
for all $i_1,i_2,i_3\in I(h,X,m)$. Moreover, since $\eta_m\geq \kappa_m\geq 1+\mu_{\mathfrak{j}(i)}(h,X)$ if $i\in I(h,X,m)$, we also have 
$$\left(\frac{2\eta_m+m-\mu_{\mathfrak{j}(i_1)}(h,X)}{2}\right)-\left(\frac{m+\mu_{\mathfrak{j}(i_2)}(h,X)}{2}\right)=\frac{2\eta_m-\mu_{\mathfrak{j}(i_1)(h,X)}-\mu_{\mathfrak{j}(i_2)}(h,X)}{2}\geq 1$$
for all $i_1,i_2\in I(h,X,m)$. Thus, 
\begin{equation}\label{lowbeta}\biggl\lfloor\frac{m+\mu_{\mathfrak{j}(i_1)}(h,X)}{2}\biggr\rfloor-\mu_{\mathfrak{j}(i_1)}(h,X)+\eta_m=\biggl\lfloor \frac{2\eta_m+m-\mu_{\mathfrak{j}(i_1)}(h,X)}{2}\biggr\rfloor>\biggl\lfloor\frac{m+\mu_{\mathfrak{j}(i_2)}(h,X)}{2}\biggr\rfloor
\end{equation}
for all $i_1,i_2\in I(h,X,m)$. Therefore, we can use Corollary  \ref{exuni} for $\beta_i=\bigl\lfloor\frac{m+\mu_{\mathfrak{j}(i)}(h,X)}{2}\bigr\rfloor=\gamma(h,X,m,i)+\mu_{\mathfrak{j}(i)}(h,X)$ if $i\in I(h,X,m)$ to have $z\in \cB_{h,X,m}$ such that $z'\in \cB_{h,X,m}(z)$.
\end{proof}
Since the volume of $\cB_{h,X,m}(z)$ is independent of $z$,  to compute the volume of  the set $$\phi_X(X+\varpi_L^{\eta_m}\cO_L^n)\setminus \left(\cA_{h,X,m}\cup\{y\mid \ord_L(g_X(y))<m-1\}\right),$$ we need to know the number of  cylinders  of the form $\cB_{h,X,m}(z)$. The following lemma shows that this number is equal to $\#\left(\cB_{h,X,m}\mod \varpi_L^m\cO_L\right)$.

\begin{lem}\label{ident} If $z,z'\in\cB_{h,X,m}$ then $\cB_{h,X,m}(z)=\cB_{h,X,m}(z')$ if and only if $z-z'\in\varpi_L^m\cO_L^{n}$. More generally, let $z,z'\in\cB_{h,X,m}$, $\beta=(\beta_{1},...,\beta_{n})\in\NN^{n}$ and $e\in\NN$ such that the following conditions hold:
\begin{itemize}
\item $e\geq \beta_i-\mu_{\mathfrak{j}(i)}(h,X,m)$ for all $i\in I(h,X,m)$. 
\item $\beta_i=m-\eta_m$ for all $i\notin I(h,X,m)$.
\item $\beta_i\geq m-\eta_m$ for all $i\in I(h,X,m)$.
\item $\beta_{i_1}+\beta_{i_2}-\mu_{\mathfrak{j}(i_1)}(h,X,m)-\mu_{\mathfrak{j}(i_2)}(h,X,m)\geq \beta_{i_3}+1$ for all $i_1,i_2,i_3\in I(h,X,m)$.
\item $\max_{i\in I(h,X,m)}\beta_i<\eta_m+\min_{i\in I(h,X,m)}(\beta_i-\mu_{\mathfrak{j}(i)}(h,X))$.
\item $\beta_i=\beta_{i'}$ for all $i,i'$ with $\mathfrak{j}(i)=\mathfrak{j}(i')$.
\end{itemize}  Then  $z-z'\in (\varpi_L^{\alpha_1(\ell,m)}\cO_L,...,\varpi_L^{\alpha_{n}(\ell,m)}\cO_L)$ if and only if $z-z'\in \varpi_L^{e}\cO_L^{n}$, where $\alpha_i(\ell,m)=\beta_i-\mu_{\mathfrak{j}(i)}(h,X,m)$ for all $i\in I(h,X,m)$ and $\alpha_i(\ell,m)=e$ for all $i\notin I(h,X,m)$.
\end{lem}
\begin{proof}By the inequalities (\ref{ieqgam}) and (\ref{lowbeta}) in the proof of Lemma \ref{part1}, the first assertion follows from the second statement by taking $\beta_i=\bigl\lfloor\frac{m+\mu_{\mathfrak{j}(i)}(h,X)}{2}\bigr\rfloor=\gamma(h,X,m,i)+\mu_{\mathfrak{j}(i)}(h,X)$ for all $i\in I(h,X,m)$, $\beta_i=m-\eta_m$ for all $i\notin I(h,X,m)$ and $e=m$.  

Now, it is clear that if $z-z'\in \varpi_L^e\cO_L^{n}$ then $z-z'\in (\varpi_L^{\alpha_1(\ell,m)}\cO_L,...,\varpi_L^{\alpha_{n}(\ell,m)}\cO_L).$ Conversely, suppose that $z-z'\in (\varpi_L^{\alpha_1(\ell,m)}\cO_L,...,\varpi_L^{\alpha_{n}(\ell,m)}\cO_L)$.  Since $z,z'\in\cB_{h,X,m}$, the two sequences $(z^{(\ell)})_{\ell\geq 0}$ and $(z'^{(\ell)})_{\ell\geq 0}$ with $z^{(\ell)}=z$ and $z'^{(\ell)}=z'$ for all $\ell\geq 0$ are of type $(\beta,m)$. Thus, by Lemma \ref{uniqne}, we have $\ord_L(z_i-z_i')\geq \min\{e,\alpha_i(\ell,m)+\ell\}$ for all $1\leq i\leq n$ and all $\ell\geq 0$. Let $\ell$ be a large enough integer, we have $z-z'\in\varpi_L^e\cO_L^{n}$. The second assertion is proved.
\end{proof}

Now, we introduce the following sets and relate them  to  $\#\left(\cB_{h,X,m}\mod \varpi_L^{m}\cO_L\right)$.
\begin{defn}\label{Cset}
For each $\ell\geq -\eta_m$, we set
\begin{align*}\cC_{h,X,m,\ell}:=&\biggl\{z\in Y+\varpi_L^{\eta_{m}}\cO_L^{n}\mid \ord_L\left(\frac{\partial g_X}{\partial x_i}(z)\right)\geq m+\ell\textnormal{ if } i\in I(h,X,m),\\&\ord_L\left(\frac{\partial g_X}{\partial x_i}(z)\right)\geq m-\eta_m\textnormal{ if } i\notin I(h,X,m),\ord_L(g_X(z))\geq m-1\biggr\}.
\end{align*}
and
\begin{align*}\tilde{\cC}_{h,X,m,\ell}:=&\biggl\{z\in X+\varpi_L^{\eta_{m}}\cO_L^{n}\mid \ord_L\left(\frac{\partial h}{\partial x_i}(z)\right)\geq m+\ell\textnormal{ if } i\in I(h,X,m),\\ &\ord_L\left(\frac{\partial h}{\partial x_i}(z)\right)\geq m-\eta_m \textnormal{ if } i\notin I(h,X,m),\ord_L(h(z))\geq m-1\biggr\}.
\end{align*}
\end{defn}
\begin{lem}\label{reduct0}Let $\ell\in\ZZ_{\geq -\eta_m}$. We have $\cC_{h,X,m,\ell}=\cB_{h,X,m}+(\varpi_L^{\lambda_1(\ell,m)}\cO_L,...,\varpi_L^{\lambda_{n}(\ell,m)}\cO_L)$, where $\lambda_i(\ell,m)=\ell+m-\mu_{\mathfrak{j}(i)}(h,X)$ if $i\in I(h,X,m)$ and $\lambda_i(\ell,m)=\ell+m$ if $i\notin I(h,X,m)$. In addition, we have 
$$\tilde{\cB}(h,X,m)+\varpi_L^{m+\ell}\cO_L^{n}\subset\tilde{\cC}_{h,X,m,\ell}\subset\tilde{\cB}(h,X,m)+\varpi_L^{\min_{i}\lambda_i(\ell,m)}\cO_L^{n} $$ and 
$$\#(\tilde{\cB}_{h,X,m}\mod\varpi_L^{m+\ell}\cO_L)q_L^{-(m+\ell)n+\sum_{1\leq j\leq e(h,X,m)}\mu_j(h,X)r_j(h,X)}=\int_{\tilde{\cC}_{h,X,m,\ell}}|dx|$$
for all $\ell\in \ZZ_{\geq -\eta_m}$.
\end{lem}
\begin{proof}
First of all, we show the first assertion that $\cC_{h,X,m,\ell}=\cB_{h,X,m}+(\varpi_L^{\lambda_1(\ell,m)}\cO_L,...,\varpi_L^{\lambda_{n}(\ell,m)}\cO_L)$ for all $\ell\in \ZZ_{\geq -\eta_m}$.  Indeed, let $\ell\in \ZZ_{\geq -\eta_m}$ and $z'\in \cC_{h,X,m,\ell}$. By Lemma-Definition \ref{kaet}, we have  $m+1\geq \eta_m+2\kappa_m$. This and the fact that $\mu_{\mathfrak{j}(i)}(h,X)\leq \kappa_m-1$ for all $i\in I(h,X,m)$ imply
\begin{equation}\label{ineqm}
m+\ell-\mu_{\mathfrak{j}(i_1)}(h,X)+m+\ell- \mu_{\mathfrak{j}(i_2)}(h,X)\geq2(m+\ell)-2(\kappa_m-1)\geq m+\ell+1
\end{equation} for all $\ell\geq -\eta_m$ and all $i_1,i_2\in I(h,X,m)$. Thus, we can use Corollary \ref{exuni} with $\beta_i=\ell+m$ for all $i\in I(h,X,m)$ and $\beta_i=m-\eta_m$ for all $i\notin I(h,X,m)$ to obtain $z\in z'+(\varpi_L^{\lambda_1(\ell,m)}\cO_L,...,\varpi_L^{\lambda_{n}(\ell,m)}\cO_L)$ such that $z\in \cB_{h,X,m}$. Conversely, let $z\in\cB_{h,X,m}$ and 
$z'\in z+(\varpi_L^{\lambda_1(\ell,m)}\cO_L,...,\varpi_L^{\lambda_{n}(\ell,m)}\cO_L)$. By the inequalities (\ref{eqi1}), (\ref{eqi2}), (\ref{eqi3}), (\ref{eqi4}), (\ref{eqi5}), (\ref{eqi6}) and (\ref{lbou}) in the proof of Lemma \ref{Cond3} and the fact that $(\beta_1,...,\beta_{n})$ satisfies the conditions of Lemma \ref{Cond3}, we can use Lemmas \ref{Cond1} and \ref{Cond2} 
to have $z'\in\cC_{h,X,m,\ell}$. 

It is easy to verify from Lemma \ref{propt} that $\tilde{\cB}_{h,X,m}=\cB_{h,X,m}\mathsf{A}(X)$ and $\tilde{\cC}_{h,X,m,\ell}=\cC_{h,X,m,\ell}\mathsf{A}(X)$. These facts and Lemma \ref{propt} together with the first assertion imply that 
$$\tilde{\cB}(h,X,m)+\varpi_L^{m+\ell}\cO_L^{n}\subset\tilde{\cC}_{h,X,m,\ell}\subset\tilde{\cB}(h,X,m)+\varpi_L^{\min_{i}\lambda_i(\ell,m)}\cO_L^{n},$$
\begin{equation}\label{modmlB}\#(\tilde{\cB}_{h,X,m} \mod \varpi_L^{m+\ell}\cO_L)=\#(\cB_{h,X,m} \mod \varpi_L^{m+\ell}\cO_L)
\end{equation}
and 
\begin{equation}\label{tildntil}\int_{\tilde{\cC}_{h,X,m,\ell}}|dx|=\int_{\cC_{h,X,m,\ell}}|dy|
\end{equation}
for all $\ell\in \NN$.

For the remaining assertion, by (\ref{modmlB}), (\ref{tildntil}) and the first assertion, it suffices to show that
$$\#(\cB_{h,X,m}\mod\varpi_L^{m+\ell}\cO_L)q_L^{-(m+\ell)n+\sum_{1\leq j\leq e(h,X,m)}\mu_j(h,X)r_j(h,X)}=\int_{\cB_{h,X,m}+(\varpi_L^{\lambda_1(\ell,m)}\cO_L,...,\varpi_L^{\lambda_{n}(\ell,m)}\cO_L)}|dx|$$
for all $\ell\in \ZZ_{\geq -\eta_m}$. Thus, we need to prove that if $z,z'\in \cB_{h,X,m}$ then $z'\in z+(\varpi_L^{\lambda_1(\ell,m)}u_1,...,\varpi_L^{\lambda_{n}(\ell,m)}u_{n})$ if and only if $z'\in z+\varpi_L^{\ell+m}\cO_L^{n}.$ Indeed, let $z,z'\in \cB_{h,X,m}$. If $z'\in z+\varpi_L^{\ell+m}\cO_L^{n}$ then $z'\in z+(\varpi_L^{\lambda_1(\ell,m)}\cO_L,...,\varpi_L^{\lambda_{n}(\ell,m)}\cO_L)$. Conversely,  suppose that $z'\in z+(\varpi_L^{\lambda_1(\ell,m)}\cO_L,...,\varpi_L^{\lambda_{n}(\ell,m)}\cO_L)$.  By (\ref{ineqm}), we can use Lemma \ref{ident} with $\beta_i=m+\ell$ for all $i\in I(h,X,m)$ and $\beta_i=m-\eta_m$ for all $i\notin I(h,X,m)$ to have $z-z'\in \varpi_L^{m+\ell}\cO_L^{n}$. 
\end{proof}

\begin{lem}\label{reduct1}We have 
$$\#(\tilde{\cB}_{h,X,m} \mod \varpi_L^m\cO_L)q_L^{\ell(n-r(h,X,m))}=\#(\tilde{\cB}_{h,X,m} \mod \varpi_L^{m+\ell}\cO_L)$$
for all $\ell\in\ZZ_{\geq -\eta_m}$.
\end{lem}
\begin{proof}Let $\ell\in\ZZ_{\geq -\eta_m}$. We set $\hat{\alpha}(\ell)=\max\{m,m+\ell\}$ and $\alpha(\ell)=\min\{m,m+\ell\}$. Moreover, we set $\beta_i=\alpha(\ell)+ \mu_{\mathfrak{j}(i)}(h,X)\geq m-\eta_m$ for all $i\in I(h,X,m)$ and $\beta_i=m-\eta_m$ for all $i\notin I(h,X,m)$. Since $m\geq 2$, we have $\beta_i<\eta_m+\alpha(\ell)$ 
for all $i\in I(h,X,m)$. Moreover, the inequalities $\ell\geq -\eta_m$ and $$2\alpha(\ell)-(\alpha(\ell)+\mu_{\mathfrak{j}(i)}(h,X)+1)\geq m-\eta_m-(\kappa_m-1)-1\geq 0$$
 imply that $\beta=(\beta_1,...,\beta_{n})$ satisfies the assumptions of Lemma \ref{Cond3}. 
By (\ref{modmlB}), it suffices to prove that 
\begin{equation}\label{cBcount}\#(\cB_{h,X,m} \mod \varpi_L^m\cO_L)q_L^{\ell(n-r(h,X,m))}=\#(\cB_{h,X,m} \mod \varpi_L^{m+\ell}\cO_L).
\end{equation}

 If $z\in\cB_{h,X,m}$ and $z'\in z+\varpi_L^{\alpha(\ell)}\cO_L^{n}$,  then  we show the fact that there exists a unique point $\overline{z}_{\ell}'$ in $$\left(\left(z'+(\varpi_L^{\alpha_1(\ell,m)}\cO_L,...,\varpi_L^{\alpha_{n}(\ell,m)}\cO_L)\right)\cap\cB_{h,X,m}\right)\mod \varpi_L^{\hat{\alpha}(\ell)}\cO_L,$$ where $\alpha_i(\ell,m)=\alpha(\ell)$ for all $i\in I(h,X,m)$ and $\alpha_i(\ell,m)=\hat{\alpha}(\ell)$ for all $i\notin I(h,X,m)$. Indeed, since $z\in \cB_{h,X,m}$,  $z'\in z+\varpi_L^{\alpha(\ell)}\cO_L^{n}$ and $\beta=(\beta_1,...,\beta_{n})$ satisfies the assumptions of Lemma \ref{Cond3},  we can use the inequalities (\ref{eqi1}), (\ref{eqi2}), (\ref{eqi3}), (\ref{eqi4}), (\ref{eqi5}), (\ref{eqi6}) and (\ref{lbou}) in the proof of Lemma \ref{Cond3} together with Lemmas \ref{Cond1} and \ref{Cond2} to have  $\ord_L\left(\frac{\partial g_X}{\partial y_i}(z')\right)\geq m-\eta_m$ for all $i\notin I(h,X,m)$, $\ord_L(g_X(z'))\geq m-1$ and $\ord_L\left(\frac{\partial g_X}{\partial y_i}(z')\right)\geq \beta_i$ for all $i\in I(h,X,m)$. Thus, the existence of $\overline{z}_{\ell}'$ follows from Corollary \ref{exuni} while the uniqueness of $\overline{z}_{\ell}'$ follows from Lemma \ref{ident}. 
 
From this fact, we have $$\#\left(\left(\left(z+(\varpi_L^{\alpha(\ell)}\cO_L,...,\varpi_L^{\alpha(\ell)}\cO_L)\right)\cap\cB_{h,X,m}\right)\mod \varpi_L^{\hat{\alpha}(\ell)}\cO_L\right)=q_L^{\sum_{1\leq i\leq n}(\alpha_i(\ell,m)-\alpha(\ell))}=q_L^{(n-r(h,X,m))|\ell|}$$
for all $z\in \cB_{h,X,m}$. This implies that
 $$q_L^{|\ell|(n-r(h,X,m))}\#(\cB_{h,X,m} \mod \varpi_L^{\alpha(\ell)}\cO_L)=\#(\cB_{h,X,m} \mod \varpi_L^{\hat{\alpha}(\ell)}\cO_L).$$
Then, we can deduce (\ref{cBcount}) by observing that $\{\alpha(\ell),\hat{\alpha}(\ell)\}=\{m,m+\ell\}$.  
\end{proof}
\begin{cor}\label{equ01}We have 
$$\int_{\tilde{\cC}_{h,X,m,\ell}}|dx|=q_L^{r(h,X,m)}\int_{\tilde{\cC}_{h,X,m,\ell+1}}|dx|$$
for all $\ell\in\ZZ_{\geq -\eta_m}$. Consequently, we have
\begin{equation*}
\int_{\tilde{\cC}_{h,X,m,\ell}}|dx|=q_L^{-mn-\ell r(h,X,m)+\sum_{1\leq j\leq e(h,X,m)}\mu_j(h,X)r_j(h,X)}\#(\tilde{\cB}_{h,X,m} \mod \varpi_L^m\cO_L)
\end{equation*}
for all $\ell\in\ZZ_{\geq -\eta_m}$.
\end{cor}
\begin{proof}
The assertions follow by combining Lemmas \ref{reduct0} and \ref{reduct1}.
\end{proof}
\begin{lem}\label{count1}We have
\begin{equation*}
\int_{\cup_{z\in \cB_{h,X,m}}\cB_{h,X,m}(z)}|dy|= q_L^{(m+\ell) r(h,X,m)-\theta(h,X,m)}\int_{\tilde{\cC}_{h,X,m,\ell}}|dx|
\end{equation*}
for all $\ell\geq -\eta_m$, where 
$$\theta(h,X,m)=\sum_{1\leq j\leq e(h,X,m)}\biggl\lfloor\frac{m+\mu_j(h,X)}{2}\biggr\rfloor r_j(h,X).$$
\end{lem}
\begin{proof}
It follows from   Lemmas \ref{part1}, \ref{ident} and  the equality (\ref{modmlB}) that  
\begin{align*}
\int_{\cup_{z\in \cB_{h,X,m}}\cB_{h,X,m}(z)}|dy|&=\#(\cB_{h,X,m} \mod \varpi_L^m\cO_L) q_L^{-\sum_{i=1}^{n} \gamma(h,X,m,i)}\\
&=\#(\tilde{\cB}_{h,X,m} \mod \varpi_L^m\cO_L) q_L^{-\sum_{i=1}^{n} \gamma(h,X,m,i)}.
\end{align*}
This equality together with Corollary \ref{equ01} imply our assertion.
\end{proof}
Now, to give an upper bound of $\left|\tilde{E}_{L,h,Z,m,\psi}\right|$ for additive characters $\psi$ of conductor $m\geq 2$, we need to take the sum over $(X\mod \varpi_L^{\eta_m}\cO_L)\in \cO_L^n/\varpi_L^{\eta_m}\cO_L^n$ of the integral 
$$\int_{\cup_{z\in \cB_{h,X,m}}\cB_{h,X,m}(z)}|dy|.$$
We will do this in the remaining of this subsection. 
\begin{defn}\label{tupe}Let $\ell$ and $m$ be positive integers. Let $\cS=(I_1,...,I_e)$ be an order tuple, where $e\geq 1$ and $I_1,...,I_e$ are mutually disjoint non-empty subsets of $\{1,...,\ell\}$.  Let $\mu=(\mu_1,...,\mu_e)\in\NN^{e}$, we say that the tuple $(\cS,\mu)$ is $(\ell,m)$-admissible if $0\leq \mu_1<...<\mu_e<\kappa_m$.  We denote by $\cD(\ell,m)$  the set consisting of all $(\ell,m)$-admissible tuples. We also consider the tuple $(\cS_\emptyset,\mu_\emptyset)$, where $\cS_\emptyset=(\emptyset)$, $\mu_\emptyset=0$.  For a $(\ell,m)$-admissible tuple $(\cS,\mu)$ with $\cS=(I_1,...,I_e)$ and $\mu=(\mu_1,...,\mu_e)$, we set $e(\cS):=e$, $\sigma(\cS):=\cup_{1\leq j\leq e}I_j$, $r_j(\cS):=\#(I_j)$ for $1\leq j\leq e$, $r(\cS):=\sum_{1\leq j\leq e}r_j(\cS)$ and $\Theta(\cS,\mu,m)=\sum_{j=1}^e r_j(\cS)\lfloor\frac{m+\mu_j}{2}\rfloor$.  We also set $\sigma(\cS_{\emptyset})=\emptyset$ and $r(S_\emptyset)=e(S_\emptyset)=\Theta(\cS_\emptyset,\mu_\emptyset,m)=0$. 
\end{defn}
\begin{defn}\label{tupset}Let $m\geq 2$ and  $(\cS,\mu)\in\cD(n,m)\cup\{(\cS_\emptyset,\mu_\emptyset)\}$,  we set 
\begin{align*}
\tilde{\cE}(h,\Phi_{L,Z},n,m,\cS,\mu):=\bigg\{&x\in \cO_L^{n}\mid  \ord_L\left(\frac{\partial h}{\partial x_i}(x)\right)\geq m-\eta_m \forall i\in \{1,...,n\},\Phi_{L,Z}(x)=1, \\&x \textnormal{ is of } (m,h)\textnormal{-type }\cS \textnormal{ with } 
(m,h)\textnormal{-multiplicity } \mu,  \ord_L(h(x))\geq m-1\bigg\}.
\end{align*}
We also set \begin{align*}\tilde{\cE}(h,\Phi_{L,Z},n,m):&=\bigg\{x\in \cO_L^{n}\mid  \ord_L\left(\frac{\partial h}{\partial x_i}(x)\right)\geq m-\eta_m \forall i,\Phi_{L,Z}(x)=1,\ord_L(h(x))\geq m-1\bigg\}\\
&=\sqcup_{(\cS,\mu)\in\cD(n,m)\cup\{(\cS_\emptyset,\mu_\emptyset)\}}\tilde{\cE}(h,\Phi_{L,Z},n,m,\cS,\mu).
\end{align*}
 \end{defn}

\begin{prop}\label{upboundL} We have
\begin{equation*}
\left|\tilde{E}_{L,h,Z,m,\psi}\right|\leq \sum_{(\cS,\mu)\in\cD(n,m)\cup\{(\cS_\emptyset,\mu_\emptyset)\}}\int_{\tilde{\cE}(f,\Phi_{L,Z},n,m,\cS,\mu)}q_L^{r(\cS)(m-\eta_m)-\Theta(\cS,\mu,m)}|dx|
\end{equation*}
for  all additive characters $\psi$ of $L$ of conductor $m\geq 2$.
\end{prop}
\begin{proof}Let $\psi$ be an additive character of $L$ of conductor $m\geq 2$. For each $\xi\in \cO_L^n/ \varpi_L^{\eta_m}\cO_L^n$, we fix an element $X_\xi\in\cO_L^n$ such that $(X_\xi\mod \varpi_L^{\eta_m}\cO_L)=\xi$. We set 
$$\cQ_{0,L,Z,m}=\{X_\xi\mid \xi\in\cO_L^n/ \varpi_L^{\eta_m}\cO_L^n, e(h,X_\xi,m)=0, \Phi_{L,Z}(X_\xi)=1\}$$
and 
$$\cQ_{1,L,Z,m}=\{X_\xi\mid \xi\in\cO_L^n/ \varpi_L^{\eta_m}\cO_L^n, e(h,X_\xi,m)>0, \Phi_{L,Z}(X_\xi)=1\}.$$
Moreover, if $(\cS,\mu)\in \cD(n,m)$, then we set 
$$\cQ_{(S,\mu),L,Z,m}=\{X_\xi\in \cQ_{1,L,Z,m}\mid X_\xi \textnormal{ is of } (m,h)\textnormal{-type }\cS \textnormal{ with } (m,h)\textnormal{-multiplicity } \mu \}.$$
By Definitions \ref{rho0}, \ref{Cset} and \ref{tupset}, we have 
\begin{equation}\label{0sets}\tilde{\cE}(h,\Phi_{L,Z},n,m,\cS_{\emptyset},\mu_{\emptyset})=\sqcup_{X_\xi\in \cQ_{0,L,Z,m}}\tilde{\cC}_{h,X_\xi,m,-\eta_m}
\end{equation}
and 
\begin{equation}\label{1sets}\tilde{\cE}(h,\Phi_{L,Z},n,m,\cS,\mu)=\sqcup_{X_\xi\in \cQ_{(S,\mu),L,Z,m}}\tilde{\cC}_{h,X_\xi,m,-\eta_m}
\end{equation}
for all $(S,\mu)\in\cD(n,m)$.

If $X_\xi\in \cQ_{(S,\mu),L,Z,m}$ for $(\cS,\mu)\in \cD(n,m)$, then it follows from Lemma-Definition \ref{rh-type1}, Definitions \ref{rh-type2}, \ref{tupe} and Lemma  \ref{count1} that $r(h,X_\xi,m)=r(\cS)$ and $\theta(h,X_\xi,m)=\Theta(\cS,\mu,m)$. Since $m\geq 2$, if $X_\xi\in \cQ_{0,L,Z,m}\cup\cQ_{1,L,Z,m}$ then $\Phi_{L,Z}(x)=1$ for all $x\in X_\xi+\varpi_L^{\eta_m}\cO_L^n$. Thus, it follows from (\ref{0sets}), (\ref{1sets}) together with Lemmas \ref{rho00}, \ref{eq11}, \ref{part1}, \ref{count1} and Definition \ref{tupe} that
\begin{align}\label{upbL}
&\left|\tilde{E}_{L,h,Z,m,\psi}\right|\nonumber\\=&\bigg|\int_{\{x\in\cO_L^n,\ord_L(h(x))\geq m-1\}}\Phi_{L,Z}(x)\psi(h(x))|dx|\bigg|\nonumber\\
=&\bigg|\sum_{\xi\in \cO_L^n/ \varpi_L^{\eta_m}\cO_L^n}\int_{\{x\in X_\xi+\varpi_L^{\eta_m}\cO_L^n\mid \ord_L(h(x))\geq m-1\}}\Phi_{L,Z}(x)\psi(h(x))|dx|\bigg|\nonumber\\
=&\bigg|\sum_{X_\xi\in\cQ_{0,L,Z,m}}\int_{\{x\in X_\xi+\varpi_L^{\eta_m}\cO_L^n\mid \ord_L(h(x))\geq m-1\}}\psi(h(x))|dx|\nonumber\\&+\sum_{X_\xi\in\cQ_{1,L,Z,m}}\int_{\{y\in X_\xi\mathsf{A}(X_\xi)^{-1}+\varpi_L^{\eta_m}\cO_L^n\mid \ord_L(g_{X_\xi}(x))\geq m-1\}}\psi(g_{X_\xi}(y))|dy|\bigg|\nonumber\\
=&\bigg|\sum_{X_\xi\in\cQ_{0,L,Z,m}}\int_{\tilde{\cC}_{h,X_\xi,m,-\eta_m}}\psi(h(x))|dx|+\sum_{X_\xi\in\cQ_{1,L,Z,m}}\int_{\cup_{z\in \cB_{h,X_\xi,m}}\cB_{h,X_\xi,m}(z)}\psi(g_{X_\xi}(y))|dy|\bigg|\nonumber\\
\leq& \sum_{X_\xi\in\cQ_{0,L,Z,m}}\int_{\tilde{\cC}_{h,X_\xi,m,-\eta_m}}|dx|+\sum_{X_\xi\in\cQ_{1,L,Z,m}}\int_{\cup_{z\in \cB_{h,X_\xi,m}}\cB_{h,X_\xi,m}(z)}|dy|\nonumber\\
=&\int_{\sqcup_{X_\xi\in\cQ_{0,L,Z,m}}\tilde{\cC}_{h,X_\xi,m,-\eta_m}}|dx|+\sum_{X_\xi\in\cQ_{1,L,Z,m}}q_L^{(m-\eta_m) r(h,X_\xi,m)-\theta(h,X_\xi,m)}\int_{\tilde{\cC}_{h,X_\xi,m,-\eta_m}}|dx|\nonumber\\
=&\sum_{(\cS,\mu)\in\cD(n,m)\cup\{(\cS_\emptyset,\mu_\emptyset)\}}\int_{\tilde{\cE}(f,\Phi_{L,Z},n,m,\cS,\mu)}q_L^{r(\cS)(m-\eta_m)-\Theta(\cS,\mu,m)}|dx|.
\end{align}
\end{proof}
\begin{lem}\label{diff}We have
$$-n+\frac{r(\cS)\kappa_m}{2}-\sum_{j=1}^{e(\cS)} r_j(\cS)\frac{\mu_j}{2}\leq r(\cS)(m-\eta_m)-\Theta(\cS,\mu,m)\leq n+\frac{r(\cS)\kappa_m}{2}-\sum_{j=1}^{e(\cS)} r_j(\cS)\frac{\mu_j}{2}$$
for all $m\geq 2$ and all $(\cS,\mu)\in\cD(n,m)\cup\{(\cS_\emptyset,\mu_\emptyset)\}$.
\end{lem}
\begin{proof}Let $(\cS,\mu)\in\cD(n,m)$.
By Lemma-Definition \ref{kaet}, we have
\begin{equation*}
\left|\frac{m-\kappa_m}{2}-\eta_m\right|\leq \frac{1}{2}
\end{equation*}
for all $m\in\ZZ$. Moreover,
\begin{equation*}
\left|\frac{a}{2}-\lfloor\frac{a}{2}\rfloor\right|\leq \frac{1}{2}
\end{equation*}
for all $a\in\ZZ$.  Thus, 
$$\left|m-\eta_m-\lfloor\frac{m+\mu_j}{2}\rfloor-\frac{\kappa_m}{2}+\frac{\mu_j}{2}\right|\leq 1$$
for all $1\leq j\leq e(\cS)$. These inequalities  and Definition \ref{tupe} imply our assertion. 
\end{proof}

\begin{lem}\label{maxmulp} If $(S,\mu)$ is a $(n,m)$-admissible tuple and $X\in \tilde{\cE}(f,\Phi_{L,Z},n,m,\cS,\mu)$ then 
$$0<\frac{r(\cS)\kappa_m}{2}-\sum_{j=1}^{e(\cS)} r_j(\cS)\frac{\mu_j}{2}=\max_{1\leq \ell\leq n}\left(\frac{\ell\kappa_m}{2}-\frac{1}{2}\ord_L\left(\mathsf{D}_{h,\ell}(X)\right)\right)$$
for all $m\geq 2$,  where $\mathsf{D}_{h,\ell}$ is the ideal of $\cO_L[x_1,...,x_{n}]$ generated by all principal minors of order $\ell$ of the matrix $\cH_{h}(x)=\left(\frac{\partial^2 h}{\partial x_i\partial x_j}(x)\right)_{1\leq i,j\leq n}$. 
\end{lem}
\begin{proof}Let $1\leq \ell\leq n$ and $X\in \tilde{\cE}(f,\Phi_{L,Z},n,m,\cS,\mu)$. With the notations of Definition \ref{typeh}, we take $i(\ell)$ to be the smallest positive integer such that $\ell\leq \sum_{1\leq j\leq i(\ell)}r_j(h,X)$. We set $r(\ell)=\sum_{1\leq j\leq i(\ell)-1}r_j(h,X)$.  It follows from Lemma \ref{detr} and Definition \ref{typeh} that 
\begin{align*}\frac{\ell\kappa_m}{2}-\frac{1}{2}\ord_L\left(\mathsf{D}_{h,\ell}(X)\right)&=\frac{\ell\kappa_m}{2}-\frac{1}{2}(\ell-r(\ell))\mu_{i(\ell)}(h,X)-\frac{1}{2}\sum_{1\leq j\leq i(\ell)-1}r_j(h,X)\mu_j(h,X)\\
&=(\ell-r(\ell))\left(\frac{\kappa_m}{2}-\frac{\mu_{i(\ell)}(h,X)}{2}\right)+\sum_{1\leq j\leq i(\ell)-1}r_j(h,X)\left(\frac{\kappa_m}{2}-\frac{\mu_{j}(h,X)}{2}\right). 
\end{align*}
Thus, by Lemma \ref{detr},  Lemma-Definition \ref{rh-type1}, Definitions \ref{rh-type2}, \ref{tupe} and \ref{tupset}, we have 
\begin{align*}\max_{1\leq \ell\leq n}\left(\frac{\ell\kappa_m}{2}-\frac{1}{2}\ord_L\left(\mathsf{D}_{h,\ell}(X)\right)\right)&=\sum_{1\leq j\leq e(h,X,m)}r_j(h,X)\left(\frac{\kappa_m}{2}-\frac{\mu_{j}(h,X)}{2}\right)\\
&=\frac{r(\cS)\kappa_m}{2}-\sum_{j=1}^{e(\cS)} r_j(\cS)\frac{\mu_j}{2}>0.
\end{align*}
\end{proof}
\begin{lem}\label{rjet}There are constants $b(n), c(n), \vartheta(n),\varsigma(n)$ depending only on $n$ such that $b(n), c(n)$ are positive and 
\begin{align*}&\sum_{(\cS,\mu)\in\cD(n,m)\cup\{(\cS_\emptyset,\mu_\emptyset)\}}\int_{\tilde{\cE}(f,\Phi_{L,Z},n,m,\cS,\mu)}q_L^{r(\cS)(m-\eta_m)-\Theta(\cS,\mu,m)}|dx|\\ \leq& c(n)q_L^{\vartheta(n)}\max_{0\leq \ell\leq n}\int_{\tilde{\cE}(h,\Phi_{L,Z},n,m)}q_L^{\frac{\ell\kappa_m}{2}-\frac{1}{2}\ord_L\left(\mathsf{D}_{h,\ell}(x)\right)}|dx|
\end{align*}
and 
\begin{align*}&\sum_{(\cS,\mu)\in\cD(n,m)\cup\{(\cS_\emptyset,\mu_\emptyset)\}}\int_{\tilde{\cE}(f,\Phi_{L,Z},n,m,\cS,\mu)}q_L^{r(\cS)(m-\eta_m)-\Theta(\cS,\mu,m)}|dx|\\ \geq& b(n)q_L^{\varsigma(n)}\max_{0\leq \ell\leq n}\int_{\tilde{\cE}(h,\Phi_{L,Z},n,m)}q_L^{\frac{\ell\kappa_m}{2}-\frac{1}{2}\ord_L\left(\mathsf{D}_{h,\ell}(x)\right)}|dx|
\end{align*}
for all $m\geq 2$, where we use the convention that $\mathsf{D}_{h,0}=(1)$.
\end{lem}
\begin{proof}
We have 
$$q_L^{\max_{0\leq \ell\leq n}\left(\frac{\ell\kappa_m}{2}-\frac{1}{2}\ord_L\left(\mathsf{D}_{h,\ell}(x)\right)\right)}\leq \sum_{0\leq \ell\leq n}q_L^{\frac{\ell\kappa_m}{2}-\frac{1}{2}\ord_L\left(\mathsf{D}_{h,\ell}(X)\right)}\leq (n+1)q_L^{\max_{0\leq \ell\leq n}\left(\frac{\ell\kappa_m}{2}-\frac{1}{2}\ord_L\left(\mathsf{D}_{h,\ell}(x)\right)\right)}$$
for all $x\in \tilde{\cE}(h,\Phi_{L,Z},n,m)=\sqcup_{(\cS,\mu)\in\cD(n,m)\cup\{(\cS_\emptyset,\mu_\emptyset)\}} \tilde{\cE}(h,\Phi_{L,Z},n,m,,\cS,\mu)$.
This and Lemmas \ref{diff}, \ref{maxmulp} imply the assertion.
\end{proof}
The following corollary helps to compare our result with the one in \cite[Theorem 6.1]{CKMu}.
\begin{cor}\label{compCMK}
There are  constants $a(n), \varrho(n)$ depending only on $n$ such that  $a(n)>0$ and
\begin{align*}&\max_{0\leq \ell\leq n}\int_{\tilde{\cE}(h,Z,n,m,L)}q_L^{\frac{\ell\kappa_m}{2}-\frac{1}{2}\ord_L\left(\mathsf{D}_{h,\ell}(x)\right)}|dx|\\ \leq& a(n)q_L^{\varrho(n)}\int_{\{x\in \cO_L^n\mid \ord_L(h(x))\geq m-1,\hspace{0.1cm} 2\ord_L\left(\frac{\partial h}{\partial x_i}(x)\right)\geq m-1 \forall i\}}\Phi_{L,Z}(x)|dx|
\end{align*}
for all $m\geq 2$. 
\end{cor} 
\begin{proof}Let $X\in \cO_L^n$ and $m\geq 2$. We set $Y=X\mathsf{A}(X)^{-1}$ as usual.

If $e(h,X,m)=0$, then we can use Lemma-Definition \ref{kaet} and Definition \ref{rho0} to have 
\begin{equation}\label{bimaj}\int_{\tilde{\cC}_{h,X,m,-\eta_m}}|dx|\leq \int_{\{x\in X+\varpi_L^{\eta_m}\cO_L^n\mid \ord_L(h(x))\geq m-1,\hspace{0.1cm} 2\ord_L\left(\frac{\partial h}{\partial x_i}(x)\right)\geq m-1 \forall i\}}|dx|.
\end{equation}

If $e(h,X,m)>0$, we can use Definitions \ref{j-func} and \ref{Aset} to have
\begin{equation*}\{z\in Y+\varpi_L^{\eta_m}\cO_L^n\mid \ord_L(g_X(z))\geq m-1,\hspace{0.1cm} \min_{1\leq \ell\leq n}2\ord_L\left(\frac{\partial g_X}{\partial y_\ell}(z)\right)<m-1\}\subset \cA_{h,X,m}.
\end{equation*}
This and Lemmas \ref{propt}, \ref{part1}, \ref{count1} imply
\begin{align}\label{bimaj1}
&\int_{\{x\in X+\varpi_L^{\eta_m}\cO_L^n\mid \ord_L(h(x))\geq m-1,\hspace{0.1cm}2\ord_L\left(\frac{\partial h}{\partial x_i}(x)\right)\geq m-1 \forall i\}}|dx|\nonumber\\
=&\int_{\{z\in Y+\varpi_L^{\eta_m}\cO_L^n\mid \ord_L(g_X(z))\geq m-1,\hspace{0.1cm}2\ord_L\left(\frac{\partial g_X}{\partial y_\ell}(z)\right)\geq m-1 \forall i\}}|dx|\nonumber\\
\geq& \int_{\cup_{z\in \cB_{h,X,m}}\cB_{h,X,m}(z)}|dy|= \int_{\tilde{\cC}_{h,X,m,-\eta_m}}q_L^{(m-\eta_m) r(h,X,m)-\theta(h,X,m)}|dx|.
\end{align}
By combining (\ref{upbL}), (\ref{bimaj}), (\ref{bimaj1}), we have 
\begin{align*}&\sum_{(\cS,\mu)\in\cD(n,m)\cup\{(\cS_\emptyset,\mu_\emptyset)\}}\int_{\tilde{\cE}(f,\Phi_{L,Z},n,m,\cS,\mu)}q_L^{-\Theta(\cS,\mu,m)+r(\cS)(m-\eta_m)}|dx|\\ \leq &\int_{\{x\in \cO_L^n\mid \ord_L(h(x))\geq m-1,\hspace{0.1cm}2 \ord_L\left(\frac{\partial h}{\partial x_i}(x)\right)\geq m-1 \forall i\}}\Phi_{L,Z}(x)|dx|
\end{align*}
for all $m\geq 2$. This and Lemma \ref{rjet} imply our assertion.
\end{proof}
\begin{prop}\label{upboundL1} There are constants $C_n, \iota_n$ depending only on $n$ such that $C_n>0$ and 
\begin{equation*}
\left|\tilde{E}_{L,h,Z,m,\psi}\right|\leq C_nq_L^{\iota_n}\max_{0\leq \ell\leq n}\int_{\tilde{\cE}(h,\Phi_{L,Z},n,m)}q_L^{\frac{\ell\kappa_m}{2}-\frac{1}{2}\ord_L\left(\mathsf{D}_{h,\ell}(x)\right)}|dx|
\end{equation*}
for  all additive characters $\psi$ of $L$ of conductor $m\geq 2$.
\end{prop}
\begin{proof}It follows immediately from Lemmas \ref{upboundL} and \ref{rjet}.
\end{proof}

More generally, we can give an upper bound of $\left|\tilde{E}_{L,h,Z,m,\psi}\right|$ related to any $1\leq n_0\leq n$ as follows. We fix $1\leq n_0\leq n$ and choose a subset $I$ of $\{1,...,n\}$ such that $\#I=n_0$. Without loss of generality, we may suppose that $I=\{1,...,n_0\}$. For each 
$z\in \cO_L^{n-n_0}$, let $h_z\in\cO_L[x_1,...,x_{n_0}]$ be the polynomial given by $h_z(x_1,...,x_{n_0})=h(x_1,...,x_{n_0},z)$. Let 
$(\cS,\mu)\in\cD(n_0,m)\cup\{(\cS_\emptyset,\mu_\emptyset)\}$,  we set 
\begin{align*}
\tilde{\cE}(h,\Phi_{L,Z},n_0,m,\cS,\mu):=\bigg\{&x=(x',z)\in \cO_L^{n_0}\times\cO_L^{n-n_0}\mid  \ord_L\left(\frac{\partial h}{\partial x_i}(x)\right)\geq m-\eta_m \hspace{0.1cm}\forall i\leq n_0, \Phi_{L,Z}(x)=1, \\&x' \textnormal{ is of } (m,h_z)\textnormal{-type }\cS \textnormal{ with } 
(m,h_z)\textnormal{-multiplicity } \mu, \ord_L(h(x))\geq m-1\bigg\}
\end{align*}
and 
\begin{align*}\tilde{\cE}(h,\Phi_{L,Z},n_0,m):=\bigg\{&x=(x',z)\in \cO_L^{n_0}\times\cO_L^{n-n_0}\mid  \ord_L\left(\frac{\partial h}{\partial x_i}(x)\right)\geq m-\eta_m \hspace{0.1cm}\forall i\leq n_0,\\ &\Phi_{L,Z}(x)=1,\ord_L(h(x))\geq m-1\bigg\}.
\end{align*}
By repeating the above argument with $h_z$ instead of $h$ for each $z\in\cO_L^{n-n_0}$, we have the following proposition.
\begin{prop}\label{rupbouL}
There are  constants $\mathsf{C}_{n_0}, \varepsilon_{n_0}$ depending only on $n_0$ such that 
\begin{equation*}
\left|\tilde{E}_{L,h,Z,m,\psi}\right|\leq \mathsf{C}_{n_0}q_L^{\varepsilon_{n_0}}\max_{0\leq \ell\leq n_0}\int_{\tilde{\cE}(h,\Phi_{L,Z},n_0,m)}q_L^{\frac{\ell\kappa_m}{2}-\frac{1}{2}\ord_L\left(\mathsf{D}_{h,n_0,\ell}(x)\right)}|dx|
\end{equation*}
for  all additive characters $\psi$ of $L$ of conductor $m\geq 2$, where $\mathsf{D}_{h,n_0,\ell}$ is the ideal of $\cO_L[x_1,...,x_{n}]$ generated by all principal minors of order $\ell$ of the matrix $\cH_{h,n_0}(x)=(\frac{\partial^2 h}{\partial x_i\partial x_j}(x))_{1\leq i,j\leq n_0}$. Here, we use the convention that $\mathsf{D}_{h,n_0,\ell}=(1)$.
\end{prop}
\begin{proof}Let $m\geq 2$ and $\psi$ be an additive character of $L$ of conductor $m$. For each $z\in\cO_L^{n-n_0}$, we consider the Schwartz-Bruhat function $\Phi_{L,Z,z}:L^{n_0}\to \CC, x'\mapsto\Phi_{L,Z}(x',z)$. 

By using fiber integration, we have
\begin{align*}
\tilde{E}_{L,h,Z,m,\psi}&=\int_{\{x\in \cO_L^{n}\mid \ord_L(h(x))\geq m-1\}}\Phi_{L,Z}(x)\psi(h(x))|dx|\\
&=\int_{\cO_L^{n-n_0}}\int_{\{(x',z)\in \cO_L^{n}\mid \ord_L(h(x',z))\geq m-1\}}\Phi_{L,Z}(x',z)\psi(h(x',z))|dx'\wedge dz|.
\end{align*}

This and Proposition \ref{upboundL1} imply 
\begin{align*}
\left|\tilde{E}_{L,h,Z,m,\psi}\right|&\leq \int_{\cO_L^{n-n_0}}\left|\int_{\{(x',z)\in \cO_L^{n}\mid \ord_L(h(x',z))\geq m-1\}}\Phi_{L,Z}(x',z)\psi(h(x',z))|dx'|\right| |dz|\\
&\leq \int_{\cO_L^{n-n_0}}C_{n_0}q_L^{\iota_{n_0}}\max_{0\leq \ell\leq n_0}\left(\int_{\tilde{\cE}(h_z,\Phi_{L,Z,z},n_0,m)}q_L^{\frac{\ell\kappa_m}{2}-\frac{1}{2}\ord_L\left(\mathsf{D}_{h_z,\ell}(x')\right)}|dx'|\right)|dz|\\
&=C_{n_0}q_L^{\iota_{n_0}}\int_{\cO_L^{n-n_0}}\max_{0\leq \ell\leq n_0}\left(\int_{\tilde{\cE}(h_z,\Phi_{L,Z,z},n_0,m)}q_L^{\frac{\ell\kappa_m}{2}-\frac{1}{2}\ord_L\left(\mathsf{D}_{h,n_0,\ell}(x',z)\right)}|dx'|\right)|dz|\\
&\leq C_{n_0}q_L^{\iota_{n_0}}\int_{\cO_L^{n-n_0}}\sum_{0\leq \ell\leq n_0}\left(\int_{\tilde{\cE}(h_z,\Phi_{L,Z,z},n_0,m)}q_L^{\frac{\ell\kappa_m}{2}-\frac{1}{2}\ord_L\left(\mathsf{D}_{h,n_0,\ell}(x',z)\right)}|dx'|\right)|dz|\\
&\leq n_0C_{n_0}q_L^{\iota_{n_0}}\max_{0\leq\ell\leq n_0}\int_{\cO_L^{n-n_0}}\left(\int_{\tilde{\cE}(h_z,\Phi_{L,Z,z},n_0,m)}q_L^{\frac{\ell\kappa_m}{2}-\frac{1}{2}\ord_L\left(\mathsf{D}_{h,n_0,\ell}(x',z)\right)}|dx'|\right)|dz|\\
&=n_0C_{n_0}q_L^{\iota_{n_0}}\max_{0\leq \ell\leq n_0}\int_{\tilde{\cE}(h,\Phi_{L,Z},n_0,m)}q_L^{\frac{\ell\kappa_m}{2}-\frac{1}{2}\ord_L\left(\mathsf{D}_{h,n_0,\ell}(x)\right)}|dx|.
\end{align*}
Thus, it is sufficient to take $\mathsf{C}_{n_0}=n_0C_{n_0}$ and $\varepsilon_{n_0}=\iota_{n_0}$.
\end{proof}

\subsection{Motivic oscillation index of polynomials}\label{moipolylo}
In this subsection, we fix a non-constant polynomial $f$ in $\cO_K[x_1,...,x_n]$ of degree $d$. Let $J_f=(\frac{\partial f}{\partial x_1},...,\frac{\partial f}{\partial x_n})\subset K[x_1,...,x_n]$ be the Jacobian ideal of $f$. As in Subsection \ref{exp-zetafunction}, we denote by $V_f\subset\overline{K}$ the set of critical values of $f$ over $\overline{K}$. Suppose that $d\geq 2$ and $f=f_0+f_1+...+f_d$, where $f_i$ is the homogeneous part of degree $i$ of $f$. In Definition  \ref{singudim}, we defined $s(f)$ to be the dimension of the closed subscheme of $\AA_K^n$ associated to $J_f$.  

Let $I$ be a non-empty subset of $\{1,...,n\}$. For each $a\in V_f$, we denote by $\cI_{f-a,I}$ the ideal of $K[a,x_1,...,x_n]$ generated by $(f-a)^2$ and $(\frac{\partial f}{\partial x_i})^3$ for $i\in I$.  We define $$\cH_{f,I}(x)=\left(\dfrac{\partial^2 f}{\partial x_i\partial x_j}(x)\right)_{i,j\in I}.$$
For each $0\leq \ell\leq\#I$, we denote by $\mathsf{D}_{f,I,\ell}$ the ideal of $\cO_K[x_1,...,x_n]$ generated by all principal minors of order $\ell$ of $\cH_{f,I}(x)$ with the convention that $\mathsf{D}_{f,I,\ell}=(1)$. We set  $$\tilde{\mathsf{D}}_{f,I,\ell}=\mathsf{D}_{f,I,\ell}\otimes_{\cO_K}K, \mathsf{D}_{f,I,\ell,\cO_L}=\mathsf{D}_{f,I,\ell}\otimes_{\cO_K}\cO_L,$$ $$\sigma_{f-a,Z,I}:=\min_{0\leq \ell\leq\#I}\left(2\lct_{Z_{K[a]}}\left(\cI_{f-a,I},\tilde{\mathsf{D}}_{f,I,\ell},\frac{1}{2}\right)-\frac{\ell}{6}\right),  \sigma_{f,Z}=\max_{\emptyset\neq I\subset \{1,...,n\}}\min_{a\in V_f}\sigma_{f-a,Z,I}$$
and $$\tilde{\sigma}_{f,Z}=\min\biggl\{\frac{d+1}{d+2}\sigma_{f,Z}+\frac{1}{2(d+2)},\frac{n-s(f)}{d}\biggr\}.$$

We recall the following result. 
\begin{lem}[\cite{Saskia-Kien}, Lemma 6.2.1]\label{sak}
There exists an integer $M$ depending only on $f$ such that for all $\cO_K$-subschemes $W$  of $\AA_{\cO_K}^n$ of finite type, all local fields $L\in\tilde{\cL}_{K,M}$ and all additive characters $\psi$ of $L$ of conductor $m\geq 2$ we have 
$$\int_{\{x\in \cO_L^n\mid \ord_L(f(x))<m-1\}}\Phi_{L,W}(x)\psi(f(x))|dx|=0.$$
\end{lem}
The following theorem generalizes and strengthens \ref{lobsig}. 
\begin{thm}\label{thmlowb}There is an integer $M$ depending only on $f,Z$ such that $\sigma_{L}(f,Z)\geq \sigma_{f,Z}$ for all $L\in\tilde{\cL}_{K,M}$. Consequently, there is an integer $M$ depending only on $f,Z$ and a positive constant $C_{n,d}$ depending only on $n,d$ such that
$$|E_{f,L,Z,\psi}|\leq C_{n,d}m_\psi^{n-1}q_L^{-\sigma_{f,Z}(m_\psi-1)-\frac{1}{2}}$$
for all $L\in\tilde{\cL}_{K,M}$ and all additive characters $\psi$ of $L$ of conductor at least $2$. \end{thm}
\begin{proof}By (\ref{decomsum}), (\ref{mimoi}), (\ref{moif}) and the definition of $\sigma_{f,Z}$, it suffices to consider the case that $f(Z(\overline{K}))=0$.  Let $\emptyset\neq I\subset \{1,...,n\}$. 
By Lemma \ref{sak} and Proposition \ref{rupbouL}, there is an integer $M_1\geq d+1$ depending only on $f$ and constants $C_I,\varepsilon_I$ depending only on $\#I$ such that  $C_I>0$ and
\begin{align*}&\left|E_{f,L,Z,\psi}\right|\\ \leq& C_Iq_L^{\varepsilon_I}\max_{0\leq \ell\leq \#I}\int_{\{x\in \cO_L^n\mid  \ord_L\left(\frac{\partial f}{\partial x_i}(x)\right)\geq m-\eta_m \hspace{0.1cm}\forall i\in I,\hspace{0.1cm} \ord_L(f(x))\geq m-1\}}\Phi_{L,Z}(x)q_L^{\frac{\ell\kappa_m}{2}-\frac{1}{2}\ord_L\left(\mathsf{D}_{f,I,\ell,\cO_L}(x)\right)}|dx|\\
\leq& C_Iq_L^{\varepsilon_I}\max_{0\leq \ell\leq \#I}\int_{\{x\in \cO_L^n\mid  \ord_L\left(\left(\frac{\partial f}{\partial x_i}\right)^3(x)\right)\geq 2(m-1) \hspace{0.1cm}\forall i\in I,\hspace{0.1cm} \ord_L\left(f^2(x)\right)\geq 2(m-1)\}}\Phi_{L,Z}(x)q_L^{\frac{\ell\kappa_m}{2}-\frac{1}{2}\ord_L\left(\mathsf{D}_{f,I,\ell,\cO_L}(x)\right)}|dx|
\end{align*}
for all $L\in \tilde{\cL}_{K,M_1}$ and all non-trivial  additive characters $\psi$ of $L$ of conductor $m_\psi=m\geq 2$. Here, we used the fact that $3(m-\eta_m)=3(m-\lceil\frac{m}{3})\rceil\geq 2(m-1)$ for all $m\in\ZZ$. 
Thus, we can use Proposition \ref{genlogc} and the inequality $$\left|\kappa_m-\frac{m}{3}\right|\leq \frac{2}{3}$$ to have an integer $M\geq M_1$ depending only on $f,Z$ and a constant $c_{d,n}$ depending only on $n,d$ such that 
\begin{align*}
&\left|E_{f,L,Z,\psi}\right|\\ \leq& C_Iq_L^{\varepsilon_{I}}\max_{0\leq \ell\leq \#I}\int_{\{x\in \cO_L^n\mid  \ord_L\left(\left(\frac{\partial f}{\partial x_i}\right)^3(x)\right)\geq 2(m-1) \hspace{0.1cm}\forall i\in I,\hspace{0.1cm} \ord_L\left(f^2(x)\right)\geq 2(m-1)\}}\Phi_{L,Z}(x)q_L^{\frac{\ell\kappa_m}{2}-\frac{1}{2}\ord_L\left(\mathsf{D}_{f,I,\ell,\cO_L}(x)\right)}|dx|\\
\leq&C_Iq_L^{\varepsilon_I}c_{d,n}\max_{0\leq \ell\leq \#I}q_L^{\frac{\ell\kappa_m}{2}}q_L^{-(2m-2)\lct_{Z_K}\left(\cI_{f,I},\tilde{\mathsf{D}}_{f,I,\ell},\frac{1}{2}\right)}\\
\leq&C_Iq_L^{\varepsilon_I}c_{d,n}q_L^{2\max_{0\leq \ell\leq\#I}\lct_{Z_K}\left(\cI_{f,I},\tilde{\mathsf{D}}_{f,I,\ell},\frac{1}{2}\right)}q_L^{\frac{\#I}{3}}q_L^{-m\min_{0\leq\ell\leq\#I}\left(2\lct_{Z_K}\left(\cI_{f,I},\tilde{\mathsf{D}}_{f,I,\ell},\frac{1}{2}\right)-\frac{\ell}{6}\right)}
\end{align*}
for all $L\in \tilde{\cL}_{K,M}$ and all non-trivial  additive characters $\psi$ of $L$ of conductor $m_\psi=m\geq 2$.
Therefore, we can use (\ref{lowbL}) to deduce that 
$$\sigma_{L}(f,Z)\geq \sigma_{f,Z,I}$$
for all $\emptyset\neq I\subset\{1,...,n\}$ and all $L\in\tilde{\cL}_{K,M}$. It implies that 
$$\sigma_{L}(f,Z)\geq \sigma_{f,Z}$$ for
all $L\in\tilde{\cL}_{K,M}$. 

The remaining assertion follows from the first assertion, Proposition \ref{transfer1} and Theorem \ref{main}.
\end{proof}

\begin{cor}There is an integer $M$ depending only on $f,Z$ and a positive constant $C_{n,d}$ depending only on $n,d$ such that
$$|E_{f,L,Z,\psi}|\leq C_{n,d}m_\psi^{n-1}q_L^{-\tilde{\sigma}_{f,Z}m_\psi}$$
for all $L\in\tilde{\cL}_{K,M}$ and all additive characters $\psi$ of $L$ of conductor at least $1$.
\end{cor}
\begin{proof}By  Theorem \ref{thmlowb} and (\ref{exmo}), we have $\moi_K(f,Z)\geq \sigma_{f,Z}$. Thus, we can use Corollary \ref{bounnsd} to deduce the assertion.
\end{proof}

In fact, the lower bound of $\sigma_L(f,Z)$ given by Theorem \ref{thmlowb} is better than the one in \cite[Theorem 6.1]{CKMu}. We can see it in the following proposition.

\begin{prop}\label{compaCKM} Suppose that $f(Z(\overline{K}))=0$.  Then, we have $\sigma_{f,Z,I_0}\geq \lct_{Z_K}\left((f)+ J_f^2\right)$,
where $I_0=\{1,...,n\}$. 

\end{prop}
\begin{proof}
By using Proposition \ref{genlogc} and Corollary \ref{compCMK} together with the fact that $3(m-\eta_m)\leq 2m$ for all $m\in\ZZ$, there is:
\begin{itemize}
\item an integer $M$ depending only on $f$,  
\item constants $a(n)>0, \varrho(n)$ depending only on $n$,  
\item constants $\Delta_{n,d}>0,\epsilon_{n,d}$ depending only on $n,d$,
\item a finite extension $L_p\in\cL_{K,M}$ of $\QQ_p$ for each $p>M$,
\end{itemize} such that for each $p>M$, the inequalities 
\begin{align*}&a(n)q_{L_p}^{\varrho(n)}\Delta_{n,d}q_{L_p}^{-(m-1)\lct_{Z_K}\left((f)+J_f^2\right)}\\
\geq &a(n)q_{L_p}^{\varrho(n)}\int_{\{x\in \cO_{L_p}^n\mid \ord_{L_p}(f(x))\geq m-1,\hspace{0.1cm} 2\ord_{L_p}(\frac{\partial f}{\partial x_i}(x))\geq m-1 \forall i\}}\Phi_{L_p,Z}(x)|dx|\\
\geq& \max_{0\leq \ell\leq n}\int_{\{x\in \cO_{L_p}^n\mid  \ord_{L_p}\left(\frac{\partial f}{\partial x_i}(x)\right)\geq m-\eta_m \hspace{0.1cm}\forall i,\hspace{0.1cm} \ord_{L_p}(f(x))\geq m-1\}}\Phi_{L_p,Z}(x)q_{L_p}^{\frac{\ell\kappa_m}{2}-\frac{1}{2}\ord_{L_p}\left(\mathsf{D}_{f,I_0,\ell,\cO_{L_p}}(x)\right)}|dx|\\
\geq &\max_{0\leq \ell\leq n}\int_{\{x\in \cO_{L_p}^n\mid  \ord_{L_p}\left(\left(\frac{\partial f}{\partial x_i}\right)^3(x)\right)\geq 2m \hspace{0.1cm}\forall i,\hspace{0.1cm} \ord_{L_p}(f^2(x))\geq 2m\}}\Phi_{L_p,Z}(x)q_{L_p}^{\frac{\ell\kappa_m}{2}-\frac{1}{2}\ord_{L_p}\left(\mathsf{D}_{f,I_0,\ell,\cO_{L_p}}(x)\right)}|dx|\\
\geq &\max_{0\leq \ell\leq n}q_{L_p}^{\epsilon_{n,d}}q_{L_p}^{\frac{\ell\kappa_m}{2}-2m\lct_{Z_K}\left(\cI_{f,I_0},\tilde{\mathsf{D}}_{f,I_0,\ell},\frac{1}{2}\right)}
\end{align*}
hold for infinitely many $m\geq 2$. Therefore, we must have 
$$\lct_{Z_K}\left((f)+J_f^2\right)\leq \min_{0\leq \ell\leq n}\left(2\lct_{Z_K}\left(\cI_{f,I_0},\tilde{\mathsf{D}}_{f,I_0,\ell},\frac{1}{2}\right)-\frac{\ell}{6}\right)=\sigma_{f,Z,I_0}$$
as desired.
\end{proof}
\subsection{Motivic oscillation index of ideals}\label{moiidelo}
In this subsection, we fix an ideal $\cI$ of $\cO_K[x_1,...,x_n]$ generated by non-constant polynomials $f_1,...,f_r$ of degree at most $d$ such that $\cI_K\neq (1)$.  We set $g(x,y)=\sum_{1\leq i\leq r}y_if_i(x)$. For each scheme $Z\subset\AA_{\cO_K}^n$ of finite type over $\cO_K$, recall that we defined $$Z^{(r)}=Z\times_{\Spec(\cO_K)}(\AA_{\cO_K}^r\setminus\Spec(\cO_K[y_1,...,y_r]/(y_1,...,y_r)).$$
Now, we fix such a scheme $Z$. By Proposition \ref{ridepol}, we have 
$$E_{L,Z,\cI}^{(r)}(m)=E_{g,L,Z,\psi}$$
for all $L\in\tilde{\cL}_{K,1}$, all  additive characters $\psi$ of $L$ of conductor $m\geq 1$. 
Thus, we can repeat the argument of Section \ref{moipolylo} for $g$ and $Z^{(r)}$ to give a lower bounds of $\moi_K(g,Z^{(r)})=\moi_{K,Z}^{(r)}(\cI)$
with respect to each non-empty subset $I$ of $\{1,...,n+r\}$. Among of such sets $I$, there are two sets that may give more  intrinsic lower bounds of $\moi_{K,Z}^{(r)}(\cI)$ than the others. Those are $I_0=\{1,...,n+r\}$ and $I_1=\{1,...,n\}$. To be more precisely, we define the ideals 
\begin{itemize}
\item $\cJ_0:= \left(g^2,\left(\frac{\partial g}{\partial x_1}\right)^3,...,\left(\frac{\partial g}{\partial x_n}\right)^3,f_1^3,...,f_r^3\right)\subset \cO_K[x_1,...,x_n,y_1,...,y_r]$,
\item $\cJ_1:=\left(g^2,\left(\frac{\partial g}{\partial x_1}\right)^3,...,\left(\frac{\partial g}{\partial x_n}\right)^3\right)\subset \cO_K[x_1,...,x_n,y_1,...,y_r]$.
\end{itemize}
We denote by $\cH_{g,0}$  the Hessian matrix of $g$ and  define
$$\cH_{g,1}(x,y)=\left(\sum_{1\leq i\leq r}y_i\frac{\partial^2 f_i}{\partial x_j\partial x_l}(x)\right)_{1\leq j,l\leq n}.$$
For each $0\leq \ell\leq n+r$, we denote by $\tilde{\mathsf{D}}_{g,0,\ell}$ the ideal of $K[x_1,...,x_n,y_1,...,y_r]$ generated by all principal minors of order $\ell$ of $\cH_{g,0}$. For each $0\leq \ell\leq n$, we denote by $\tilde{\mathsf{D}}_{g,1,\ell}$ the ideal of $K[x_1,...,x_n,y_1,...,y_r]$ generated by all principal minors of order $\ell$ of $\cH_{g,1}$. We still use the convention that $\tilde{\mathsf{D}}_{g,0,0}=\tilde{\mathsf{D}}_{g,1,0}=(1)$.

We set $$\sigma_{\cI,Z,0}:=\min_{0\leq \ell\leq n+r}\left(2\lct_{Z^{(r)}_K}\left(\cJ_0,\tilde{\mathsf{D}}_{g,0,\ell},\frac{1}{2}\right)-\frac{\ell}{6}\right), \sigma_{\cI,Z,1}:=\min_{0\leq \ell\leq n}\left(2\lct_{Z^{(r)}_K}\left(\cJ_1,\tilde{\mathsf{D}}_{g,1,\ell},\frac{1}{2}\right)-\frac{\ell}{6}\right)$$
and
$$\sigma_{\cI,Z}=\max\{\sigma_{\cI,Z,0},\sigma_{\cI,Z,1}\}.$$

Moreover, if $f_1,...,f_r$ are of degree $d\geq 1$, then  we set
 $$\tilde{\sigma}_{\cI,Z}=\min\biggl\{\frac{d+1}{d+2}\sigma_{\cI,Z}+\frac{r}{d+2},\frac{n-s(f_1,...,f_r)}{d}\biggr\},$$
where $s(f_1,...,f_r)$ is defined in Definition \ref{singudim}.

We end this section by the following theorem.
\begin{thm}\label{thmlowbide}With the above notations and assumptions, we have $\moi_{K,Z}^{(r)}(\cI)\geq \sigma_{\cI,Z}$. Consequently, there is an integer $M$ depending only on $\cI,Z$ and a positive constant $C_{n,d}$ depending only on $n,d$ such that
$$|E_{L,Z,\cI}^{(r)}(m)|\leq C_{n,d}m^{n-1}q_L^{-\sigma_{\cI,Z}(m-1)-r}$$
for all $L\in\tilde{\cL}_{K,M}$ and all integers $m\geq 2$. 
In particular, if $f_1,...,f_r$ are of degree $d$ then 
$$|E_{L,Z,\cI}^{(r)}(m)|\leq C_{n,d}m^{n-1}q_L^{-\tilde{\sigma}_{\cI,Z}m}$$
for all $L\in\tilde{\cL}_{K,M}$ and all integers $m\geq 1$. 
\end{thm}
\begin{proof}We observe that $0$ is the only critical value of $g(x,y)=\sum_{1\leq i\leq r}y_if_i(x)$. 
The first assertion follows from Definition \ref{defmoi0} and Theorem \ref{thmlowb}. The second assertion follows  from the first assertion and Corollary \ref{AvIgnrati}. The last assertion is trivial if $s(f_1,...,f_r)=0$. If $s(f_1,...,f_r)>0$ then $\cI_K$ defines a complete intersection in $\AA_K^n$ of dimension $n-r$.  Thus, the first assertion and Theorem \ref{bounalph} imply the last assertion. 
\end{proof}

\section{Examples and Questions}\label{EXQu}
In this section, we give some examples to illustrate the calculation in Section \ref{lobmoti}. We also give some further questions.
\begin{example}\label{exmpol}
Let $f\in\ZZ[x_1,...,x_n]$ be a homogeneous polynomial of degree $d\geq 2$. We set $Z=\AA_\ZZ^n$ and $I=\{1,...,n\}$. Suppose that $f$ defines a smooth hypersurface in $\PP_\QQ^{n-1}$. We recall the following notations in Subsection \ref{moipolylo}. We denote by $\cI_{f,I}$ the ideal of $\ZZ[x_1,...,x_n]$ generated by $f^2$ and $(\frac{\partial f}{\partial x_i})^3$ for $1\leq i\leq n$.  We define $$\cH_{f}(x)=\left(\dfrac{\partial^2 f}{\partial x_i\partial x_j}(x)\right)_{1\leq i,j\leq n}.$$
For each $0\leq \ell\leq n$, we denote by $\mathsf{D}_{f,I,\ell}$ the ideal of $\ZZ[x_1,...,x_n]$ generated by all principal minors of order $\ell$ of $\cH_{f}(x)$ with the convention that $\tilde{\mathsf{D}}_{f,I,\ell}=(1)$. We set  $$\tilde{\cI}_{f,I}=\cI_{f,I}\otimes_\ZZ\QQ, \tilde{\mathsf{D}}_{f,I,\ell}=\mathsf{D}_{f,\ell}\otimes_{\ZZ}\QQ, \cI_{f,I,p}=\cI_{f,I}\otimes_{\ZZ}\FF_p[[t]], \mathsf{D}_{f,I,\ell,p}=\mathsf{D}_{f,I,\ell}\otimes_{\ZZ}\FF_p[[t]],$$ $$\textnormal{and } \sigma_{f,Z,I}:=\min_{0\leq \ell\leq n}\left(2\lct_{Z_\QQ}\left(\tilde{\cI}_{f,I},\tilde{\mathsf{D}}_{f,I,\ell},\frac{1}{2}\right)-\frac{\ell}{6}\right).$$
Since $f$ defines a smooth hypersurface in $\PP_\QQ^{n-1}$, there is an integer $M$ such that $f_p:=(f\mod p)$ defines a smooth hypersurface in $\PP_{\FF_p}^n$ if $p>M$. Let $p>M$. If $x\in \FF_p[[t]]^n$ then we have $$\min_{1\leq i\leq n}\ord_t\left(\frac{\partial f_p}{\partial x_i}(x)\right)=(d-1)\ord_t(x):=(d-1)\min_{1\leq i\leq n}\ord_t(x_i)$$
and 
$$\ord_t(\mathsf{D}_{f,I,\ell,p}(x))=\ell(d-2)\ord_t(x)+\ord_t(\mathsf{D}_{f,I,\ell,p}(p^{-\ord_t(x)}x))\geq \ell(d-2)\ord_t(x)$$
for all $0\leq \ell\leq n$. 
Moreover, since $f_p$ defines  a smooth hypersurface in $\PP_{\FF_p}^n$, it is easy to see that   
$$\int_{\{x\in \FF_p[[t]]^n\mid \ord_t(x)=0,\hspace{0.1cm} \ord_t(f(x))\geq m\}}|dx|=\#\{x\in \FF_p^n\setminus \{0\}\mid f_p(x)=0\}p^{-n-m+1}$$
if $m\geq 1$. From these facts, if $d\geq 3$ then we have 
\begin{align*}
&\int_{\{x\in\FF_p[[t]]^n\mid \ord_t(\cI_{f,I,p})\geq m\}}p^{-\frac{1}{2}\ord_t(\mathsf{D}_{f,I,\ell,p}(x))}|dx|\nonumber\\=&\sum_{a\geq \frac{m}{3(d-1)}}\int_{\{x\in\FF_p[[t]]^n\mid 2\ord_t(f_p(x))\geq m,\hspace{0.1cm}\ord_t(x)=a\}}p^{-\frac{1}{2}\ord_t(\mathsf{D}_{f,I,\ell,p}(x))}|dx|\nonumber\\
=&\sum_{\frac{m}{2d}\geq a\geq \frac{m}{3(d-1)}}\int_{\{x\in\FF_p[[t]]^n\mid 2\ord_t(f_p(x))\geq m,\hspace{0.1cm}\ord_t(x)=a\}}p^{-\frac{1}{2}\ord_t(\mathsf{D}_{f,I,\ell,p}(x))}|dx|\nonumber\\
&+ \sum_{a>\frac{m}{2d}}\int_{\{x\in\FF_p[[t]]^n\mid 2\ord_t(f_p(x))\geq m,\hspace{0.1cm}\ord_t(x)=a\}}p^{-\frac{1}{2}\ord_t(\mathsf{D}_{f,I,\ell,p}(x))}|dx|\nonumber\\
=&\sum_{\frac{m}{2d}\geq a\geq \frac{m}{3(d-1)}}p^{-an-\frac{\ell(d-2)a}{2}}\int_{\{x\in\FF_p[[t]]^n\mid 2\ord_t(f_p(x))\geq m-2ad,\hspace{0.1cm}\ord_t(x)=0\}}p^{-\frac{1}{2}\ord_t(\mathsf{D}_{f,I,\ell,p}(x))}|dx|\nonumber\\
&+ \int_{\{x\in\FF_p[[t]]^n\mid \ord_t(x)>\frac{m}{2d}\}}p^{-\frac{1}{2}\ord_t(\mathsf{D}_{f,I,\ell,p}(x))}|dx|\nonumber\\
=&\sum_{\frac{m}{2d}>a\geq \frac{m}{3(d-1)}}p^{-an-\frac{\ell(d-2)a}{2}}\int_{\{x\in\FF_p[[t]]^n\mid 2\ord_t(f_p(x))\geq m-2ad,\hspace{0.1cm}\ord_t(x)=0\}}p^{-\frac{1}{2}\ord_t(\mathsf{D}_{f,I,\ell,p}(x))}|dx|\nonumber\\
&+ \11_{2d\ZZ}(m)p^{-\frac{mn}{2d}-\frac{\ell(d-2)m}{4d}}+ C_pp^{-n(\lfloor\frac{m}{2d}\rfloor+1)-\frac{\ell(d-2)}{2}(\lfloor\frac{m}{2d}\rfloor+1)}\nonumber\\
\end{align*}
\begin{align}\label{bouFFp}
&\hspace{-2.8cm}\leq\sum_{\frac{m}{2d}>a\geq \frac{m}{3(d-1)}}p^{-an-\frac{\ell(d-2)a}{2}}\#\{x\in \FF_p^n\setminus \{0\}\mid f_p(x)=0\}p^{-n-\lceil\frac{m}{2}\rceil+ad+1}\nonumber\\&\hspace{-2.4cm}+ \11_{2d\ZZ}(m)p^{-\frac{mn}{2d}-\frac{\ell(d-2)m}{4d}}+ C_pp^{-n(\lfloor\frac{m}{2d}\rfloor+1)-\frac{\ell(d-2)}{2}(\lfloor\frac{m}{2d}\rfloor+1)}
\end{align}
for all $0\leq \ell\leq n$ and a constant $C_p$ with $|C_p|\leq 1$. Moreover, if $\ell=0$ then the inequality in (\ref{bouFFp}) becomes an equality. Now, we can use (\ref{bouFFp}) and Lemma \ref{genlogc} to conclude that if $d\geq 3$ then
$$\lct_{Z_\QQ}\left(\tilde{\cI}_{f,I},\tilde{\mathsf{D}}_{f,I,\ell},\frac{1}{2}\right)\geq \begin{cases}\frac{n}{2d}+\frac{\ell(d-2)}{4d} \textnormal{ if } d\geq n\\
\frac{n-d}{3(d-1)}+\frac{1}{2}+\frac{\ell(d-2)}{6(d-1)} \textnormal{ if } n>d
\end{cases}$$
for all $0\leq \ell\leq n$ and the equality holds if $\ell=0$. Thus, if $d\geq 3$, then we have 
$$\sigma_{f,Z,I}=\begin{cases} \frac{n}{d} \textnormal{ if } n\leq d,\\
\frac{2n+d-3}{3(d-1)} \textnormal{ if } n>d.
\end{cases}$$
If $d=2$, we have $\mathsf{D}_{f,I,\ell,p}=(1)$ for all $0\leq \ell\leq n$. By a similar computation, we have 
$$\lct_{Z_\QQ}\left(\tilde{\cI}_{f,I},\tilde{\mathsf{D}}_{f,I,\ell},\frac{1}{2}\right)=\frac{n}{3}$$
for all $0\leq \ell\leq n$ if $d=2$.  Thus, we have $$\sigma_{f,Z,I}=\min_{0\leq\ell\leq n}\left(\frac{2n}{3}-\frac{\ell}{6}\right)=\frac{n}{2}$$
if $d=2$.

\end{example}

\begin{Qu}\label{qus}Let $k$ be a field of characteristic zero and $f\in k[x_1,...,x_n]$ be a polynomial of degree $d\geq 3$.
With the notations of Subsection \ref{moipolylo}, can we have 
$$\sigma_{f,\AA_k^n,I}\geq \min\biggl\{\frac{n-s(f)}{d},\frac{2(n-s(f))+d-3}{3(d-1)}\biggr\}$$
for $I=\{1,...,n\}$? More generally, we also ask whether this inequality holds for all non-empty subsets $I$ of $\{1,...,n\}$. 
\end{Qu}
Note that $\lct(f)=\min\{\tilde{\alpha}_f,1\}$ where	 $\tilde{\alpha}_f$ is the minimal exponent of $f$  mentioned in the introduction. 
So it follows from \cite[Proposition 2.4]{CluckerNguyen} that $\lct(f)\geq \min\bigl\{1,\frac{n-s(f)}{d}\bigr\}$. Thus,  we have a positive answer for the first part of Question \ref{qus}  if $n-s(f)\leq d$ by using Proposition \ref{compaCKM} and the fact that $\lct((f)+J_f^2)\geq \lct(f)$. 

Lastly, motivated by the strong monodromy conjecture and the work in \cite{CKMu}, we may expect that the number $\sigma_{f,\AA_k^n,I}$ in Question \ref{qus} is a lower bound of the minimal exponent  $\tilde{\alpha}_f$ of $f$. This is  the content of the following question.
\begin{Qu}
With the notations and assumptions of Question \ref{qus}, do we have
$$\sigma_{f,\AA_k^n,I}\leq \tilde{\alpha}_f$$
for all non-empty subsets $I$ of $\{1,...,n\}$?

\end{Qu}

\section*{Acknowledgements}The author is  supported by  Fund for Scientific Research-Flanders (Belgium) (F.W.O.) 1270923N. 




\bibliographystyle{amsplain}
\bibliography{anbib}
\end{document}